\documentclass{amsart}
\usepackage{etex}
\usepackage{fixltx2e}

\usepackage[usenames,dvipsnames]{xcolor}
\usepackage[bookmarksopen,bookmarksdepth=2]{hyperref}
\hypersetup{colorlinks=true,citecolor=NavyBlue,linkcolor=BrickRed,urlcolor=Green} 
\usepackage[mathscr]{eucal}
\usepackage{tikz}
\usepackage{tikz-cd}
\usepackage{tikzsymbols}
\usepackage{adjustbox}
\usepackage{enumitem}
\usepackage{dsfont}

\usepackage{microtype}

\renewcommand{\mathbb}{\mathbf}

\usepackage{amssymb,amsmath,amsfonts,amsthm,epsfig,amscd}
\usepackage{stmaryrd}
\usepackage[all,cmtip,poly]{xy}
\usepackage{color}

\newcommand{\Fpbartimes}{\overline{\F}^\times_p}
\newcommand{\hyp}{\operatorname{hyp}}
\newcommand{\Iw}{\operatorname{Iw}}
\newcommand{\gL}{\mathfrak{L}}

\newcommand{\teta}{\tilde{\eta}}

\newcommand{\tx}{\tilde{x}}

\newcommand{\irred}{\operatorname{irred}}
\newcommand{\red}{\operatorname{red}}

\newcommand{\fbar}{\overline{f}}
\newcommand{\Ann}{\operatorname{Ann}}
\newcommand{\Adist}{{A^{\operatorname{dist}}}}
\newcommand{\Akfree}{{A^{\operatorname{k-free}}}}
\newcommand{\Bdist}{{B^{\operatorname{dist}}}}
\newcommand{\Btwist}{{B^{\operatorname{twist}}}}
\newcommand{\Bkfree}{{B^{\operatorname{k-free}}}}
\newcommand{\Ckfree}{{C^{\operatorname{k-free}}}}
\newcommand{\gEdist}{{\widetilde{\gE}^{\operatorname{dist}}}}
\newcommand{\gEkfree}{{\widetilde{\gE}^{\operatorname{k-free}}}}
\newcommand{\dist}{{\operatorname{dist}}}
\newcommand{\kfree}{{{\operatorname{k-free}}}}

\newcommand{\To}{\longrightarrow}
\newcommand{\isoto}{\stackrel{\sim}{\To}}
\newcommand{\cotimes}{\, \widehat{\otimes}}
\newcommand{\pr}{\operatorname{pr}}

\newcommand{\Rep}{\operatorname{Rep}}

\newcommand{\WBT}{W^{\operatorname{BT}}}
\newcommand{\BT}{\operatorname{BT}}
\newcommand{\ocZ}{\overline{\cZ}}
\newcommand{\cZbar}{\overline{\cZ}}
\newcommand{\cTbar}{\overline{\cT}}
\newcommand{\Rbarinfty}{{\bar{R}_\infty}}

\newcommand{\Kbar}{\bar{K}}
\newcommand{\gP}{\mathfrak{P}}
\newcommand{\ghat}{\hat{g}}

\newcommand{\JH}{\operatorname{JH}}
 \newcommand{\sigmabar   }{\overline{\sigma}}   
\def\iso{\buildrel \sim \over \longrightarrow}

\newcommand{\xbar}{\bar{x}}
\newcommand{\ybar}{\bar{y}}

\newcommand{\id}{\operatorname{id}}

\newtheorem{thm}[subsubsection]{Theorem}
\newtheorem{lemma}[subsubsection]{Lemma}
\newtheorem{lem}[subsubsection]{Lemma}

\newtheorem{cor}[subsubsection]{Corollary}

\newtheorem{prop}[subsubsection]{Proposition}

\newtheorem{athm}[subsection]{Theorem}
\newtheorem{alemma}[subsection]{Lemma}
\newtheorem{alem}[subsection]{Lemma}
\newtheorem{adefn}[subsection]{Definition}

\newtheorem{aprop}[subsection]{Proposition}

\theoremstyle{definition}
\newtheorem{df}[subsubsection]{Definition}
\newtheorem{defn}[subsubsection]{Definition}
\newtheorem{adf}[subsection]{Definition}
\theoremstyle{remark}
\newtheorem{remark}[subsubsection]{Remark}
\newtheorem{rem}[subsubsection]{Remark}
\newtheorem{aremark}[subsection]{Remark}

\newtheorem{example}[subsubsection]{Example}

\newtheorem{assumption}[subsubsection]{Assumption}

\def\numequation{\addtocounter{subsubsection}{1}\begin{equation}}
\def\nummultline{\addtocounter{subsubsection}{1}\begin{multline}}
\def\anumequation{\addtocounter{subsection}{1}\begin{equation}}
\def\anummultline{\addtocounter{subsection}{1}\begin{multline}}
\renewcommand{\theequation}{\arabic{section}.\arabic{subsection}.\arabic{subsubsection}}

\newif\iffinalrun
\iffinalrun
\else
\fi

\iffinalrun
  \newcommand{\need}[1]{}
  \newcommand{\mar}[1]{}
\else
  \newcommand{\need}[1]{{\tiny *** #1}}
  \newcommand{\mar}[1]{\marginpar{\raggedright\tiny fixme #1}}
\fi

\newcommand{\F}{\FF}

\newcommand{\Q}{\QQ}

\newcommand{\Z}{\ZZ}

\newcommand{\m}{\frakm}
\newcommand{\p}{\frakp}

\newcommand{\FF}{{\mathbb F}}
\newcommand{\GG}{{\mathbb G}}

\newcommand{\QQ}{{\mathbb Q}}

\newcommand{\ZZ}{{\mathbb Z}}

\renewcommand{\bf}{\ensuremath{\mathbf{f}}}

\newcommand{\cC}{{\mathcal C}}
\newcommand{\cD}{{\mathcal D}}
\newcommand{\cE}{{\mathcal E}}
\newcommand{\cF}{{\mathcal F}}
\newcommand{\cG}{{\mathcal G}}

\newcommand{\cI}{{\mathcal I}}

\newcommand{\cL}{{\mathcal L}}

\newcommand{\cM}{{\mathcal M}}

\newcommand{\cO}{{\mathcal O}}
\newcommand{\cP}{{\mathcal P}}

\newcommand{\cR}{{\mathcal R}}
\newcommand{\cS}{{\mathcal S}}
\newcommand{\cT}{{\mathcal T}}
\newcommand{\cU}{{\mathcal U}}
\newcommand{\cV}{{\mathcal V}}
\newcommand{\cW}{{\mathcal W}}
\newcommand{\cX}{{\mathcal X}}
\newcommand{\cY}{{\mathcal Y}}
\newcommand{\cZ}{{\mathcal Z}}

\newcommand{\frakm}{\mathfrak{m}}

\newcommand{\frakp}{\mathfrak{p}}

\newcommand{\gE}{\mathfrak{E}}
\newcommand{\gM}{\mathfrak{M}}
\newcommand{\gN}{\mathfrak{N}}
\newcommand{\gS}{\mathfrak{S}}

\newcommand{\tB}{\mathrm{B}}

\newcommand{\Fbar}{\overline{\F}}
\newcommand{\Qbar}{\overline{\Q}}
\newcommand{\Zbar}{\overline{\Z}}

\newcommand{\Fp}{\F_p}

\newcommand{\Fpbar}{\Fbar_p}

\newcommand{\Fpbarx}{\Fpbar^{\times}}

\newcommand{\Zp}{\Z_p}

\newcommand{\Zpbar}{\Zbar_p}

\newcommand{\Qp}{\Q_p}

\newcommand{\Qpbar}{\Qbar_p}

\DeclareMathOperator{\Aut}{Aut}
\DeclareMathOperator{\coker}{coker}

\DeclareMathOperator{\End}{End}
\DeclareMathOperator{\Ext}{Ext}
\DeclareMathOperator{\kExt}{ker-Ext}
\DeclareMathOperator{\Fil}{Fil}
\DeclareMathOperator{\Gal}{Gal}
\DeclareMathOperator{\GL}{GL}

\DeclareMathOperator{\Hom}{Hom}
\DeclareMathOperator{\Tor}{Tor}
\DeclareMathOperator{\im}{im}
\DeclareMathOperator{\Ind}{Ind}

\DeclareMathOperator{\Spec}{Spec}

\DeclareMathOperator{\Spf}{Spf}

\DeclareMathOperator{\Sym}{Sym}

\DeclareMathOperator{\val}{val}

\newcommand{\ab}{\mathrm{ab}}
\newcommand{\cris}{\mathrm{cris}}

\newcommand{\un}{\mathrm{un}}
\newcommand{\ur}{\mathrm{ur}}

\newcommand{\Bcris}{\tB_{\cris}}

\newcommand{\Dpcris}{\operatorname{D_{pcris}}}

\newcommand{\rhobar}{\overline{\rho}}

\newcommand*{\invlim}{\varprojlim} 
\newcommand{\into}{\hookrightarrow}
\newcommand{\onto}{\twoheadrightarrow}
\newcommand{\toisom}{\buildrel\sim\over\to}

\newcommand{\longnto}[1]{\buildrel#1\over\longrightarrow}

\newcommand{\Ga}{\GG_a}
\newcommand{\Gm}{\GG_m}

\newcommand{\hchar}{h} 

\newcommand{\dd}{\mathrm{dd}}
\newcommand{\loc}{\mathrm{loc}}

\newcommand{\Art}{{\operatorname{Art}}}
\newcommand{\col}{\colon}

\newcommand{\varepsilonbar}{\overline{\varepsilon}}
\newcommand{\rbar}{\overline{r}}
\newcommand{\Res}{\operatorname{Res}}
\newcommand{\K}[1]{\mathcal{K}(#1)}

\newcommand{\scrC}{\mathscr{C}}
\newcommand{\Czero}{\scrC^0}
\newcommand{\Czerofrac}{\Czero_{u}}
\newcommand{\Cone}{\scrC^1}
\newcommand{\Conefrac}{\Cone_{u}}
\newcommand{\Hzero}{\mathscr{H}}
\newcommand{\mumap}{\partial}

\newcommand{\tgP}{\widetilde{\gP}}
\newcommand{\czero}{c}

\newcommand{\overphi}{/_{\hskip -2pt \phi \hskip 2pt}}

\begin{document}
\title[Moduli stacks of two-dimensional Galois representations]{Moduli stacks of two-dimensional Galois
representations}

\author[A. Caraiani]{Ana Caraiani}\email{a.caraiani@imperial.ac.uk}
\address{Department of Mathematics, Imperial College London, London SW7 2AZ, UK}

\author[M. Emerton]{Matthew Emerton}\email{emerton@math.uchicago.edu}
\address{Department of Mathematics, University of Chicago,
5734 S.\ University Ave., Chicago, IL 60637, USA}

\author[T. Gee]{Toby Gee} \email{toby.gee@imperial.ac.uk} \address{Department of
  Mathematics, Imperial College London,
  London SW7 2AZ, UK}

\author[D. Savitt]{David Savitt} \email{savitt@math.jhu.edu}
\address{Department of Mathematics, Johns Hopkins University}
\thanks{The first author was supported in part by NSF grant DMS-1501064, 
  by a Royal Society University Research Fellowship,
  and by ERC Starting Grant 804176. The second author was supported in part by the
  NSF grants DMS-1303450 and DMS-1601871. The third author was 
  supported in part by a Leverhulme Prize, EPSRC grant EP/L025485/1, Marie Curie Career
  Integration Grant 303605, 
  ERC Starting Grant 306326, and a Royal Society Wolfson Research
  Merit Award. The fourth author was supported in part by NSF CAREER
  grant DMS-1564367 and  NSF grant 
  DMS-1702161.}

\maketitle
\begin{abstract}
 We construct moduli stacks of two-dimensional mod~$p$ representations of the
 absolute Galois group of a $p$-adic local field, 
 and relate their geometry to
 the weight part of Serre's conjecture for~$\GL_2$.
\end{abstract}
\setcounter{tocdepth}{1}
\tableofcontents

\section{Introduction}\subsection{Moduli of Galois
  representations}Let~$K/\Qp$ be a finite extension,
let~$\overline{K}$ be an algebraic closure of~$K$, and let
$\rbar:\Gal(\overline{K}/K)\to\GL_d(\Fpbar)$ be a continuous representation. 
The
theory of deformations of $\rbar$ --- that is, liftings of $\rbar$ to
continuous representations $r:\Gal(\overline{K}/K)\to\GL_d(A)$, where $A$ is
a complete local ring with residue field $\Fpbar$ --- is extremely
important in the Langlands program, and in particular is crucial for
proving automorphy lifting theorems via the Taylor--Wiles method.
Proving such theorems often comes down to studying the moduli spaces
of those deformations which satisfy various $p$-adic Hodge-theoretic
conditions. 

From the point of view of algebraic geometry, it seems unnatural to
study only formal deformations of this kind, and Kisin observed about
fifteen years ago that results on the reduction modulo $p$ of
two-dimensional crystalline representations suggested that there
should be moduli spaces of $p$-adic representations in which the
residual representations $\rbar$ should be allowed to vary. In
particular, the special fibres of these moduli spaces would be 
would be moduli
spaces of mod~$p$ representations of 
$\Gal(\Kbar/K)$. 

In this paper we construct such a space (or rather stack) $\cZ$ of mod $p$ representations in the
case~$d=2$, and describe its geometry. In particular, we show that their
irreducible components are naturally labelled by Serre weights, and
that our spaces give a geometrisation of the weight part of
Serre's conjecture.  More precisely, we prove the following
theorem (see Proposition~\ref{prop: dimensions of the Z stacks} and
Theorem~\ref{thm:stack version of geometric Breuil--Mezard}; we
explain the definition of a Serre weight, and see Section~\ref{subsec:
  intro serre weights} below for the notion of a Serre weight
associated to a Galois representation).
\begin{thm}
  \label{thm:intro thm on Z}The stack $\cZ$ is an algebraic stack of
  finite type over~$\Fpbar$, and
  is equidimensional of
  dimension~$[K:\Qp]$.  The irreducible components of~$\cZ$ are
  labelled by the Serre weights~$\sigmabar$, in such a way that the
  $\Fpbar$-points 
of the component $\cZ(\sigmabar)$ labelled by~$\sigmabar$ are
  precisely the representations $\rbar:G_K\to\GL_2(\Fpbar)$ having
  $\sigmabar$ as a Serre weight. 
\end{thm}
We also show that generic points of the irreducible
components admit a simple description (they are extensions of
characters whose restrictions to inertia are determined by the
corresponding Serre weight).

In the course of proving Theorem~\ref{thm:intro thm on Z}, we study a partial resolution of the
moduli spaces inspired by a construction of Kisin~\cite{kis04} in the
setting of formal deformations, and show that its irreducible
components are also naturally labelled by Serre weights. 
We use this resolution to show that our moduli spaces are generically
reduced, and as an illustration of the utility of our constructions,
we use this to prove the corresponding result for the special fibres
of tamely potentially Barsotti--Tate deformation rings (see
Proposition~\ref{prop: generically reduced special fibre deformation
  ring}). It seems hard to prove this result purely in the setting of
formal deformations, and we anticipate that it will
have applications to the theory of mod~$p$ Hilbert modular forms.

\subsection{The construction}The reason that we restrict to the case
of two-dimensional representations is that in this case one knows that
most mod~$p$ representations are ``tamely potentially finite flat'';
that is, after restriction to a finite tamely ramified extension, they
come from the generic fibres of finite flat group schemes.  Indeed,
the only representations not of this form are the so-called tr\`es
ramifi\'ee representations, which are twists of extensions of the
trivial character by the mod~$p$ cyclotomic character, and can be
described explicitly in terms of Kummer theory. (This is a local
Galois-theoretic analogue of the well-known fact that, up to twist,
modular forms of any weight and level $\Gamma_1(N)$, with $N$ prime to
$p$, are congruent modulo~$p$ to modular forms of weight two and level
$\Gamma_1(Np)$; the corresponding modular curves acquire semistable
reduction over a tamely ramified extension of~$\Qp$.)

These Galois representations, and the corresponding finite flat group
schemes, can be described in terms of semilinear algebra data. Such
descriptions also exist for more general $p$-adic Hodge theoretic
conditions (such as being crystalline of prescribed Hodge--Tate
weights), although they are more complicated,
and can be used to construct 
analogues,
for higher dimensional representations,
of the moduli stacks we construct here;
this construction is the subject of 
the forthcoming
paper~\cite{EGmoduli}.

The semilinear algebra data that we use in the present paper are Breuil--Kisin modules
and \'etale $\varphi$-modules. A Breuil--Kisin module is a module with
Frobenius over a power series ring, satisfying a condition on the
cokernel of the Frobenius which depends on a fixed integer, called the
\emph{height} of the Breuil--Kisin module. Inverting the formal variable in the power
series ring gives a functor from the category of Breuil--Kisin modules to the
category of 
\'etale $\varphi$-modules. By Fontaine's
theory~\cite{MR1106901}, these \'etale $\varphi$-modules correspond to
representations of~$\Gal(\overline{K}/K_\infty)$, where~$K_\infty$ is
an infinite non-Galois extension of~$K$ obtained by
extracting $p$-power roots of a uniformiser. By work of Breuil and
Kisin (in particular~\cite{kis04}), for \'etale
$\varphi$-modules that arise from a Breuil--Kisin module of height at most~$1$ 
the corresponding
representations admit a natural extension to~$\Gal(\overline{K}/K)$,
and in this way one obtains precisely the finite flat
representations. 
This is the case that we will consider throughout this paper, extended
slightly to
incorporate descent data from a finite tamely ramified
extension~$K'/K$ and thereby allowing us to study tamely potentially finite
flat representations.

Following Pappas and Rapoport~\cite{MR2562795}, we then consider the moduli stack~$\cC$ of rank two
projective Breuil--Kisin modules, and the moduli stack~$\cR$ of \'etale
$\varphi$-modules, together with the natural map~$\cC\to\cR$. We
deduce from the results of~\cite{MR2562795} that the stack~$\cC$ is algebraic (that is, it
is an Artin stack); however~$\cR$ is not algebraic, and indeed is infinite-dimensional. (In fact, we consider
versions of these stacks with $p$-adic coefficients, in which
case~$\cC$ is a $p$-adic formal algebraic stack, but we suppress
this for the purpose of this introduction.) The analogous construction
without tame descent data was considered
in~\cite{EGstacktheoreticimages}, where it was shown that one can
define a notion of the ``scheme-theoretic image'' of the morphism $\cC\to\cR$,
and that the scheme-theoretic image is algebraic. Using similar
arguments, we define our moduli stack~$\cZ$ of two-dimensional Galois
representations to be the scheme-theoretic image of the morphism~$\cC\to\cR$.

By construction, we know that the closed points of~$\cZ$ are in
bijection with the (non-tr\`es ramifi\'ee) representations
$\Gal(\overline{K}/K)\to\GL_2(\Fpbar)$, and by using standard results on
the corresponding formal deformation problems, we know that~$\cZ$ is
equidimensional of dimension~$[K:\Qp]$. The closed points of~$\cC$
correspond to potentially finite flat models of these Galois
representations, and we are able to deduce that~$\cC$ is also
equidimensional of dimension~$[K:\Qp]$ (at least morally, this is by
Tate's theorem on the uniqueness of prolongations of $p$-divisible
groups). 

These constructions
are relatively formal. To go further, we combine results from the
theory of local models of Shimura varieties and Taylor--Wiles patching
with an explicit construction of families of extensions of
characters. We begin by describing the last of these.

Intuitively, a natural source of ``families'' of representations
$\rhobar:\Gal(\overline{K}/K)\to\GL_2(\Fpbar)$ is given by the
extensions of two fixed characters. Indeed, given two
characters~$\chi_1,\chi_2:\Gal(\overline{K}/K)\to\Fpbartimes$, the
$\Fpbar$-vector space $\Ext^1_{\Gal(\overline{K}/K)}(\chi_2,\chi_1)$
is usually $[K:\Qp]$-dimensional, and a back of the envelope
calculation 
suggests that this should give a~$([K:\Qp]-2)$-dimensional
substack of~$\cZ$ (the difference between an extension and a
representation counts for a $-1$, as does the~$\Gm$ of endomorphisms). Twisting~$\chi_1,\chi_2$ independently
by unramified characters gives a candidate for a $[K:\Qp]$-dimensional
family; since~$\cZ$ is equidimensional of dimension~$[K:\Qp]$, the
closure of such a family should be an irreducible component of~$\cZ$.

Since there are only finitely many possibilities for the restrictions
of the~$\chi_i$ to the inertia subgroup~$I(\overline{K}/K)$, this gives a finite list of
maximal-dimensional families. On the other hand, there are up to
unramified twist only finitely many irreducible two-dimensional
representations of~$\Gal(\overline{K}/K)$, which suggests that the
irreducible representations should correspond to $0$-dimensional
substacks. Together these considerations suggest that the irreducible
components of our moduli stack should be given by the closures of the
families of extensions considered in the previous paragraph, and in
particular that the irreducible representations should arise as limits
of reducible representations. This could not literally be the case
for families of Galois representations, rather than families of
\'etale $\varphi$-modules, and may seem surprising at first glance,
but it is indeed what happens.

\subsection{Serre weights}\label{subsec: intro serre weights}
In the body of the paper we make this analysis 
rigorous, and we show
that the different families that we have constructed exhaust the
irreducible components. 
We can therefore label the irreducible components of~$\cZ$
as follows. Let~$k$ be the residue field of~$K$;\ a \emph{Serre
  weight} is then an irreducible $\Fpbar$-representation of $\GL_2(k)$ (or
rather an isomorphism class thereof). Such a representation is
specified by its highest weight, which can be thought of as a pair of
characters~$k^\times\to\Fpbartimes$, which via local class field
theory corresponds to a pair of
characters~$I(\overline{K}/K)\to\Fpbar^\times$, and thus to an
irreducible component of~$\cZ$ (in fact, we need to make a
shift in this dictionary, corresponding to half the sum of the
positive roots of~$\GL_2(k)$, but we ignore this for the purposes of
this introduction).

This might seem artificial, but in fact it is completely natural, for
the following reason. Following the pioneering work of
Serre~\cite{MR885783} and Buzzard--Diamond--Jarvis~\cite{bdj} (as
extended in~\cite{MR2430440} and~\cite{gee061}), we now know how to
associate a set $W(\rbar)$ of Serre weights to each continuous
representation $\rbar:G_K\to\GL_2(\Fpbar)$, with the property
that if $F$ is a totally real field and $\rhobar:G_F\to\GL_2(\Fpbar)$
is an irreducible representation coming from a Hilbert modular form,
then the possible weights of Hilbert modular forms giving rise
to~$\rhobar$ are precisely determined by the sets
$W(\rhobar|_{G_{F_v}})$ for places $v|p$ of~$F$ (see for example
\cite{blggu2,geekisin,gls13}).

Going back to our labelling of irreducible components
above, we have associated a Serre weight~$\sigmabar$ to each
irreducible component of~$\cZ$. 
One of our main theorems is that the representations~$\rbar$ on the
irreducible component labelled by~$\sigmabar$ are precisely the
representations with $\sigmabar\in W(\rbar)$,

We emphasise that the existence of such a geometric interpretation of
the sets~$W(\rbar)$ is far from
obvious, and indeed we know of no direct proof using any of the
explicit descriptions of~$W(\rbar)$ in the literature; it seems hard to understand in any explicit way which Galois
representations arise as the limits of a family of extensions of given
characters, and the description of the sets~$W(\rbar)$ is
very complicated (for example, the description
in~\cite{bdj} relies on certain $\Ext$ groups of crystalline
characters). Our proof is indirect, and ultimately makes use of a
description of~$W(\rbar)$ given in~\cite{geekisin}, which is in terms
of potentially Barsotti--Tate deformation rings of~$\rbar$ and is
motivated by the Taylor--Wiles method. We interpret this description
in the geometric language of~\cite{emertongeerefinedBM}, which we 
in turn interpret as the formal completion of a ``geometric
Breuil--M\'ezard conjecture'' for our stacks. 

We also study the irreducible components of the stack~$\cC$. This
stack admits a decomposition as a disjoint union of
substacks~$\cC^\tau$, indexed by the tame inertial types~$\tau$ (the
substack~$\cC^\tau$ is the moduli of those Breuil--Kisin modules which have
descent data given by~$\tau$). The inertial local Langlands
correspondence assigns a finite set of Serre weights
$\JH(\sigmabar(\tau))$ to~$\tau$ (the Jordan--H\"older factors of the
reduction mod~$p$ of the representation~$\sigma(\tau)$ of~$\GL_2(\cO_K)$
corresponding to~$\tau$), and we show that the scheme-theoretic image
of the morphism $\cC^\tau\to\cZ$ is
$\cZ^\tau=\cup_{\sigmabar\in\JH(\sigmabar(\tau))}\cZ(\sigmabar)$. 

The set~$\JH(\sigmabar(\tau))$ can be identified with a subset of the
power set~$\cS$ of the set of embeddings~$k\into\Fpbar$. For generic
choices of~$\tau$, it is equal to~$\cS$, and in this case we show that
the morphism~$\cC^\tau\to\cZ^\tau$ is a generic isomorphism on the
source. 
We are able to show (using the theory of Dieudonn\'e modules) that for
any non-scalar type~$\tau$, the irreducible components of~$\cC^\tau$
can be identified with~$\cS$, and those irreducible components not
corresponding to elements of~$\JH(\sigmabar(\tau))$ have image
in~$\cZ^\tau$ of positive codimension. (In the case of scalar types,
both~$\cC^\tau$ and~$\cZ^\tau$ are irreducible.) It follows from the
results described that~$\cZ^\tau$ is generically reduced, which
is not at all obvious from its definition.

An important tool in our proofs is that~$\cC$ has rather mild
singularities, and in particular is Cohen--Macaulay and
reduced. 
We show this by relating the singularities of the various~$\cC^\tau$
to the local
models at Iwahori level of Shimura varieties of $\GL_2$-type; such a
relationship was first found in~\cite{kis04} (in the context of formal
deformation spaces, with no descent data) and~\cite{MR2562795} (in the
context of the stacks~$\cC$, although again
without descent data) and developed further
by the first author and Levin in~\cite{2015arXiv151007503C}. 

\subsection{An outline of the paper}
In
Section~\ref{sec:kisin modules with descent data} we recall the theory
of Breuil--Kisin modules and \'etale $\varphi$-modules, and explain how it
extends to the setting of tame descent data. In
Section~\ref{sec:moduli stacks} we define the stacks~$\cC$, $\cR$
and~$\cZ$, and prove some of their basic properties
following~\cite{EGstacktheoreticimages}. We relate the singularities
of~$\cC$ to those of local
models, define the Dieudonn\'e stack, and explain how the morphism
from~$\cC$ to the Dieudonn\'e stack can be thought of in terms of
effective Cartier divisors.

In Section~\ref{sec: extensions of rank one Kisin modules} we build
our families of reducible Galois representations, and show that they
are dense in~$\cZ$. We begin with a thorough study of spaces of
extensions of Breuil--Kisin modules, before considering their scheme-theoretic
images in~$\cR$. After some general considerations we specialise to
the case of extensions of rank one Breuil--Kisin modules, where we explicitly
calculate the dimensions of the extension groups. We also show that
the Kisin variety corresponding to an irreducible Galois
representation has ``small'' dimension, by using a base change
argument, and proving an upper bound on the Kisin variety for
reducible representations via an explicit calculation.

In Section~\ref{sec: picture} we prove our main results, by
combining the hands-on study of~$\cC$ and~$\cZ$ of Section~\ref{sec:
  extensions of rank one Kisin modules} with the results on the weight
part of Serre's conjecture and the Breuil--M\'ezard conjecture
from~\cite{geekisin}. 

We finish with several appendices, summarising results that we use
earlier in the paper. Appendix~\ref{sec:formal stacks} recalls some
properties of formal algebraic stacks from~\cite{Emertonformalstacks},
and proves a technical result that we use in Section~\ref{sec:
  picture}. 
Appendix~\ref{sec: appendix on tame
  types} recalls some standard facts about Serre weights and the
inertial local Langlands correspondence, and finally
Appendix~\ref{sec:appendix on geom BM} combines the results
of~\cite{geekisin} and~\cite{emertongeerefinedBM} to prove a geometric
Breuil--M\'ezard result for tamely potentially Barsotti--Tate
deformation rings, which we use in Section~\ref{sec:
  picture}.

\subsection{Final comments}
As explained above, our construction excludes the tr\`es ramifi\'ee
representations, which are twists of certain extensions of the trivial
character by the mod~$p$ cyclotomic character. From the point of view
of the weight part of Serre's conjecture, they are precisely the
representations which admit a twist of the Steinberg representation as
their only Serre weight. In accordance with the picture described
above, this means that the full moduli stack of 2-dimensional
representations of~$\Gal(\overline{K}/K)$ can be obtained from our
stack by adding in the irreducible components consisting of the tr\`es
ramifi\'ee representations. This is carried out in~\cite{EGmoduli},
and the geometrisation of the weight part of Serre's conjecture
described above is extended to this moduli stack, using the results of
this paper as an input.

We assume that~$p>2$ in much of the paper; while we expect that our
results should also hold if~$p=2$, there are several reasons to
exclude this case. We are frequently able to considerably simplify our
arguments by assuming that the extension~$K'/K$ is not just tamely
ramified, but in fact of degree prime to~$p$; this is problematic
when~$p=2$, as the consideration of cuspidal types involves a
quadratic unramified extension. We also use results on the
Breuil--M\'ezard conjecture which ultimately depend on automorphy
lifting theorems that are not available in the case $p=2$ at present
(although it is plausible that the methods of~\cite{Thornep=2} could
be used to prove them). 

\subsection{Acknowledgements}We would like to thank Ulrich G\"ortz,
Wansu Kim and Brandon Levin
for helpful conversations and correspondence.
\subsection{Notation and conventions}\label{subsec: notation}

\subsubsection*{Topological groups} If~$M$ is an abelian
topological group with a linear topology, then as
in~\cite[\href{https://stacks.math.columbia.edu/tag/07E7}{Tag
  07E7}]{stacks-project} we say that~$M$ is {\em complete} if the
natural morphism $M\to \varinjlim_i M/U_i$ is an isomorphism,
where~$\{U_i\}_{i \in I}$ is some (equivalently any) fundamental
system of neighbourhoods of~$0$ consisting of subgroups. Note that in
some other references this would be referred to as being~{\em complete
  and separated}. In particular, any $p$-adically complete ring~$A$ is
by definition $p$-adically separated.

\subsubsection*{Galois theory and local class field theory} If $M$ is a field, we let $G_M$ denote its
absolute Galois group.
If~$M$ is a global field and $v$ is a place of $M$, let $M_v$ denote
the completion of $M$ at $v$. If~$M$ is a local field, we write~$I_M$
for the inertia subgroup of~$G_M$. 

 Let $p$ be a prime number. 
 Fix a finite extension $K/\Qp$, with
 ring of integers $\cO_K$ and residue field $k$.  Let $e$ and $f$
 be the  ramification and inertial degrees of $K$, respectively, and
 write $\# k=p^f$ for the cardinality of~$k$.   
Let $K'/K$ be a finite
tamely ramified Galois extension. Let $k'$ be the residue field of $K'$, and let $e',f'$ be the
ramification and inertial degrees of $K'$ respectively.

Our representations of $G_K$ will have coefficients in $\Qpbar$,
a fixed algebraic closure of $\Qp$ whose residue field we denote by~$\Fpbar$. Let $E$ be a finite
extension of $\Qp$ contained in $\Qpbar$ and containing the image of every
embedding of $K'$ into $\Qpbar$. Let $\cO$ be the ring of integers in
$E$, with uniformiser $\varpi$ and residue field $\F \subset
\Fpbar$.  

Fix an embedding $\sigma_0:k'\into\F$, and recursively define
$\sigma_i:k'\into\F$ for all $i\in\Z$ so that
$\sigma_{i+1}^p=\sigma_i$; of course, we have $\sigma_{i+f'}=\sigma_i$
for all~$i$. We let $e_i\in k'\otimes_{\Fp} \F$ denote the idempotent
satisfying $(x\otimes 1)e_i=(1\otimes\sigma_i(x))e_i$ for all $x\in
k'$; note that $\varphi(e_i)=e_{i+1}$. We also denote by $e_i$ the
natural lift of $e_i$ to an idempotent in
$W(k')\otimes_{\Zp}\cO$. If $M$ is an
$W(k')\otimes_{\Zp}\cO$-module, then we write $M_i$ for
$e_iM$.

We write $\Art_K \col K^\times\to W_K^{\ab}$ for
the isomorphism of local class field theory, normalised so that
uniformisers correspond to geometric Frobenius elements. 

\begin{lemma}\label{lem:cft} Let  $\pi$ be any uniformiser
of $\cO_K$. 
The composite $I_K \to \cO_K^{\times} \to k^{\times}$, where  the map
$I_K \to \cO_K^\times$ is induced by the restriction of $\Art_K^{-1}$,
sends an element $g \in I_K$ to the image in $k^{\times}$ of
$g(\pi^{1/(p^f-1)})/\pi^{1/(p^f-1)}$. 
\end{lemma}

\begin{proof}
This follows (for example) from the construction in \cite[Prop.~4.4(iii), Prop.~4.7(ii), Cor.~4.9, Def.~4.10]{MR2487860}.
\end{proof}

 For each $\sigma\in \Hom(k,\Fpbar)$ we
define the fundamental character $\omega_{\sigma}$   to~$\sigma$ to be the composite \[\xymatrix{I_K \ar[r] & \cO_{K}^{\times}\ar[r] & k^{\times}\ar[r]^{\sigma} & \Fpbarx,}\]
where the map $I_K \to \cO_K^\times$ is induced by the restriction of $\Art_K^{-1}$.
Let $\varepsilon$ denote the $p$-adic cyclotomic
character and $\varepsilonbar$ the mod~$p$ cyclotomic
character, so that $\prod_{\sigma \in \Hom(k,\Fpbar)}
\omega_{\sigma}^{e} = \varepsilonbar$.   We will often identify
characters $I_K \to \Fpbar^{\times}$ with characters $k^{\times}
\to \Fpbar^{\times}$ via the Artin map, and similarly for
their Teichm\"uller lifts.

\subsubsection*{Inertial local Langlands} A two-dimensional \emph{tame inertial type} is (the isomorphism
class of) a tamely ramified representation
$\tau : I_K \to \GL_2(\Zpbar)$ that extends to a representation of $G_K$ and
whose kernel is open. Such a representation is of the form $\tau
\simeq \eta \oplus \eta'$, and we say that $\tau$ is a \emph{tame principal series
  type} if 
$\eta,\eta'$ both extend to characters of $G_K$. Otherwise,
$\eta'=\eta^q$, and ~$\eta$ extends to a character
of~$G_L$, where~$L/K$ is a quadratic unramified extension. 
In this case we say
that~$\tau$ is a \emph{tame cuspidal type}.

Henniart's appendix to \cite{breuil-mezard}
associates a finite dimensional irreducible $E$-representation $\sigma(\tau)$ of
$\GL_2(\cO_K)$ to each inertial type $\tau$; we refer to this association as the {\em
  inertial local Langlands correspondence}. Since we are only working
with tame inertial types, this correspondence can be made very
explicit as follows. 

If $\tau
\simeq \eta \oplus \eta'$ is a tame principal series type, then we
also write $\eta,\eta':k^\times\to\cO^\times$ for the 
multiplicative characters determined by
$\eta\circ\Art_K|_{\cO_{K}^\times},\eta'\circ\Art_K|_{\cO_{K}^\times}$
respectively. If $\eta=\eta'$, then we set
$\sigma(\tau)=\eta\circ\det$. Otherwise, we write $I$ for the Iwahori subgroup of $\GL_2(\cO_K)$ consisting of
matrices which are upper triangular modulo a uniformiser~$\varpi_K$
of~$K$, and write $\chi = \eta'\otimes \eta:
I\to\cO^\times$ for the character \[
\begin{pmatrix}
  a&b\\\varpi_K c&d
\end{pmatrix}\mapsto \eta'(\overline{a})\eta(\overline{d}).\] Then $\sigma(\tau) := \Ind_I^{\GL_2(\cO_K)}
\chi$. 

If $\tau=\eta\oplus\eta^q$ is a tame cuspidal type, then as above we
write~$L/K$ for a quadratic unramified extension, and~$l$ for the
residue field of~$\cO_L$. We write
$\eta :l^\times\to\cO^\times$ for the 
multiplicative character determined by
$\eta\circ\Art_L|_{\cO_{L}^\times}$; then $\sigma(\tau)$ is the
inflation to $\GL_2(\cO_K)$ of the cuspidal representation of $\GL_2(k)$
denoted by~$\Theta(\eta)$ in~\cite{MR2392355}.

\subsubsection*{$p$-adic Hodge theory} We normalise Hodge--Tate weights so that all Hodge--Tate weights of
the cyclotomic character are equal to $-1$. We say that a potentially
crystalline representation $\rho:G_K\to\GL_2(\Qpbar)$ has \emph{Hodge
  type} $0$, or is \emph{potentially Barsotti--Tate}, if for each
$\varsigma :K\into \Qpbar$, the Hodge--Tate weights of $\rho$ with
respect to $\varsigma$ are $0$ and $1$.  (Note that this is a more
restrictive definition of potentially Barsotti--Tate than is sometimes
used; however, we will have no reason to deal with representations
with non-regular Hodge-Tate weights, and so we exclude them from
consideration. Note also that it is more usual in the literature to
say that $\rho$ is potentially Barsotti--Tate if it is potentially
crystalline, and $\rho^\vee$ has Hodge type $0$.) 

We say
that a potentially crystalline representation
$\rho:G_K\to\GL_2(\Qpbar)$ has  \emph{inertial type} $\tau$ if the traces of
elements of $I_K$ acting on~$\tau$ and on
\[\Dpcris(\rho)=\varinjlim_{K'/K}(\Bcris\otimes_{\Qp}V_\rho)^{G_{K'}}\] are
equal (here~$V_\rho$ is the underlying vector space
of~$V_\rho$). 
A representation $\rbar:G_K\to\GL_2(\Fpbar)$ \emph{has a potentially
    Barsotti--Tate lift of
    type~$\tau$} if and
  only if $\rbar$ admits a lift to a representation
  $r:G_K\to\GL_2(\Zpbar)$ of Hodge type~$0$ and inertial type~$\tau$.
\subsubsection*{Serre weights}
By definition, a \emph{Serre weight} is an irreducible
$\F$-representation of $\GL_2(k)$. Concretely, such a
representation is of the form
\[\sigmabar_{\vec{t},\vec{s}}:=\otimes^{f-1}_{j=0}
(\det{\!}^{t_j}\Sym^{s_j}k^2) \otimes_{k,\sigma_{j}} \F,\]
where $0\le s_j,t_j\le p-1$ and not all $t_j$ are equal to
$p-1$. We say that a Serre weight is \emph{Steinberg} if $s_j=p-1$ for all $j$,
and \emph{non-Steinberg} otherwise.

\subsubsection*{A remark on normalisations}  Given a continuous representation $\rbar:G_K\to\GL_2(\Fpbar)$, there
is an associated (nonempty) set of Serre weights~$W(\rbar)$ whose
precise definition we will recall in Appendix~\ref{sec: appendix on tame types}. There are in fact
several different definitions of~$W(\rbar)$ in the literature; as a
result of the papers~\cite{blggu2,geekisin,gls13}, these definitions
are known to be equivalent up to normalisation. 

However, the normalisations of
Hodge--Tate weights and of inertial local Langlands used in
\cite{geekisin,gls13,emertongeesavitt} are not all the same, and so
for clarity we lay out how they differ, and how they compare to the normalisations of
this paper. 

Our conventions for Hodge--Tate weights and
inertial types agree with those of~\cite{geekisin, emertongeesavitt}, but our
representation~$\sigma(\tau)$ is the
representation~$\sigma(\tau^\vee)$ of~\cite{geekisin, emertongeesavitt}
(where~$\tau^\vee=\eta^{-1}\oplus(\eta')^{-1}$);\ to see this, note the
dual in the definition of~$\sigma(\tau)$ in~\cite[Thm.\
2.1.3]{geekisin} and the discussion in \S 1.9 of
\cite{emertongeesavitt}.\footnote{However, this dual is erroneously
  omitted when the inertial local Langlands correspondence is made
  explicit  at the end of  \cite[\S3.1]{emertongeesavitt}. See
  Remark~\ref{arem: wtf were we thinking in EGS}.}

In all cases one chooses to normalise the set of Serre weights so
that the condition of Lemma~\ref{lem: list of things we need to know about Serre
  weights}(1) holds.  Consequently, our set of weights~$W(\rbar)$ is the
set of duals of the weights~$W(\rbar)$ considered
in~\cite{geekisin}. In turn, the paper~\cite{gls13} has the opposite
convention for the signs of Hodge--Tate weights to our convention (and
to the convention of~\cite{geekisin}), so we find that our set of
weights~$W(\rbar)$ is the set of duals of the weights~$W(\rbar^\vee)$
considered in~\cite{gls13}.

\subsubsection*{Stacks}We follow the terminology of~\cite{stacks-project}; in
particular, we write ``algebraic stack'' rather than ``Artin stack''. More
precisely, an algebraic stack is a stack in groupoids in the \emph{fppf} topology,
whose diagonal is representable by algebraic spaces, which admits a smooth
surjection from a
scheme. See~\cite[\href{http://stacks.math.columbia.edu/tag/026N}{Tag
  026N}]{stacks-project} for a discussion of how this definition relates to
others in the literature, and~\cite[\href{http://stacks.math.columbia.edu/tag/04XB}{Tag
  04XB}]{stacks-project} for key properties of morphisms
representable by algebraic spaces.

For a commutative ring $A$, an \emph{fppf stack over $A$} (or
\emph{fppf} $A$-stack) is a stack fibred in groupoids over the big \emph{fppf}
site of $\Spec A$. 

\subsubsection*{Scheme-theoretic images}We briefly remind the reader of some 
definitions from~\cite[\S3.2]{EGstacktheoreticimages}. 
Let
 $\cX \to \cF$ be a proper morphism of stacks over a locally
Noetherian base-scheme~$S$, 
where $\cX$ is an algebraic stack which is locally of finite presentation over~$S$,
and the diagonal of $\cF$ is representable by algebraic spaces and locally of
finite presentation.

We refer to~\cite[Defn.\ 3.2.8]{EGstacktheoreticimages} for the
definition of the \emph{scheme-theoretic image}~$\cZ$ of the proper morphism $\cX \to
\cF$. By definition, it is a full subcategory in groupoids of~$\cF$, and in fact
by~\cite[Lem.\ 3.2.9]{EGstacktheoreticimages} it is a Zariski substack
of~$\cF$. By~\cite[Lem.\ 3.2.14]{EGstacktheoreticimages}, the finite type points
of~$\cZ$ are precisely the finite type points of~$\cF$ for which the
corresponding fibre of~$\cX$ is nonzero. 

 The results of~\cite[\S3.2]{EGstacktheoreticimages} give criteria
for~$\cZ$ to be an algebraic stack, and prove a number of associated results
(such as universal properties of the morphism $\cZ\to\cF$, and a description of
versal deformation rings for~$\cZ$); rather than recalling these results in detail
here, we will refer to them as needed in the body of the paper.

\section{Integral $p$-adic Hodge theory with tame descent data}\label{sec:kisin modules with descent
  data}In this section we introduce various objects in semilinear algebra which
arise in the study of potentially Barsotti--Tate Galois representations with tame
descent data. Much of this material is standard, and none of it will surprise an
expert, but we do not know of a treatment in the literature in the level of
generality that we require; in particular, we are not aware of a treatment of the theory of tame descent data for Breuil--Kisin modules. However, the arguments are almost identical
to those for strongly divisible
modules and Breuil modules, so we will be brief.

The various equivalences of categories between the objects we consider and
finite flat group schemes or $p$-divisible groups will not be
relevant to our main arguments, except at a motivational level, so we
largely ignore them.

\subsection{Breuil--Kisin modules and \texorpdfstring{$\varphi$}{phi}-modules
  with descent data}\label{subsec: kisin modules with dd} 
Recall that we have a finite
tamely ramified Galois extension $K'/K$. Suppose further that there exists a uniformiser $\pi'$ of
$\cO_{K'}$ such that $\pi:=(\pi')^{e(K'/K)}$ is an element of~$K$, 
where $e(K'/K)$ is the ramification index of
$K'/K$. 
Recall that $k'$ is the residue field of $K'$, while $e',f'$ are the
ramification and inertial degrees of $K'$ respectively.
 Let $E(u)$ be the minimal polynomial of $\pi'$ over $W(k')[1/p]$. 

Let $\varphi$ denote the arithmetic Frobenius automorphism of $k'$, which lifts uniquely
to an automorphism of $W(k')$ that we also denote by $\varphi$. Define
$\gS:=W(k')[[u]]$, and extend $\varphi$ to $\gS$ by \[\varphi\left(\sum a_iu^i\right)=\sum
\varphi(a_i)u^{pi}.\] By our assumptions that $(\pi')^{e(K'/K)} \in K$ 
and that $K'/K$ is Galois, for each
$g\in\Gal(K'/K)$ we can write $g(\pi')/\pi'=h(g)$ with $h(g)\in
\mu_{e(K'/K)}(K') \subset W(k')$,
and we
let $\Gal(K'/K)$ act on $\gS$ via \[g\left(\sum a_iu^i\right)=\sum g(a_i)h(g)^iu^i.\]

Let $A$ be a $p$-adically complete $\Zp$-algebra, set $\gS_A:=(W(k')\otimes_{\Zp} A)[[u]]$, and extend
the actions of $\varphi$ and $\Gal(K'/K)$ on $\gS$ to actions on $\gS_A$ in the
obvious ($A$-linear) fashion. 

\begin{lemma}\label{lem:projectivity descends}
An $\gS_A$-module is 
projective if and only if it is projective as an 
$A[[u]]$-module. 
\end{lemma}

\begin{proof}
 Suppose that $\gM$ is an $\gS_A$-module that is projective as an
 $A[[u]]$-module. Certainly $W(k') \otimes_{\Zp} \gM$ is projective
 over $\gS_A$, and we claim that it has $\gM$ as an $\gS_A$-module direct summand. 
 Indeed, this  follows by rewriting $\gM$ as $W(k')\otimes_{W(k')}
 \gM$ and noting that $W(k')$ is a $W(k')$-module direct summand of $W(k')
 \otimes_{\Zp} W(k')$.
\end{proof}

The actions of $\varphi$ and $\Gal(K'/K)$ on $\gS_A$
extend to actions on $\gS_A[1/u]=(W(k')\otimes_{\Zp} A)((u))$ in the obvious
way.  It will sometimes be necessary to consider the subring $\gS_A^0
:=(W(k)\otimes_{\Zp} A)[[v]]$ of $\gS_A$
  consisting of power series in 
$v:=u^{e(K'/K)}$, on which  $\Gal(K'/K)$ acts
  trivially. 

\begin{defn}\label{defn: Kisin module with descent data}
Fix a $p$-adically complete $\Zp$-algebra~$A$. A \emph{weak Breuil--Kisin module with
  $A$-coefficients and descent data from $K'$ to $K$}
  is a triple $(\gM,\varphi_{\gM},\{\hat{g}\}_{g\in\Gal(K'/K)})$ consisting of
  a 
  $\gS_A$-module~$\gM$ and a $\varphi$-semilinear map
  $\varphi_{\gM}:\gM\to\gM$ 
   such that:
  \begin{itemize}
  \item the $\gS_A$-module $\gM$ is finitely generated and $u$-torsion
    free, and 
  \item the  induced
  map $\Phi_{\gM} = 1 \otimes \varphi_{\gM} :\varphi^*\gM\to\gM$ is an isomorphism after
  inverting $E(u)$  (here as usual we write $\varphi^*\gM:=\gS_A \otimes_{\varphi,\gS_A}\gM$), 
  \end{itemize}
together with
  additive bijections  $\hat{g}:\gM\to\gM$, satisfying the further
  properties that 
    the maps $\hat{g}$ commute with $\varphi_\gM$, satisfy 
    $\hat{g_1}\circ\hat{g_2}=\widehat{g_1\circ g_2}$, and have
    $\hat{g}(sm)=g(s)\hat{g}(m)$ for all $s\in\gS_A$, $m\in\gM$. 
   We say that $\gM$ is has \emph{height at most $h$} if the cokernel of 
   $\Phi_{\gM}$ is killed by $E(u)^h$. 

 If $\gM$ as above is projective as an $\gS_A$-module (equivalently, if
   the condition that  the $\gM$ is $u$-torsion free is replaced with the condition
   that $\gM$ is projective) then we say 
 that $\gM$ is a \emph{Breuil--Kisin
   module  with $A$-coefficients and descent data from $K'$ to $K$},
 or even simply that $\gM$ is a \emph{Breuil--Kisin module}. 

The Breuil--Kisin module $\gM$ is said to be  of rank~$d$ if the underlying
finitely generated projective $\gS_A$-module has constant rank~$d$. It
is said to be free if the underlying $\gS_A$-module is free.
\end{defn}

A morphism of (weak) Breuil--Kisin modules with descent data is
a  morphism 
of~$\gS_A$-modules 
that commutes with $\varphi$ and with the~$\hat{g}$. 
In the case that $K'=K$ the data of the $\hat{g}$ is trivial, so
it can be forgotten, giving the category of (\emph{weak}) \emph{Breuil--Kisin modules with
  $A$-coefficients.} In this case it will sometimes be convenient to elide the difference between a
Breuil--Kisin module with trivial descent data, and a Breuil--Kisin module without
descent data, in order to avoid making separate definitions in the
case of Breuil--Kisin modules without descent data; the same convention will
apply to the \'etale $\varphi$-modules considered below.

\begin{lem}\label{lem:kisin injective} Suppose either that $A$ is a $\Z/p^a\Z$-algebra for some
  $a \ge 1$, or that $A$ is $p$-adically separated and $\gM$ is projective. Then in
  Definition~\ref{defn: Kisin module with descent data} the condition
  that $\Phi_{\gM}$ is an isomorphism after inverting $E(u)$ may
  equivalently be replaced with the condition that $\Phi_\gM$ is
  injective and its cokernel is killed by a power
  of $E(u)$.
\end{lem}

  \label{rem: Kisin modules are etale} 

  \begin{proof}
   
  If $A$ is a $\Z/p^a\Z$-algebra for some $a\ge 0$, then $E(u)^h$
  divides $u^{e(a+h-1)}$ in $\gS_A$ (see ~\cite[Lem.\
  5.2.6]{EGstacktheoreticimages} and its proof), so that $\gM[1/u]$ is
  \'etale in the sense that the induced
  map \[\Phi_{\gM}[1/u]:\varphi^*\gM[1/u]\to\gM[1/u]\]
  is an isomorphism. The injectivity of $\Phi_{\gM}$ now follows
  because $\gM$, and therefore $\varphi^*\gM$, is $u$-torsion free.

If instead  $A$ is $p$-adically complete, then no Eisenstein
polynomial over $W(k')$ is a zero divisor in $\gS_A$:\ this is plainly
true if $p$ is nilpotent in $A$, from which one deduces the same for
$p$-adically complete $A$. Assuming that $\gM$ is projective, it
follows that the maps $\gM \to \gM[1/E(u)]$ and $\varphi^*\gM \to
(\varphi^*\gM)[1/E(u)]$ are injective, and we are done.
  \end{proof}

\begin{rem} 
\label{rem:projectivity for Kisin modules} We refer the reader
to~\cite[\S5.1]{EGstacktheoreticimages} for a
discussion of foundational results concerning finitely generated modules
over the power series ring $A[[u]]$. In particular (using
Lemma~\ref{lem:projectivity descends}) we note the
following.
\begin{enumerate}
\item An $\gS_A$-module $\gM$ is finitely generated and projective if
  and only if it is $u$-torsion free and $u$-adically complete, and $\gM/u\gM$ is a finitely generated projective
  $A$-module (\cite[Prop.~5.1.8]{EGstacktheoreticimages}).

\item  If the $\gS_A$-module $\gM$ is projective of
rank~$d$, then it is Zariski locally free of rank~$d$ in the sense that there is a cover of $\Spec A$
by affine opens $\Spec B_i$ such that each of the base-changed modules
 $\gM\otimes_{\gS_A}\gS_{B_i}$ is free of rank $d$  (\cite[Prop.~5.1.9]{EGstacktheoreticimages}).

\item If $A$ is coherent (so in
    particular, if $A$ is Noetherian), then $A[[u]]$ is faithfully
    flat over~$A$, and so $\gS_A$ is faithfully flat over~$A$, but
    this need not hold if $A$ is not coherent.
\end{enumerate}

\end{rem}

\begin{df}
\label{def:completed tensor}
If $Q$ is any (not necessarily finitely generated) $A$-module,
and $\gM$ is an $A[[u]]$-module,
then  we let $\gM\cotimes_A Q$ denote the $u$-adic completion
of $\gM\otimes_A Q$.
\end{df} 

\begin{lem}
  \label{rem: base change of locally free Kisin module is a
    locally free Kisin module}
If $\gM$ is a Breuil--Kisin module and $B$ is an $A$-algebra, then the base
change $\gM \cotimes_A B$ is a Breuil--Kisin module.
\end{lem}

\begin{proof}

We claim that $\gM \cotimes_A B \cong \gM \otimes_{A[[u]]} B[[u]]$ for
any finitely generated projective $A[[u]]$-module; the lemma then follows
immediately from Definition~\ref{defn: Kisin module with descent
  data}.

To check the claim,
we must see that
the finitely generated $B[[u]]$-module
$\gM\otimes_{A[[u]]} B[[u]]$ is $u$-adically complete.
But $\gM$ is a direct summand of a free $A[[u]]$-module of finite rank,
in which case
$\gM\otimes_{A[[u]]} B[[u]]$ is a direct summand of a free $B[[u]]$-module
of finite rank and hence is $u$-adically complete.
\end{proof}

\begin{remark}\label{rem:Kisin-mod-ideal}
If $I \subset A$ is a finitely generated ideal then  $A[[u]] \otimes_A
A/I \cong (A/I)[[u]]$, and $\gM \otimes_A A/I
\cong \gM \otimes_{A[[u]]} (A/I)[[u]] \cong \gM \cotimes_A A/I$;
so in this case $\gM \otimes_A A/I$ itself is a Breuil--Kisin module.
\end{remark}

Note that the base change (in the sense of Definition~\ref{def:completed
  tensor})  of a weak Breuil--Kisin
  module may not be a weak
  Breuil--Kisin module, because  the property of being $u$-torsion free is not
  always preserved by base change. 

We make the following two further remarks concerning base change.

\begin{remark}
\label{rem:completed tensor} 
(1) If $A$ is Noetherian, if $Q$ is finitely generated over $A$,
and if $\gN$ is 
finitely generated over $A[[u]]$, then $\gN\otimes_A Q$
is finitely generated over $A[[u]]$, and hence (by the Artin--Rees
lemma) is automatically $u$-adically complete.  Thus 
in this case the natural morphism $\gN\otimes_A Q \to \gN\cotimes_A Q$
is an isomorphism.

\smallskip

(2)
Note that $A[[u]]\cotimes_A Q = Q[[u]]$ (the $A[[u]]$-module
consisting of power series with coefficients in the $A$-module $Q$),
and so if $\gN$ is Zariski locally free on $\Spec A$,
then $\gN\cotimes_A Q$ is Zariski locally isomorphic to a direct sum
of copies of $Q[[u]]$, and hence is $u$-torsion free (as well as
being $u$-adically complete). In particular, by
Remark~\ref{rem:projectivity for Kisin modules}(2), this holds if~$\gN$
is projective.
\end{remark}

\begin{defn}\label{defn: etale phi module} 
Let $A$ be a $\Z/p^a\Z$-algebra for some $a\ge 1$.  A \emph{weak \'etale
  $\varphi$-module} with $A$-coefficients
  and descent data from $K'$ to $K$ is a triple
  $(M,\varphi_M,\{\hat{g}\})$ consisting of: 
  \begin{itemize}
\item 
a finitely generated
  $\gS_A[1/u]$-module $M$; 
\item a  $\varphi$-semilinear map $\varphi_M:M\to M$ with the
  property that the induced
  map   \[\Phi_M = 1 \otimes \varphi_M:\varphi^*M:=\gS_A[1/u]\otimes_{\varphi,\gS_A[1/u]}M\to M\]is an
  isomorphism,
  \end{itemize}
together with   additive bijections  $\hat{g}:M\to M$ for $g\in\Gal(K'/K)$, satisfying the further
  properties that the maps $\hat{g}$ commute with $\varphi_M$, satisfy
    $\hat{g_1}\circ\hat{g_2}=\widehat{g_1\circ g_2}$, and have
    $\hat{g}(sm)=g(s)\hat{g}(m)$ for all $s\in\gS_A[1/u]$, $m\in M$.

If $M$ as above is projective as an $\gS_A[1/u]$-module then we say
simply that $M$ is an \'etale $\varphi$-module. The \'etale $\varphi$-module $M$ is said to be  of rank~$d$ if the underlying
finitely generated projective $\gS_A[1/u]$-module has constant rank~$d$.

\end{defn}

\begin{rem} 
  \label{rem: completed version if $p$ not nilpotent}We could also
  consider \'etale $\varphi$-modules for general $p$-adically complete $\Zp$-algebras~$A$, but
  we would need to replace $\gS_A[1/u]$ by its $p$-adic completion. As
  we will not need to consider these modules in this paper, we do not
  do so here, but we refer the interested reader to~\cite{EGmoduli}. 
\end{rem}
A morphism
of weak \'etale
$\varphi$-modules with $A$-coefficients and descent data from $K'$ to
$K$ 
is a morphism of~$\gS_A[1/u]$-modules 
that commutes with $\varphi$ and with the
$\hat{g}$. Again, in the case $K'=K$ the descent data is trivial, and we
obtain the  usual category of \'etale $\varphi$-modules with
$A$-coefficients. 

Note that
if $A$ is a $\Z/p^a\Z$-algebra, and $\gM$ is a  Breuil--Kisin module (resp.,
weak Breuil--Kisin module)  with descent data, then $\gM[1/u]$ naturally has the
structure of an \'etale $\varphi$-module (resp., weak \'etale $\varphi$-module) with descent data.

Suppose that $A$ is an $\cO$-algebra (where $\cO$ is as in
Section~\ref{subsec: notation}). In making calculations, it is often
convenient to use the idempotents~$e_i$ (again as in
Section~\ref{subsec: notation}). In particular if $\gM$ is a Breuil--Kisin
module, then writing as usual
$\gM_i:=e_i\gM$, we write $\Phi_{\gM,i}:\varphi^*(\gM_{i-1})\to\gM_{i}$ for
the morphism induced by~$\Phi_{\gM}$.  Similarly if $M$ is an
\'etale $\varphi$-module  then we write
$M_i:=e_iM$, and we write $\Phi_{M,i}:\varphi^*(M_{i-1}) \to M_{i}$ for
the morphism induced by~$\Phi_{M}$.

\subsection{Dieudonn\'e modules}\label{subsec:Dieudonne modules} 
Let $A$ be a $\Zp$-algebra. We define a \emph{Dieudonn\'e module of rank $d$ with $A$-coefficients and
  descent data from $K'$ to $K$} to be a finitely generated projective
$W(k')\otimes_{\Zp}A$-module $D$ of constant rank 
$d$ on $\Spec A$, together with:

\begin{itemize}
\item $A$-linear endomorphisms $F,V$ satisfying $FV = VF = p$ such that $F$ is $\varphi$-semilinear and $V$ is
  $\varphi^{-1}$-semilinear for the action of $W(k')$, and
\item a $W(k')\otimes_{\Zp}A$-semilinear action of $\Gal(K'/K)$
which commutes with $F$ and $V$. 
\end{itemize}

\begin{defn}\label{def: Dieudonne module formulas}
If $\gM$ is a Breuil--Kisin module of height at most~$1$ and rank~$d$ with descent data, 
then there is a
corresponding Dieudonn\'e module $D=D(\gM)$ of rank~$d$ defined as follows. We set
$D:=\gM/u\gM$
with the induced action of $\Gal(K'/K)$, and $F$ given by the induced
action of $\varphi$. 
The endomorphism $V$ is determined as follows.  Write $E(0) = \czero p$,
so that we have $p \equiv \czero^{-1}E(u) \pmod{u}$. The
condition that the cokernel of $\varphi^*\gM\to\gM$ is killed by $E(u)$
allows us to factor the multiplication-by-$E(u)$ map on $\gM$ uniquely
as $\mathfrak{V} \circ \varphi$,  and $V$ is
defined to be
$\czero^{-1} \mathfrak{V}$ modulo~$u$. 
\end{defn}

\subsection{Galois representations}\label{subsec: etale phi modules
  and Galois representations}
The theory of fields of norms~\cite{MR526137} 
was used in~\cite{MR1106901} to relate \'etale $\varphi$-modules with descent data to representations
of a certain absolute Galois group; not the group $G_K$,
but rather the group $G_{K_{\infty}}$,
where $K_{\infty}$ is a certain infinite extension of~$K$ (whose
definition is recalled below).  Breuil--Kisin modules of height $h \leq 1$ are closely related to finite
flat group schemes (defined over $\cO_{K'}$, but with descent data to $K$
on their generic fibre). Passage from a Breuil--Kisin module to its associated
\'etale $\varphi$-module can morally be interpreted as the
passage from a finite flat group scheme (with descent data) to
its corresponding Galois representation (restricted to $G_{K_{\infty}}$).
Since the generic fibre of a finite flat group scheme over $\cO_{K'}$,
when equipped with descent data to $K$,
in fact gives rise to a representation of $G_K$,
in the case $h = 1$ we may relate Breuil--Kisin modules with descent data
(or, more precisely, their associated \'etale $\varphi$-modules),
not only to representations of $G_{K_{\infty}}$, but to representations
of~$G_K$.

 In 
this subsection, we recall
some results coming from this connection,
and draw some conclusions for Galois deformation rings.

\subsubsection{From \'etale $\varphi$-modules to $G_{K_{\infty}}$-representations}
We begin by recalling from \cite{kis04} some constructions arising in $p$-adic Hodge theory
and the theory of fields of norms, which go back to~\cite{MR1106901}. 
Following Fontaine,
we write $R:=\varprojlim_{x\mapsto
  x^p}\cO_{\Kbar}/p$. 
Fix a compatible system $(\! \sqrt[p^n]{\pi}\,)_{n\ge 0}$ of
$p^n$th roots of $\pi$ in $\Kbar$ (compatible in the obvious sense that 
$\bigl(\! \sqrt[p^{n+1}]{\pi}\,\bigr)^p = \sqrt[p^n]{\pi}\,$),
and let
$K_{\infty}:=\cup_{n}K(\sqrt[p^n]{\pi})$, and
also $K'_\infty:=\cup_{n}K'(\sqrt[p^n]{\pi})$. Since $(e(K'/K),p)=1$, the compatible system
$(\! \sqrt[p^n]{\pi}\,)_{n\ge 0}$ determines a unique compatible system $(\!
\sqrt[p^n]{\pi'}\,)_{n\ge 0}$ of $p^n$th roots of~$\pi'$ such that $(\!
\sqrt[p^n]{\pi'}\,)^{e(K'/K)} =\sqrt[p^n]{\pi}$.
Write
$\underline{\pi}'=(\sqrt[p^n]{\pi'})_{n\ge 0}\in R$, and $[\underline{\pi}']\in
W(R)$ for its Teichm\"uller representative. We have a Frobenius-equivariant
inclusion $\gS\into W(R)$ by sending $u\mapsto[\underline{\pi}']$.  We can naturally identify
$\Gal(K'_\infty/K_\infty)$ with $\Gal(K'/K)$, and doing this we see that the
action of $g\in G_{K_\infty}$ on $u$ is via $g(u)=h(g)u$.

We let $\cO_{\cE}$ denote the $p$-adic completion of $\gS[1/u]$, and let $\cE$ be the
field of fractions of~$\cO_\cE$. The inclusion $\gS\into W(R)$ extends to an
inclusion $\cE\into W(\operatorname{Frac}(R))[1/p]$. Let $\cE^{\text{nr}}$ be
the maximal unramified extension of $\cE$ in $ W(\operatorname{Frac}(R))[1/p]$,
and let $\cO_{\cE^{\text{nr}}}\subset W(\operatorname{Frac}(R))$ denote
its ring of
integers. Let $\cO_{\widehat{\cE^{\text{nr}}}}$ be the $p$-adic completion of
$\cO_{\cE^{\text{nr}}}$. Note that $\cO_{\widehat{\cE^{\text{nr}}}}$ is stable
under the action of $G_{K_\infty}$. 

 \begin{defn} 
Suppose that $A$ is a $\Z/p^a\Z$-algebra for some $a \ge 1$.  If $M$
is a weak \'etale $\varphi$-module with $A$-coefficients and descent data, set
  $T_A(M):=\left(\cO_{\widehat{\cE^{\text{nr}}}}\otimes_{\gS[1/u]}M
  \right)^{\varphi=1}$, an $A$-module with a
  $G_{K_\infty}$-action (via the diagonal action on
  $\cO_{\widehat{\cE^{\text{nr}}}}$ and $M$, the latter given by
  the~$\hat{g}$). If $\gM$ is a weak Breuil--Kisin module with
  $A$-coefficients and descent data,
  set 
  $T_A(\gM):=T_A(\gM[1/u])$. 
\end{defn}

\begin{lem}
  \label{lem: Galois rep is a functor if A is actually finite local} Suppose
  that $A$ is a local $\Zp$-algebra and that $|A|<\infty$. Then~$T_A$ induces an 
  equivalence of categories from the category of weak \'etale
  $\varphi$-modules with $A$-coefficients and descent data to the category of continuous
  representations of $G_{K_\infty}$ on finite $A$-modules. If $A\to A'$
  is finite, then there is a natural isomorphism $T_A(M)\otimes_A A'\iso
  T_{A'}(M\otimes_A A')$. A weak \'etale $\varphi$-module with
  $A$-coefficients and descent
  data~$M$ is free of rank~$d$ if and only if $T_A(M)$ is a free
  $A$-module of rank~$d$. 
\end{lem}
\begin{proof}
  This is due to Fontaine~\cite{MR1106901}, and can be proved in exactly the same way as~\cite[Lem.\ 1.2.7]{kis04}.
\end{proof}

We will frequently simply write $T$ for $T_A$. Note that if we let
$M'$ be the \'etale $\varphi$-module obtained from $M$ by forgetting the
descent data, then by definition we have
$T(M')=T(M)|_{G_{K'_\infty}}$.

\subsubsection{Relationships between $G_K$-representations
and $G_{K_{\infty}}$-representations}\label{subsubsec: deformation
rings and Kisin modules}
We will later need to study deformation rings for representations
of~$G_K$ in terms of the deformation rings for the restrictions of
these representations to~$G_{K_\infty}$. Note that the representations
of~$G_{K_\infty}$ coming from Breuil--Kisin modules of height at most~$1$
admit canonical extensions to~$G_K$ by~\cite[Prop.\ 1.1.13]{kis04}.
\begin{lem}
  \label{lem: restricting to K_infty doesn't lose information about
    rbar}If $\rbar,\rbar':G_K\to\GL_2(\Fpbar)$ are continuous
  representations, both of which arise as the reduction mod~$p$ of
  potentially Barsotti--Tate representations of tame inertial type,
 and there is an isomorphism $\rbar|_{G_{K_\infty}}\cong \rbar'|_{G_{K_\infty}}
  $, then $\rbar\cong\rbar'$.
\end{lem}
\begin{proof}
  The extension $K_\infty/K$ is totally wildly ramified. Since the
  irreducible $\Fpbar$-representations of~$G_K$ are induced from
  tamely ramified characters, we see that~$\rbar|_{G_{K_\infty}}$ is
  irreducible if and only if~$\rbar$ is irreducible, and if
  $\rbar$ or $\rbar'$ is irreducible then we are done. In the
  reducible case, we see that $\rbar$ and $\rbar'$ are extensions of
  the same characters, and the result then follows from~\cite[Lem.\
  5.4.2]{gls13} and Lemma~\ref{lem: list of things we need to know about Serre weights}~(2). 
\end{proof}
Let $\rbar:G_K\to\GL_2(\F)$
be a continuous representation, let $R_{\rbar}$ denote the universal
framed deformation $\cO$-algebra for~$\rbar$, and let
$R_{\rbar}^{[0,1]}$ 
be the quotient with the property that if~$A$ is
an Artinian local~$\cO$-algebra with residue field~$\F$, 
then
a local $\cO$-morphism $R_{\rbar}\to A$ factors
through $R_{\rbar}^{[0,1]}$ if and only if the corresponding
$G_K$-module (ignoring the $A$-action) admits a $G_K$-equivariant surjection from a potentially crystalline
$\cO$-representation  all of whose Hodge--Tate weights are equal to $0$ or
$1$, and whose restriction to~$G_{K'}$ is crystalline. (The existence of this quotient follows as in~\cite[\S2.1]{MR2782840}.)

Let $R_{\rbar|_{G_{K_\infty}}}$ be the universal framed deformation
$\cO$-algebra for~$\rbar|_{G_{K_\infty}}$, and let
$R_{\rbar|_{G_{K_\infty}}}^{\le 1}$ denote the quotient with the property that if~$A$ is
an Artinian local~$\cO$-algebra with residue field~$\F$, then a morphism $R_{\rbar|_{G_{K_\infty}}}\to A$ factors
through $R_{\rbar|_{G_{K_\infty}}}^{\le 1}$ if and only if the corresponding
$G_{K_\infty}$-module is isomorphic to  $T(\gM)$ for some weak Breuil--Kisin
module~$\gM$ of height at most one with $A$-coefficients and descent
data from~$K'$ to~$K$. 
(The existence of this quotient
follows exactly as for~\cite[Thm.~1.3]{MR2782840}.)

\begin{prop}
  \label{prop: restriction gives an equivalence of height at most 1
    deformation}The natural map induced by restriction from $G_K$ to
  $G_{K_\infty}$ induces an isomorphism $\Spec R_{\rbar}^{[0,1]}\to
  \Spec R_{\rbar|_{G_{K_\infty}}}^{\le 1}$. 
\end{prop}
\begin{proof}
  This can be proved in exactly the same way as~\cite[Cor.\
  2.2.1]{MR2782840} (which is the case that $E=\Qp$ and $K'=K$).
\end{proof}

\section{Moduli stacks of Breuil--Kisin modules and  \texorpdfstring{$\varphi$}{phi}-modules with descent data}\label{sec:moduli stacks}

In this section we define moduli stacks of Breuil--Kisin modules with tame
descent data, following~\cite{MR2562795,EGstacktheoreticimages} (which
consider the case without descent data). In particular, we define various
stacks~$\cZ$ in Section~\ref{subsec:results from EG
  and PR on finite flat moduli}, as scheme-theoretic images of
morphisms from moduli stacks of Breuil--Kisin modules to moduli stacks
of \'etale $\varphi$-modules; these stacks are the main objects of
interest in the rest of the paper. In the rest of the section, we use the theories of
local models of Shimura varieties and Dieudonn\'e modules to begin our
study of the geometry of these stacks.

\subsection{Moduli stacks of Breuil--Kisin modules}
We begin by defining the 
moduli stacks of Breuil--Kisin modules, with and without descent data. We will
make use of the notion of a $\varpi$-adic formal algebraic stack, which
is recalled in Appendix~\ref{sec:formal stacks}.

\begin{defn}
  \label{defn: C^dd,a }For each integer $a\ge 1$, we let $\cC_{d,h,K'}^{\dd,a}$ be
  the {\em fppf} stack over~$\cO/\varpi^a$ which associates to any $\cO/\varpi^a$-algebra $A$
the groupoid $\cC_{d,h,K'}^{\dd,a}(A)$ 
of rank~$d$ Breuil--Kisin modules of height at most~$h$ with $A$-coefficients and descent data from
$K'$ to $K$. 

By~\cite[\href{http://stacks.math.columbia.edu/tag/04WV}{Tag 04WV}]{stacks-project}, 
we may also regard each of the stacks $\cC_{d,h,K'}^{\dd,a}$ as an {\em fppf}
stack over $\cO$,
and we then write $\cC_{d,h,K'}^{\dd}:=\varinjlim_{a}\cC_{d,h,K'}^{\dd,a}$; this
is again an {\em fppf} stack over $\cO$.

We will frequently omit any (or all) of the subscripts $d,h,K'$ from this notation
when doing so will not cause confusion.
In the case that $K=K'$, we write $\cC^a$ for $\cC^{\dd,a}$ and $\cC$
for $\cC^{\dd}$.
\end{defn} 

The natural morphism $\cC^{\dd} \to \Spec \cO$ factors through $\Spf \cO$,
and by construction, there is an isomorphism $\cC^{\dd,a} \iso \cC^{\dd}\times_{\Spf \cO}
\Spec \cO/\varpi^a,$ for each $a \geq 1$; in particular, each of the morphisms
$\cC^{\dd,a} \to \cC^{\dd,a+1}$
is a thickening
(in the sense that its pullback under any test morphism $\Spec A \to \cC^{\dd,a+1}$
becomes a thickening of schemes, as defined
in~\cite[\href{http://stacks.math.columbia.edu/tag/04EX}{Tag
  04EX}]{stacks-project}\footnote{Note that for morphisms of algebraic
stacks --- and we will see below that $\cC^{\dd,a}$ and $\cC^{\dd,a+1}$ {\em are}
algebraic stacks --- this notion of thickening coincides with the notion defined
in~\cite[\href{http://stacks.math.columbia.edu/tag/0BPP}{Tag
  0BPP}]{stacks-project},
by~\cite[\href{http://stacks.math.columbia.edu/tag/0CJ7}{Tag
  0CJ7}]{stacks-project}.}).
  In Corollary~\ref{cor: C^dd,a is an algebraic stack} below we show that for each integer $a\ge 1$,
  $\cC^{\dd,a}$ is in fact an algebraic stack of finite type over
  $\Spec\cO/\varpi^a$, and that $\cC^{\dd}$ (which is then
  {\em a priori} an Ind-algebraic
 stack, endowed with a morphism to $\Spf \cO$ which is representable
by algebraic stacks) 
is in fact a 
  $\varpi$-adic formal algebraic stack, in the sense of Definition~\ref{defn: pi adic formal alg stack}. 

Our approach will be to deduce the statements in the case with descent 
data from the corresponding statements in the case with no descent
data, which follow from the methods of Pappas and Rapoport~\cite{MR2562795}.  
More precisely, in that reference it is proved that each
$\cC^a$ is an algebraic stack over~$\cO/\varpi^a$~\cite[Thm.~0.1 (i)]{MR2562795},
and thus that $\cC := \varinjlim_a \cC^a$ is a $\varpi$-adic
Ind-algebraic stack (in the sense that it is an Ind-algebraic stack
with a morphism to~$\Spf\cO$ that is representable by algebraic
stacks). 
(In \cite{MR2562795} the stack
	$\cC$ is described as being a $p$-adic formal algebraic stack. 
	However, in that reference, this
	term is used synonomously with our notion of a $p$-adic 
	Ind-algebraic stack; the question of the existence of a smooth cover
	of $\cC$ by a $p$-adic formal algebraic space is not discussed.
	As we will see, though, the existence of such a cover is easily
	deduced from the set-up of \cite{MR2562795}.)

We thank Brandon Levin for pointing out the following
result to us.  The proof is essentially (but somewhat 
implicitly) contained in the proof of 
\cite[Thm.~3.5]{2015arXiv151007503C},
but we take the opportunity to make it explicit. Note that it could
also be directly deduced from the results of~\cite{MR2562795} using
Lemma~\ref{lem:ind to formal}, but the proof that we give has the
advantage of giving an explicit cover by a formal algebraic space. 

\begin{prop}
  \label{prop: C without dd is a $p$-adic formal  algebraic
    stack}For any choice of $d,h$, $\cC$ is a 
$\varpi$-adic formal algebraic stack of finite type over~$\Spf\cO$ with affine diagonal.
\end{prop}
\begin{proof}
  We begin by recalling some results from~\cite[\S 3.b]{MR2562795} (which
  is where the
  proof that each~$\cC^a$ is an algebraic stack of finite type
  over~$\cO/\varpi^a$ is given). If~$A$ is an $\cO/\varpi^a$-algebra for some $a\ge 1$,
  then we set\[L^+G(A):=GL_d(\gS_A),\] \[LG^{h,K'}(A):=\{X\in
    M_d(\gS_A) \mid  X^{-1}\in E(u)^{-h}M_d(\gS_A)\},\]and let $g\in L^+G(A)$ act
  on the right on $LG^{h,K'}(A)$ by $\varphi$-conjugation as $g^{-1}\cdot
  X\cdot \varphi(g)$. Then we may write 
  \[\cC=[LG^{h,K'} \overphi L^+G].\]

  For each $n\ge 1$ we have the principal congruence subgroup $U_n$ of
  $L^+G$ given by $U_n(A)=I+u^n\cdot M_d(\gS_A)$. As in~\cite[\S
  3.b.2]{MR2562795}, for any integer $n(a)>eah/(p-1)$ we have a
  natural identification
  \numequation\label{eqn:U_n phi conjugacy}
[LG^{h,K'}\overphi U_{n(a)}]_{\cO/\varpi^a}\cong [LG^{h,K'}/
    U_{n(a)}]_{\cO/\varpi^a}\end{equation}where
  the $U_{n(a)}$-action on the right hand side is by left
  translation by the inverse; 
  moreover this quotient stack is represented by a finite type
  scheme $(X^{h,K'}_{n(a)})_{\cO/\varpi^a}$, and we find that
  \[\cC^a\cong
    [(X^{h,K'}_{n(a)})_{\cO/\varpi^a}\overphi(\cG_{{n(a)}})_{\cO/\varpi^a}],\]
  where $(\cG_{{n(a)}})_{\cO/\varpi^a}=(L^+G/U_{n(a)})_{\cO/\varpi^a}$ is a
  smooth finite type group scheme over $\cO/\varpi^a$.

Now define $Y_a := [(X_{n(a)}^{h,K'})_{\cO/\varpi^a}\overphi (U_{n(1)})_{\cO/\varpi^a}].$
If $a \geq b,$ then there is a natural isomorphism
$(Y_a)_{\cO/\varpi^b} \cong Y_b.$ 
Thus we may form the $\varpi$-adic Ind-algebraic stack
$Y := \varinjlim_a Y_a.$
Since $Y_1 := (X_{n(1)}^{h,K'})_{\F}$ is a scheme,
each $Y_a$ is in fact a
scheme~\cite[\href{http://stacks.math.columbia.edu/tag/0BPW}{Tag 0BPW}]{stacks-project}, 
and thus $Y$ is a $\varpi$-adic formal scheme.
(In fact, it is easy to check directly that $U_{n(1)}$ acts freely
on $X_{n(a)}^{h,K'},$ and thus to see that $Y_a$ is an algebraic space.)
The natural morphism $Y \to \cC$ is then representable by algebraic
spaces; indeed, any morphism from an affine scheme to~$\cC$ factors
through some~$\cC^a$, and representability by algebraic spaces then follows from the
representability by algebraic spaces of~$Y_a\to\cC^a$, and the
Cartesianness of the diagram \[\xymatrix{Y_a\ar[r]\ar[d]&Y\ar[d]\\ \cC^a\ar[r]&\cC}\]
Similarly, the morphism $Y\to\cC$ is smooth and surjective, and so witnesses the claim that
$\cC$ is a $\varpi$-adic formal algebraic stack.

To check that $\cC$ has affine diagonal, it suffices to check
that each $\cC^a$ has affine diagonal, which follows from the fact that
$(\cG_{n(a)})_{\cO/\varpi^a}$ is in fact an affine group scheme
over $\cO/\varpi^a$ (indeed, as in~\cite[\S 2.b.1]{MR2562795}, it is a
Weil restriction of~$\GL_d$). 
\end{proof}

We next introduce the moduli stack of \'etale $\varphi$-modules, again both with
and without descent data.

\begin{defn}\label{defn: R^dd}
For each integer~$a\ge 1$, we let 
 $\cR^{\dd,a}_{d,K'}$ be the \emph{fppf} $\cO/\varpi^a$-stack which
  associates
  to any $\cO/\varpi^a$-algebra $A$ the groupoid $\cR^{\dd,a}_{d,K'}(A)$ of rank $d$ \'etale
  $\varphi$-modules  with $A$-coefficients and  descent data from $K'$ to
  $K$.

By~\cite[\href{http://stacks.math.columbia.edu/tag/04WV}{Tag 04WV}]{stacks-project}, 
we may also regard each of the stacks $\cR_{d,K'}^{\dd,a}$ as an {\em fppf}
$\cO$-stack, 
and we then write 
  $\cR^{\dd}:=\varinjlim_{a}\cR^{\dd,a}$,
  which is again an {\em fppf} $\cO$-stack.

  We will omit $d, K'$ from the
  notation wherever doing so will not cause confusion, and when $K'=K$, we write~$\cR$ for~$\cR^{\dd}$.
\end{defn}

Just as in the case of $\cC^{\dd}$, the morphism $\cR^{\dd} \to \Spec \cO$ factors through
$\Spf \cO$, and for each $a \geq 1$, there is a natural isomorphism
$\cR^{\dd,a} \iso \cR^{\dd}\times_{\Spf \cO} \Spec \cO/\varpi^a.$
Thus each of the morphisms $\cR^{\dd,a} \to \cR^{\dd,a+1}$ is a thickening.

  There is a natural morphism $\cC_{d,h,K'}^{\dd}\to\cR_d^{\dd}$,
defined via
 $$ (\gM,\varphi,\{\hat{g}\}_{g\in\Gal(K'/K)}) \mapsto
(\gM[1/u],\varphi,\{\hat{g}\}_{g\in\Gal(K'/K)}),$$ and natural morphisms
$\cC^{\dd}\to\cC$ and $\cR^{\dd}\to\cR$ given by forgetting the descent
data.  
In the optic of Section~\ref{subsec: etale phi modules
  and Galois representations}, the stack~$\cR^{\dd}_d$ may
morally be thought of as a moduli of $G_{K_{\infty}}$-representations,
and the morphisms
$\cC^{\dd}_{d,h,K'}\to\cR^{\dd}_d$ correspond to passage from a Breuil--Kisin module to
its underlying Galois representation.

\begin{prop}
  \label{prop: relative representability of descent data for R} For
  each $a\ge 1$, the natural morphism
  $\cR^{\dd,a}\to\cR^a$ is representable by algebraic spaces,
  affine, and of finite presentation.
\end{prop}
\begin{proof} To see this, consider the pullback along some morphism $\Spec
  A\to\cR^a$ (where $A$ is a $\cO/\varpi^a$-algebra); we must show
  that given an \'etale $\varphi$-module $M$ of rank $d$ without descent data,
  the
data of giving 
additive bijections $\hat{g}:M\to M$, satisfying the further
  property that:
  \begin{itemize}
  \item the maps $\hat{g}$ commute with $\varphi$, satisfy
    $\hat{g_1}\circ\hat{g_2}=\widehat{g_1\circ g_2}$, and we have
    $\hat{g}(sm)=g(s)\hat{g}(m)$ for all $s\in\gS_A[1/u]$, $m\in M$
  \end{itemize}is represented by an affine algebraic space (i.e.\ an affine scheme!)
of finite presentation over~$A$.

To see this, note first that such maps $\hat{g}$ are by definition
$\gS_A^0[1/v]$-linear. 
The data of giving
an $\gS^0_A[1/v]$-linear automorphism of $M$ which commutes with~$\varphi$ is representable
by an affine scheme of finite presentation over~$A$ by~\cite[Prop.\ 5.4.8]{EGstacktheoreticimages} 
and so the
data of a finite collection of automorphisms is also representable by
a finitely presented affine scheme over $A$. The further commutation and composition
conditions on the $\hat{g}$ cut out a closed subscheme, as does the condition of
$\gS_A[1/u]$-semi-linearity, so the result follows.
\end{proof}

\begin{cor}
\label{cor:diagonal of R}
The diagonal of $\cR^{\dd}$ is representable by algebraic spaces, affine,
and of finite presentation.
\end{cor}
\begin{proof}
Since $\cR^{\dd} = \varinjlim_a\cR^{\dd,a} \iso \varinjlim_a \cR^{\dd}\times_{\Spf \cO}
\Spec \cO/\varpi^a$,
and since the transition morphisms are closed immersions (and hence monomorphisms),
we have a Cartesian diagram
$$\xymatrix{ \cR^{\dd,a} \ar[r]\ar[d] & \cR^{\dd,a}\times_{\cO/\varpi^a} \cR^{\dd,a} 
\ar[d] \\
\cR^\dd \ar[r] & \cR^{\dd} \times_{\cO} \cR^{\dd}}$$
for each~$a~\geq~1$,
and the diagonal morphism of $\cR^{\dd}$ is the inductive limit of the diagonal
morphisms of the various $\cR^{\dd,a}$.
Any morphism from an affine scheme $T$ to $\cR^{\dd} \times_{\cO} \cR^{\dd}$
thus factors through
one of the $\cR^{\dd,a}\times_{\cO/\varpi^a} \cR^{\dd,a}$, and the fibre product
$\cR^{\dd} \times_{\cR^{\dd} \times \cR^{\dd}} T$
may be identified with
$\cR^{\dd,a} \times_{\cR^{\dd,a} \times \cR^{\dd,a}} T$.
It is thus equivalent to prove that each of the diagonal morphisms
$\cR^{\dd,a} \to \cR^{\dd,a} \times_{\cO/\varpi^a} \cR^{\dd,a}$
is representable by algebraic spaces, affine, and of finite presentation.

The diagonal of $\cR^{\dd,a}$ may be obtained
by composing the pullback over $\cR^{\dd,a}\times_{\cO}
\cR^{\dd,a}$ of the diagonal $\cR^a \to \cR^a \times_{\cO} \cR^a$ with
the relative diagonal of the morphism $\cR^{\dd,a} \to \cR^a$.
The first of these morphisms is representable by algebraic spaces,
affine, and of finite presentation, by \cite[Thm.~5.4.11~(2)]{EGstacktheoreticimages},
and the second is also representable by algebraic
spaces, affine, and of finite presentation,
since it is the relative diagonal of a morphism which has these properties,
by~Proposition~\ref{prop: relative representability of descent data for R}.
\end{proof}

\begin{cor}
  \label{cor: C^dd,a is an algebraic stack}
\begin{enumerate}
\item
  For each $a\ge 1$,
  $\cC^{\dd,a}$ is an algebraic stack of finite presentation over
  $\Spec\cO/\varpi^a$, with affine diagonal.
\item
  The Ind-algebraic stack $\cC^{\dd} := \varinjlim_a \cC^{\dd,a}$ 
  is furthermore a $\varpi$-adic formal algebraic stack.
\item
The morphism $\cC^{\dd}_h \to \cR^{\dd}$ is representable by algebraic spaces and proper.
\end{enumerate}
\end{cor}
\begin{proof}By Proposition~\ref{prop: C without dd is a $p$-adic formal  algebraic
    stack}, $\cC^a$ is an algebraic stack of
finite type over $\Spec\cO/\varpi^a$ with affine diagonal.
In particular it has quasi-compact diagonal, and so is quasi-separated.
Since $\cO/\varpi^a$ is Noetherian, it follows from~\cite[\href{http://stacks.math.columbia.edu/tag/0DQJ}{Tag 0DQJ}]{stacks-project} that $\cC^a$ is 
in fact of finite presentation over $\Spec \cO/\varpi^a$.

By Proposition~\ref{prop: relative representability of descent data for
  R}, the  morphism $\cR^{\dd,a}\times_{\cR^a}\cC^a\to\cC^a$ is
representable by algebraic spaces and of finite presentation, so it follows from~\cite[\href{http://stacks.math.columbia.edu/tag/05UM}{Tag
    05UM}]{stacks-project} that $\cR^{\dd,a}\times_{\cR^a}\cC^a$ is
  an algebraic stack of finite presentation over $\Spec\cO/\varpi^a$. In order to show
that $\cC^{\dd,a}$ is an algebraic stack of finite presentation over $\Spec\cO/\varpi^a$,
it therefore suffices 
  to show 
  that the natural monomorphism 
\numequation
\label{eqn:closed}
\cC^{\dd,a}\to \cR^{\dd,a}\times_{\cR^a}\cC^a
\end{equation}
 is representable by
algebraic spaces and of finite presentation.  We will in fact show that it is
a closed immersion (in the sense that its pull-back under 
any morphism from a scheme to its target becomes a closed immersion
of schemes); since the target is locally Noetherian, and closed
immersions are automatically of finite type and quasi-separated, 
it follows
from~\cite[\href{http://stacks.math.columbia.edu/tag/0DQJ}{Tag
  0DQJ}]{stacks-project} that this closed immersion is of finite
presentation, as required.

By~\cite[\href{http://stacks.math.columbia.edu/tag/0420}{Tag
  0420}]{stacks-project}, the property of being a closed immersion can  be checked
after pulling back to an affine scheme, and then working
\emph{fpqc}-locally. 
The claim then follows easily from the proof of~\cite[Prop.\ 5.4.8]{EGstacktheoreticimages}, as
  \emph{fpqc}-locally the condition that a lattice in an \'etale
  $\varphi$-module of rank $d$ with descent data is preserved by the action
  of the $\hat{g}$ is determined by the vanishing of the coefficients of
  negative powers of~$u$ in a matrix.

To complete the proof of~(1), it suffices to show that the diagonal of
$\cC^{\dd,a}$ is affine.  Since (as we have shown)
the morphism~\eqref{eqn:closed} is a closed immersion, and thus a monomorphism,
it is equivalent to show that the diagonal 
of $\cR^{\dd,a}\times_{\cR^a} \cC^a$ is affine.
To ease notation, we denote this fibre product by $\cY$.
We may then factor the diagonal of $\cY$ as the composite 
of the pull-back over $\cY \times_{\cO/\varpi^a} \cY$
of the diagonal morphism $\cC^a \to \cC^a \times_{\cO/\varpi^a} \cC^a$
and the relative diagonal $\cY \to \cY\times_{\cC^a} \cY$.
The former morphism is affine, by~\cite[Thm.~5.4.9~(1)]{EGstacktheoreticimages},
and the latter morphism is also affine, since it is the pullback via $\cC^a \to
\cR^a$ 
of the relative diagonal morphism $\cR^{\dd,a} \to \cR^{\dd,a}\times_{\cR^a} \cR^{\dd,a}$,
which is affine (as already observed in the proof of Corollary~\ref{cor:diagonal of R}).

To prove~(2),
consider the morphism $\cC^{\dd} \to \cC.$  
This is a morphism of $\varpi$-adic Ind-algebraic stacks,
and by what we have already proved, it is representable 
by algebraic spaces.    Since the target is a $\varpi$-adic formal
algebraic stack, it follows from ~\cite[Lem.\
7.9]{Emertonformalstacks} that the source is also a $\varpi$-adic formal
algebraic stack, as required. 

To prove~(3), since each of $\cC^{\dd,a}$ and $\cR^{\dd,a}$ is obtained from
$\cC^{\dd}$ and $\cR^{\dd}$ via pull-back over $\cO/\varpi^a$,
it suffices to prove that each of the morphisms $\cC^{\dd,a} \to \cR^{\dd,a}$
is representable by algebraic spaces and proper.   Each of these morphisms
factors as
$$\cC^{\dd,a} \buildrel \text{\eqref{eqn:closed}} \over
\longrightarrow \cR^{\dd,a} \times_{\cR^a} \cC^a \buildrel \text{proj.} \over
\longrightarrow  \cR^{\dd,a}.$$
We have already shown that the first of these morphisms is a closed immersion,
and hence representable by algebraic spaces and proper.  The second morphism is also 
representable by algebraic spaces and proper, since it is a base-change 
of the morphism $\cC^a \to \cR^a$, which has these properties
by~\cite[Thm.~5.4.11~(1)]{EGstacktheoreticimages}.
\end{proof}

The next lemma gives a concrete interpretation of the points of $\cC^{\dd}$ over 
$\varpi$-adically complete $\cO$-algebras, extending the tautological interpretation
of the points of each $\cC^{\dd,a}$ prescribed by Definition~\ref{defn: C^dd,a }.

\begin{lemma}\label{lem:points-of-Cdd}
 If $A$ is a $\varpi$-adically complete $\cO$-algebra then the
 $\Spf(A)$-points of $\cC^{\dd}$ are the Breuil--Kisin modules
 of rank $d$ and height $h$ with $A$-coefficients and descent
 data. 
\end{lemma} 

\begin{proof} 

Let $\gM$ be a Breuil--Kisin module of rank $d$ and height $h$ with
$A$-coefficients and descent data. Then the sequence
$\{\gM/\varpi^a\gM\}_{a\ge 1}$ defines a $\Spf(A)$-point of $\cC^{\dd}$ 
 (\emph{cf.}\ Remark~\ref{rem:Kisin-mod-ideal}), and since $\gM$ is
 $\varpi$-adically complete  it is recoverable from the sequence
 $\{\gM/\varpi^a\gM\}_{a \ge 1}$. 

In the other direction, suppose that $\{\gM_a\}$ is a $\Spf(A)$-point
of $\cC^{\dd}$, so that $\gM_a \in \cC^{\dd,a}(A/\varpi^a)$. Define
$\gM = \invlim_{a} \gM_a$, and similarly define $\varphi_\gM$ and
$\{\ghat\}$ as inverse limits. Observe that $\varphi^* \gM = \invlim_a
\varphi^*\gM_a$ (since $\varphi : \gS_A \to \gS_A$ makes $\gS_A$ into
a free $\gS_A$-module). Since each $\Phi_{\gM_a}$ is injective with
cokernel killed by $E(u)^h$ the same holds for $\Phi_\gM$. 

Since the required properties of the descent data are immediate, to complete the proof it remains to check that $\gM$ is a projective
$\gS_A$-module (necessarily of rank~$d$, since its rank will equal
that of $\gM_1$), which is a consequence of \cite[Prop.\ 0.7.2.10(ii)]{MR3075000}.
\end{proof}

We now temporarily reintroduce~$h$ to the notation. 

\begin{df} For each $h\ge 0$,  write~$\cR^a_h$ for the
scheme-theoretic image of $\cC^a_h\to\cR^a$ in the sense of~\cite[Defn.\
3.2.8]{EGstacktheoreticimages}; then by~\cite[Thms.\ 5.4.19,
5.4.20]{EGstacktheoreticimages}, $\cR^a_h$ is an algebraic stack of
finite presentation over~$\Spec\cO/\varpi^a$, the morphism 
$\cC^a_h\to\cR^a$ factors through $\cR^a_h$, 
and we may write
$\cR^a \cong \varinjlim_h\cR^a_h$ as an inductive limit of closed substacks,
the natural transition morphisms being closed immersions. 

We similarly write $\cR^{\dd,a}_h$ for the scheme-theoretic image of the morphism
$\cC_h^{\dd,a}\to\cR^{\dd,a}$ in the sense of~\cite[Defn.\
3.2.8]{EGstacktheoreticimages}.
\end{df}

  \begin{thm}
    \label{thm: R^dd is an ind algebraic stack}For each $a\ge 1$,
    $\cR^{\dd,a}$ is an Ind-algebraic stack. Indeed, we can write
    $\cR^{\dd,a}=\varinjlim_h\cX_h$ as an inductive limit of
    algebraic stacks of finite presentation over $\Spec\cO/\varpi^a$, the
    transition morphisms being closed immersions.
  \end{thm}
  \begin{proof}
    As we have just recalled, by~\cite[Thm.\
    5.4.20]{EGstacktheoreticimages} 
    we can write
    $\cR^{a}=\varinjlim_h\cR^{a}_h$, so that if we set
    $\cX^{\dd,a}_h:=\cR^{\dd,a}\times_{\cR^a}\cR^a_h$, then
    $\cR^{\dd,a}=\varinjlim_h\cX^{\dd,a}_h$, and the transition morphisms are
    closed immersions.  Since~$\cR^a_h$ is of finite presentation
    over~$\Spec\cO/\varpi^a$, and a composite of morphisms of finite
    presentation is of finite presentation, it follows from
    Proposition~\ref{prop: relative representability of descent data
      for R} and
    \cite[\href{http://stacks.math.columbia.edu/tag/05UM}{Tag
      05UM}]{stacks-project} that $\cX^{\dd,a}_h$ is an algebraic stack of
    finite presentation over $\Spec\cO/\varpi^a$, as required.
  \end{proof}

\begin{thm}
  \label{thm: R^dd_h is an algebraic stack}$\cR^{\dd,a}_h$ is an algebraic stack of finite presentation over
    $\Spec\cO/\varpi^a$. It is a closed substack of $\cR^{\dd,a}$, and the
    morphism $\cC_h^{\dd,a}\to\cR^{\dd,a}$ factors through
a morphism 
    $\cC_h^{\dd,a} \to \cR^{\dd,a}_h$ which is representable by algebraic
spaces, scheme-theoretically dominant, and proper. 
\end{thm}
\begin{proof}
As in the proof of Theorem~\ref{thm: R^dd is an ind algebraic stack},
if we set $\cX^{\dd,a}_h:=\cR^{\dd,a}\times_{\cR^a}\cR^a_h$, then~$\cX^{\dd,a}_h$ is an algebraic stack of
    finite presentation over $\Spec\cO/\varpi^a$, and the natural
    morphism $\cX^{\dd,a}_h\to\cR^{\dd,a}$ is a closed immersion. The
    morphism~$\cC_h^{\dd,a}\to\cR^{\dd,a}$ factors through~$\cX^{\dd,a}_h$
    (because the morphism $\cC_h^a\to\cR^{a}$ factors through its
    scheme-theoretic image~$\cR^a_h$), so by~\cite[Prop.\
    3.2.31]{EGstacktheoreticimages}, 
    $\cR^{\dd,a}_h$ is the
    scheme-theoretic image of the morphism of algebraic stacks
    $\cC_h^{\dd,a}\to\cX^{\dd,a}_h$. The required properties now follow
    from \cite[Lem.\ 3.2.29]{EGstacktheoreticimages}
    (using representability by algebraic spaces and properness of the morphism
$\cC_h^{\dd,a} \to \cR^{\dd,a}$,
as proved in Corollary~\ref{cor: C^dd,a is an algebraic stack}~(3),
to see that the induced morphism $\cC_h^{\dd,a} \to \cR_h^{\dd,a}$ is representable
by algebraic spaces and proper,
along with~\cite[\href{http://stacks.math.columbia.edu/tag/0DQJ}{Tag
      0DQJ}]{stacks-project}, and the fact that~$\cX^{\dd,a}_h$ is of finite presentation
    over~$\Spec\cO/\varpi^a$, to see that~$\cR^{\dd,a}_h$ is of finite presentation). 
\end{proof}

\subsection{Representations of tame groups}\label{subsec: tame groups}
Let $G$ be a finite group.

\begin{df}
We let $\Rep_d(G)$ denote the algebraic stack
classifying $d$-dimensional
representations of $G$ over $\cO$: if $X$ is any $\cO$-scheme,
then $\Rep_d(G)(X)$ is the groupoid consisting 
of locally free sheaves of rank $d$ over $X$ endowed with an $\cO_X$-linear
action of $G$ (rank $d$ locally free $G$-sheaves, for short);
morphisms are $G$-equivariant isomorphisms of vector bundles.
\end{df}

We now suppose that $G$ is tame, i.e.\ that it has prime-to-$p$ order.
In this case (taking into account the fact that $\F$ has characteristic~$p$,
and that $\cO$ is Henselian),
the isomorphism classes of $d$-dimensional $G$-representations
of $G$ over $E$ and over $\F$ are in natural bijection.  Indeed, any 
finite-dimensional representation $\tau$
of $G$ over $E$ contains a $G$-invariant
$\cO$-lattice $\tau^{\circ}$,
and the associated representation of $G$ over $\F$ is given
by forming $\overline{\tau}:= \F\otimes_{\cO} \tau^{\circ}$.

\begin{lemma}
\label{lem:tame reps}
Suppose that~ $G$ is tame, and that~ $E$ is chosen large enough so
 that each irreducible
representation of $G$ over $E$ is absolutely irreducible {\em (}or,
equivalently, so that each irreducible representation of $G$ over $\F$
is absolutely irreducible{\em )}, and so that each irreducible
representation of~$G$ over~$\Qpbar$ is defined over~$E$
\emph{(}equivalently, so that each irreducible representation of~$G$
over~$\Fpbar$ is defined over~$\F$\emph{)}. 
\begin{enumerate}
\item $\Rep_d(G)$ is the disjoint union of irreducible
  components 
$\Rep_d(G)_{\tau}$, where $\tau$ ranges over the finite set of isomorphism
classes of $d$-dimensional representations of $G$ over $E$.
\item A morphism $X \to \Rep_d(G)$
factors through $\Rep_d(G)_{\tau}$
if and only if the associated locally free $G$-sheaf on $X$
is Zariski locally isomorphic to $\tau^{\circ}\otimes_{\cO} \cO_X$.
\item If we write $G_{\tau} := \Aut_{\cO[G]}(\tau^{\circ})$, 
then $G_{\tau}$ is a smooth
\emph{(}indeed reductive\emph{)} group scheme over $\cO$,
and
$\Rep_d(G)_{\tau}$ is isomorphic to the classifying space
$[\Spec \cO/G_{\tau}]$. 
\end{enumerate}
\end{lemma} 
\begin{proof}
Since~$G$ has order prime to~$p$, the
representation~$P:= \oplus_\sigma\sigma^\circ$ is a projective generator of
the category of $\cO[G]$-modules, where~$\sigma$ runs over a set of
representatives for the isomorphism classes of irreducible
$E$-representations of~$G$. (Indeed, each~$\sigma^\circ$ is
projective, because the fact that~$G$ has order prime to~$p$ means
that all of the~$\Ext^1$s against~$\sigma^\circ$ vanish. To see that
$\oplus_\sigma\sigma^\circ$ is a generator, we need to show that every
$\cO[G]$-module admits a non-zero map from some~$\sigma^\circ$. We can
reduce to the case of a finitely generated module~$M$, and it is
then enough (by projectivity) to prove that~$M\otimes_\cO\F$ admits
such a map, which is clear.)
Our assumption that each $\sigma$ is absolutely irreducible furthermore
shows that $\End_G(\sigma^{\circ}) = \cO$ for each $\sigma$,
so that $\End_G(P) = \prod_{\sigma} \cO$.  

Standard Morita theory then shows that the functor
$M \mapsto \Hom_G(P,M)$ induces an equivalence between the category
of $\cO[G]$-modules and the category of $\prod_{\sigma} \cO$-modules.
Of course, a $\prod_{\sigma}\cO$-module is just given by a tuple
$(N_{\sigma})_{\sigma}$ of $\cO$-modules,
and in this optic, the functor $\Hom_G(P,\text{--})$ 
can be written as $M \mapsto \bigl(\Hom_G(\sigma^{\circ},M)\bigr)_{\sigma}$,
with a quasi-inverse functor being given by $(N_{\sigma}) \mapsto
\bigoplus_{\sigma} \sigma^{\circ} \otimes_{\cO} N_{\sigma}.$
It is easily seen (just using the fact that $\Hom_G(P,\text{--})$
induces an equivalence of categories)
that $M$ is a finitely generated projective $A$-module,
for some $\cO$-algebra $A$,
if and only if each $\Hom_G(\sigma^{\circ},M)$
is a finitely generated projective $A$-module.

The preceding discussion shows that
giving a rank $d$ representation of $G$ over an $\cO$-algebra $A$
amounts to giving a tuple $(N_{\sigma})_{\sigma}$ of projective
$A$-modules, of ranks~$n_{\sigma}$, such that $\sum_{\sigma} n_{\sigma}
\dim \sigma = d.$
For each such tuple of ranks $(n_\sigma)$, we obtain a corresponding
moduli stack
$\Rep_{(n_{\sigma})}(G)$ classifying rank $d$ representations of $G$
which decompose in this manner,
and $\Rep_d(G)$ is isomorphic to the disjoint union
of the various stacks~$\Rep_{(n_{\sigma})}(G).$

If we write $\tau = \oplus_{\sigma} \sigma^{n_{\sigma}}$, 
then we may relabel $\Rep_{(n_{\sigma})}(G)$ as~$\Rep_{\tau}(G)$; 
statements~(1) and~(2) are then proved.
By construction, there is an isomorphism  
$$\Rep_{\tau}(G) = \Rep_{(n_{\sigma})}(G) \iso \prod_{\sigma}
[\Spec \cO/\GL_{n_{\sigma}}].$$
Noting that $G_{\tau} := \Aut(\tau) = 
\prod_\sigma\GL_{n_\sigma/\cO}$, 
we find that statement~(3) follows as well.
\end{proof}

For each~$\tau$, it follows from the identification of $\Rep_d(G)_{\tau}$ with
$[\Spec \cO/G_{\tau}]$ that there is a natural map
$\Rep_d(G)_{\tau}\to\Spec\cO$. We let~$\pi_0(\Rep_{d}(G))$ denote the
disjoint union of copies of~$\Spec\cO$, one for each isomorphism
class~$\tau$; then there is a natural map
$\Rep_d(G)\to\pi_0(\Rep_d(G))$. While we do not want to develop a
general theory of the \'etale~$\pi_0$ groups of algebraic stacks, we
note that it is natural to regard~$\pi_0(\Rep_d(G))$ as the
\'etale~$\pi_0$ of~$\Rep_d(G)$.
\subsection{Tame inertial types}
\label{subsec:tame inertial types,
  general height}Write $I(K'/K)$ for the inertia subgroup of~$\Gal(K'/K)$.
Since we are assuming that~$E$ is large enough that it contains the
image of every embedding $K'\into\Qpbar$, it follows in particular that every
$\Qpbar$-character of~$I(K'/K)$ is defined
over~$E$.

Recall from Subsection~\ref{subsec: notation} that if $A$ is an~$\cO$-algebra, and~$\gM$ is a Breuil--Kisin module with
$A$-coefficients, then we write $\gM_i$ for  $e_i\gM$. Since $I(K'/K)$
acts trivially on~$W(k')$, the $\hat{g}$ for~$g\in I(K'/K)$ stabilise
each~$\gM_i$, inducing an action of $I(K'/K)$ on $\gM_i/u\gM_i$. 

Write \[\Rep_{d,I(K'/K)}:=\prod_{i = 0}^{f'-1}
\Rep_d\bigl(I(K'/K)\bigr),\]the fibre product being taken over~$\cO$.  If~$\{\tau_i\}$ is an $f'$-tuple of isomorphism classes 
of $d$-dimensional representations of $I(K'/K)$, we write \[\Rep_{d,I(K'/K),\{\tau_i\}}:=\prod_{i = 0}^{f'-1}
\Rep_d\bigl(I(K'/K)\bigr)_{\tau_i}.\] Lemma~\ref{lem:tame reps} shows that we may write
$$\Rep_{d,I(K'/K)} = \coprod_{\{\tau_i\}}
\Rep_{d,I(K'/K),\tau_i}.$$
 Note that since $K'/K$ is tamely ramified,
$I(K'/K)$ is abelian of prime-to-$p$ order, and each~$\tau_i$ is just
a sum of characters. If all of the~$\tau_i$ are equal to some
fixed~$\tau$, then we write $\Rep_{d,I(K'/K),\tau}$ for
$\Rep_{d,I(K'/K),\{\tau_i\}}$. We have corresponding
stacks~$\pi_0(\Rep_{d,I(K'/K)})$,
~$\pi_0(\Rep_{d,I(K'/K),\{\tau_i\}})$
and~$\pi_0(\Rep_{d,I(K'/K),\tau})$, defined in the obvious way.

 If $\gM$ is a Breuil--Kisin module of rank~$d$ with descent data and $A$-coefficients,
then $\gM_i/u\gM_i$ is projective $A$-module of rank $d$, endowed
with an $A$-linear action of $I(K'/K)$,
and so is an $A$-valued point of $\Rep_d\bigl(I(K'/K)\bigr).$ Thus we obtain a morphism 
\numequation
\label{eqn:mixed type morphism}
\cC_d^{\dd} \to \Rep_{d,I(K'/K)},
\end{equation}
defined via $\gM \mapsto ( \gM_0/u\gM_0, \ldots, \gM_{f'-1}/u\gM_{f'-1}).$

\begin{defn}
  \label{defn: Kisin module of type tau general rank}Let $A$ be
  an~$\cO$-algebra, and let $\gM$ be a 
  Breuil--Kisin module of rank~$d$ with $A$-coefficients. We say that $\gM$ has
  \emph{mixed type}~$(\tau_i)_i$
  if the composite $\Spec A \to \cC_d^{\dd} \to \Rep_{d,I(K'/K)}$ (the first arrow being the morphism
that classifies $\gM$, and the second arrow being~(\ref{eqn:mixed 
type morphism}))
factors through $\Rep_{d,I(K'/K),\{\tau_i\}}$.
Concretely, this is equivalent to requiring that,
Zariski locally on
  $\Spec A$, 
there is an $I(K'/K)$-equivariant isomorphism $\gM_i/u\gM_i \cong
A\otimes_{\cO}\tau_i^{\circ}$ for each $i$.

If each $\tau_i=\tau$ for some fixed~$\tau$, then we say that the type
of $\gM$ is unmixed, or simply that
$\gM$ has \emph{type}~$\tau$.
\end{defn}
\begin{rem}
  \label{rem: Kisin modules don't always have a type} If $A=\cO$ then
  a Breuil--Kisin module necessarily has some (unmixed) type~$\tau$, since
after 
  inverting $E(u)$ and reducing modulo $u$ the map   $\Phi_{\gM,i}$
  gives an
  $I(K'/K)$-equivariant $E$-vector space isomorphism
  $\varphi^*(\gM_{i-1}/u\gM_{i-1})[\frac 1p] \toisom (\gM_{i}/u\gM_{i})[\frac 1p]$.
  However if $A=\F$
  there are Breuil--Kisin modules which have a genuinely mixed type; indeed, it is easy to write down examples of free Breuil--Kisin
  modules of rank one of any mixed type (see
  also~\cite[Rem. 3.7]{2015arXiv151007503C}), which necessarily cannot
  lift to characteristic zero. This shows that~$\cC_d^\dd$ is not flat
  over~$\Zp$. In the following sections, when $d=2$ and $h=1$ we define a
  closed substack~$\cC_d^{\dd,\BT}$ of~$\cC^{\dd}$ which \emph{is} flat
  over~$\Zp$, and can be thought of as taking the Zariski closure of
   $\Qpbar$-valued Galois representations that become Barsotti--Tate
   over $K'$ and such that all pairs of labeled Hodge--Tate
   weights are $\{0,1\}$ (see Remark~\ref{rem: strong determinant condition motivated by
      flat closure} below). 
\end{rem}

\begin{defn} 
  \label{defn: Ctau in general}Let $\cC_d^{(\tau_i)}$ be the
  \'etale substack of $\cC_d^{\dd}$ which associates to
  each $\cO$-algebra $A$ the subgroupoid $\cC_d^{(\tau_i)}(A)$ of
  $\cC_d^{\dd}(A)$ consisting of those Breuil--Kisin modules which are of
  mixed type~$(\tau_i)$. If each $\tau_i = \tau$ for some fixed
  $\tau$, we write $\cC_d^\tau$ for $\cC_d^{(\tau_i)}$.
\end{defn}

\begin{prop}
  \label{prop: Ctau is an algebraic stack}Each $\cC_d^{(\tau_i)}$ is an open
and closed
  substack of $\cC_d^\dd$, and $\cC_d^\dd$ is the disjoint union of its substacks~$\cC_d^{(\tau_i)}$. 
\end{prop}
\begin{proof}  
By Lemma~\ref{lem:tame reps}, $\Rep_{d,I(K'/K)}$ is the disjoint
union of its open and closed substacks $\Rep_{d,I(K'/K),\{\tau_i\}}$.
By definition $\cC_d^{(\tau_i)}$ is the preimage of 
$\Rep_{d,I(K'/K),\{\tau_i\}}$
under the morphism~(\ref{eqn:mixed type morphism});
the lemma follows.
\end{proof}

\subsection{Local models:\ generalities}\label{subsec: local models
  generalities}Throughout this section we allow~$d$ to be arbitrary; in Section~\ref{subsec: local
  models rank two} we specialise to the case~$d=2$, where we relate
the local models considered in this section to the local models
considered in the theory of Shimura varieties. We will usually
omit~$d$ from our notation, writing for example~$\cC$ for~$\cC_d$,
without any further comment. We begin with the
following lemma, for which we allow~$h$ to be arbitrary.

\begin{lem}
  \label{lem:cokernel is projective} Let $\gM$ be a rank~$d$ Breuil--Kisin
  module of height~$h$ with
  descent data over an $\cO$-algebra $A$. Assume further 
  either that~$A$ is $p^n$-torsion for some~$n$, or that~$A$ is Noetherian
  and $p$-adically complete. 
Then
  $\im \Phi_{\gM}/E(u)^h\gM$ is a finite projective
  $A$-module, and is a direct summand of~$\gM/E(u)^h\gM$ as an
  $A$-module. 
\end{lem}
\begin{proof}
  We follow the proof of~\cite[Lem.\ 1.2.2]{kis04}. We have a short
  exact sequence
  \[0\to \im \Phi_{\gM}/E(u)^h\gM\to
    \gM/E(u)^h\gM\to\gM/\im \Phi_{\gM}\to 0 \] in which the
  second term is a finite projective $A$-module (since it is a finite
  projective $\cO_{K'}\otimes_{\Zp}A$-module), so it is enough to show
  that the third term is a projective $A$-module. It is therefore enough to show
  that the finitely generated $A$-module $\gM/\im \Phi_{\gM}$
 is finitely presented and
  flat.

To see that it is finitely presented, note that we have the
equality \[\gM/\im \Phi_{\gM}=(\gM/E(u)^h)/(\im \Phi_{\gM}/E(u)^h), \]
and the right hand side admits a presentation
by finitely generated projective
$A$-modules  \[\varphi^*(\gM/E(u)^h)\to\gM/E(u)^h\to(\gM/E(u)^h) /(\im \Phi_{\gM}/E(u)^h)\to 0. \]

To see that it is flat, it is enough to show that for every finitely generated ideal~$I$
of~$A$, the map \[I\otimes_A\gM/\im \Phi_{\gM} \to
  \gM/\im \Phi_{\gM}\] is injective. It follows easily (for
example, from the snake lemma)
that it is enough to check that the complex \[0\to
  \varphi^*\gM\to \gM\to \gM/\im \Phi_{\gM}\to 0, \] 
which is exact by Lemma~\ref{lem:kisin injective},
remains
exact after tensoring with~$A/I$.  Since~$I$ is finitely generated we
have $\gM \otimes_A A/I \cong \gM \cotimes_A A/I$ by
Remark~\ref{rem:Kisin-mod-ideal}, and the desired exactness amounts
 to the injectivity of
$\Phi_{\gM \cotimes_A A/I}$ for the Breuil--Kisin module $\gM \cotimes_A A/I$.
This  follows immediately  from
  Lemma~\ref{lem:kisin injective} in the case that~$A$ is killed by
  $p^n$, and otherwise follows from the same lemma once we check that $A/I$ is
  $p$-adically complete, which 
follows from the
Artin--Rees lemma (as $A$ is assumed Noetherian and $p$-adically
complete). 
\end{proof}
We assume from now on that~$h=1$, but 
we continue to allow arbitrary~$d$. We allow~$K'/K$ to be any 
Galois extension such that~$[K':K]$ is prime to~$p$ (so in particular~$K'/K$
is tame). 

\begin{defn}\label{defn: }
  We let~$\cM_{\loc}^{K'/K,a}$ be the algebraic stack of finite presentation
over~$\Spec\cO/\varpi^a$ defined as follows:
  if~$A$ is an $\cO/\varpi^a$-algebra, then $\cM_{\loc}^{K'/K}(A)$ is
  the groupoid of tuples~$(\gL,\gL^+)$, where:
  \begin{itemize}
  \item $\gL$ is a rank $d$ projective $\cO_{K'}\otimes_{\Zp} A$-module,
    with a $\Gal(K'/K)$-semilinear, $A$-linear action of~$\Gal(K'/K)$;
  \item $\gL^+$ is an $\cO_{K'}\otimes_{\Zp} A$-submodule of~$\gL$, which is
    locally on~$\Spec A$ a direct summand of~$\gL$ as an $A$-module
    (or equivalently, for which $\gL/\gL^+$ is  projective as an $A$-module),
    and is preserved by~$\Gal(K'/K)$.
  \end{itemize}
  We set~$\cM_{\loc}^{K'/K}:=\varinjlim_a\cM_{\loc}^{K'/K,a}$, so
  that~$\cM_{\loc}^{K'/K}$ is a $\varpi$-adic formal algebraic stack,
of finite presentation over $\Spf \cO$
  (indeed, it is easily seen to be the $\varpi$-adic completion of an algebraic stack
of finite presentation over $\Spec \cO$).
\end{defn}

\begin{defn}By Lemma~\ref{lem:cokernel is projective}, we have a natural morphism
  $\Psi:\cC_{1,K'}^{\dd}\to\cM_{\loc}^{K'/K}$, which takes a Breuil--Kisin module
  with descent data~$\gM$ of height~$1$ to the pair\[(\gM/E(u)\gM, \im\Phi_{\gM}/E(u)\gM).\]
\end{defn}

\begin{remark}\label{rem:Mloc-is-too-big}
  The definition of the stack $\cM_{\loc}^{K'/K,a}$ does not include
  any condition that mirrors the commutativity between the Frobenius
  and the   descent data on a Breuil--Kisin module,   and so in general the
  morphism $\Psi_A:\cC_{1,K'}^{\dd}(A)\to\cM_{\loc}^{K'/K}(A)$ cannot be essentially surjective.
\end{remark}

It will be convenient to consider
the twisted group rings~$\gS_A[\Gal(K'/K)]$
and~$(\cO_{K'}\otimes_{\Zp}A)[\Gal(K'/K)]$, 
in which the elements $g\in\Gal(K'/K)$ obey the following
commutation relation with elements $s\in \gS_A$ (resp.\
$s\in\cO_{K'}$):
\[g\cdot s=g(s)\cdot g. \]
(In the literature these twisted group rings are more often written as
$\gS_A * \Gal(K'/K)$, $(\cO_{K'}\otimes_{\Zp}A) * \Gal(K'/K)$, in order
to distinguish them from the usual (untwisted) group rings, but as we
will only use the twisted versions in this paper, we prefer to use
this notation for them.)

By definition, endowing a finitely generated ~$\gS_A$-module~$P$ with a semilinear
$\Gal(K'/K)$-action is equivalent to giving it the structure of a
left~$\gS_A[\Gal(K'/K)]$-module. If~$P$ is projective as an
$\gS_A$-module, then it is also projective as
an $\gS_A[\Gal(K'/K)]$-module. Indeed, $\gS_A$ is a direct summand
of~$\gS_A[\Gal(K'/K)]$ as a $\gS_A[\Gal(K'/K)]$-module, given by the
central idempotent~$\frac{1}{|\Gal(K'/K)|}\sum_{g\in\Gal(K'/K)}g$,
so~$P$ is a direct summand of the projective module
$\gS_A[\Gal(K'/K)]\otimes_{\gS_A}P$. Similar remarks apply to the case
of  $(\cO_{K'}\otimes_{\Zp}A)[\Gal(K'/K)]$-modules.

\begin{thm}\label{thm: map from Kisin to general local model is
    formally smooth}The
  morphism~$\Psi:\cC_{1,K'}^{\dd}\to\cM_{\loc}^{K'/K}$ is representable by
algebraic spaces and smooth.  
\end{thm} 
\begin{proof}
We first show that the morphism~$\Psi$ is formally smooth,
in the sense that it satisfies the infinitesimal lifting criterion for nilpotent
thickenings of affine test objects \cite[Defn.\
2.4.2]{EGstacktheoreticimages}.
For this, 
we follow the proof of~\cite[Prop.\ 2.2.11]{kis04} (see also the proof
  of~\cite[Thm.\ 4.9]{2015arXiv151007503C}). Let $A$ be an $\cO/\varpi^a$-algebra and $I\subset A$ be a nilpotent ideal. 
  Suppose that we are given  $\gM_{A/I}\in \cC_{1,K'}^{\dd}(A/I)$
  and $(\gL_A,\gL^+_A)\in \cM_{\loc}^{K'/K}(A)$ 
together with an isomorphism
  \[
  \Psi(\gM_{A/I})\isoto (\gL_{A},\gL^+_{A})\otimes_{A}A/I=:(\gL_{A/I},\gL^+_{A/I}).
  \]
  We must show that there exists
  $\gM_{A}\in \cC_{1,K'}^{\dd}(A)$ together with an isomorphism 
  $\Psi(\gM_{A})\isoto
  (\gL_{A},\gL^+_{A})$ lifting the
  given isomorphism. 
  
As explained above, we can and do think of~$\gM_{A/I}$ as a finite
projective $\gS_{A/I}[\Gal(K'/K)]$-module, and~$\gL_A$ as a finite
projective $(\cO_{K'}\otimes_{\Zp}A)[\Gal(K'/K)]$-module. Since the
closed 2-sided ideal 
of~$\gS_{A}[\Gal(K'/K)]$ generated by~$I$ consists of
nilpotent elements, we may lift
$\gM_{A/I}$ to a finite
projective~$\gS_{A}[\Gal(K'/K)]$-module~$\gM_A$. (This is presumably
standard, but for lack of a reference we indicate a proof. In fact the
proof of ~\cite[\href{https://stacks.math.columbia.edu/tag/0D47}{Tag
  0D47}]{stacks-project} goes over unchanged to our setting. Writing
~$\gM_{A/I}$ as a direct summand of a finite free
$\gS_{A/I}[\Gal(K'/K)]$-module $F$, it is enough to lift the
corresponding idempotent in~$\End_{\gS_{A/I}[\Gal(K'/K)]}(F)$, which
is possible
by~\cite[\href{http://stacks.math.columbia.edu/tag/05BU}{Tag
  05BU}]{stacks-project}; see also~\cite[Thm.\ 21.28]{MR1125071} for another proof of the existence
of lifts of idempotents in this generality.) Note that since $\gM_{A/I}$ is of
rank~$d$ as a projective $\gS_{A/I}$-module, $\gM_A$ is of rank~$d$ as
a projective~$\gS_A$-module.

Since~$\gM_A/E(u)\gM_A$ is a projective
$(\cO_{K'}\otimes_{\Zp}A)[\Gal(K'/K)]$-module, we may lift the
composite \[\gM_A/E(u)\gM_A \onto \gM_{A/I}/E(u)\gM_{A/I} \isoto
\gL_{A/I}\] to a morphism $\theta : \gM_A/E(u)\gM_A \to \gL_A$.  
Since the composite of
$\theta$ with $\gL_A \onto \gL_{A/I}$ is
surjective, it follows by Nakayama's lemma that $\theta$ is
surjective. But a surjective map of projective modules of the same
rank is an isomorphism, and so $\theta$ is
an isomorphism lifting the given isomorphism $\gM_{A/I}/E(u)\gM_{A/I} \isoto
\gL_{A/I}$. 

We let~$\gM_A^+$ denote the
preimage in~$\gM_A$ of~$\theta^{-1}(\gL^+_A)$. The image of the induced
map $f : \gM_A^+ \subset \gM_A \onto \gM_A \cotimes_A A/I \cong  \gM_{A/I}$ is precisely $\im
\Phi_{\gM_{A/I}}$, since the same is true modulo $E(u)$ and because
$\gM^+_A$, $\im \Phi_{\gM_{A/I}}$
contain $E(u)\gM_A$, $E(u) \gM_{A/I}$ respectively. Observing that 
\[ \gM_A / \gM_A^+ \cong (\gM_A/E(u)\gM_A)/(\gM_A^+/E(u)\gM_A) \cong
\gL_A/\gL^+_A \] 
we deduce that $\gM_A/\gM_A^+$ is projective as an $A$-module, whence
$\gM_A^+$ is an $A$-module direct summand of $\gM_A$. By the same
argument $\im \Phi_{\gM_{A/I}}$ is an $A$-module direct summand of
$\gM_{A/I}$, and we conclude that the map $\gM_A^+ \cotimes_{A} A/I \to
\im \Phi_{\gM_{A/I}}$ induced by $f$ is an isomorphism.

Finally, we have the diagram
\[\xymatrix{\varphi^*\gM_A\ar[d]&\gM^+_A\ar[d]\\
    \varphi^*\gM_{A/I}\ar[r]& \im \Phi_{\gM_{A/I}}} \]
where the horizontal arrow is given by~$\Phi_{\gM_{A/I}}$, and the
right hand vertical arrow is~$f$. 
Since~$\varphi^*\gM_A$ is a
projective $\gS_{A}[\Gal(K'/K)]$-module, we may find a morphism of
$\gS_{A}[\Gal(K'/K)]$-modules $\varphi^*\gM_A\to \gM_A^+$ which fills in
the commutative square. Since the composite $\varphi^*\gM_A\to
\gM_A^+\to \im \Phi_{\gM_{A/I}} \cong \gM_A^+ \cotimes_{A} A/I $ is surjective, it follows by Nakayama's
lemma 
that
$\varphi^*\gM_A\to \gM_A^+$ is also surjective, and the composite
$\varphi^*\gM_A\to \gM_A^+\subset\gM_A$ gives a map
$\Phi_{\gM_A}$. Since $\Phi_{\gM_A}[1/E(u)]$ is a surjective
map of projective modules of the same rank, it is an isomorphism, and
we see that $\gM_A$ together with $\Phi_{\gM_A}$ is our required lifting to a
Breuil--Kisin module of rank~$d$ with descent data.

Since the source and target of $\Psi$ are of finite presentation
over $\Spf \cO$, and $\varpi$-adic, we see that $\Psi$ is representable
by algebraic spaces
(by \cite[Lem.~7.10]{Emertonformalstacks}) and locally of finite presentation
(by \cite[Cor.~2.1.8]{EGstacktheoreticimages}
and~\cite[\href{http://stacks.math.columbia.edu/tag/06CX}{Tag
  06CX}]{stacks-project}).  Thus $\Psi$ is in fact smooth (being formally
smooth and locally of finite presentation).
 \end{proof}

We now show that the inertial type of a Breuil--Kisin module is visible on the
local model.

\begin{lem}\label{lem: local model types map}There is a natural
  morphism  $\cM_{\loc}^{K'/K}\to \pi_0(\Rep_{I(K'/K)})$. 
\end{lem}
\begin{proof} 
  The morphism $\cM_{\loc}^{K'/K}\to \pi_0(\Rep_{I(K'/K)})$ is defined by sending
$(\gL,\gL^+)\mapsto\gL/\pi'$. More precisely, $\gL/\pi'$ is a rank~$d$
projective $k'\otimes_{\Zp}A$-module with a linear action of
$I(K'/K)$, so determines an $A/p$-point of~$\pi_0(\Rep_{I(K'/K)})=\prod_{i = 0}^{f'-1}
\pi_0(\Rep_d\bigl(I(K'/K)\bigr))$. Since the target is a disjoint
union of copies of~$\Spec\cO$, the morphism $\Spec A/p\to
\pi_0(\Rep_{I(K'/K)})$ lifts uniquely to a morphism $\Spec A \to
\pi_0(\Rep_{I(K'/K)})$, as required. 
\end{proof}

\begin{defn}\label{defn: local model stack fixed type}We let \[\cM_{\loc}^{(\tau_i)}:=\cM_{\loc}^{K'/K}\otimes_{\pi_0(\Rep_{I(K'/K)})}\pi_0(\Rep_{I(K'/K),\{\tau_i\}}).\]  If each $\tau_i = \tau$ for some fixed
  $\tau$, we write $\cM_{\loc}^\tau$ for $\cM_{\loc}^{(\tau_i)}$. By
  Lemma~\ref{lem:tame reps}, $\cM_{\loc}^{K'/K}$ is the disjoint union
  of the open and closed substacks~$\cM_{\loc}^{(\tau_i)}$.
  \end{defn}

\begin{lem} 
 We have~$\cC^{(\tau_i)}=\cC^{\dd}\times_{\cM_{\loc}^{K'/K}}\cM_{\loc}^{(\tau_i)}$. 
\end{lem}
\begin{proof}
  This is immediate from the definitions.
\end{proof}

In particular, $\cC^{(\tau_i)}$ is a closed substack of $\cC^{\dd}$.

\subsection{Local models:\ determinant conditions}\label{subsec: local
  models rank two}
Write~$N=K\cdot W(k')[1/p]$, so that~$K'/N$ is
totally ramified. Since~$I(K'/K)$ is cyclic of order prime to~$p$  and
acts trivially on~$\cO_N$, we may
write  \numequation\label{eqn: breaking up local model over
  xi}(\gL,\gL^+)=\oplus_{\xi}(\gL_{\xi},\gL^+_\xi)\end{equation}where the sum
is over all characters~$\xi:I(K'/K)\to\cO^\times$, and  
$\gL_\xi$ (resp.\ $\gL^+_\xi$) is the $\cO_N\otimes A$-submodule of~$\gL$
(resp.\ of~$\gL^+$) on which~$I(K'/K)$ acts through~$\xi$.
\begin{defn}\label{defn: strong determinant condition general local
    model} 
We say that an object $(\gL,\gL^+)$  of~$\cM_{\loc}^{K'/K}(A)$
\emph{satisfies the strong determinant condition} if Zariski
  locally on $\Spec A$ the
  following condition holds: for all $a\in\cO_N$ and all $\xi$, we
  have \numequation\label{eqn: strong det condn}
  \det{\!}_A(a|\gL^+_\xi)
  =\prod_{\psi:N\into E}\psi(a) 
  \end{equation}
  as polynomial functions on $\cO_N$ 
  in the sense of~\cite[\S5]{MR1124982}. 
  \end{defn}

  \begin{remark}\label{rem: explicit version of determinant condition
      local models version} 
An explicit version of this determinant condition is stated, in this generality, in~\cite[\S 2.2]{kis04}, specifically in the proof of \cite[Prop.\ 2.2.5]{kis04}. 
  We recall this here, with our notation. 
  We have a direct sum decomposition
  \[
  \cO_N\otimes_{\Z_p}A\toisom \bigoplus_{\sigma_i:k'\hookrightarrow \F}\cO_{N}\otimes_{W(k'),\sigma_i}A
  \]
  Recall that $e_{i}\in \cO_{N}\otimes_{\Z_p}\cO$ denotes the idempotent that identifies $e_i\cdot  \cO_N\otimes_{\Z_p}A$
  with the summand $\cO_{N}\otimes_{W(k'),\sigma_i}A$. For $j=0,1,\dots, e-1$, let $X_{j,\sigma_i}$ be an indeterminate. 
  Then the strong determinant condition on~$(\gL,\gL^+)$ is that for
  all~$\xi$, we have
  
\numequation\label{eqn: explicit strong det condn}  
\det{\!}_{A}\left(\sum_{j,\sigma_i}e_i\pi^jX_{j,\sigma_i}
\mid \gL^+_{\xi}\right) = 
\prod_{\psi}\sum_{j,\sigma_i}\left(\psi(e_i\pi^j)X_{j,\sigma_i}\right),
 \end{equation}
  where $j$ runs over $0,1,\dots,e-1$, $\sigma_i$ over embeddings $k'\hookrightarrow \F$, and $\psi$ over
  embeddings $\cO_N\hookrightarrow \cO$. Note that $\psi(e_i)=1$ if $\psi|_{W(k')}$ lifts $\sigma_i$
  and is equal to $0$ otherwise. 
  \end{remark}

\begin{defn}\label{defn: C^tau BT}
We write~$\cM_{\loc}^{K'/K,\BT}$ for the substack of~$\cM_{\loc}^{K'/K}$
given by those~$(\gL,\gL^+)$ which satisfy the strong determinant
condition. For each (possibly mixed) type~$(\tau_i)$, we write
$\cM_{\loc}^{(\tau_i),\BT}:=\cM_{\loc}^{(\tau_i)}\times_{\cM_{\loc}^{K'/K}}\cM_{\loc}^{K'/K,\BT}$.

Suppose for the remainder of this section that $d=2$ and $h=1$, so
that $\cC^{\dd}$ consists of Breuil--Kisin modules of rank two and height at
most $1$.
We then
  set~$\cC^{\dd,\BT}:=\cC^{\dd}\times_{\cM_{\loc}^{K'/K}}\cM_{\loc}^{K'/K,\BT}$,
  and
  $\cC^{(\tau_i),\BT}:=\cC^{(\tau_i)}\times_{\cM_{\loc}^{(\tau_i)}}\cM_{\loc}^{(\tau_i),\BT}$. 

A Breuil--Kisin module $\gM \in
  \cC^{\dd}(A)$ is said to satisfy the strong determinant
  condition if and only if its image $\Psi(\gM) \in
  \cM_{\loc}^{K'/K}(A)$ does, i.e.\ if and only if it lies in $\cC^{\dd,\BT}$.

\end{defn}
\begin{prop}
  \label{prop:C tau BT is algebraic}$\cC^{(\tau_i),\BT}$ \emph{(}resp.\ $\cC^{\dd,\BT}$\emph{)} is a closed
  substack of $\cC^{(\tau_i)}$ \emph{(}resp.\ $\cC^{\dd}$\emph{)}; in particular, it is a $\varpi$-adic formal
   algebraic
  stack of finite presentation over~$\cO$. 
\end{prop}
\begin{proof}
  This is immediate from Corollary~\ref{cor: C^dd,a is an algebraic
    stack} and the definition of the strong determinant
  condition as an equality of polynomial functions.
\end{proof}

  \begin{rem}\label{rem: strong determinant condition motivated by
      flat closure} 
The motivation for imposing the strong determinant condition is as
follows. One can take the flat part (in the sense of~\cite[Ex.\
9.11]{Emertonformalstacks}) of the $\varpi$-adic formal
stack~$\cC^{\dd}$, and on this flat part, one can impose the condition
that the corresponding Galois representations have all pairs of labelled
Hodge--Tate weights equal to~$\{0,1\}$; that is, we can consider the
substack of~$\cC^{\dd}$ corresponding to the Zariski closure of the
these Galois representations. 

We will soon see that $\cC^{\dd,\BT}$ is flat (Corollary~\ref{cor:
  Kisin moduli consequences of local models}).  By Lemma~\ref{lem: O points of moduli
  stacks} below, it follows that the substack of the previous paragraph is equal
to~$\cC^{\dd,\BT}$; so we may think of the strong determinant
condition as being precisely the condition which imposes this condition on the
labelled Hodge--Tate weights, and results in a formal stack which is
flat over~$\Spf\cO$. Since the inertial
types of $p$-adic Galois representations are unmixed, it is natural
from this perspective to expect
that $\cC^{\dd,\BT}$ should be the disjoint union of the stacks
$\cC^{\tau,\BT}$ for \emph{unmixed} types, and indeed this will be
proved shortly at Corollary~\ref{cor: BT C is union of tau C}.
  \end{rem}

To compare the strong determinant condition to the condition that the
type of a Breuil--Kisin module is unmixed, we make some observations about these conditions in the case of finite
field coefficients.

\begin{lem}
  \label{lem: determinant condition explicit finite field local models
  version}Let $\F'/\F$
  be a finite extension, and let~$(\gL,\gL^+)$ be an object
  of~$\cM_{\loc}^{K'/K}(\F')$. Then~$(\gL,\gL^+)$
satisfies the strong determinant condition if and only if the
  following property holds: for each~$i$ and for each $\xi: I(K'/K)
  \to \cO^\times$ we have $\dim_{\F'} (\gL^+_i)_{\xi} = e$. 
\end{lem}
\begin{proof} This is proved in a similar way to~\cite[Lemma
  2.5.1]{kis04},  using the explicit formulation of the strong
    determinant condition from Remark~\ref{rem: explicit version
      of determinant condition local models version}. 
    In the
    notation of that remark, we see that the strong determinant
    condition holds at~$\xi$ if and only if for each embedding
    $\sigma_i:k'\into\F$ we have
\numequation\label{eqn: explicit strong det condn slightly rewritten}  
\det{\!}_{A}\left(\sum_{j}\pi^jX_{j,\sigma_i}
\mid (\gL^+_i)_{\xi}\right) = 
\prod_{\psi}\sum_{j}\left(\psi(\pi)^jX_{j,\sigma_i}\right),
 \end{equation} where the product runs over the embeddings~$\psi:\cO_N\into\cO$
    with the property that~$\psi|_{W(k')}$ lifts~$\sigma_i$.
Since $\pi$ induces a nilpotent 
endomorphism of $(\gL^+_i)_{\xi}$ the left-hand side of \eqref{eqn: explicit strong det condn slightly rewritten} evaluates to $X_{0,\sigma_i}^{\dim_{\F'} (\gL^+_i)_{\xi}}$
while the right-hand side, which can be viewed as a norm
from $\cO_N \otimes_{\Zp} \F'$ down to $W(k') \otimes_{\Zp} \F'$, is equal to
$X_{0,\sigma_i}^e$.
     \end{proof}

\begin{lem} 
  \label{lem: determinant condition explicit finite field}Let $\F'/\F$
  be a finite extension, and let $\gM$ be a Breuil--Kisin module of rank~2 and
  height at most one with $\F'$-coefficients and descent data.
  
\begin{enumerate}
\item  $\gM$ satisfies the strong determinant condition if and only if the
  following property holds: for each~$i$ and for each $\xi
  : I(K'/K) \to \cO^\times$ we have $\dim_{\F'} (\im
  \Phi_{\gM,i}/E(u)\gM_i)_{\xi} = e$.
\item If~$\gM$ satisfies the strong determinant condition, then the determinant of~$\Phi_{\gM,i}$
  with respect to some choice of basis has $u$-adic
  valuation~$e'$. 
  \end{enumerate}
\end{lem}
\begin{proof} 
  The first part is immediate from Lemma~\ref{lem: determinant condition explicit finite field local models
  version}. 
  For the second part, let $\Phi_{\gM,i,\xi}$ be the restriction
  of $\Phi_{\gM, i}$ to $\varphi^*(\gM_{i-1})_{\xi}$. We think of
  $\gM_{i}$ and $\varphi^*(\gM_{i-1})$ as free $\F'[\![v]\!]$-modules of rank $2e(K'/K)$,
  where $v=u^{e(K'/K)}$. 
  We have
  \[
  \det{\!}_{\F'[\![v]\!]}(\Phi_{\gM,i}) =\left(\det{\!}_{\F'[\![u]\!]}(\Phi_{\gM,i})\right)^{e(K'/K)}.
  \] 
  Since $\Phi_{\gM,i}$ commutes with the descent datum, we also have 
  \[
  \det{\!}_{\F'[\![v]\!]}(\Phi_{\gM, i}) =\prod_{\xi} \det{\!}_{\F'[\![v]\!]}(\Phi_{\gM,i,\xi}),
  \]
  where $\xi$ runs over the $e(K'/K)$ characters $I(K'/K)\to \cO^\times$. 
  
  The proof of the 
  second part of~\cite[Lemma 2.5.1]{kis04} implies
  that, for each $\xi$, $\det_{\F'[\![v]\!]}(\Phi_{\gM,i,\xi})$ is $v^e=u^{e'}$ times a unit. 
  Indeed, each $\gM_{i,\xi}$ is a free $\F'[\![v]\!]$-module of rank $2$. It admits
  a basis $\{e_{1,\xi},e_{2,\xi}\}$ such that $\im\Phi_{\gM,i,\xi} = \langle v^{i}e_{1,\xi}, v^je_{2,\xi}\rangle$ 
  for some non-negative integers $i,j$. The strong determinant condition on 
  $(\im \Phi_{\gM,\i,\xi}/v^e\gM_{i,\xi})$
  implies that $i+j = 2e-e =e$, and this is precisely the $v$-adic valuation of $\det_{\F'[\![v]\!]}(\Phi_{\gM,i,\xi})$.
  We deduce that the $u$-adic valuation of  
  $\left(\det{\!}_{\F'[\![u]\!]}(\Phi_{\gM,i})\right)^{e(K'/K)}$ is $e(K'/K)\cdot e'$, which implies
  the second part of the lemma.  
   \end{proof}

By contrast, we have the following criterion for the type of a
Breuil--Kisin module to be unmixed.

 \begin{prop}\label{prop:types are unmixed}
  Let $\F'/\F$ be a finite extension, and let $\gM$ be a Breuil--Kisin module
  of rank~2 and height at most one with $\F'$-coefficients and descent
  data.  Then
the type of 
  $\gM$ is unmixed if and only if $\dim_{\F'} (\im
  \Phi_{\gM,i}/E(u)\gM_i)_{\xi}$ is independent of $\xi$ for each
  fixed $i$.  In particular, if
  $\gM$ satisfies the strong determinant condition, then the type of
  $\gM$ is unmixed. 
\end{prop}

\begin{proof} We begin the proof of the first part with the following observation. Let $\Lambda$ be a rank two free $\F'[[u]]$-module with an action of
 $I(K'/K)$ that is $\F'$-linear and $u$-semilinear with respect to a
 character $\chi$  (i.e., such that $g(u\lambda) = \chi(g)u
 g(\lambda)$ for $\lambda \in \Lambda$). In particular $I(K'/K)$ acts on $\Lambda/u\Lambda$ through a
 pair of characters which we call $\eta$ and $\eta'$. Let $\Lambda' \subset \Lambda$ be a rank two $I(K'/K)$-sublattice. We claim that there are integers $m, m'
 \ge 0$ such that the multiset of characters of $I(K'/K)$ occurring
in $\Lambda/\Lambda'$ has the shape $$\{ \eta \chi^i : 0 \le i < m\} \cup \{ \eta'
\chi^i : 0 \le j < m'\}$$ and the multiset of
characters occurring in $\Lambda'/u\Lambda'$ is $\{\eta \chi^m, \eta' \chi^{m'}\}$.

To check the claim we proceed by induction on $\dim_{\F'}
\Lambda/\Lambda'$, the case $\Lambda=\Lambda'$ being
trivial. Suppose $\dim_{\F'} \Lambda/\Lambda'=1$, so that $\Lambda'$ lies between $\Lambda$ and
$u\Lambda$. Consider the chain of containments $\Lambda \supset \Lambda' \supset u\Lambda
\supset u\Lambda'$. If without loss of generality $I(K'/K)$ acts on $\Lambda/\Lambda'$
via $\eta$, then it acts on $\Lambda'/u\Lambda$ by $\eta'$ and $u\Lambda/u\Lambda'$ by $\chi
\eta$, proving the claim with $m=1$ and $m'=0$. The general case follows by iterated application of the case
$\dim_k \Lambda/\Lambda'=1$, noting that since $I(K'/K)$ is abelian
the quotient $\Lambda/\Lambda'$ has a filtration by $I(K/K')$-submodules
whose graded pieces have dimension $1$.

Now return to the statement of the proposition. Let $(\tau_i)$ be the mixed type of $\gM$ and write $\tau_{i-1} = \eta
\oplus \eta'$. We apply the preceding observation with $\Lambda = \im
  \Phi_{\gM,i}$ and $\Lambda' = E(u)\gM_i = u^{e'} \gM_i$. Note that
  $\chi$ is a generator of the cyclic group $I(K'/K)$ of order $e'/e$. 
  Since
  $\Phi_{\gM}$ commutes with descent data, the group $I(K'/K)$ acts
  on $\Lambda/u\Lambda$ via $\eta$ and $\eta'$.   Then the
  the multiset 
\[\{ \eta \chi^i : 0 \le i < m\} \cup \{ \eta'
\chi^i : 0 \le j < m'\}\]
contains each character of $I(K'/K)$ with equal multiplicity if and
only if  one of $\eta,\eta'$  is the successor to $\eta\chi^{m-1}$ in the list
$\eta,\eta\chi,\eta\chi^2,\ldots$, and the other is the successor to
$\eta'\chi^{m'-1}$ in the list $\eta',\eta'\chi,\eta'\chi^2,\ldots$,
i.e., if and only if  $\{\eta\chi^m,\eta'\chi^{m'}\} = \{ \eta,\eta'\}$. Since $\gM_i/u\gM_i
\cong u^{e'}\gM_i/u^{e'+1}\gM_i = \Lambda'/u\Lambda'$, this occurs if
and only if
that $\tau_i = \tau_{i-1}$.

Finally, the last part of the proposition follows immediately from the
  first part and Lemma~\ref{lem: determinant condition explicit finite
    field}.\end{proof}

\begin{cor} 
  \label{cor: BT C is union of tau C}$\cC^{\dd,\BT}$ is the disjoint
  union of its closed substacks $\cC^{\tau,\BT}$. 
\end{cor}
\begin{proof}
  This follows from Propositions~\ref{prop: Ctau is
    an algebraic stack} and~\ref{prop:types are unmixed}. Indeed, from Proposition~\ref{prop: Ctau is
    an algebraic stack}, it suffices to show that if~$(\tau_i)$ is a
  mixed type,
  and~$\cC^{(\tau_i),\BT}$ is
  nonzero, then $(\tau_i)$ is in fact an unmixed type. Indeed,
  note that if $\cC^{(\tau_i),\BT}$ is nonzero, then it contains a dense set of
  finite type points, so in particular contains an $\F'$-point for
  some finite extension $\F'/\F$. It follows from
  Proposition~\ref{prop:types are unmixed} that the type is unmixed,
  as required.
\end{proof}

  \begin{rem}
    \label{rem: we don't consider mixed type local models} Since our
    primary interest is in Breuil--Kisin modules, we will have no further need
    to consider the stacks $\cM^{(\tau_i),\BT}_{\loc}$ or
    $\cC^{(\tau_i),\BT}$ for types that are not unmixed.
  \end{rem}

Let~$\tau$ be a tame
type; since $I(K'/K)$ is cyclic, we can write~$\tau=\eta\oplus\eta'$
for (possibly equal) characters $\eta,\eta':I(K'/K)\to \cO^\times$. 
Let~$(\gL,\gL^+)$ be an object of $\cM_{\loc}^{\tau}(A)$. Suppose that~$\xi\ne\eta,\eta'$. Then elements of $\gL_{\xi}$ are
divisible by $\pi'$ in $\gL$, and so multiplication by~$\pi'$
induces an isomorphism 
 of
projective $e_i(\cO_N \otimes A)$-modules of equal rank  \[p_{i,\xi} :
  e_i \gL_{\xi\chi_i^{-1} }
\isoto e_i\gL_{\xi} \]
where~$\chi_i:I(K'/K)\to\cO^{\times}$ is the
character sending $g\mapsto \sigma_i(h(g))$.  The 
induced map  \[p_{i,\xi}^+ : e_i \gL^+_{\xi\chi_i^{-1} }
\longrightarrow  e_i\gL^+_{\xi} \] is in particular an injection.
 The following lemma will be useful for checking the strong
 determinant condition when comparing various different stacks of local model
 stacks. 

\begin{lem} \label{lem: local model of type only have to check the strong
      determinant condition on that type}Let~$(\gL,\gL^+)$ be an object of $\cM_{\loc}^{\tau}(A)$. Then
  $(\gL,\gL^+)$ is an object of $\cM_{\loc}^{\tau,\BT}(A)$ if and only
  if both
  \begin{enumerate}
\item   the condition~\eqref{eqn: strong det condn} holds 
  for~$\xi=\eta,\eta'$, and 
  \item the injections $p_{i,\xi}^+ : e_i \gL^+_{\xi\chi_i^{-1} }
\longnto{\pi'}  e_i\gL^+_{\xi} $ are isomorphisms for all $\xi \neq
\eta,\eta'$ and for all $i$.
  \end{enumerate}
The second condition is equivalent to
\begin{enumerate}
\item[($\textrm{2}\,'$)] we have $(\gL^+/\pi' \gL^+)_{\xi} = 0$ for all $\xi \neq \eta,\eta'$.
\end{enumerate}
   \end{lem}
     \begin{proof}

The equivalence between (2) and ($2'$) is straightforward.
Suppose now that $\xi \neq \eta,\eta'$. Locally on $\Spec A$ the module $e_i \gL^{+}_{\xi \chi_i^{-1}}$ is by
definition a
  direct summand of $e_i \gL_{\xi \chi_i^{-1}}$.  Since $p_{i,\xi}$ is an isomorphism, the image of $p^+_{i,\xi}$ is locally on $\Spec A$ a direct summand of
$e_i\gL^+_{\xi}$. 
Under the assumption that~(\ref{eqn: explicit strong det condn slightly
  rewritten}) holds for $i$ and~ $\xi \chi_i^{-1}$, the condition~(\ref{eqn: explicit strong det condn slightly
  rewritten}) for $i$ and $\xi$ is therefore  equivalent to the surjectivity of
$p^+_{i,\xi}$.  The lemma follows upon  noting that $\chi_i$ is a generator
of the
group of characters $I(K'/K)\to \cO^{\times}$.  
\end{proof}

To conclude this section we describe the $\cO_{E'}$-points of $\cC^{\dd,\BT}$, for $E'/E$
a finite extension; recall that our convention is that a two-dimensional Galois
representation is Barsotti--Tate if all its labelled pairs of
Hodge--Tate weights are equal to~$\{0,1\}$ (and not just that all of
the labelled Hodge--Tate weights are equal to~$0$ or~$1$).
\begin{lem} 
\label{lem: O points of moduli stacks}Let $E'/E$ be a finite
extension. Then the $\Spf(\cO_{E'})$-points of $\cC^{\dd,\BT}$ correspond
precisely to the potentially Barsotti--Tate Galois representations
$G_K\to\GL_2(\cO_{E'})$ which become Barsotti--Tate over~$K'$; and the
$\Spf(\cO_{E'})$-points of $\cC^{\tau,\BT}$ correspond to those
representations which are potentially Barsotti--Tate of type~$\tau$.
\end{lem}
\begin{proof} 
In light of Lemma~\ref{lem:points-of-Cdd} and the first sentence of
Remark~\ref{rem: Kisin modules don't always have a type}, we are
reduced to checking that a Breuil--Kisin module of rank $2$ and height $1$ with $\cO_{E'}$-coefficients
and descent data corresponds to a potentially Barsotti--Tate
representation if and only if it satisfies the strong determinant
condition, as well as checking that the descent data on the Breuil--Kisin module
matches the type of the corresponding Galois representation. 

Let $\gM_{\cO_{E'}} \in\cC^{\dd,\BT}(\Spf(\cO_{E'}))$. Plainly
$\gM_{\cO_{E'}}$ satisfies the strong determinant condition if and
only if $\gM := \gM_{\cO_{E'}}[1/p]$ satisfies the strong determinant
condition (with the latter having the obvious meaning). 
Consider the filtration
\[
\mathrm{Fil}^1(\varphi^*(\gM_i)):=\{m\in \varphi^*(\gM_i)\mid \Phi_{\gM,i+1}(m)\in 
E(u)\gM_{i+1}\}\subset \varphi^*\gM_{i}
\]
inducing
\[
\mathrm{Fil}^1_i\subset \varphi^*(\gM_{i})/ E(u)\varphi^*(\gM_{i}).
\]
Note that $\varphi^*(\gM_{i})/E(u)\varphi^*(\gM_{i})$ 
is isomorphic to a free $K'\otimes_{W(k'),\sigma_i}E'$-module 
of rank $2$. Then $\gM$ corresponds to a Barsotti--Tate Galois representation
\[
G_{K'}\to \GL_{2}(E')
\]
if and only, if for every $i$, $\mathrm{Fil}^1_i$ is isomorphic to $K'\otimes_{W(k'),\sigma_i}E'$ as a 
$K'\otimes_{W(k'),\sigma_i}E'$-submodule of $\varphi^*\gM_{i}/E(u)\varphi^*\gM_{i}$. 
This follows, for example, by specialising the proof of~\cite[Cor. 2.6.2]{kisindefrings}
to the Barsotti--Tate case (taking care to note that the conventions
of \emph{loc.\ cit.} for Hodge--Tate weights and for Galois
representations associated to Breuil--Kisin modules are both dual to ours). 

Let $\xi: I(K'/K)\to\cO^\times$ be a character. Consider the filtration
\[
\mathrm{Fil}^1_{i,\xi}\subset \varphi^*(\gM_i)_{\xi}/E(u)\varphi^*(\gM_i)_{\xi}\simeq N^2\otimes_{W(k'),\sigma_i}E'
\]
induced by $\mathrm{Fil}^1_i$. The strong determinant condition on 
$(\im \Phi_{\gM,i+1}/E(u)\gM_{i+1})_{\xi}$
holds if and only if $\mathrm{Fil}^1_{i,\xi}$ 
is isomorphic to $N\otimes_{W(k'),\sigma_i}E'$. By~\cite[Lemma 5.10]{2015arXiv151007503C},
we have an isomorphism of $K'\otimes_{W(k'),\sigma_i}E'$-modules 
\[
\mathrm{Fil}^1_{i}\simeq K' \otimes_{N} \mathrm{Fil}^1_{i,\xi}. 
\] 
This, together with the previous paragraph, allows us to conclude. 
Note that, since $u$ acts invertibly when working with $E'$-coefficients
and after quotienting by $E(u)$, the argument is independent of the
choice of character $\xi$.  

For the statement about types, let $S_{K'_0}$ be Breuil's period ring
(see e.g.\ \cite[\S5.1]{MR1804530}) endowed with the evident action of
$\Gal(K'/K)$ compatible with the embedding $\gS \into S_{K'_0}$.
Here $K'_0$ is the maximal unramified extension in $K'$.  Recall that by \cite[Cor.\
3.2.3]{MR2388556} there is a canonical  $(\varphi,N)$-module isomorphism
\numequation\label{eq:compare-M-and-D}  S_{K'_0} \otimes_{\varphi,\gS} \gM \cong S_{K'_0} \otimes_{K'_0}
  D_{\mathrm{pcris}}(T(\gM)). \end{equation}
One sees from its construction 
that the 
isomorphism \eqref{eq:compare-M-and-D} is in fact equivariant for the
action of $I(K'/K)$, 
and
the claim follows by reducing modulo $u$ and all its divided
powers.
\end{proof}

\subsection{Change of extensions}\label{subsec: change of extension K
  prime over K}
We now discuss the behaviour of various of our constructions under
change of~$K'$. Let~$L'/K'$ be a tame extension such that~$L'/K$ is
Galois. We suppose that we have fixed a uniformiser~$\pi''$ of~$L'$
such that
$(\pi'')^{e(L'/K')}=\pi'$. Let~$\gS'_A:=(W(l')\otimes_{\Zp}A)[[u]]$,
where~$l'$ is the residue field of~$L'$, and let~$\Gal(L'/K)$ and~$\varphi$ act
on~$\gS'_A$ via the prescription of Section~\ref{subsec: kisin modules
  with dd} (with~$\pi''$ in place of~$\pi'$). 

There is a natural injective ring
homomorphism~$\cO_{K'}\otimes_{\Zp}A\to\cO_{L'}\otimes_{\Zp}A$, which
is equivariant for the action of~$\Gal(L'/K)$ (acting on the source via the
natural surjection~$\Gal(L'/K)\to\Gal(K'/K)$). There is also an
obvious injective ring homomorphism $\gS_A\to\gS'_A$
sending~$u\mapsto u^{e(L'/K')}$, which is equivariant for the actions
of~$\varphi$ and~$\Gal(L'/K)$; we have $(\gS'_{A})^{\Gal(L'/K')}=\gS_A$. 
If~$\tau$ is an inertial type for~$I(K'/K)$,
we write~$\tau'$ for the corresponding type for~$I(L'/K)$, obtained by
inflation.

For any $(\gL,\gL^+)\in\cM_{\loc}^{K'/K}$, we define 
$(\gL',(\gL')^+)\in\cM_{\loc}^{L'/K}$
by \[(\gL',(\gL')^+):=\cO_{L'}\otimes_{\cO_{K'}}(\gL,\gL^+),\]with the diagonal
action of~$\Gal(L'/K)$.  Similarly, for
any~$\gM\in\cC^{\dd}(A)$, we let~$\gM':=\gS'_A\otimes_{\gS_A}\gM$,
with~$\varphi$ and~$\Gal(L'/K)$ again acting diagonally. 

\begin{prop} 
  \label{prop: change of K'}\hfill
  \begin{enumerate}
  \item The assignments~$(\gL,\gL^+)\to(\gL',(\gL')^+)$ and 
  $\gM\mapsto\gM'$ induce compatible monomorphisms 
  $\cM_{\loc}^{K'/K}\to\cM_{\loc}^{L'/K}$
  and~$\cC^{\dd}_{K'}\to\cC^{\dd}_{L'}$, i.e., as functors they are 
  fully faithful. 

\item The monomorphism $\cC^{\dd}_{K'}\to\cC^{\dd}_{L'}$ induces an
  isomorphism~$\cC^{\tau}\to\cC^{\tau'}$, as well as a monomorphism
$\cC^{\dd,\BT}_{K'}\to\cC^{\dd,\BT}_{L'}$ 
and an isomorphism $\cC^{\tau,\BT}\to\cC^{\tau',\BT}$. 
  \end{enumerate}
\end{prop}
\begin{proof} 
(1) One checks easily that the  assignments~$(\gL,\gL^+)\to(\gL',(\gL')^+)$ and 
  $\gM\mapsto\gM'$ are compatible. For the rest of the claim, we consider the case of the functor $\gM\mapsto\gM'$; the
  arguments for the local models case are similar but slightly easier,
  and we leave them to the reader. Let $A$ be a $\varpi$-adically
  complete $\cO$-algebra. 
If~$\gN$ is a
rank~$d$ Breuil--Kisin module with descent data from~$L'$ to~$K$, consider
the Galois
invariants~$\gN^{\Gal(L'/K')}$. Since~$(\gS'_{A})^{\Gal(L'/K')}=\gS_A$,
these invariants are naturally a~$\gS_A$-module, and moreover they
naturally carry a Frobenius and descent data satisfying the conditions
required of a Breuil--Kisin module of height at most $h$. In general the
invariants need not be projective of rank $d$, and so need not be
rank~$d$ Breuil--Kisin module with descent data from~$K'$ to~$K$. However, in
the case~$\gN=\gM'$ we
have \[(\gM')^{\Gal(L'/K')}=(\gS'_A\otimes_{\gS_A}\gM)^{\Gal(L'/K')}=(\gS'_A)^{\Gal(L'/K')}\otimes_{\gS_A}\gM=\gM.\]
Here the second equality holds e.g.\ because $\Gal(L'/K')$ has order
prime to $p$, so that taking $\Gal(L'/K')$-invariants is exact, and
indeed is given by
multiplication by an idempotent $i \in \gS'_A$ (use the decomposition
$\gS'_A = i \gS'_A \oplus (1-i) \gS'_A$ and note that $i$ kills the
latter summand).
It
follows immediately that the functor~$\gM\to\gM'$ is fully faithful,
so $\cC^{\dd,a}_{K'}\to\cC^{\dd,a}_{L'}$ is a monomorphism.

(2) Suppose now that~$\gN$ has type~$\tau'$. In view of what we have
proven so far, in order to prove that
$\cC^{\tau}\to\cC^{\tau'}$ is an isomorphism, it is enough to
show that $\gN^{\Gal(L'/K')}$ is a rank~$d$ Breuil--Kisin module of type~$\tau$, and that the natural
map of $\gS'_A$-modules
\numequation\label{eqn: tensoring up invariants}\gS'_{A}\otimes_{\gS_A}\gN^{\Gal(L'/K')}\to\gN\end{equation}
is an
isomorphism. For the remainder of this proof, for clarity we write $u_{K'}$,
$u_{L'}$ instead of $u$ for the variables of $\gS_A$ and $\gS'_A$ respectively.
Since the type~$\tau'$ of $\gN$ is inflated from $\tau$, the
action of $\Gal(L'/K')$ on $\gN/u_{L'} \gN$ factors through
$\Gal(l'/k')$; noting that $W(l')$ has a normal basis for $\Gal(l'/k')$ over
$W(k')$,  we obtain an isomorphism
\numequation\label{eq:moduL'} 
 W(l')  \otimes_{W(k')}  (\gN/u_{L'}\gN)^{\Gal(L'/K')} \toisom
  \gN/u_{L'} \gN.\end{equation}
In particular the $W(k')\otimes_{\Zp} A$-module $(\gN/u_{L'} \gN)^{\Gal(L'/K')}$ is projective of rank $d$.

Observe however that $(\gN/u_{K'}\gN)^{\Gal(L'/K')} =
(\gN/u_{L'}\gN)^{\Gal(L'/K')}$. To see this, by the exactness of
taking $\Gal(L'/K')$ invariants it suffices to check that $u_{L'}^i
\gN/u_{L'}^{i+1}$ has trivial $\Gal(L'/K')$-invariants  for $0 < i <
e(L'/K')$. Multiplication by $u_{L'}^i$ gives an isomorphism 
$\gN/u_{L'} \gN \cong u_{L'}^i
\gN/u_{L'}^{i+1}$, so that for $i$ in the above range, the inertia group 
$I(L'/K')$ acts linearly on $ u_{L'}^i \gN/u_{L'}^{i+1}$ through a twist of
$\tau'$ by a nontrivial character; so there are no
$I(L'/K')$-invariants, and thus no $\Gal(L'/K')$-invariants either.

It
follows that the isomorphism \eqref{eq:moduL'} is the map \eqref{eqn: tensoring up
  invariants} modulo~$u_{L'}$. By 
Nakayama's lemma it follows that \eqref{eqn: tensoring up
  invariants} is surjective. Since~$\gN$ is projective,  the surjection~\eqref{eqn: tensoring up
  invariants} is split, and is therefore an isomorphism, since it is
an isomorphism modulo~$u_{L'}$. This isomorphism
exhibits~$\gN^{\Gal(L'/K')}$ as a direct summand (as
an~$\gS_A$-module) of the projective module~$\gN$, so it is also
projective; and it is projective of rank~$d$, since this holds modulo~$u_{K'}$. 

Finally, we need to check the compatibility of these maps with the
strong determinant condition.  By Corollary~\ref{cor: BT C is union of
  tau C}, it is enough to prove this for the case of the morphism
$\cC^{\tau}\to\cC^{\tau'}$ for some~$\tau$; by the compatibility in
part (1), this comes down to the same for the corresponding map of
local model stacks $\cM_{\loc}^{\tau}\to\cM_{\loc}^{\tau'}$. If
$(\gL,\gL^+) \in \cM_{\loc}^{\tau}$, it therefore suffices to show 
to show that conditions (1) and ($2'$) of  Lemma~\ref{lem: local model of type only have to check the strong
      determinant condition on that type} for $(\gL,\gL^+)$ and
    $(\gL', (\gL')^+) := \cO_{L'} \otimes_{\cO_{K'}} (\gL,\gL^+)$ are
    equivalent. This is immediate for condition ($2'$), since we have
    $(\gL')^+/\pi'' (\gL')^+ \cong l' \otimes_{k'} (\gL^+/\pi' \gL^+)$
 as $I(L'/K)$-representations.

Writing~$\tau=\eta\oplus\eta'$, it remains to relate the
    strong determinant conditions on the~$\eta,\eta'$-parts over
    both~$K'$ and~$L'$. Unwinding the definitions using Remark~\ref{rem: explicit version of determinant condition
      local models version}, one finds that the condition over $L'$ is
    a product of $[l':k']$ copies (with different sets of variables)
    of the condition over~$K'$. Thus the strong determinant condition
    over $K'$ implies the condition over $L'$, while the condition
    over $L'$ implies the condition over $K'$ up to an $[l':k']$th
    root of unity. Comparing the terms involving only copies of 
    $X_{0,\sigma_i}$'s shows that this root of unity must be $1$.
\end{proof}

\begin{remark}
 The morphism of local model stacks $\cM^{\tau}_{\loc} \to
 \cM^{\tau'}_{\loc}$ is not an isomorphism (provided that the extension
 $L'/K'$ is nontrivial). The issue is that, as we observed in the
 preceding proof, local models
 $(\gL',(\gL')^+)$ in the image
 of the morphism $\cM^{\tau}_{\loc} \to
 \cM^{\tau'}_{\loc}$ can have  $((\gL')^+/\pi'' (\gL')^+)_{\xi} \neq
 0$ only for characters $\xi : I(L'/K) \to \cO^{\times}$ that are
 inflated from $I(K'/K)$. However, one does obtain an isomorphism
 from the substack of $\cM^{\tau}_{\loc}$ of pairs $(\gL,\gL^+)$ satisfying condition~(2)
 of Lemma~\ref{lem: local model of type only have to check the strong
      determinant condition on that type} to the analogous substack of
    $\cM^{\tau'}_{\loc}$; therefore the induced map
    $\cM^{\tau,\BT}\rightarrow \cM^{\tau',\BT}$ \emph{will} also be an
    isomorphism.  Analogous remarks will apply to the maps 
    of local model stacks in \S\ref{subsec: explicit local models}.
\end{remark}

\subsection{Explicit local models}\label{subsec: explicit local models}
We now explain the connection between the moduli stacks $\cC^\tau$
and local models for Shimura varieties at hyperspecial and Iwahori level. 
This idea has been developed in~\cite{2015arXiv151007503C} for
Breuil--Kisin modules of arbitrary rank with tame principal series descent data, inspired by 
~\cite{kis04}, which analyses the case without descent data. 

The results
of~\cite{2015arXiv151007503C} relate the moduli stacks $\cC^{\tau}$(in
the case that~$\tau$ is a principal series type) via
a local model diagram to a mixed-characteristic deformation of the affine flag variety,
introduced in this generality by Pappas and Zhu~\cite{pappas-zhu}. The local models in~\cite[\S 6]{pappas-zhu}
are defined in terms of Schubert cells in the generic fibre of this mixed-characteristic deformation,
by taking the Zariski closure of these Schubert cells. The disadvantage of this
approach is that it does not give a direct moduli-theoretic interpretation of the special
fibre of the local model. Therefore, it is hard to check directly that our stack $\cC^{\tau, \BT}$, 
which has a moduli-theoretic definition,
corresponds to the desired local model under the diagram introduced in~\cite[Cor. 4.11]{2015arXiv151007503C} 
\footnote{One should be able to check this by adapting the ideas
in~\cite[\S 2.1]{haines-ngo} and~\cite[Prop. 6.2]{pappas-zhu} to $\Res_{K/\Q_p}\GL_n$
where $K/\Q_p$ can be ramified.}. 

In our rank~2 setting, the local models admit a much more explicit condition, using
lattice chains and Kottwitz's determinant condition, and in the cases
of nonscalar types, we will relate our local models to 
the naive local model at Iwahori level
for the Weil restriction of $\GL_2$, in the sense of~\cite[\S
2.4]{MR3135437}.

We begin with the simpler case of scalar inertial types. Suppose  that~$K'/K$ is totally ramified, and that~$\tau$ is a scalar
inertial type, say~$\tau=\eta\oplus\eta$. In this case we define the
local model stack~$\cM_{\loc,\hyp}$ (``hyp'' for ``hyperspecial'') to be the \emph{fppf} stack
over~$\Spf\cO$ (in fact, a $p$-adic formal algebraic stack),
which to each $p$-adically complete
$\cO$-algebra $A$ associates the groupoid~$\cM_{\loc,\hyp}(A)$
consisting of pairs~$(\gL,\gL^+)$,
where
\begin{itemize}
\item $\gL$ is a rank~$2$ projective~$\cO_K\otimes_{\Zp}A$-module, and
\item $\gL^+$ is an $\cO_K\otimes_{\Zp}A$-submodule of $\gL$,
	which is locally on~$\Spec A$ a direct summand of~$\gL$ as an $A$-module (or equivalently, for which the quotient $\gL/\gL^+$ is projective
	as an $A$-module).
\end{itemize}
We let~$\cM^{\BT}_{\loc,\hyp}$ be the substack of
pairs~$(\gL,\gL^+)$ with the property that for all $a\in\cO_K$,
we have \numequation\label{eqn: strong det condn for scalar local model} \det{\!}_A(a|\gL^+)
=\prod_{\psi:K\into E}\psi(a)
  \end{equation}
  as polynomial functions on $\cO_K$.

\begin{lem}
  \label{lem: local model scalar type is equivalent to what you
    expect}
The functor
  $(\gL',(\gL')^+) \mapsto ((\gL')_{\eta},(\gL')^+_{\eta})$
  defines a morphism
$\cM^\tau_{\loc}\to\cM_{\loc,\hyp}$ which induces an isomorphism
  $\cM^{\tau,\BT}_{\loc} \to \cM^{\BT}_{\loc,\hyp}$. 
{\em (}We remind the reader that $K'/K$ is assumed totally ramified,
and that $\tau$ is assumed to be a scalar inertial type associated
to the character $\eta$.{\em )}
\end{lem}
\begin{proof}
If $(\gL',(\gL')^+)$ is an object of $\cM^{\tau}_{\loc}$,
the proof that $((\gL')_\eta,(\gL')^+_\eta)$
  is indeed an object of~$\cM_{\loc,\hyp}(A)$ is
  very similar to the proof of Proposition~\ref{prop: change of K'},
  and is left to the reader.
  Similarly, the reader may verify that the functor
  \[(\gL,\gL^+) \mapsto
(\gL',(\gL')^+):=\cO_{K'}\otimes_{\cO_K}(\gL,\gL^+),\]
where the
action of~$\Gal(K'/K)$ is given by the tensor product of the natural
action on~$\cO_{K'}$ with the action on~$(\gL,\gL^+)$ given by the character~$\eta$,
defines a morphism
$\cM_{\loc,\hyp} \to \cM^\tau_{\loc}$.

 The composition
$\cM_{\loc,\hyp} \to \cM^\tau_{\loc}\to\cM_{\loc,\hyp}$ is evidently
the identity. The composition in the other order is not, in general, naturally equivalent to the
identity morphism, because for $(\gL,\gL^+) \in
\cM^{\tau}_{\loc}(A)$ one cannot necessarily recover $\gL^+$ from the
projection to its $\eta$-isotypic part. However, this 
\emph{will} hold if $\gL^+$ satisfies condition (2) of  Lemma~\ref{lem:
      local model of type only have to check the strong determinant
      condition on that type} (and so in particular will hold after
    imposing the strong determinant condition). 

Indeed, suppose $(\gL,\gL^+) \in
\cM^{\tau}_{\loc}(A)$. Then there is a natural
$\Gal(K'/K)$-equivariant map of projective $\cO_{K'}\otimes_{\Zp} A$-modules
\numequation\label{eq:recover-L} \cO_{K'} \otimes_{\cO_K} \gL_{\eta} \to \gL\end{equation}
of the same rank (in which $\Gal(K'/K)$ acts by $\eta$ on $\gL_{\eta}$). This map is
surjective because it is surjective on $\eta$-parts and the maps $p_{i,\xi}$ are surjective for all
$\xi\neq \eta$; therefore it is an isomorphism. One further has a 
a natural
$\Gal(K'/K)$-equivariant map of $\cO_{K'}\otimes_{\Zp} A$-modules
\numequation\label{eq:recover-L+} \cO_{K'} \otimes_{\cO_K}
\gL^+_{\eta} \to \gL^+ \end{equation}
that is injective because locally on $\Spec(A)$ it is a direct summand of the isomorphism \eqref{eq:recover-L}.
If one further assumes that $\gL^+$ satisfies condition (2) of  Lemma~\ref{lem:
      local model of type only have to check the strong determinant
      condition on that type} then \eqref{eq:recover-L+} is an
    isomorphism, as claimed.

  It remains to check the compatibility of these maps with the strong
  determinant condition. If $(\gL,\gL^+) \in \cM_{\loc,\hyp}$ then 
certainly condition (2) of  Lemma~\ref{lem:
      local model of type only have to check the strong determinant
      condition on that type} holds for
    $(\gL',(\gL')^+):=\cO_{K'}\otimes_{\cO_K}(\gL,\gL^+) \in
    \cM^\tau_{\loc}$.  By Lemma~\ref{lem:
      local model of type only have to check the strong determinant
      condition on that type} the strong determinant condition holds
    for~$(\gL',(\gL')^+)$ if and only if~(\ref{eqn: strong det condn}) holds
    for~$\gL'$ with~$\xi=\eta$; but this is exactly the
    condition~(\ref{eqn: strong det condn for scalar local model}) for
    $\gL$, as required.
\end{proof}

Next, we consider the case of principal series types. We suppose
that~$K'/K$ is totally ramified. We begin by
defining a $p$-adic formal algebraic stack $\cM_{\loc,\Iw}$
over~$\Spf\cO$ (``Iw'' for Iwahori). For each complete $\cO$-algebra $A$, we
let~$\cM_{\loc,\Iw}(A)$ be the groupoid of tuples~$(\gL_1,\gL^+_1,\gL_2,\gL^+_2,f_1,f_2)$,
where
\begin{itemize}
\item $\gL_1$, $\gL_2$ are rank~$2$
  projective~$\cO_K\otimes_{\Zp}A$-modules, 
\item $f_1:\gL_1\to\gL_2$, $f_2:\gL_2\to\gL_1$ are morphisms of
  $\cO_K\otimes_{\Zp}A$-modules, satisfying~$f_1\circ f_2=f_2\circ f_1=\pi$,
\item both $\coker f_1$ and~$\coker f_2$ are rank one projective $k\otimes_{\Zp}A$-modules,
\item $\gL^+_1$, $\gL^+_2$ are $\cO_K\otimes_{\Zp}A$-submodules of $\gL_1$, $\gL_2$, which are locally
  on~$\Spec A$ direct summands  as $A$-modules
  (or equivalently, for which the quotients
  $\gL_i/\gL^+_i$ ($i = 1,2$) are projective $A$-modules), and moreover
  for which
  the morphisms~$f_1,f_2$ restrict to morphisms~$f_1:\gL^+_1\to \gL^+_2$, $f_2:\gL^+_2\to \gL^+_1$.
\end{itemize}
We let~$\cM^{\BT}_{\loc,\Iw}$ be the substack of
tuples with the property that for all $a\in\cO_K$ and $i=1,2$,
we have \numequation\label{eqn: strong det condn for principal series local model} \det{\!}_A(a|\gL^+_i)
=\prod_{\psi:K\into E}\psi(a)
  \end{equation}
  as polynomial functions on $\cO_K$.

Write~$\tau=\eta\oplus\eta'$ with~$\eta\ne\eta'$. 
Recall that the character $h:\Gal(K'/K)=I(K'/K)\to W(k)^\times$ is
given by~$h(g)=g(\pi')/\pi'$. 
Since we are assuming
that~$\eta\ne\eta'$, for each embedding~$\sigma:k\into\F$ (which we
also think of as $\sigma:W(k)\into\cO$) there are
integers $0<a_\sigma,b_\sigma<e(K'/K)$ with the properties that
$\eta'/\eta=\sigma\circ h^{a_\sigma}$, $\eta/\eta'=\sigma\circ
h^{b_\sigma}$; in particular, $a_\sigma+b_\sigma=e(K'/K)$. Recalling
that~$e_\sigma\in W(k)\otimes_{\Zp}\cO\subset \cO_K$ is the idempotent
corresponding to~$\sigma$, we
set~$\pi_1=\sum_\sigma(\pi')^{a_\sigma}e_\sigma$,
$\pi_2=\sum_\sigma(\pi')^{b_\sigma}e_\sigma$; so we have
~$\pi_1\pi_2=\pi$, and~$\pi_1\in(\cO_{K'}\otimes_{\Zp}\cO)^{I(K'/K)=\eta'/\eta}$,
$\pi_2\in(\cO_{K'}\otimes_{\Zp}\cO)^{I(K'/K)=\eta/\eta'}$.

We define a
morphism $\cM^\tau_{\loc}\to\cM_{\loc,\Iw}$ as follows. Given a
pair~$(\gL,\gL^+)\in\cM^\tau_{\loc}(A)$, we
set~$(\gL_1,\gL^+_1)=(\gL_\eta,\gL^+_\eta)$,
$(\gL_2,\gL^+_2)=(\gL_{\eta'},\gL^+_{\eta'})$, and we let~$f_1$, $f_2$ be given by
multiplication by~$\pi_1$, $\pi_2$ respectively. The only point that
is perhaps not immediately obvious is to check that~$\coker f_1$ and~$\coker f_2$
are  rank one projective $k\otimes_{\Zp}A$-modules. To see this, note
that by~\cite[\href{http://stacks.math.columbia.edu/tag/05BU}{Tag
  05BU}]{stacks-project}, we can lift $(\gL/\pi'\gL)_{\eta'}$ to a
rank one summand~$U$ of the
projective~$\cO_{K'}\otimes_{\Zp}A$-module~$\gL$. Let~$U_\eta$
be the direct summand of~$\gL_\eta$ obtained by projection of~$U$ to
the $\eta$-eigenspace~$\gL_\eta$; then the projective $\cO_K
\otimes_{\Zp} A$-module $U_\eta$ has rank one, as can be checked
modulo $\pi$.
Similarly we
may lift $(\gL/\pi'\gL)_{\eta'}$ to a rank one
summand~$V$ of~$\gL$, and we let~$V_{\eta'}$ be the 
projection to the $\eta'$-part.

The natural map 
\numequation\label{eqn: decomposing principal series type over powers
  of pi
  prime}
\cO_{K'} \otimes_{\cO_K} (U_\eta \oplus V_{\eta'}) \to \gL 
\end{equation}
is an isomorphism,  since both sides are projective
$\cO_{K'}\otimes_{\Zp}A$-modules of rank two, and the given
map is an isomorphism modulo~$\pi'$.
It follows immediately from~(\ref{eqn: decomposing principal series type over powers
  of pi prime}) that~$\gL_\eta=U_\eta\oplus\pi_2V_{\eta'}$
and~$\gL_{\eta'}=\pi_1U_\eta\oplus V_{\eta'}$, so that~$\coker f_1$ and~$\coker f_2$ are projective of rank
one, as claimed.

\begin{prop}
  \label{prop: principal series case local model isomorphic to Iwahori
  local model}The morphism $\cM^\tau_{\loc}\to\cM_{\loc,\Iw}$ 
induces an isomorphism $\cM^{\tau,\BT}_{\loc}\to\cM^{\BT}_{\loc,\Iw}$.
{\em (}We remind the reader that $K'/K$ is assumed totally ramified,
and that $\tau$ is assumed to be a principal series inertial type.{\em )}
\end{prop} 
\begin{proof} 
We begin by constructing a morphism
  $\cM_{\loc,\Iw}\to\cM^\tau_{\loc}$, inspired by the arguments
  of~\cite[App.\ A]{rapoport-zink}. 
We define an~$\cO_K\otimes_{\Zp}A$-module~$\gL$ by
 \[\gL= \oplus_\sigma\left(\oplus_{i=0}^{a_\sigma-1}e_\sigma
    \gL_1^{(i)} \oplus_{j=0}^{b_\sigma-1}e_\sigma
    \gL_2^{(j)}\right), \] where the $\gL_1^{(i)}$'s and
  $\gL_2^{(j)}$'s are copies of~$\gL_1$, $\gL_2$ respectively. 

We can
  upgrade the ~$\cO_K\otimes_{\Zp}A$-module structure on~$\gL$ to that
  of an~$\cO_{K'}\otimes_{\Zp}A$-module by specifying how~$\pi'$
  acts.  If $i<a_\sigma-1$, then we let~$\pi':e_\sigma
    \gL_1^{(i)}\to e_\sigma
    \gL_1^{(i+1)}$ be the map induced by the identity
    on~$\gL_1$, and if~$j<b_\sigma-1$, then we let~$\pi':e_\sigma
    \gL_2^{(j)}\to e_\sigma
    \gL_2^{(j+1)}$ be the map induced by the identity
    on~$\gL_2$. We let $\pi':e_\sigma
    \gL_1^{(a_\sigma-1)}\to e_\sigma
    \gL_2^{(0)}$ be the map induced by~$f_1:\gL_1\to\gL_2$, and we let $\pi':e_\sigma
    \gL_2^{(b_\sigma-1)}\to e_\sigma
    \gL_1^{(0)}$ be the map induced by~$f_2:\gL_2\to\gL_1$. That this indeed gives~$\gL$ the structure of an
$\cO_{K'}\otimes_{\Zp}A$-module follows from our assumption that
$f_1\circ f_2=f_2\circ f_1=\pi$. We give~$\gL$ a semilinear action of~$\Gal(K'/K)=I(K'/K)$ 
by letting
it act via~$(\sigma \circ h)^{i} \cdot \eta$ on each~$e_\sigma
    \gL_1^{(i)} $ and via~$(\sigma \circ h)^j  \cdot \eta'$ on each~$e_\sigma
    \gL_2^{(j)} $.

We claim that~$\gL$ is a rank~$2$ projective
$\cO_{K'}\otimes_{\Zp}A$-module.  Since~$\coker f_2$ is projective by
  assumption, we can choose a section to the $k\otimes_{\Zp}A$-linear
  morphism $\gL_1/\pi\to\coker f_2$, with image~$\overline{U}_\eta$,
  say. Similarly we choose a section to $\gL_2/\pi\to\coker f_1$ with
  image~$\overline{V}_{\eta'}$. We choose lifts~$U_\eta$, $V_{\eta'}$
  of~$\overline{U}_\eta$, $\overline{V}_{\eta'}$ to direct summands of the
  $\cO_K\otimes_{\Zp}A$-modules $\gL_1$, $\gL_2$ respectively. There
  is a map of $\cO_{K'} \otimes_{\Zp} A$-modules
\numequation\label{eqn:
    reconstructing
    L} \lambda: \cO_{K'} \otimes_{\cO_K} (U_\eta \oplus V_{\eta'})  \to \gL
\end{equation}
induced by the map identifying $U_\eta, V_{\eta'}$ with their copies in $\gL_1^{(0)}$ and
$\gL_2^{(0)}$ respectively.  The map $\lambda$ is surjective modulo $\pi'$ by
construction, hence surjective by Nakayama's lemma. Regarding $\lambda$ as a map
of projective $\cO_K \otimes_{\Zp} A$-modules of equal rank $2e(K'/K)$, we
deduce that $\lambda$ is an isomorphism. Since the source of $\lambda$
is a projective $\cO_{K'}\otimes_{\Zp} A$-module, the claim follows.

We now set \[\gL^+=
  \oplus_\sigma\left(\oplus_{i=0}^{a_\sigma-1} e_\sigma
    (\gL_1^{(i)})^+ \oplus_{j=0}^{b_\sigma-1} e_\sigma
    (\gL_2^{(j)})^+\right) \subset \gL\]
It is
immediate from the construction that $\gL^+$ is preserved by
$\Gal(K'/K)$. The hypothesis that $f_1,f_2$ and preserve $\gL_1^+,\gL_2^+$ implies
that $\gL^+$ is an $\cO_{K'}\otimes_{\Zp} A$-submodule of $\gL$, while
the hypothesis that each $\gL_i/\gL_i^+$  is a projective $A$-module
implies the same for $\gL/\gL^+$.  This completes the construction of
our morphism 
  $\cM_{\loc,\Iw}\to\cM^\tau_{\loc}$.

Just as in the proof of Proposition~\ref{lem: local model scalar type is equivalent to what you
    expect}, the morphism $\cM_{\loc,\Iw}\to
\cM^{\tau}_{\loc}$ followed by our morphism
$\cM^{\tau}_{\loc}\to\cM_{\loc,\Iw}$ is the identity, while the composition in the other order is not, in general, naturally equivalent to the
identity morphism. 
However, it follows immediately from Lemma~\ref{lem: local model of type only have to check the strong
      determinant condition on that type} and the construction of $\gL^+$ that our morphisms $\cM_{\loc,\Iw}\to
\cM^{\tau}_{\loc}$ and
$\cM^{\tau}_{\loc}\to\cM_{\loc,\Iw}$ respect the strong determinant
condition, and so induce maps  $\cM_{\loc,\Iw}^{\BT} \to
\cM^{\tau,\BT}_{\loc}$ and
$\cM^{\tau,\BT}_{\loc} \to\cM_{\loc,\Iw}^{\BT}$.  To see that the
composite $\cM^{\tau,\BT}_{\loc} \to\cM_{\loc,\Iw}^{\BT} \to
\cM^{\tau,\BT}_{\loc}$ is naturally equivalent to the identity,
suppose that 
$(\gL,\gL^+) \in \cM^\tau_{\loc}(A)$ and observe that
there is a natural $\Gal(K'/K)$-equivariant isomorphism of $\cO_K\otimes_A\Zp$-modules
\numequation \label{eq:ps-natural-isom} \oplus_\sigma\left(\oplus_{i=0}^{a_\sigma-1}e_\sigma
    \gL_\eta^{(i)} \oplus_{j=0}^{b_\sigma-1}e_\sigma
    \gL_{\eta'}^{(j)}\right) \toisom \gL \end{equation}
induced by the maps $\gL_\eta^{(i)} \longnto{(\pi')^i} (\pi')^i
\gL_\eta$ and $\gL_{\eta'}^{(j)} \longnto{(\pi')^j} (\pi')^j
\gL_{\eta'}$. The commutativity of the diagram
\[ 
\xymatrix{ e_\sigma \gL_{\eta}^{(a_\sigma-1)} \ar[d]_{\pi_1}
  \ar[r]^{(\pi')^{a_\sigma-1}} & e_\sigma (\pi')^{a_\sigma-1} \gL_{\eta}
  \ar[d]^{\pi'} \\
e_\sigma \gL^{(0)}_{\eta'} \ar[r]_{\textrm{id}} & e_\sigma\gL_{\eta'} }\]
implies that the map in \eqref{eq:ps-natural-isom} is in fact an
$\cO_{K'}\otimes_{\Zp} A$-module isomorphism. The map
\eqref{eq:ps-natural-isom} induces an inclusion
\[ \oplus_\sigma\left(\oplus_{i=0}^{a_\sigma-1}e_\sigma
    (\gL_\eta^{(i)})^+ \oplus_{j=0}^{b_\sigma-1}e_\sigma
    (\gL_{\eta'}^{(j)})^+ \right) \rightarrow \gL^+.\]
If furthermore $(\gL,\gL^+) \in \cM^{\tau,\BT}_{\loc}$ then this is an isomorphism because $\gL^+$ satisfies condition (2) of Lemma~\ref{lem: local model of type only have to check the strong
      determinant condition on that type}.
  \end{proof}

Finally, we turn to the case of a cuspidal type. Let~$L$ as usual be a
quadratic unramified extension of~$K$, and
set~$K'=L(\pi^{1/(p^{2f}-1)})$. The field $N$ continues to denote the
maximal unramified extension of $K$ in $K'$, so that $N=L$.  Let~$\tau$ be a cuspidal type, so that
~$\tau=\eta\oplus\eta'$, where~$\eta\ne\eta'$ but~$\eta'=\eta^{p^f}$.

\begin{prop}
  \label{prop: cuspidal local model formally smooth over Iwahori}There
  is a morphism 
  $\cM_{\loc}^{\tau,\BT}\to\cM_{\loc,\Iw}^{\BT}$
which is representable by algebraic spaces and smooth.
{\em (}We remind the reader that $\tau$ is now assumed to be a cuspidal
inertial type.{\em )}
\end{prop}
\begin{proof}Let~$\tau'$ be the type~$\tau$, considered as a (principal series) type for the
  totally ramified extension~$K'/N$. Let~$c\in\Gal(K'/K)$ be the
  unique element which fixes~$\pi^{1/(p^{2f}-1)}$ but acts
  nontrivially on~$N$. For any map $\alpha : X\to Y$ of
  $\cO_N$-modules we write $\alpha^c$ for the twist  $1\otimes \alpha :
  \cO_N \otimes_{\cO_N,c} X \to \cO_N \otimes_{\cO_N,c} Y$.

We may think of an
  object~$(\gL,\gL^+)$ of~$\cM_{\loc}^{\tau,\BT}$ as an object~$(\gL',(\gL')^+)$
  of~$\cM_{\loc}^{\tau',\BT}$ equipped with the additional data of an
  isomorphism of $\cO_{K'}\otimes_{\Zp}A$-modules
  $\theta:\cO_{K'}\otimes_{\cO_{K'},c}\gL'\to\gL'$ which is compatible
  with~$(\gL')^+$, which satisfies~$\theta\circ\theta^c=\id$, and which is compatible with the
  action of~$\Gal(K'/N)=I(K'/N)$ in the sense that $\theta \circ (1
  \otimes g) = g^{p^f} \circ \theta$.

  Employing the isomorphism of Proposition~\ref{prop: principal series case local model isomorphic to Iwahori
  local model}, we think of~$(\gL',(\gL')^+)$ as a
tuple~$(\gL'_1, (\gL'_1)^+,\gL'_2,(\gL'_2)^+,f_1,f_2)$, where the~$\gL'_i,(\gL'_i)^+$ are
$\cO_N\otimes_{\Zp}A$-modules; by construction, the map~$\theta$
induces  isomorphisms $\theta_1:\cO_N\otimes_{\cO_N,c}\gL'_1\isoto\gL_2'$,
$\theta_2:\cO_N\otimes_{\cO_N,c}\gL'_2\isoto\gL'_1$, which are
compatible with~$(\gL'_1)^+,(\gL'_2)^+$ and~$f_1,f_2$, and
satisfy~$\theta_1\circ\theta_2^c=\id$.

Choose for each embedding~$\sigma:k\into\F$ an extension to an
embedding~$\sigma^{(1)}:k'\into\F$, 
set~$e_1=\sum_\sigma e_{\sigma^{(1)}}$, and write~$e_2=1-e_1$. Then
the map~$\theta_1$  induces isomorphisms
$\theta_{11}:e_1\gL'_1\isoto e_2\gL_2'$ and
$\theta_{12}:e_2\gL'_1\isoto e_1\gL_2'$, while~$\theta_2$  induces isomorphisms
$\theta_{21}:e_1\gL'_2\isoto e_2\gL_1'$ and
$\theta_{22}:e_2\gL'_2\isoto e_1\gL_1'$. The condition
that~$\theta_1\circ\theta_2^c=\id$ translates to $\theta_{22} =
\theta_{11}^{-1}$ and $\theta_{21}=\theta_{12}^{-1}$, and
compatibility with~$(\gL'_1)^+,(\gL'_2)^+$ implies that
$\theta_{11},\theta_{21}$ induce isomorphisms $e_1 (\gL'_1)^+ \isoto
e_2(\gL_2')^+$ and $e_1 (\gL'_2)^+ \isoto
e_2(\gL_1')^+$ respectively.

Furthermore $f_1,f_2$ induce maps
$e_1f_1:e_1\gL'_1\to e_1\gL_2'$, $e_1g:e_1\gL_2'\to e_1\gL_1'$.
 It
follows that there is a map $\cM_{\loc}^{\tau,\BT}\to\cM_{\loc,\Iw}$,
sending~$(\gL,\gL^+)$ to the tuple
$(e_1\gL_1',e_1(\gL_1')^+,e_1\gL_2',e_1(\gL_2')^+,e_1f_1,e_1f_2)$.
To see that it respects the strong determinant condition, 
one has to check that the
    conditions on~$\gL^+_1$, $\gL^+_2$ imply those for~$e_1(\gL'_1)^+$,
    $e_2(\gL'_2)^+$
    coincide, and this follows from the definitions (via Remark~\ref{rem: explicit version of determinant condition
      local models version}). We therefore obtain a map $\cM_{\loc}^{\tau,\BT}\to\cM^{\BT}_{\loc,\Iw}$,

Since this morphism is
given by forgetting the data of $e_2 \gL'_1,e_2 \gL'_2$ and the pair of isomorphisms~$\theta_{11},\theta_{21}$, it is evidently formally
smooth.
It is also a morphism between $\varpi$-adic formal algebraic stacks that
are locally of finite presentation, and so is representable by algebraic spaces
(by \cite[Lem.~7.10]{Emertonformalstacks}) and locally of finite presentation
(by \cite[Cor.~2.1.8]{EGstacktheoreticimages}
and~\cite[\href{http://stacks.math.columbia.edu/tag/06CX}{Tag
  06CX}]{stacks-project}).  Thus this morphism is in fact smooth.
\end{proof}

\subsection{Local models:\ local geometry}\label{subsec: local models
  relation to explicit PR ones}We now deduce our main results on the
local structure of our moduli stacks from results in the literature on local models for Shimura
varieties.

  \begin{prop}\label{prop: explicit description of local models stack
      as naive local model} 
  We can identify $\cM_{\loc,\Iw}^{\BT}$ with the quotient 
  of \emph{(}the $p$-adic formal completion of\emph{)} the naive local model for $\Res_{K/\Q_p}\GL_2$
  \emph{(}as defined in~\cite[\S 2.4]{MR3135437}\emph{)} by a smooth
  group scheme over $\cO$.
  \end{prop}
  
  \begin{proof} Let $\widetilde{\cM}_{\loc,\Iw}^{\BT}$ be the $p$-adic formal
    completion of the naive local model for $\Res_{K/\Q_p}\GL_2$
  corresponding to a standard lattice chain~ $\cL$, as defined in~\cite[\S 2.4]{MR3135437}. 
  By~\cite[Prop. A.4]{rapoport-zink}, the automorphisms of the standard lattice chain $\cL$ 
  are represented by a smooth group scheme $\cP_{\cL}$ over $\cO$. (This is in fact a parahoric subgroup scheme
  of $\mathrm{Res}_{\cO_K/\Z_p}\GL_2$, and in particular it is affine.) 
  Also by~\emph{loc.\ cit.}, every lattice chain of type $(\cL)$ is Zariski locally isomorphic to $\cL$. 
  By comparing the two moduli problems, we see
  that $\widetilde{\cM}_{\loc,\Iw}^{\BT}$ is a $\cP_{\cL}$-torsor over
  $\cM_{\loc,\Iw}^{\BT}$ 
  for the Zariski topology
  and the proposition follows. 
  \end{proof}
  
The following theorem describes the various regularity properties of local models.
Since we are working in the context of formal algebraic stacks,
we use the terminology developed in~\cite[\S 8]{Emertonformalstacks}
(see in particular~\cite[Rem.~8.21]{Emertonformalstacks}
and~\cite[Def.~8.35]{Emertonformalstacks}).

\begin{thm} 
  \label{thm: local consequences of local models}Suppose that~$d=2$
  and that $\tau$
  is a tame inertial type. Then
  \begin{enumerate}
  \item $\cM_{\loc}^{\tau,\BT}$ is residually Jacobson and analytically normal,
 and Cohen--Macaulay.
  \item The special fibre $\cM_{\loc}^{\tau,\BT,1}$ is reduced.
  \item  $\cM_{\loc}^{\tau,\BT}$ is flat over~$\cO$.
  \end{enumerate}

\end{thm}
\begin{proof}
For scalar types, this follows from~\cite[Prop.\
  2.2.2]{kis04} by Lemma~\ref{lem: local model scalar type is equivalent to what you
    expect}, 
and so we turn to studying the case of a non-scalar type.
The properties in question can be checked smooth locally
(see \cite[\S 8]{Emertonformalstacks} for~(1),
and \cite[\href{http://stacks.math.columbia.edu/tag/04YH}{Tag 04YH}]{stacks-project} for~(2);
for~(3), note that morphisms that are representable by algebraic spaces and smooth
are also flat, and take into account the final statement
of~\cite[Lem.~8.34]{Emertonformalstacks}),
and so 
  by Propositions~\ref{prop: principal series case local model isomorphic to Iwahori
  local model} and~\ref{prop: cuspidal local model formally smooth over
  Iwahori}
we reduce to checking the assertions of the theorem
for $\cM^{\BT}_{\loc,\Iw}.$  
Proposition~\ref{prop: explicit description of local models stack
    as naive local model}
then reduces us to checking these assertions for the $\varpi$-adic completion
of the naive local model at Iwahori level for $\Res_{K/\Q_p}\GL_2$.

Since this naive local model is a scheme of finite presentation over $\cO$,
its special fibre (i.e.\ its base-change to $\F$) is Jacobson,
and it is excellent; thus its $\varpi$-adic completion satisfies the properties
of~(1) if and only if the naive local model itself is normal and Cohen--Macaulay.
The special fibre of its $\varpi$-adic completion is of course just equal
to its own special fibre, and so verifying~(2) for the first of these special fibres
is equivalent to verifying it for the second.  Finally, the $\varpi$-adic completion
of an $\cO$-flat Noetherian ring is again $\cO$-flat, and so the $\varpi$-adic
completion of the naive local model will be $\cO$-flat if the naive local
model itself is.

All the properties of the naive local model that are described in the preceding
paragraph,
other than the Cohen--Macaulay property, 
 are contained 
  in~\cite[Thm.\ 2.17]{MR3135437},
  if we can identify the naive local models at Iwahori level with the vertical local
  models at Iwahori level, in the sense of~\cite[\S 8]{pappas-rapoport1} 
  (see also the discussion above~\emph{loc.\ cit.}\
  in~\cite{MR3135437}).
  
  The vertical local models are obtained
  by intersecting the preimages at Iwahori level of the flat local models at hyperspecial level. Since the naive 
  local models for the Weil restriction of~$\GL_2$
  at hyperspecial level are already flat by a special case
  of~\cite[Cor.\ 4.3]{pappas-rapoport1}, the naive local models at Iwahori level
  are identified with the vertical ones and~\cite[Thm.\
  2.17]{MR3135437} applies to them directly.  (To be precise, the results
  of~\cite{pappas-rapoport1} apply to restrictions of scalars
  $\Res_{F/F_0}\GL_2$ with $F/F_0$ totally ramified. However, thanks
  to the decomposition $\cO_K \otimes_{\Zp} A \cong \oplus_{\sigma :
    W(k) \into \cO} \cO_K \otimes_{W(k),\sigma} A$ the local
  model for
  $\Res_{K/\Qp}\GL_2$ decomposes as a product of local models for totally
  ramified extensions.) 

  Finally, Cohen--Macaulayness can be proved as in~\cite[Prop.\
  4.24]{MR1871975} and the discussion immediately following. We thank
  U.\ G\"ortz for explaining this argument to us.  As in the
  previous paragraph we reduce to the case of local models at Iwahori
  level for $\Res_{F/F_0}\GL_2$ with $F/F_0$ totally ramified of
  degree~$e$. In this
  setting the admissible set $M$ for the coweight $\mu = (e,0)$ has precisely
   one element of length $0$ and two elements of each length between $1$
   and $e$.  
Moreover, for elements $x,y \in M$ we have
$x<y$ in the Bruhat order if and only if $\ell(x)<\ell(y)$. One checks easily
that $M$ is $e$-Cohen--Macaulay in the sense of~\cite[Def.\
  4.23]{MR1871975}, and we conclude by~\cite[Prop.\
  4.24]{MR1871975}. Alternatively, this also follows from the much more
  general recent results of Haines--Richarz \cite{HainesRicharz}.
\end{proof}

\begin{cor} \label{cor: Kisin moduli consequences of local models}Suppose that~$d=2$
  and that $\tau$
  is a tame inertial type. Then
  \begin{enumerate}
  \item $\cC^{\tau,\BT}$ is analytically normal, and Cohen--Macaulay.
  \item The special fibre $\cC^{\tau,\BT,1}$ is reduced.
  \item  $\cC^{\tau,\BT}$ is flat over~$\cO$.
  \end{enumerate}
  \end{cor}
\begin{proof}
This follows from Theorems~\ref{thm: map from Kisin to general local model is
    formally smooth} and~\ref{thm: local consequences of local models},
since all of these properties can be verified smooth locally
(as was already noted in the proof of the second of these theorems).
\end{proof}

\begin{remark}
  There is another structural result about vertical local models
that is proved in~\cite[\S 8]{pappas-rapoport1} but which we haven't incorporated
into our preceding results, namely, that the irreducible components of
the special fibre are normal.  Since the notion of irreducible component
is not \'etale local (and so in particular not smooth local), this statement
does not imply the corresponding statement for the special fibres of
$\cM_{\loc}^{\tau,\BT}$ or $\cC^{\tau,\BT}$.  Rather, it implies the weaker,
and somewhat
more technical, statement that each of the analytic branches passing through each
closed point of the special fibre of these $\varpi$-adic formal algebraic
stacks is normal.  We won't discuss this further here, since we don't 
need this result.
\end{remark}

\subsection{Scheme-theoretic images}\label{subsec:results from EG
  and PR on finite flat moduli}
We continue to fix $d=2$, $h=1$, and we set
$K'=L(\pi^{1/{p^{2f-1}}})$, where $L/K$ is the unramified quadratic
extension, and~$\pi$ is a uniformiser of~$K$. (This is the choice of
$K'$ that we made before in the cuspidal case, and contains the choice
of $K'$ that we made in the principal series case; since the category
of Breuil--Kisin
modules with descent data for the smaller extension is by Proposition~\ref{prop: change of K'} naturally a
full subcategory of the category of Breuil--Kisin modules with descent data
for the larger extension, we can consider both principal series and
cuspidal types in this setting.)

\begin{df}
\label{def:defining the Z's}
For each $a\ge 1$ we write $\cZ^{\dd,a}$ and $\cZ^{\tau,a}$ for the scheme-theoretic
images (in the sense of~\cite[Defn.\ 3.2.8]{EGstacktheoreticimages}) of the morphisms
$\cC^{\dd,\BT,a}\to\cR^{\dd,a}$ and
$\cC^{\tau,\BT,a}\to\cR^{\dd,a}$ respectively. We write $\ocZ$,
$\ocZ^\tau$ for $\cZ^1, \cZ^{\tau,1}$ respectively. 
\end{df}

The following theorem records some basic properties of these
scheme-theoretic images. We refer to Appendix~\ref{sec: appendix on tame
  types} for the notion of representations admitting a potentially
Barsotti--Tate lift of a given type, and for the definition of tr\`es ramifi\'ee representations.
\begin{thm} 
  \label{thm: existence of picture with descent data and its basic
    properties}
  \begin{enumerate}
  \item  For each $a\ge 1$, $\cZ^{\dd,a}$ is an algebraic
    stack of finite presentation over 
    $\cO/\varpi^a$, and is a closed substack
    of~$\cR^{\dd,a}$. In turn, each $\cZ^{\tau,a}$ is
    a closed substack of~$\cZ^{\dd,a}$, and thus in particular is an
    algebraic stack of finite presentation over $\cO/\varpi^a$; and $\cZ^{\dd,a}$ is the union of
    the $\cZ^{\tau,a}$. 
  \item The morphism $\cC^{\dd,\BT,a}\to\cR^{\dd,a}$ factors through a morphism
$\cC^{\dd,\BT,a}\to \cZ^{\dd,a}$
which is representable by algebraic spaces, scheme-theoretically dominant, and proper.
Similarly, the morphism
    $\cC^{\tau,\BT,a}\to\cR^{\dd,a}$ factors through a morphism
    $\cC^{\tau,\BT,a}\to\cZ^{\tau,a}$
which is representable by algebraic spaces, scheme-theoretically dominant, and proper.
  \item The $\Fpbar$-points of $\ocZ$ are naturally in bijection with the
    continuous representations $\rbar:G_K\to\GL_2(\Fpbar)$ which are not a
    twist of a tr\`es ramifi\'ee extension of the trivial character by the
    mod~$p$ cyclotomic character. Similarly, the $\Fpbar$-points of $\ocZ^\tau$ are naturally in bijection with the
    continuous representations $\rbar:G_K\to\GL_2(\Fpbar)$ which have
    a potentially Barsotti--Tate lift of type~$\tau$. 
  
  \end{enumerate}

\end{thm}
\begin{proof} 

Part~(1) follows easily from Theorem~\ref{thm: R^dd_h is an algebraic
  stack}. Indeed, by \cite[Prop.\ 3.2.31]{EGstacktheoreticimages} we may think of $\cZ^{\dd,a}$ as the scheme-theoretic image of the
proper morphism of algebraic stacks~$\cC^{\dd,\BT,a}\to\cR^{\dd,a}_1$ and
similarly for each $\cZ^{\tau,a}$. The
existence of the factorisations in~(2) is then formal.

By~\cite[Lem.\ 3.2.14]{EGstacktheoreticimages}, for each finite
extension $\F'/\F$, 
the $\F'$-points of $\ocZ$ (respectively $\ocZ^\tau$) correspond
to the \'etale $\varphi$-modules with descent data of the form~$\gM[1/u]$,
where~$\gM$ is a Breuil--Kisin module of rank $2$ with descent data and
$\F$-coefficients which satisfies the strong
determinant condition (respectively, which satisfies the strong determinant
condition and is of type~$\tau$). By Lemma~\ref{lem: O points of
  moduli stacks} and Corollary~\ref{cor: Kisin moduli consequences of local models}, these precisely
correspond to the Galois representations $\rbar:G_K\to\GL_2(\F)$ which
admit potentially Barsotti--Tate lifts of some tame type (respectively, of
type~$\tau$). The result follows from Lemma~\ref{lem: list of things
  we need to know about Serre weights}. 
\end{proof}

The thickenings 
$\cC^{\dd, \BT, a} \hookrightarrow \cC^{\dd, \BT, a+1}$
and $\cR^{\dd, a} \hookrightarrow \cR^{\dd, a+1}$
induce closed immersions
$\cZ^{\dd, a} \hookrightarrow \cZ^{\dd, a+1}$.
Similarly, the thickenings $\cC^{\tau,\BT,a}\hookrightarrow \cC^{\tau,\BT, a+1}$
give rise to closed immersions $\cZ^{\tau,a} \hookrightarrow \cZ^{\tau,a+1}.$

\begin{lemma}
\label{lem:Z thickenings}
Fix $a \geq 1$.
Then the morphism
$\cZ^{\dd,a} \hookrightarrow \cZ^{\dd,a+1}$ is a thickening,
and for each tame type~$\tau,$
the morphism
$\cZ^{\tau,a} \hookrightarrow \cZ^{\tau,a+1}$ is a thickening.
\end{lemma}
\begin{proof}
In each case, the claim of the lemma follows from the following more general
statement:
if $$\xymatrix{ \cX \ar[r] \ar[d] & \cX' \ar[d] \\ \cY \ar[r] & \cY'}$$
is a diagram of morphisms of algebraic stacks in which the upper horizontal arrow
is a thickening, the lower horizontal arrow is a closed immersion,
and each of the vertical arrows is representable by algebraic spaces, quasi-compact,
and scheme-theoretically
dominant, then the lower horizontal arrow is also a thickening.

Since the property of being a thickening may be checked smooth locally, and since
scheme-theoretic dominance of quasi-compact morphisms is preserved 
by flat base-change, we may show this
after pulling the entire diagram back over a smooth surjective morphism
$V' \to \cY$  whose source is a scheme, and thus reduce to the case in which
the lower arrow is a morphism of schemes, and the upper arrow is a morphism
of algebraic spaces.  A surjecive \'etale morphism is also scheme-theoretically
dominant, and so pulling back the top arrow over a surjective \'etale morphism
$U' \to V'\times_{\cY'} \cX'$ whose source is a scheme, we finally reduce to
considering a diagram of morphisms of schemes
$$\xymatrix{U \ar[r]\ar[d] & U' \ar[d] \\ V \ar[r] & V'}$$
in which the top arrow is a thickening,
the vertical arrows are quasi-compact and scheme-theoretically dominant,
and the bottom arrow is a closed immersion.

Pulling back over an affine open subscheme of $V'$, and then pulling back the top arrow
over the disjoint union of the members of a finite affine open
cover of the preimage of this affine open in $U'$ (note that this preimage is quasi-compact),
we further reduce
to the case when all the schemes involved are affine.  That is, we have a diagram
of ring morphisms
$$\xymatrix{A' \ar[r] \ar[d] & A \ar[d] \\ B' \ar[r] & B}$$
in which the vertical arrows are injective, the horizontal arrows are surjective,
and the bottom arrow has nilpotent kernel.  One immediately verifies that
the top arrow has nilpotent kernel as well.
\end{proof}

We write $\cC^{\dd, \BT} := \varinjlim_a\cC^{\dd,\BT,a}$
and $\cZ^{\dd}:=\varinjlim_a\cZ^{\dd,a}$; 
we then
have evident morphisms of 
Ind-algebraic stacks
$$ {\cC^{\dd ,\BT}} \to {\cZ^{\dd}} \to
{\cR^{\dd}}$$
lying over $\Spf \cO$,
both representable by algebraic spaces, with the first being furthermore
proper and scheme-theoretically dominant in the sense of~\cite[Def.\ 6.13]{Emertonformalstacks}, 
and the second being a closed immersion.

Similarly, for each choice of tame type $\tau$,
we set
$\cC^{\tau,\BT}=\varinjlim_a\cC^{\tau,a}$, and $\cZ^{\tau}:=\varinjlim_a\cZ^{\tau,a}$.
We again have morphisms 
$$ {\cC^{\tau ,\BT}} \to {\cZ^{\tau}} \to
{\cR^{\dd}}$$
of Ind-algebraic stacks over $\Spf \cO$,
both being representable by algebraic spaces,
the first being proper and scheme-theoretically dominant, and the second
being a closed immersion. Note that by Corollary~\ref{cor: BT C is
  union of tau C}, $\cC^{\dd,\BT}$ is the disjoint union of the
~$\cC^{\tau,\BT}$, so it follows that~$\cZ^{\dd,\BT}$ is the union
(but \emph{not} the disjoint union) of
the~$\cZ^{\tau}$.

Proposition~\ref{prop:C tau BT is algebraic} shows that $\cC^{\tau,\BT} $ is a $\varpi$-adic formal algebraic stack
of finite presentation over $\Spf \cO$.
Each $\cZ^{\tau,a}$ is an algebraic stack of finite presentation
over $\Spec \cO/\varpi^a$ by Theorem~\ref{thm: existence of picture with descent data and its basic
    properties}. 
Analogous remarks apply in the case of $\cC^{\dd,\BT}$ and~$\cZ^{\dd}$.

\begin{prop} 
	\label{prop: Z is a p-adic formal stack}
	$\cZ^{\dd}$, and each $\cZ^{\tau}$, are $\varpi$-adic formal algebraic stacks,
	of finite presentation over $\Spf \cO$.
\end{prop}
\begin{proof}We give the argument for~$\cZ^\dd$, the argument
  for~$\cZ^{\tau}$ being identical. (Alternatively,
  this latter case follows from the former and the fact
  that the canonical morphism $\cZ^{\tau} \hookrightarrow \cZ^{\dd}$
  is a closed immersion.) That $\cZ^{\dd}$ is a $\varpi$-adic formal
  algebraic stack will follow from Proposition~\ref{prop:Y is p-adic
    formal} once we show that the morphism $\cC^{\dd,\BT} \to
  \cZ^{\dd}$ fits into the framework of Section~\ref{subsec:situation}.
	It follows from Lemma~\ref{lem:ind to formal}
	that $\cZ^{\dd}$ is a formal algebraic stack,
	and by construction it is locally Ind-finite type over $\Spec \cO$.
	Furthermore, since each $\cZ^{\dd,a}$ is quasi-compact and
	quasi-separated (being of finite presentation over $\cO/\varpi^a$),
	we see that $\cZ^{\dd}$ is quasi-compact and
        quasi-separated. Thus Proposition~\ref{prop:Y is p-adic
    formal} indeed applies.

	The isomorphism $\cZ^{\dd} \iso \varinjlim_a \cZ^{\dd,a}$ 
	induces an isomorphism
	$$\cZ^{\dd}\times_{\cO} \cO/\varpi^b \iso
	\varinjlim_a \cZ^{\dd,a}\times_{\cO} \cO/\varpi^b,$$
	for any fixed $b \geq 1$.
	Since $\cZ^{\dd}$ is quasi-compact and quasi-separated,
	so is $\cZ^{\dd}\times_{\cO} \cO/\varpi^b,$ and thus this isomorphism
	factors through $\cZ^{\dd,a}\times_{\cO} \cO/\varpi^b$ for some $a$.
	Thus the directed system $\cZ^{\dd,a}\times_{\cO}
        \cO/\varpi^b$ in $a$ 
	eventually stabilises, and so we see that
	$$\cZ^{\dd}\times_{\cO} \cO/\varpi^b \iso \cZ^{\dd,a} \times_{\cO}
	\cO/\varpi^b$$
	for sufficiently large values of $a$.  Since $\cZ^{\dd,a}$
	is of finite presentation over $\cO/\varpi^a$,
	we find that $\cZ^{\dd}\times_{\cO} \cO/\varpi^b$
	is of finite presentation over $\cO/\varpi^b$.  Consequently,
	we conclude that $\cZ^{\dd}$ is of finite presentation
	over $\Spf \cO$, as claimed.
\end{proof}

\begin{remark}
As observed in the general context of Subsection~\ref{subsec:situation},
the thickening $\cZ^{\dd,a} \hookrightarrow \cZ^{\dd}\times_{\cO}
\cO/\varpi^a$ need not be an isomorphism {\em a priori},
and we have no reason to expect that it is.  
Nevertheless, in Subsection~\ref{subsec: generically reduced}
we will prove that this thickening
is {\em generically} an isomorphism for {\em every} value of $a \geq 1$,
and we will furthermore show that each 
$(\cZ^{\dd,a})_{/\F}$ is generically reduced;
see
Proposition~\ref{prop: existence of dense open substack of R such that C is a mono}
and Remark~\ref{rem:U^a = U mod pi^a} below.
The proof of this result involves an application of Proposition~\ref{prop:opens},
and 
depends on the detailed analysis of the irreducible
	components of the algebraic stacks $\cC^{\tau,a}$ 
	and $\cZ^{\tau,a}$ that we will make in
	Section~\ref{sec: extensions of rank one Kisin modules}.
\end{remark}

We conclude this subsection
by establishing some basic lemmas about the reduced substacks underlying 
each of $\cC^{\tau, \BT}$ and $\cZ^{\tau}$.

\begin{lem}
    \label{lem: underlying reduced substack mod p}Let $\cX$ be an
    algebraic stack over $\cO/\varpi^a$, and let $\cX_{\red}$ be
    the underlying reduced substack of~$\cX$. Then $\cX_{\red}$ is a
    closed substack of $\cX_{/\F}:=\cX\times_{\cO/\varpi^a}\F$.
  \end{lem}
  \begin{proof}The structural morphism $\cX\to\Spec\cO/\varpi^a$
    induces a natural morphism $\cX_{\red}\to
    (\Spec\cO/\varpi^a)_{\red}=\Spec\F$, so the natural morphism
    $\cX_{\red}\to \cX$ factors through $\cX_{/\F}$.  Since the
    morphisms $\cX_{\red}\to\cX$ and $\cX_{/\F}\to\cX$ are both closed
    immersions, so is the morphism $\cX_{\red}\to\cX_{/\F}$.
      \end{proof}
\begin{lem}
  \label{lem: scheme-theoretic image commutes with passing to
    underlying reduced substacks}If $f:\cX\to\cY$ is a quasi-compact morphism of
  algebraic stacks, and $\cW$ is the scheme-theoretic image of $f$,
  then the scheme-theoretic image of the induced morphism of
  underlying reduced substacks $f_{\red}:\cX_{\red}\to\cY_{\red}$ is the
  underlying reduced substack $\cW_{\red}$. 
\end{lem}
\begin{proof} 
  Since the definitions of
  the scheme-theoretic image and of the underlying reduced substack
  are both smooth local (in the former case
  see~\cite[Rem. 3.1.5(3)]{EGstacktheoreticimages}, and in the latter
  case it follows immediately from the construction in~\cite[\href{http://stacks.math.columbia.edu/tag/0509}{Tag 0509}]{stacks-project}), we immediately reduce to the case of
  schemes, which follows from~\cite[\href{http://stacks.math.columbia.edu/tag/056B}{Tag 056B}]{stacks-project}. 
  \end{proof}

\begin{lem}
  \label{lem: C 1 and Z 1 are the underlying reduced substacks}
For each $a\ge 1$, $\cC^{\dd,\BT,1}$ is the underlying reduced
  substack of~$\cC^{\dd,\BT,a}$, and $\cZ^{\dd,1}$ is the
  underlying reduced substack of~$\cZ^{\dd,a}$; consequently,
 $\cC^{\dd,\BT,1}$ is the underlying reduced substack of $\cC^{\dd,\BT}$,
 and $\cZ^{\dd,1}$ is the underlying reduced substack of~$\cZ^{\dd}$.
 Similarly, for each tame type $\tau,$
  $\cC^{\tau,\BT,1}$ is the underlying reduced
  substack of each~$\cC^{\tau,\BT,a}$, and of~$\cC^{\tau,\BT}$,
  while $\cZ^{\tau,1}$ is the
  underlying reduced substack of each ~$\cZ^{\tau,a}$, and of~$\cZ^{\tau}$.
\end{lem}
\begin{proof} The statements for the $\varpi$-adic formal algebraic stacks
	follow directly from the corresponding statements for the various
	algebraic stacks modulo $\varpi^a$, and so we focus on proving 
	these latter statements,
	beginning with the case of~$\cC^{\tau,\BT,a}$. Note that
  $\cC^{\tau,\BT,1}=\cC^{\tau,\BT,a}\times_{\cO/\varpi^a}\F$ is
  reduced by Corollary~\ref{cor: Kisin moduli consequences of
    local models}, so
  $\cC^{\dd,\BT,1}=\cC^{\dd,\BT,a}\times_{\cO/\varpi^a}\F$ is
  also reduced by Corollary~\ref{cor:
  BT C is union of tau C}.  The claim follows for $\cC^{\dd,\BT,a}$
and $\cC^{\tau,\BT,a}$  from Lemma~\ref{lem:
    underlying reduced substack mod p}. 
  
The claims for $\cZ^{\tau,a}$  and $\cZ^{\dd,a}$ are then immediate from Lemma~\ref{lem: scheme-theoretic image commutes with passing to
    underlying reduced substacks}, applied to the morphisms
  $\cC^{\tau,\BT,a}\to\cZ^{\tau,a}$ and  $\cC^{\dd,\BT,a}\to\cZ^{\dd,a}$.
\end{proof}

\subsection{Versal rings and equidimensionality}\label{subsec:Galois
  deformation rings}
We now show that $\cC^{\dd,\BT}$ and $\cZ^{\dd,\BT}$ (and their
substacks $\cC^{\tau,\BT}$, $\cZ^{\tau}$) are equidimensional, and
compute their dimensions, by making use of their versal
rings. In~\cite[\S 5]{EGstacktheoreticimages} these versal rings were
constructed in a more general setting in terms of liftings of \'etale
$\varphi$-modules; in our particular setting, we will find it
convenient to interpret them as Galois deformation rings.

  Fix a finite type point $x:\Spec\F'\to\cZ^{\tau,a}$, where $\F'/\F$
is a finite extension; we also denote the induced finite type point of 
$\cR^{\dd,a}$ 
by~$x$.  Let $\rbar:G_K\to\GL_2(\F')$ be the
 Galois representation corresponding to~$x$ by Theorem~\ref{thm: existence of picture with descent data and its basic
    properties}~(3). Let~$E'$ be the compositum of~$E$
  and~$W(\F')[1/p]$, with ring of integers~$\cO_{E'}$ and residue
  field~$\F'$. 

  As in Appendix~\ref{sec:appendix on geom BM}, we 
  have the universal framed deformation $\cO_{E'}$-algebra~$R^\square_{\rbar}$, and we let $R^\square_{\rbar,0,\tau}$ be the reduced
  and $p$-torsion free quotient of $R^\square_{\rbar}$ whose
  $\Qpbar$-points correspond to the potentially Barsotti--Tate lifts
  of~$\rbar$ of type~$\tau$. In this section we will denote
  $R^\square_{\rbar,0,\tau}$ by the more suggestive name
  $R^{\tau,\BT}_{\rbar}$. We recall, for instance from
  \cite[Thm.~3.3.8]{BellovinGee}, that the ring $R^{\tau,\BT}_{\rbar}[1/p]$
  is regular.

As in Section~\ref{subsec: etale phi modules
  and Galois representations}, we write $R_{\rbar|_{G_{K_\infty}}}$
for the universal framed deformation
$\cO_{E'}$-algebra 
for~$\rbar|_{G_{K_\infty}}$. By Lemma~\ref{lem: Galois rep is a
  functor if A is actually finite local}, we have a natural morphism
\numequation
\label{eqn:versal morphism to R}
\Spf R_{\rbar|_{G_{K_\infty}}}\to\cR^{\dd}.
\end{equation}

\begin{lemma}
\label{lem:versal morphism to R}
The morphism~{\em \eqref{eqn:versal morphism to R}} 
is versal {\em (}at~$x${\em )}.
\end{lemma}
\begin{proof}
By definition, it suffices to show that if  
$\rho: G_{K_\infty} \to \GL_d(A)$ is a representation with $A$ a
finite Artinian $\cO_{E'}$-algebra, and if $\rho_B: G_{K_\infty} \to \GL_d(B)$
is a second representation, with $B$ a finite Artinian $\cO_{E'}$-algebra
admitting a surjection onto $A$, such that the base change
$\rho_A$ of $\rho_B$ to $A$ is isomorphic to $\rho$
(more concretely, so that there exists $M \in \GL_d(A)$ with
$\rho = M \rho_A M^{-1}$),
then we may find $\rho': G_{K_\infty} \to \GL_d(B)$ which lifts $\rho$,
and is isomorphic to $\rho_B$.  This is straightforward:\ the natural
morphism $\GL_d(B) \to \GL_d(A)$ is surjective, 
and so if $M'$ is any lift of $M$ to an element of $\GL_d(B)$,
then we may set $\rho' = M' \rho_B (M')^{-1}$.
\end{proof}

\begin{df}
\label{def:GL_2-hat}
For any pro-Artinian $\cO_{E'}$-algebra $R$ with residue field $\F'$ we let  $\widehat{\GL_2}_{/R}$ denote the completion
of~$(\GL_2)_{/R}$ along the closed
subgroup of its special fibre given by the centraliser of
$\rbar|_{G_{K_\infty}}$. 
\end{df}

\begin{remark}\label{rem:pullback-GL_2}
For $R$ as above we have $\widehat{\GL_2}_{/R} = \Spf R \times_{\cO_{E'}} \widehat{\GL_2}_{/\cO_{E'}}$.
Indeed, if $R \cong \varprojlim_i A_i,$ then
$\Spf R \times_{\cO_{E'}} \widehat{\GL_2}_{/\cO_{E'}} \cong \varinjlim_i
\Spec A_i \times_{\cO_{E'}} \widehat{\GL_2}_{/\cO_{E'}}$,
and 
$\Spec A_i \times_{\cO_{E'}} \widehat{\GL_2}_{/\cO_{E'}}$ agrees with the 
completion of $(\GL_2)_{/A_i}$ because $A_i$ is a finite $\cO_{E'}$-module.

It follows from this that $\widehat{\GL_2}_{/R}$ has nice base-change properties
more generally:
if $R \to S$ is a morphism of pro-Artinian $\cO_{E'}$-algebras each
with residue field~$\F'$, then there is an isomorphism $\widehat{\GL_2}_{/S}
\cong \Spf S \times_{\Spf R} \widehat{\GL_2}_{/R}.$
We apply this fact without further comment in various arguments below.
\end{remark}

There is a pair of morphisms
$\widehat{\GL_2}_{/R_{\rbar|_{G_{K_\infty}}}} \rightrightarrows
\Spf R_{\rbar|_{G_{K_\infty}}},$
the first being simply the projection to~$\Spf R_{\rbar|_{G_{K_\infty}}},$
and the second being given by ``change of framing''. 
Composing such changes of framing endows $\widehat{\GL_2}_{/R_{\rbar|_{G_{K_\infty}}}}$ with
the structure of a groupoid over~$\Spf R_{\rbar|_{G_{K_\infty}}}.$
Note that the two morphisms 
$$\widehat{\GL_2}_{/R_{\rbar|_{G_{K_\infty}}}} \rightrightarrows \Spf R_{\rbar|_{G_{K_\infty}}}
\buildrel \text{\eqref{eqn:versal morphism to R}} \over \longrightarrow
\cR^{\dd}$$
coincide, since changing the framing does not change the isomorphism
class (as a Galois representation) of a deformation of~$\rbar|_{G_{K_\infty}}$.
Thus there is an induced morphism of groupoids over~$\Spf R_{\rbar|_{G_{K_\infty}}}$
\numequation
\label{eqn:GL_2 hat equivariance}
\widehat{\GL_2}_{/R_{\rbar|_{G_{K_\infty}}}} \to
\Spf R_{\rbar|_{G_{K_\infty}}}\times_{\cR^{\dd}}\Spf R_{\rbar|_{G_{K_\infty}}}.
\end{equation}

\begin{lemma}
\label{lem:GL_2 hat equivariance}
The morphism~{\em \eqref{eqn:GL_2 hat equivariance}} is an isomorphism.
\end{lemma}
\begin{proof}
If $A$ is an Artinian $\cO_{E'}$-algebra, with residue field $\F'$,
then a pair of $A$-valued points of 
$\Spf R_{\rbar|_{G_{K_\infty}}}$ map to the same point of $\cR^{\dd}$
if and only if they give rise to isomorphic deformations of $\rhobar$,
once we forget the framings.  But this precisely means that the second of them is obtained
from the first by changing the framing via an $A$-valued point of $\widehat{\GL_2}$.
\end{proof}

It follows from Lemma~\ref{lem:versal morphism to R} that,
for each~$a\ge 1$, the quotient $R_{\rbar|_{G_{K_\infty}}}/\varpi^a$ is a
(non-Noetherian) versal ring for~$\cR^{\dd,a}$ at~$x$. By~\cite[Lem.\
3.2.16]{EGstacktheoreticimages}, for each~$a\ge 1$ a versal ring for
$\cZ^{\tau,a}$ at~$x$ is given by the
scheme-theoretic image of the
morphism
\numequation
\label{eqn:pull-back morphism mod a}
\cC^{\tau,\BT,a}\times_{\cR^{\dd,a}}\Spf R_{\rbar|_{G_{K_\infty}}}/\varpi^a \to\Spf
R_{\rbar|_{G_{K_\infty}}}/\varpi^a,
\end{equation}
in the sense that we now explain.

In general, the notion of scheme-theoretic image for morphisms 
of formal algebraic stacks can be problematic; at the very least it should
be handled with care.  But in this particular context,
a definition is given in~\cite[Def.\ 3.2.15]{EGstacktheoreticimages}: 
we write $R_{\rbar|_{G_{K_\infty}}}/\varpi^a$ as an inverse limit of Artinian local
rings $A$, form the corresponding scheme-theoretic images
of the induced morphisms 
$\cC^{\tau,\BT,a}\times_{\cR^{\dd,a}}\Spec A \to\Spec A,$
and then take the inductive limit of these scheme-theoretic images;
this is a formal scheme, which is in fact of the form $\Spf R^{\tau,a}$
for some quotient $R^{\tau,a}$ of $R_{\rbar_{| G_{K_{\infty}}}}/\varpi^a$
(where {\em quotient} should be understood in the sense of topological rings),
and is by definition the scheme-theoretic
image in question.

The closed
immersions $\cC^{\tau,\BT,a}\into \cC^{\tau,\BT,a+1}$ induce
corresponding closed immersions
$$
\cC^{\tau,\BT,a}\times_{\cR^{\dd,a}}\Spf R_{\rbar|_{G_{K_\infty}}}/\varpi^a \to
\cC^{\tau,\BT,a+1}\times_{\cR^{\dd,a+1}}\Spf R_{\rbar|_{G_{K_\infty}}}/\varpi^{a+1},$$
and hence closed immersions of scheme-theoretic images
$\Spf R^{\tau,a}\to \Spf R^{\tau,a+1}$, corresponding to surjections
$R^{\tau,a+1} \to R^{\tau,a}$.
(Here we are using the fact that
an projective limit of surjections of finite Artin
rings is surjective.)
Thus we may form the pro-Artinian ring 
$\varprojlim_a R^{\tau,a}.$
This projective limit is a quotient (again in the sense of topological rings)
of $R_{\rbar_{| G_{K_{\infty}}}}$, 
and the closed formal subscheme $\Spf (\varprojlim R^{\tau,a})$ 
of $\Spf R_{\rbar_{| G_{K_{\infty}}}}$ 
is the scheme-theoretic image (computed in the sense 
described above) of the projection
\numequation
\label{eqn:pull-back morphism}
{\cC^{\tau,\BT}}\times_{{\cR^{\dd}}}\Spf
R_{\rbar|_{G_{K_\infty}}}\to \Spf R_{\rbar|_{G_{K_\infty}}}
\end{equation}
(This is a formal consequence of the construction
of the $\Spf R^{\tau,a}$ as scheme-theoretic images,
since any discrete Artinian quotient of
$R_{\rbar |_{G_{K_{\infty}}}}$
is a discrete Artinian quotient of 
$R_{\rbar |_{G_{K_{\infty}}}}/\varpi^a$,
for some~$a~\geq~1$.)
It also follows formally (for example, by the
same argument as in the proof of~\cite[Lem.\
4.2.14]{EGstacktheoreticimages})  that $\varprojlim R^{\tau,a}$ is a versal
ring to ${\cZ^\tau}$ at~$x$.
Our next aim is to identify 
this projective limit
with $R^{\tau,\BT}_{\rbar}$. 

Before we do this, we have to establish some preliminary facts 
related to the various objects and morphisms we have just introduced.

\begin{lemma}  
\label{lem:R-tau-a properties}\leavevmode
\begin{enumerate}
\item
Each of the rings $R^{\tau,a}$ is a complete local Noetherian ring,
endowed with its $\mathfrak m$-adic topology,
and the same is true of the inverse limit $\varprojlim_a R^{\tau,a}$.
\item
For each $a \geq 1,$ the morphism $\Spf R^{\tau,a} \to \Spf R_{\rbar_{|G_{K_{\infty}}}}$
induces an isomorphism
$$
\cC^{\tau,\BT,a}\times_{\cR^{\dd,a}}\Spf R^{\tau,a} \iso
\cC^{\tau,\BT,a}\times_{\cR^{\dd,a}}\Spf R_{\rbar|_{G_{K_\infty}}}/\varpi^a.
$$
\item
For each $a \geq 1$,
the morphism $\Spf R^{\tau,a} \to \cR^{\dd,a}$ is effective, i.e.\
may be promoted {\em (}in a unique manner{\em )} to a morphism
$\Spec R^{\tau,a} \to \cR^{\dd,a}$,
and the induced morphism $$\cC^{\tau,a} \times_{\cR^{\dd,a}} \Spec R^{\tau,a} 
\to \Spec R^{\tau,a}$$ is proper and scheme-theoretically dominant.
\item
Each transition morphism $\Spec R^{\tau,a} \hookrightarrow \Spec R^{\tau,a+1}$
is a thickening.
\end{enumerate}
\end{lemma}
\begin{proof}
Recall that in Section~\ref{subsubsec: deformation rings and Kisin
  modules} we defined a Noetherian quotient $R_{\rbar|_{G_{K_\infty}}}^{\le 1}$
of $R_{\rbar|_{G_{K_\infty}}}$, which is naturally identified with the framed
deformation ring $R_{\rbar}^{[0,1]}$ by Proposition~\ref{prop:
  restriction gives an equivalence of height at most 1
  deformation}. It follows from~\cite[Lem.\
5.4.15]{EGstacktheoreticimages}  (via an argument almost identical to
the one in the proof of~\cite[Prop.\ 5.4.17]{EGstacktheoreticimages}) that the morphism
$\Spf \varprojlim R^{\tau,a}\into\Spf R_{\rbar|_{G_{K_\infty}}}$ factors through
$\Spf R_{\rbar|_{G_{K_\infty}}}^{\le 1}=\Spf R_{\rbar}^{[0,1]}$, and indeed
that $\varprojlim R^{\tau,a}$ is a quotient of $\Spf R_{\rbar}^{[0,1]};$
this proves~(1).

It follows by the very construction of the $R^{\tau,a}$ that the
morphism~\eqref{eqn:pull-back morphism mod a}
factors through the closed subscheme $\Spf R^{\tau,a}$
of $\Spf R_{\rbar_{|G_{K_{\infty}}}}/\varpi^a$.
The claim of~(2) is a formal consequence of this.

We have already observed that the morphism $\Spf R^{\tau,a} \to \cR^{\dd,a}$
factors through $\cZ^{\tau,a}$.  This latter stack is algebraic, and of finite
type over~$\cO/\varpi^a$.  
    It follows
from~\cite[\href{https://stacks.math.columbia.edu/tag/07X8}{Tag
  07X8}]{stacks-project}
that the morphism $\Spf R^{\tau,a} \to \cZ^{\tau,a}$ is effective.
Taking into account part~(1) of the present lemma,
we deduce from the theorem on formal functions that
the formal completion of the scheme-theoretic image of the projection
$$\cC^{\tau,a}\times_{\cR^{\dd,a}} \Spec R^{\tau,a} \to \Spec R^{\tau,a}$$
at the closed point of $\Spec R^{\tau,a}$ coincides
with the scheme-theoretic image of the morphism
$$\cC^{\tau,a}\times_{\cR^{\dd,a}} \Spf R^{\tau,a} \to \Spf R^{\tau,a}.$$
Taking into account~(2), we see that this latter scheme-theoretic image coincides with
$\Spf R^{\tau,a}$ itself.  This completes the proof of~(3).

The claim of~(4) follows from a consideration of the diagram
$$\xymatrix{\cC^{\tau,a} \times_{\cR^{\dd,a}} \Spec R^{\tau,a} \ar[r] \ar[d]
& \cC^{\tau,a+1} \times_{\cR^{\dd,a+1}} \Spec R^{\tau,a+1}  \ar[d] \\
\Spec R^{\tau, a} \ar[r] & \Spec R^{\tau,a+1} }$$
just as in the proof of Lemma~\ref{lem:Z thickenings}.
\end{proof}

\begin{lemma}
\label{lem:algebraization}\leavevmode
\begin{enumerate}
\item The projection
$\cC^{\tau,\BT} \times_{\cR^{\dd}} 
\Spf R_{\rbar_{|G_{K_{\infty}}}}
\to
\Spf R_{\rbar_{|G_{K_{\infty}}}}$
factors through a morphism
$\cC^{\tau,\BT} \times_{\cR^{\dd}} 
\Spf R_{\rbar_{|G_{K_{\infty}}}}
\to \Spf(\varprojlim R^{\tau,a})$,
which is scheme-theoretically dominant
in the sense that its scheme-theoretic image {\em (}computed in the manner
described above{\em )} is equal to its target.
\item There is a projective
morphism of schemes $X_{\rbar} \to \Spec (\varprojlim R^{\tau,a})$,
which is uniquely 
determined, up to unique isomorphism, by the requirement that its
$\mathfrak m$-adic completion {\em (}where $\mathfrak m$ denotes
the maximal ideal of $\varprojlim R^{\tau,a}${\em )} may be identified with
the morphism
$\cC^{\tau,\BT} \times_{\cR^{\dd}} 
\Spf R_{\rbar_{|G_{K_{\infty}}}}
\to \Spf(\varprojlim R^{\tau,a})$
of~(1).
\end{enumerate}
\end{lemma}
\begin{proof}
Part~(1) follows formally from the various constructions and definitions
of the objects involved (just like part~(2) of Lemma~\ref{lem:R-tau-a properties}).

We now consider the morphism
\[{\cC^{\tau,\BT}}\times_{{\cR^{\dd}}}\Spf R_{\rbar|_{G_{K_\infty}}}
\to \Spf (\varprojlim R^{\tau,a}).\]  Once we recall
that
$\varprojlim R^{\tau,a}$
is Noetherian, by Lemma~\ref{lem:R-tau-a properties}~(1),
it follows exactly as in the proof
of~\cite[Prop.\ 2.1.10]{kis04} (which treats the case that
$\tau$ is the trivial type), 
via an application of formal GAGA~\cite[Thm.\ 5.4.5]{MR217085},
that this morphism arises as the formal completion along the maximal ideal of
$\varprojlim R^{\tau,a}$
of a projective morphism
$X_{\rbar}\to \Spec (\varprojlim R^{\tau,a})$
(and $X_{\rbar}$ is unique up to unique isomorphism, by~\cite[Thm.\ 5.4.1]{MR217085}). 
\end{proof}

We next establish various properties of the scheme $X_{\rbar}$
constructed in the previous lemma.
To ease notation going forward, we write $\widehat{X}_{\rbar}$ to
denote the fibre product 
${\cC^{\tau,\BT}}\times_{{\cR^{\dd}}}\Spf R_{\rbar|_{G_{K_\infty}}}$
(which is reasonable, since this fibre product is isomorphic to the
formal completion of $X_{\rbar}$).

\begin{lemma}
\label{lem:X is normal and flat}
The scheme $X_{\rbar}$ is Noetherian, normal, and flat over $\cO_{E'}$.
\end{lemma}
\begin{proof}
Since $X_{\rbar}$ is projective over the Noetherian ring $\varprojlim R^{\tau,a}$,
it is Noetherian.   The other claimed properties of~$X_{\rbar}$
will be deduced from the corresponding properties of $\cC^{\tau,\BT}$ that
are proved in 
Corollary~\ref{cor: Kisin moduli consequences of local models}.

To this end, we first note that,
since the morphism $\cC^{\tau,\BT} \to \cR^{\dd}$ factors through~$\cZ^{\tau}$,
it follows (for example as in the proof of~\cite[Lem.\
3.2.16]{EGstacktheoreticimages}) that we have isomorphisms
\begin{multline*}
\cC^{\tau,\BT}\times_{\cZ^{\tau}}\Spf (\varprojlim R^{\tau,a})
\iso
\cC^{\tau,\BT}\times_{\cZ^{\tau}} \cZ^{\tau}\times_{\cR^{\dd}}
\Spf R_{\rbar|_{G_{K_\infty}}}
\\
\iso
\cC^{\tau,\BT}\times_{\cR^{\dd}} 
\Spf R_{\rbar|_{G_{K_\infty}}} =: \widehat{X}_{\rbar}.
\end{multline*}
In summary, we may identify
$\widehat{X}_{\rbar}$ with the fibre product
$\cC^{\tau,\BT}\times_{\cZ^{\tau}}\Spf (\varprojlim R^{\tau,a})$.

We now show that
$\widehat{X}_{\rbar}$ 
is analytically normal.   To see this,
let $\Spf B \to \widehat{X}_{\rbar}$
be a morphism whose source is a Noetherian affine formal algebraic
space, which is
representable by algebraic spaces and smooth. 
We must show
that the completion $\widehat{B}_{\mathfrak n}$ is normal,
for each maximal ideal $\mathfrak n$ of $B$.
In fact, it suffices to verify this for some collection of such $\Spf B$
which cover $\widehat{X}_{\rbar}$, and so without loss of generality
we may choose our $B$ as follows:
first, choose a collection of morphisms 
$\Spf A \to \cC^{\tau,\BT}$
whose sources are Noetherian affine formal algebraic spaces, and which are
representable by algebraic spaces and smooth, which, taken together, cover~$\cC^{\tau,\BT}$.
Next, for each such $A$, choose a collection of morphisms
$$\Spf B \to \Spf_A \times_{\cC^{\tau,\BT}} \widehat{X}_{\rbar}$$
whose sources are Noetherian affine formal algebraic spaces, and which are
representable by algebraic spaces and smooth, which, taken together, cover the
fibre product.  Altogether (considering all such $B$ associated to all such $A$),
the composite morphisms
$$\Spf B \to \Spf_A \times_{\cC^{\tau,\BT}} \widehat{X}_{\rbar} \to \widehat{X}_{\rbar}$$
are representable by algebraic spaces and smooth,
and cover~$\widehat{X}_{\rbar}$. 

Now, let $\mathfrak n$ be a maximal ideal in one of these rings $B$,
lying over a maximal ideal $\mathfrak m$ in the corresponding ring $A$.
The extension of residue fields $A/\mathfrak m \to B/\mathfrak n$ is
finite, and each of these fields is finite over $\F'$. Enlarging $\F'$ sufficiently,
we may assume that in fact each of these residue fields coincides with $\F'$.
(On the level of rings, this amounts to forming various tensor products of the
form $\text{--}\otimes_{W(\F')} W(\F'')$, which doesn't affect the question
of normality.)  
The morphism $\Spf B_{\mathfrak n} \to \Spf A_{\mathfrak n}$ is then
seen to be smooth
in the sense of~\cite[\href{https://stacks.math.columbia.edu/tag/06HG}{Tag
  06HG}]{stacks-project}, i.e., it satisfies the infinitesimal lifting property for 
finite Artinian $\cO'$-algebras with residue field $\F'$:\ this follows from
the identification of $\widehat{X}_{\rbar}$ above as a fibre product,
and the fact that $\Spf (\varprojlim R^{\tau,a}) \to \cZ^{\tau}$ is versal at
the closed point~$x$.
Thus $\Spf B_{\mathfrak n}$ is a formal power series ring
over $\Spf A_{\mathfrak m}$, 
by~\cite[\href{https://stacks.math.columbia.edu/tag/06HL}{Tag
  06HL}]{stacks-project}, 
and hence $\Spf B_{\mathfrak n}$ is indeed normal,
since $\Spf A_{\mathfrak m}$ is so,
by Corollary~\ref{cor: Kisin moduli consequences of local models}.
By Lemma~\ref{lem:GAGA facts} below,
this implies that the algebraization $X_{\rbar}$ of $\widehat{X}_{\rbar}$ is normal.

We next claim that the morphism
\numequation
\label{eqn:versal morphism to Z}
\Spf (\varprojlim R^{\tau,a}) \to \cZ^{\tau}
\end{equation}
is a flat morphism of formal algebraic stacks,
in the sense of \cite[Def.~8.35]{Emertonformalstacks}.
Given this, we find that the base-changed morphism
$\widehat{X}_{\rbar} \to \cC^{\tau,\BT}$
is also flat.
Since
Corollary~\ref{cor: Kisin moduli consequences of local models} shows that
$\cC^{\tau,\BT}$ is flat over $\cO_{E'},$ 
we conclude that the same is true of~$\widehat{X}_{\rbar}$.
Again, by Lemma~\ref{lem:GAGA facts}, this implies that the algebraization $X_{\rbar}$
is also flat over~$\cO_{E'}$.

It remains to show the claimed flatness.  To this end, we note first
that for each $a \geq 1$, the morphism
\numequation
\label{eqn:versal morphism to Z-a}
\Spf R^{\tau,a} \to \cZ^{\tau,a}
\end{equation}
is a versal morphism from a complete Noetherian local ring to an
algebraic stack which is locally of finite type over $\cO/\varpi^a$.
We already observed in the proof of Lemma~\ref{lem:R-tau-a properties}~(3) 
that~\eqref{eqn:versal morphism to Z-a} is effective,
i.e.\ can be promoted to a morphism
$\Spec R^{\tau,a} \to \cZ^{\tau,a}$.  It then follows 
from~\cite[\href{https://stacks.math.columbia.edu/tag/0DR2}{Tag
  0DR2}]{stacks-project} that this latter morphism is flat,
and thus that~\eqref{eqn:versal morphism to Z-a}
is flat in the sense of \cite[Def.~8.35]{Emertonformalstacks}.
It follows easily that the morphism~\eqref{eqn:versal morphism to Z}
is also flat:\ use the fact that a morphism
 of $\varpi$-adically complete local Noetherian $\cO$-algebras
 which becomes flat upon reduction modulo $\varpi^a$, for each $a \geq 1$,
 is itself flat, which follows from (for example) \cite[\href{https://stacks.math.columbia.edu/tag/0523}{Tag 0523}]{stacks-project}.
\end{proof}

The following lemma is standard, and is presumably well-known.  We sketch the proof,
since we don't know a reference.

\begin{lemma}
\label{lem:GAGA facts}
If $S$ is a complete Noetherian local $\cO$-algebra and $Y \to \Spec S$ is a proper
morphism of schemes, then $Y$ is flat over $\Spec \cO$
{\em (}resp.\ normal{\em )} if and only
$\widehat{Y}$ {\em (}the $\mathfrak m_S$-adic completion of~$Y${\em )}
is flat over $\Spf \cO$ {\em (}resp.\ is analytically normal{\em )}.
\end{lemma}
\begin{proof}
The properties of $Y$ that are in question can be tested by considering the 
various local rings $\cO_{Y,y}$, as $y$ runs over the points of $Y$; namely,
we have to consider whether or not these rings are flat over~$\cO$,
or normal.  Since
any point $y$ specializes to a closed point $y_0$ of $Y$, so that 
$\cO_{Y,y}$ is a localization of $\cO_{Y,y_0}$, and thus $\cO$-flat (resp.\
normal) if $\cO_{Y,y_0}$ is, it suffices to consider the rings $\cO_{Y,y_0}$
for closed points $y_0$ of~$Y$.  
Note also that since $Y$ is proper over $\Spec S$, any closed point of $Y$ lies over
the closed point of $\Spec S$.  

Now let $\Spec A$ be an affine neighbourhood of a closed point $y_0$ of $Y$;
let $\mathfrak m$ be the corresponding maximal ideal of~$A$.  As we noted,
$\mathfrak m$ lies over~$\mathfrak m_S$, and so gives rise to a maximal
ideal $\widehat{\mathfrak m} := \mathfrak m \widehat{A}$ of~$\widehat{A}$,
the $\mathfrak m_S$-adic completion of~$A$;
and any
maximal ideal of $\widehat{A}$ contains $\mathfrak m_S \widehat{A}$,
and so arises from a maximal ideal of~$A$ in this manner (since $A/\mathfrak m_S \iso
\widehat{A}/\mathfrak m_S$).  Write $\widehat{A}_{\mathfrak m}$ to denote
the $\mathfrak m$-adic completion of $A$ (which maps isomorphically to
the $\widehat{\mathfrak m}$-adic completion of $\widehat{A}$).  Then $\widehat{A}$
is faithfully flat over the localization $A_{\mathfrak m} = \cO_{Y,y_0}$,
and hence $A_{\mathfrak m}$ is flat over $\cO$ if and only if
$\widehat{A}_{\mathfrak m}$ is. Consequently we see that $Y$ is flat over 
$\cO$ if and only if, for each affine open subset $\Spec A$ of $Y$, 
the corresponding $\mathfrak m_S$-adic completion $\widehat{A}$ becomes flat over
$\cO$ after completing at each of its maximal ideals.  Another application
of faithful flatness of completions of Noetherian local rings shows
that this holds if and only if each such $\widehat{A}$ is flat over $\cO$ after localizing
at each of its maximal ideals, which holds if and only each
such $\widehat{A}$ is flat over~$\cO$.  This is precisely what it means for
$\widehat{Y}$ to be flat over~$\cO$.

The proof that analytic normality of $\widehat{Y}$ implies that $Y$ is normal
is similar.  Indeed, analytic normality by definition means that the
completion of $\widehat{A}$ at each of its maximal ideals is normal.  This completion
is faithfully flat over the localization of $\Spec A$ at its corresponding 
maximal ideal,
and so~\cite[\href{https://stacks.math.columbia.edu/tag/033G}{Tag 033G}]{stacks-project}
implies that this localization is also normal.  The discussion of the first paragraph
then implies that $Y$ is normal.  For the converse direction, we have to deduce
normality of the completions $\widehat{A}_{\mathfrak m}$ from the normality of 
the corresponding localizations~$A_{\mathfrak m}$.  This follows from that fact
that  $Y$
is an excellent scheme (being of finite type over the complete local ring~$S$),
so that each $A$ is an excellent
ring~\cite[\href{https://stacks.math.columbia.edu/tag/0C23}{Tag
    0C23}]{stacks-project}.
\end{proof}

\begin{prop}\label{prop: projective morphism as in Kisin}The projective morphism $X_{\rbar}\to \Spec R_{\rbar}^{[0,1]}$ factors
  through a projective and scheme-theoretically dominant morphism
  \numequation
  \label{eqn:projective morphism as in Kisin}
  X_{\rbar}\to \Spec R^{\tau,\BT}_{\rbar}
  \end{equation}
  which becomes an isomorphism after inverting~$\varpi$.  \end{prop}
\begin{proof}
We begin by showing the existence of~\eqref{eqn:projective morphism as in Kisin},
and that it induces a bijection on closed points after inverting~$\varpi$.
    Since $X_{\rbar}$ 
    is $\cO$-flat, by Lemma~\ref{lem:X is normal and flat},
it suffices to show that the induced morphism
    $$\Spec E \times_{\cO} X_{\rbar} \to \Spec R^{[0,1]}_{\rbar}[1/\varpi]$$
    factors through a morphism
  \numequation
  \label{eqn:projective morphism as in Kisin p inverted}
    \Spec E\times_{\cO} X_{\rbar} \to \Spec R^{\tau,\BT}_{\rbar}[1/\varpi],
\end{equation}
    which induces a bijection on closed points.

  This can be proved in exactly the same way as~\cite[Prop.\
2.4.8]{kis04}, 
    which treats the case that $\tau$ is
    trivial. 
    Indeed,
the computation of the $D_{\cris}$ of a Galois
    representation in the proof
    of~\cite[Prop.\ 2.4.8]{kis04} goes over essentially unchanged to
    the case of a Galois representation coming from $\cC^{\tau,\BT}$, and
    finite type points of $\Spec R_{\rbar}^{\tau,\BT}[1/\varpi]$ yield $p$-divisible
    groups and thus Breuil--Kisin modules exactly as in the proof
    of~\cite[Prop.\ 2.4.8]{kis04} (bearing in mind Lemma~\ref{lem: O
        points of moduli stacks} above). 
    The tame descent data comes along
    for the ride.

The morphism~\eqref{eqn:projective morphism as in Kisin p inverted}
is a projective morphism whose target is Jacobson, and which induces a bijection
on closed points.  It is thus proper and quasi-finite, and hence finite.
Its source is reduced (being even normal, by Lemma~\ref{lem:X is normal and flat}),
and its target is normal (as it is even regular, as we noted above).
A finite morphism whose source is reduced,
whose target is normal and Noetherian, and which
induces a bijection on finite type points, is indeed an isomorphism.
(The connected components of a normal scheme are integral,
and so base-changing over the connected components of the target,
we may assume that the target is integral.  The source is a union of finitely many 
irreducible components, each of which has closed image in the target.
Since the morphism is surjective on finite type  points, it is surjective,
and thus one of these closed images coincides with the target.  The injectivity
on finite type points then shows that the source is also irreducible,
and thus integral, as it is reduced.   
It follows from~\cite[\href{https://stacks.math.columbia.edu/tag/0AB1}{Tag 0AB1}]{stacks-project} that the morphism is an isomorphism.)
Thus~\eqref{eqn:projective morphism as in Kisin p inverted}
is an isomorphism.
Finally, since $R^{\tau,\BT}_{\rbar}$ is also flat over $\cO$ (by its definition),
this implies
that~\eqref{eqn:projective morphism as in Kisin}
is scheme-theoretically dominant.
  \end{proof}

\begin{cor}
  \label{cor: R tau BT is a versal ring to Z-hat}$\varprojlim
  R^{\tau,a}=R^{\tau,\BT}_{\rbar}$; thus $R^{\tau,\BT}_{\rbar}$ is a
  versal ring to ${\cZ^\tau}$ at~$x$.
\end{cor}
\begin{proof}
The theorem on formal functions shows that if we write the
scheme-theoretic image
of~\eqref{eqn:projective morphism as in Kisin}
in the form $\Spec B$, for some quotient $B$ of $R_{\rbar_{|G_{K_{\infty}}}}$,
then the scheme-theoretic image of the morphism~\eqref{eqn:pull-back morphism}
coincides with $\Spf B$.
The corollary then follows from Proposition~\ref{prop: projective morphism
as in Kisin}, which shows that~\eqref{eqn:projective morphism as in Kisin}
is scheme-theoretically dominant.
  \end{proof}

\begin{prop}
  \label{prop: dimensions of the Z stacks}
The algebraic stacks $\cZ^{\dd,a}$ and  $\cZ^{\tau,a}$ are equidimensional of
dimension~$[K:\Qp]$.
\end{prop}
\begin{proof}
Let $x$ be a finite type point of $\cZ^{\tau,a}$, defined over some finite extension
$\F'$ of $\F$, and corresponding to a Galois representation $\rbar$
with coefficients in~$\F'$.
By Corollary~\ref{cor: R tau BT is a versal ring to Z-hat}
the ring $R^{\tau,\BT}_{\rbar}$ coincides with the versal ring $\varprojlim_a R^{\tau, a}$
at $x$ of the $\varpi$-adic formal
algebraic stack~$\cZ^{\tau}$,
and so $\Spf R^{\tau,a} \iso \Spf R^{\tau,\BT}_{\rbar} \times_{\cZ^{\tau}} \cZ^{\tau,a}.$
Since $\cZ^{\tau}$ is a $\varpi$-adic formal algebraic stack,
the natural morphism $\cZ^{\tau,1} \to \cZ^{\tau}\times_{\Spf \cO}\F$
is a thickening, and thus the same is true of the morphism
$\Spf R^{\tau,1} \to \Spf R^{\tau,\BT}_{\rbar}/\varpi$
obtained by pulling the former morphism back over~$\Spf R^{\tau,\BT}_{\rbar}/\varpi$.

  Since $R^{\tau,\BT}_{\rbar}$ is
  flat over $\cO_{E'}$ and 
  equidimensional of dimension $5+[K:\Qp]$, it follows that
  $R^{\tau,1}$ 
  is equidimensional of dimension~$4~+~[K:\Qp]$.
  The same is then true of each~$R^{\tau,a}$, since these are thickenings
  of~$R^{\tau,1}$, by Lemma~\ref{lem:R-tau-a properties}~(4).

We have a versal morphism $\Spf
  R^{\tau,a}\to\cZ^{\tau,a}$ at the finite type point~$x$ of $\cZ^{\tau,a}$. It follows from Lemma~\ref{lem:GL_2 hat
    equivariance} that 
\[\widehat{\GL_2}_{/\Spf R^{\tau,a}} \iso
\Spf R^{\tau,a}\times_{\cZ^{\tau,a}}\Spf R^{\tau,a}.\]

To find the dimension of $\cZ^{\tau,a}$ it
suffices to compute its dimension at finite type points (\emph{cf}.\ \cite[\href{https://stacks.math.columbia.edu/tag/0DRX}{Tag
  0DRX}]{stacks-project}, recalling the definition of the dimension of
an algebraic stack,
\cite[\href{https://stacks.math.columbia.edu/tag/0AFP}{Tag
  0AFP}]{stacks-project}).  It follows from~\cite[Lem.\ 2.40]{2017arXiv170407654E}
applied to the presentation
  $[\Spf R^{\tau,a} / \widehat{\GL_2}_{/\Spf R^{\tau,a}}]$ 
  of $\widehat{\cZ}^{\tau,a}_{x}$, together with Remark~\ref{rem:pullback-GL_2},
that $\cZ^{\tau,a}$ is
equidimensional of dimension $[K:\Qp]$. Since $\cZ^{\dd,a}$ is the
 union of the $\cZ^{\tau,a}$ by Theorem~\ref{thm: existence of picture with descent data and its basic
    properties},  $\cZ^{\dd,a}$ is also
equidimensional of dimension $[K:\Qp]$
by~\cite[\href{https://stacks.math.columbia.edu/tag/0DRZ}{Tag
  0DRZ}]{stacks-project}. 
  \end{proof}

\begin{prop}
  \label{prop: C tau is equidimensional of the expected dimension}The
  algebraic stacks $\cC^{\tau,\BT,a}$ are equidimensional of
  dimension $[K:\Qp]$.
\end{prop}
\begin{proof}
Let $x'$ be a finite type point of  $\cC^{\tau,\BT,a}$, defined
over some finite extension $\F'$ of $\F$,
lying over the
finite type point~$x$ of~$\cZ^{\tau,a}$.  Let $\rbar$ be the Galois
representation with coefficients in $\F'$ corresponding to~$x$,
and recall that $X_{\rbar}$ denotes a projective $\Spec R_{\rbar}^{\tau,\BT}$-scheme
whose pull-back $\widehat{X}_{\rbar}$
over $\Spf R_{\rbar_{|G_{K_{\infty}}}}$ is isomorphic to
${\cC^{\tau,\BT}}\times_{{\cR^{\dd}}}\Spf R_{\rbar|_{G_{K_\infty}}}.$
The point $x'$ gives rise to a closed point $\tx$ of $X_{\rbar}$ (of which $x'$
is the image under the morphism $X_{\rbar} \to \cC^{\tau,\BT}$).
Let
$\widehat{\cO}_{X_{\rbar},\tx}$ denote the complete local ring to
$X_{\rbar}$ at the point~$\tx$; then the natural morphism $\Spf
\widehat{\cO}_{X_{\rbar},\tx}\to {\cC^{\tau,\BT}}$
is versal at~$\tx$,
so that $\widehat{\cO}_{X_{\rbar},\tx}/\varpi^a$ is a versal ring for the point
$x'$ of $\cC^{\tau,\BT,a}$.

The isomorphism~\eqref{eqn:GL_2 hat equivariance} induces
(after pulling back over $\cC^{\tau,\BT}$)
an isomorphism 
$$\widehat{\GL_2}_{/\widehat{X}_{\rbar}} \iso
\widehat{X}_{\rbar}\times_{\cC^{\tau,\BT}} \widehat{X}_{\rbar},$$
and thence an isomorphism
$$\widehat{\GL_2}_{/\widehat{\cO}_{X_{\rbar},\tx}} \iso
\widehat{\cO}_{X_{\rbar},\tx}
\times_{\cC^{\tau,\BT}} 
\widehat{\cO}_{X_{\rbar},\tx}.$$

Since $R^{\tau,\BT}$ is  equidimensional of dimension $5+[K:\Qp]$,
it follows from Proposition~\ref{prop: projective morphism as in Kisin} that
$X_{\rbar}$ is equidimensional of dimension $5+[K:\Qp]$,
and thus (taking into account the flatness statement of
Lemma~\ref{lem:X is normal and flat}) that
$\widehat{\cO}_{X_{\rbar},\tx}/\varpi^a$ is equidimensional of
dimension $4+[K:\Qp]$.
As in the proof of 
  Proposition~\ref{prop: dimensions of the Z stacks},
an application of~\cite[Lem.\ 2.40]{2017arXiv170407654E}
shows that $\dim_{x'}{\cC^{\tau,\BT,a}}$ is equal to $[K:\Qp]$. 
Since~$x'$ was an arbitrary finite type point, the result follows.
\end{proof}

\subsection{The Dieudonn\'e stack}\label{subsec: Dieudonne
  stack} 

We now specialise the choice of $K'$ in the following way. Choose a
tame inertial type $\tau=\eta\oplus\eta'$.
Fix a uniformiser $\pi$ of~$K$. If $\tau$ is a tame
principal series type, we take $K'=K(\pi^{1/(p^f-1)})$, while
if~$\tau$ is a tame cuspidal type, we let $L$ be an unramified
quadratic extension of~$K$, and set $K'=L(\pi^{1/(p^{2f}-1)})$. Let
$N$ be the maximal unramified extension of $K$ in $K'$. In
either case $K'/K$ is a Galois extension; in the principal series
case, we have $e'=(p^f-1)e$, $f'=f$, and in the cuspidal case we have
$e'=(p^{2f}-1)e$, $f'=2f$.  We refer to this choice of extension as the
\emph{standard choice} (for the fixed type $\tau$ and uniformiser
$\pi$). 

For the rest of this section we assume that
$\eta\ne\eta'$ (we will not need to consider Dieudonn\'e modules for
scalar types).

Let $\gM$ be a Breuil--Kisin module with $A$-coefficients and descent data of type
$\tau$ and height at most $1$, and let $D := \gM/u\gM$ be
its corresponding Dieudonn\'e module as in Definition~\ref{def: Dieudonne module formulas}. 
If we write
$D_i:=e_iD$, then this Dieudonn\'e module is given by rank two
projective modules
 $D_j$ over $A$ ($j = 0,\ldots, f'-1$) with linear maps $F: D_j \to D_{j+1}$
and $V: D_j \to D_{j-1}$ (subscripts understood modulo $f'$) 
such that $FV = VF = p$.

Now,
  $I(K'/K)$ is abelian of order prime to $p$, so we can write
  $D=D_\eta\oplus D_{\eta'}$, where $D_\eta$ is the submodule on which
  $I(K'/K)$ acts via~$\eta$. Since $\gM_\eta$ is obtained from the projective $\gS_A$-module $\gM$
by applying a projector, each $D_{\eta,j}$ is an invertible
$A$-module, and $F,V$ induce linear
maps $F:D_{\eta,j}\to D_{\eta,j+1}$ and $V: D_{\eta,j+1} \to D_{\eta,j}$ 
such that $FV = VF = p$. 

We can of course apply the same construction with $\eta'$ in the place
of $\eta$, obtaining a Dieudonn\'e module $D_{\eta'}$.
We now prove some lemmas relating these various Dieudonn\'e
modules. We will need to make use of a variant of the strong
determinant condition, so we begin by discussing this and its relationship to
the strong determinant condition of Subsection~\ref{subsec: local
  models rank two}. 

\begin{defn}
  \label{defn: determinant condition over L} 
Let~$(\gL,\gL^+)$ be a
  pair consisting of a rank two projective
  $\cO_{K'}\otimes_{\Zp}A$-module $\gL$, and an
  $\cO_{K'}\otimes_{\Zp}A$-submodule $\gL^+ \subset \gL$, such that
  Zariski locally on~$\Spec A$, $\gL^+$ is a direct summand of~$\gL$
  as an $A$-module. 

Then we say that the pair~$(\gL,\gL^+)$ \emph{satisfies the Kottwitz
  determinant condition over~$K'$} if for all~$a\in\cO_{K'}$, we
have \[\det{\!}_A(a|\gL^+)=\prod_{\psi:K'\into E}\psi(a) \]as polynomial functions on $\cO_{K'}$ in the sense of~\cite[\S
  5]{MR1124982}.
\end{defn}

There is a finite type stack~$\cM_{K',\det}$ over~$\Spec \cO$, with
$\cM_{K',\det}(\Spec A)$ being the groupoid of pairs~$(\gL,\gL^+)$ as
above which satisfy the Kottwitz determinant condition over~$K'$. As we
have seen above, by a
result of Pappas--Rapoport, this
stack is flat over~$\Spec\cO$ (see~\cite[Prop.\ 2.2.2]{kis04}).

\begin{lem}
  \label{lem: determinant condition  generic fibre}If $A$ is an
  $E$-algebra, then a pair~$(\gL,\gL^+)$ as in Definition~{\em \ref{defn: determinant condition over L}} satisfies the Kottwitz
  determinant condition over~$K'$ if and only if~$\gL^+$ is a rank one
  projective $\cO_{K'}\otimes_{\Zp}A$-module.
\end{lem}
\begin{proof}
  We may write
  $\cO_{K'}\otimes_{\Zp}A=K'\otimes_{\Qp}A\cong\prod_{\psi:K'\into E}A$, where
  the embedding $\psi:K'\into E$ corresponds to an idempotent
  $e_\psi\in K'\otimes_{\Qp}A$. Decomposing $\gL^+$ as $\oplus_{\psi}
  e_{\psi} \gL^+$, the left-hand side of the Kottwitz determinant condition
becomes $\prod_{\psi} \det{\!}_A(a|e_\psi \gL^+) = \prod_{\psi}
\psi(a)^{\mathrm{rk}_A e_{\psi} \gL^+}$. It follows that the Kottwitz
  determinant condition is satisfied if and only if the projective
  $A$-module $e_\psi
  \gL^+$ has rank one for all $\psi$, which is equivalent to~$\gL^+$ being a
  rank one projective~$K'\otimes_{\Qp}A$-module, as required.
\end{proof}

\begin{prop}
  \label{prop: for us strong det implies det}If~$\gM$ is an object
  of~$\cC^{\tau,\BT}(A)$, 
  then
  the pair \[(\gM/E(u)\gM,\im\Phi_{\gM}/E(u)\gM)\] satisfies the
  Kottwitz determinant condition for~$K'$.
\end{prop}
\begin{proof}Let~$\cC^{\tau,\BT'}$ be the closed substack
  of~$\cC^{\tau}$ consisting of those~$\gM$ for which the
  pair~$(\gM/E(u)\gM,\im\Phi_{\gM}/E(u)\gM)$ satisfies the Kottwitz
  determinant condition for~$K'$. We need to show
  that~$\cC^{\tau,\BT}$ is a closed substack
  of~$\cC^{\tau,\BT'}$. Since~$\cC^{\tau,\BT}$ is flat over~$\Spf\cO$
  by Corollary~\ref{cor: Kisin moduli consequences of local models},
  it is enough to show that if~$A$ is an $E$-algebra, then
  $\cC^{\tau,\BT}(A)=\cC^{\tau,\BT'}(A)$. 
  
To see this, let~$\gM$ be an object of~$\cC^{\tau}(A)$. By Lemma~\ref{lem: determinant condition
  generic fibre}, $\gM$ is an object of~$\cC^{\tau,\BT'}(A)$ if and
only if $\im\Phi_{\gM}/E(u)\gM$ is a rank one
projective~$K'\otimes_{\Qp}A$-module. Similarly,  $\gM$ is an object of~$\cC^{\tau,\BT}(A)$ if and
only if for each~$\xi$, $(\im\Phi_{\gM})_{\xi}/E(u)\gM_{\xi}$ is a rank one
projective~$N\otimes_{\Qp}A$-module. Since \[\im\Phi_{\gM}/E(u)\gM=\oplus_{\xi}(\im\Phi_{\gM})_{\xi}/E(u)\gM_{\xi},\]the
equivalence of these two conditions is clear.
\end{proof}

\begin{lemma}
\label{lem:dets for local models}
If $(\gL,\gL^+)$ is an object of $\cM_{K',\det}(A)$ {\em (}i.e.\
satisfies the Kottwitz determinant condition over $K'${\em )}, 
then the morphism $\bigwedge^2_{\cO_{K'}\otimes_{\Z_p} A} \gL^+ \to
\bigwedge^2_{\cO_{K'}\otimes_{\Z_p} A} \gL$ induced
by the inclusion $\gL^+ \subset \gL$ is identically zero.
\end{lemma}

\begin{remark}
Note that, although $\gL^+$ need not be locally free over
$\cO_{K'}\otimes_{\Z_p} A$, its exterior square is nevertheless defined,
so that the statement of the lemma makes sense.
\end{remark}

\begin{proof}[Proof of Lemma~{\ref{lem:dets for local models}}]
Since~$\cM_{K',\det}$ is $\cO$-flat, 
it is enough to treat the case that $A$ is $\cO$-flat. In this case $\gL$, and thus
also $\bigwedge^2 \gL$, are $\cO$-flat.  Given this additional assumption,
it suffices to prove that the morphism of the lemma becomes zero after tensoring
with $\Q_p$ over $\Z_p$.   This morphism may naturally be identified with
the morphism
$\bigwedge^2_{K'\otimes_{\Z_p} A} \gL^+ \to \bigwedge^2_{K'\otimes_{\Z_p} A} \gL$
induced by the injection
$\Q_p\otimes_{\Z_p} \gL^+ \hookrightarrow \Q_p\otimes_{\Z_p} \gL$.
Locally on $\Spec A$, this is the embedding of a free $K'\otimes_{\Z_p} A$-module
of rank one as a direct summand of a free $K'\otimes_{\Z_p} A$-module of rank
two.  Thus $\bigwedge^2$ of the source in fact vanishes, and hence so does $\bigwedge^2$ of the embedding.
\end{proof}

\begin{lemma}
\label{lem:BT condition gives control of det Phi}
If $\gM$ is an object of $\cC^{\tau,\BT}(A)$, then
$\bigwedge^2 \Phi_{\gM}: \bigwedge^2 \varphi^*\gM  \hookrightarrow \bigwedge^2 \gM$
is exactly divisible by $E(u)$, i.e.\ can be written as $E(u)$ times an
isomorphism of $\gS_A$-modules.
\end{lemma}
\begin{proof}
It follows from Proposition~\ref{prop: for us strong det implies det} and
Lemma~\ref{lem:dets for local models} that the reduction of
$\bigwedge^2 \Phi_{\gM}$ modulo $E(u)$ vanishes, so we can think
of~$\bigwedge^2\Phi_{\gM}$ as a morphism
$\bigwedge^2\varphi^*\gM\to E(u)\bigwedge^2\gM$.
We need to show that the cokernel~$X$ of this
morphism vanishes. Since~$\im\Phi_{\gM}\supseteq E(u)\gM$, $X$ is a finitely generated $A$-module, so that in order to prove
that it vanishes, it is enough to prove that~$X/pX=0$. 

Since the formation of cokernels is compatible with
base change, this means that we can (and do) assume that~$A$ is an
$\F$-algebra. Since the special fibre~$\cC^{\tau,\BT}$ is of finite
type over~$\F$, we can and do assume that~$A$ is furthermore of finite
type over~$\F$. 
The special fibre of~$\cC^{\tau,\BT}$ is reduced by
Corollary~\ref{cor: Kisin moduli consequences of local models}, 
so we
may assume that~$A$ is reduced, 
and it is therefore enough to prove that~$X$ vanishes
modulo each maximal ideal of~$A$. Since the residue fields at such
maximal ideals are finite, 
we are reduced to the case that~$A$ is a finite
field, when the result follows from~\cite[Lem.\ 2.5.1]{kis04}. 
\end{proof}

\begin{lemma}
\label{lem:Dieudonne iso}
There is a canonical isomorphism
$$\text{``}(F\otimes F)/p\text{''}
:D_{\eta,j} \otimes_A
D_{\eta',j} \iso D_{\eta,j+1} \otimes_A D_{\eta',j+1},$$
characterised by the fact that it is compatible with change of scalars,
and that $$p \cdot \text{``}(F\otimes F)/p\text{''}= F\otimes F.$$
\end{lemma}
\begin{proof}
Since $\cC^{\BT}$ is flat over $\cO$, we see that in the universal case,
the formula $$p \cdot \text{``}(F\otimes F)/p\text{''}= F\otimes F$$
uniquely determines the isomorphism
$\text{``}(F\otimes F)/p\text{''}$ (if it exists).  Since any 
Breuil--Kisin module with descent data is obtained from the universal case by
change of scalars, we see that the isomorphism
$\text{``}(F\otimes F)/p\text{''}$ is indeed characterised by the properties
stated in the lemma, provided that it exists. 

To check that the isomorphism exists, we can again consider the universal case,
and hence assume that $A$ is a flat $\cO$-algebra. 
In this case, it suffices to check that the morphism
$F\otimes F: D_{\eta,j}\otimes_A D_{\eta',j} \to D_{\eta,j+1}\otimes_A D_{\eta,
j+1}$ is divisible by $p$, and that the formula
$(F\otimes F)/p$ is indeed an isomorphism.
Noting that the direct sum over $j = 0,\ldots, f'-1$ of these morphisms may
be identified with the reduction modulo $u$ of the morphism
$\bigwedge^2 \Phi_{\gM}: \bigwedge^2 \varphi^*\gM \to \bigwedge^2
\gM$, 
this follows from Lemma~\ref{lem:BT condition gives control of det Phi}.
\end{proof}

The isomorphism $\text{``}(F\otimes F)/p \text{''}$ of the preceding lemma
may be rewritten as an isomorphism of invertible $A$-modules
\numequation
\label{eqn:Dieudonne hom iso}
\Hom_A(D_{\eta,j}, D_{\eta,j+1}) \iso \Hom_A(D_{\eta',j+1},D_{\eta',j}).
\end{equation}

\begin{lem}
\label{lem:swapping F and V}
The isomorphism~{\em (\ref{eqn:Dieudonne hom iso})} takes $F$ to $V$.
\end{lem}
\begin{proof}
The claim of the lemma is equivalent to showing that
the composite 
$$D_{\eta,j} \otimes_A D_{\eta',j+1} \buildrel
\id\otimes V \over \longrightarrow 
D_{\eta,j}\otimes_A D_{\eta',j} \buildrel \text{``}(F\otimes F)/p \text{''}
\over \longrightarrow D_{\eta,j+1}\otimes_A D_{\eta',j+1}$$
coincides with the morphism $F \otimes \id.$  
It suffices to check this in the universal case, and thus we may assume that $p$ is a non-zero divisor in $A$, and hence verify the required identity of morphisms after multiplying each of them by $p$.
The identity to be verified then becomes
$$(F\otimes F) \circ (\id \otimes V) \buildrel ? \over = p (F\otimes \id),$$
which follows immediately from the formula $F V = p$.
\end{proof}

We now consider the moduli stacks classifying the Dieudonn\'e modules
with the properties we have just established, and the maps from the
moduli stacks of Breuil--Kisin modules to these stacks.

Suppose first that we are in the principal series case. Then there is a moduli stack classifying the data of the $D_{\eta,j}$
together with the
$F$ and $V$, namely the stack 
$$\cD_{\eta} :=
\Big[
\bigl( \Spec W(k)[X_0,Y_0,\ldots,X_{f-1},Y_{f-1}]/(X_j Y_j - p)_{j = 0,\ldots, f-1}) \bigr) /
\mathbb G_m^f  \big],$$
where the $f$ copies of $\mathbb G_m$ act as follows:
$$(u_0,\ldots,u_{f-1}) \cdot (X_j,Y_j) \mapsto (u_j u_{j+1}^{-1} X_j, u_{j+1} u_j^{-1} Y_j).$$
To see this, recall that the stack
$$[\text{point}/\Gm]$$ classifies line bundles, so the $f$ copies of $\mathbb G_m$
in $\cD_\eta$ correspond to $f$ line bundles, which are the line bundles $D_{\eta,j}$
($j = 0,\ldots,f-1$).  If we locally trivialise these line bundles, then the
maps $F: D_{\eta,j} \to D_{\eta,j+1}$ and $V:D_{\eta,j+1} \to D_{\eta,j}$ act by
scalars, which we denote by $X_j$ and $Y_j$ respectively.  The $f$ copies of
$\mathbb G_m$ are then encoding possible changes of trivialisation, by units
$u_j$, which induce the indicated changes on the $X_j$'s and $Y_j$'s.

There is then a natural map
\[\cC^\tau
 \to \cD_{\eta},\]
classifying the Dieudonn\'e modules
underlying the Breuil--Kisin modules with descent data. 

There is a more geometric way to think about what $\cD_{\eta}$ 
classifies.
To begin with, we just rephrase what we've already indicated:
it represents the functor which associates to a $W(k)$-scheme the
groupoid whose objects are $f$-tuples of line bundles
$(D_{\eta,j})_{j = 0,\ldots,f-1}$ equipped with morphisms $X_j: D_{\eta,j} \to D_{\eta,j+1}$
and $Y_j:D_{\eta,j+1} \to D_{\eta,j}$ such that $Y_j X_j = p$.
(Morphisms in the groupoid are just isomorphisms between collections of such data.)
Equivalently, we can think of this as giving the line bundle
$D_{\eta,0}$, and then the $f$ line bundles $\cD_j
:= D_{\eta,j+1}\otimes D_{\eta,j}^{-1}$,
equipped with sections $X_j \in \cD_j$ and $Y_j \in \cD_j^{-1}$ whose
product in $\cD_j \otimes \cD_j^{-1} = \cO$
(the trivial line bundle) is equal to the element~$p$.
Note that it superficially looks like we are remembering $f+1$ line bundles,
rather than $f$, but this is illusory, since in fact
$\cD_0 \otimes \cdots \otimes \cD_{f-1}$ is trivial; indeed,
the isomorphism 
$\cD_0 \otimes \cdots \otimes \cD_{f-1} \iso \cO$ is part of the data we should remember.

It will be helpful to introduce another stack,
the stack $\cG_{\eta}$ of $\eta$-gauges.  This classifies
$f$-tuples of line bundles $\cD_j$ ($j = 0,\ldots,f-1$) equipped
with sections $X_j \in \cD_j$ and $Y_j \in \cD_j^{-1}$.
Explicitly, it can be written as the quotient stack
$$\cG_{\eta} :=
\Big[
\bigl( \Spec W(k)[X_0,Y_0,\ldots,X_{f-1},Y_{f-1}]/(X_j Y_j - p)_{j = 0,\ldots, f-1}) \bigr) /
\mathbb G_m^f  \big],$$
where the $f$ copies of $\mathbb G_m$ act as follows:
$$(v_0,\ldots,v_{f-1}) \cdot (X_j,Y_j) \mapsto (v_j X_j, v_j^{-1} Y_j).$$
There is a natural morphism of stacks $\cD_{\eta} \to \cG_{\eta}$
given by forgetting forgetting $D_0$ and the isomorphism
$\cD_0 \otimes \cD_1 \otimes \dots\otimes \cD_{f-1} \iso \cO$.
In terms of the explicit descriptions via quotient stacks, 
we have a morphism $\Gm^f \to \Gm^f$ given by
$(u_j)_{j = 0,\ldots,f-1} \mapsto (u_j u_{j+1}^{-1})_{j = 0,\ldots,f-1}$,
which is compatible with the actions of these two groups on
$\Spec W(k)[(X_j,Y_j)_{j=0,\ldots,f-1}]/(X_j Y_j - p)_{j = 0,
\ldots , f-1},$ and we are just considering the map from the quotient
by the first $\Gm^f$ to the quotient by the second~$\Gm^f$.

Composing our morphism $\cC^\tau \to \cD_{\eta}$ with the forgetful morphism
$\cD_{\eta} \to \cG_{\eta}$, we obtain a morphism $\cC^\tau \to \cG_{\eta}$.

We now turn to the case that~$\tau$ is a cuspidal type. In this case
our Dieudonn\'e modules have unramified as well as inertial descent
data; accordingly, we let $\varphi^f$ denote the element of
$\Gal(K'/K)$ which acts trivially on $\pi^{1/(p^{2f}-1)}$ and
non-trivially on $L$. Then the descent data of~ $\varphi^f$ induces
isomorphisms $D_j\isoto D_{j+f}$, which are compatible with the $F,V$,
and which identify $D_{\eta,j}$ with $D_{\eta',f+j}$. 

If we choose local trivialisations of the line bundles
$D_{\eta,0},\dots,D_{\eta,f}$, then the maps $F : D_{\eta,j} \to
D_{\eta,j+1}$ and $V : D_{\eta,j+1} \to D_{\eta,j}$ for $0 \le j \le
f-1$ are given by scalars $X_j$ and $Y_j$ respectively. The identification of $D_{\eta,j}$ and $D_{\eta',f+j}$ given
by~$\varphi^f$ identifies $D_{\eta,j}\otimes D_{\eta,j+1}^{-1}$ with
$D_{\eta',f+j}\otimes D_{\eta',f+j+1}^{-1}$, which via the isomorphsim
\eqref{eqn:Dieudonne hom iso} is identified with
$D_{\eta,f+j+1}\otimes D_{\eta,f+j}^{-1}$. It follows that for $0\le j\le f-2$ the
data of $D_{\eta,j}$, $D_{\eta,j+1}$ and $D_{\eta,f+j}$ recursively
determines $D_{\eta,f+j+1}$. From Lemma~\ref{lem:swapping F and V}
we see, again recursively  for $0\le j\le f-2$, that there are unique trivialisations
of $D_{\eta,f+1},\dots,D_{\eta,2f-1}$ such that
$F:D_{\eta,f+j}\to D_{\eta,f+j+1}$ is given by~$Y_{j}$, and
$V:D_{\eta,f+j+1}\to D_{\eta,f+j}$ is given by~$X_j$. Furthermore,
there is some unit~$\alpha$ such that $F:D_{\eta,2f-1}\to D_{\eta,0}$
is given by~$\alpha Y_{f-1}$, and
$V:D_{\eta,0}\to D_{\eta,2f-1}$ is given by~$\alpha^{-1}
X_{f-1}$. Note that the map $F^{2f} : D_{\eta,0} \to D_{\eta,0}$ is
precisely $ p^f\alpha$.

Consequently, we see that the data of the  $D_{\eta,j}$ (together with
the $F,V$) is classified by the stack  $$\cD_{\eta} :=
\Big[
\bigl( \Spec W(k)[X_0,Y_0,\ldots,X_{f-1},Y_{f-1}]/(X_j Y_j - p)_{j = 0,\ldots, f-1}) \times\Gm\bigr) /
\mathbb G_m^{f+1}  \big],$$
where the $f+1$ copies of $\mathbb G_m$ act as follows:
$$(u_0,\ldots,u_{f-1},u_f) \cdot ((X_j,Y_j),\alpha) \mapsto ((u_j u_{j+1}^{-1}
X_j, u_{j+1} u_j^{-1} Y_j),\alpha ). $$

We again define $$\cG_{\eta} :=
\Big[
\bigl( \Spec W(k)[X_0,Y_0,\ldots,X_{f-1},Y_{f-1}]/(X_j Y_j - p)_{j = 0,\ldots, f-1}) \bigr) /
\mathbb G_m^f  \big],$$
where the $f$ copies of $\mathbb G_m$ act as
$$(v_0,\ldots,v_{f-1}) \cdot (X_j,Y_j) \mapsto (v_j X_j, v_j^{-1} Y_j).$$
There are again  natural morphisms of stacks $\cC^\tau\to\cD_{\eta} \to
\cG_{\eta}$, where the second morphism is given 
in terms of the explicit descriptions via quotient stacks as follows:
we have a morphism $\Gm^{f+1} \to \Gm^f$ given by
$(u_j)_{j = 0,\ldots,f} \mapsto (u_j u_{j+1}^{-1})_{j = 0,\ldots,f-1}$,
and the morphism $\cD_{\eta} \to
\cG_{\eta}$ is the obvious one which forgets the factor of~$\Gm$
coming from~$\alpha$. 

For our analysis of the irreducible components of the stacks
$\cC^{\tau,\BT,1}$ at the end of Section~\ref{sec: extensions of rank one Kisin
  modules}, it will be useful to have a 
more directly geometric interpretation
of a morphism $S \to \cG_{\eta}$, in the case that
the source is a {\em flat} $W(k)$-scheme, or, more generally,
a flat $p$-adic formal algebraic stack over~$\Spf
W(k)$. In order to do this we will need some basic material on
effective Cartier divisors for (formal) algebraic stacks; while it is
presumably possible to develop this theory in considerable generality,
 we only need a very special case, and we limit ourselves to this
 setting.

The property of a closed subscheme being an effective Cartier divisor is not 
preserved under arbitrary pull-back, but it is preserved under flat
pull-back. More precisely, we have the following result.

\begin{lemma}\label{lem:Cartier divisors are flat local}
       	If $X$ is a scheme,
	and $Z$ is a closed subscheme of $X$,
	then the following are equivalent:
	\begin{enumerate}
		\item $Z$ is an effective Cartier divisor on $X$.
		\item For any flat morphism of schemes $U \to X$,
			the pull-back $Z\times_{X} U$
			is an effective Cartier divisor on $U$.
	 	\item For some fpqc covering $\{X_i \to X\}$ of $X$,
	 		each of the pull-backs $Z\times_{X} X_i$
	 		is an effective Cartier divisor on $X_i$.
	\end{enumerate}
\end{lemma}
\begin{proof}
	Since $Z$ is an effective Cartier divisor if and only if its ideal sheaf
	$\cI_Z$ is an invertible sheaf on $X$, this follows from
	the fact that the invertibility of a quasi-coherent sheaf
	is a local property in the {\em fpqc} topology.
\end{proof}

\begin{lemma}
	\label{lem:comparing closed subsets}
	If $A$ is a Noetherian adic topological ring,
	then pull-back under the natural morphism $\Spf A \to \Spec A$
	induces a bijection between the closed subschemes of 
	$\Spec A$ and the closed subspaces of
	$\Spf A$.
\end{lemma}
\begin{proof}
	It follows
	from~\cite[\href{http://stacks.math.columbia.edu/tag/0ANQ}{Tag
 0ANQ}]{stacks-project}
that closed immersions $Z \to \Spf A$
are necessarily of the form $\Spf B \to \Spf A$,
and correspond to continuous morphisms $A \to B$, for some complete 
linearly topologized
ring $B$, which are taut (in the sense
	of~\cite[\href{http://stacks.math.columbia.edu/tag/0AMX}{Tag
 0AMX}]{stacks-project}),
have closed kernel, and dense image.
Since $A$ is adic, it admits a countable basis of neighbourhoods of the origin,
and so it follows 
	from~\cite[\href{http://stacks.math.columbia.edu/tag/0APT}{Tag
 0APT}]{stacks-project} (recalling also~\cite[\href{http://stacks.math.columbia.edu/tag/0AMV}{Tag
 0AMV}]{stacks-project})  that $A\to B$ is surjective.  
Because any ideal of definition $I$ of $A$ is finitely generated, it follows 
	from~\cite[\href{http://stacks.math.columbia.edu/tag/0APU}{Tag
 0APU}]{stacks-project} that $B$ is endowed with the $I$-adic topology.
Finally, since $A$ is Noetherian, any ideal in $A$ is $I$-adically closed.
Thus closed immersions $\Spf B \to \Spf A$ are determined by giving
the kernel of the corresponding morphism $A \to B$, which can be arbitrary.
The same is true of closed immersions $\Spec B \to \Spec A$,
and so the lemma follows.
\end{proof}

\begin{df} If $A$ is a Noetherian adic topological ring,
then we say that a closed subspace of $\Spf A$
is an {\em effective Cartier divisor} on $\Spf A$ if the corresponding closed
subscheme of $\Spec A$ is an effective Cartier divisor on $\Spec A$.
\end{df}

\begin{lemma}
	Let $\Spf B \to \Spf A$ be a flat adic morphism of Noetherian
	affine formal algebraic spaces.
	If $Z \hookrightarrow \Spf A$ is a Cartier divisor,
	then $Z \times_{\Spf A} \Spf B \hookrightarrow \Spf B$ 
	is a Cartier divisor.  Conversely, if $\Spf B \to \Spf A$
	is furthermore surjective, and if $Z \hookrightarrow \Spf A$
	is a closed subspace for which the base-change
	$Z \times_{\Spf A} \Spf B \hookrightarrow \Spf B$ 
	is a Cartier divisor,
	then $Z$ is a Cartier divisor on $\Spf A$.
\end{lemma}
\begin{proof}
	The morphism $\Spf B \to \Spf A$ corresponds to an adic flat morphism
	$A \to B$ 
(\cite[\href{http://stacks.math.columbia.edu/tag/0AN0}{Tag
0AN0}]{stacks-project}
and \cite[Lem.\ 8.18]{Emertonformalstacks})
	and hence is induced by a flat morphism $\Spec B \to \Spec A$,
	which is furthermore faithfully flat if and only if 
	$\Spf B \to \Spf A$ is surjective
(again by \cite[Lem.\ 8.18]{Emertonformalstacks}).
	The present lemma thus follows from Lemma~\ref{lem:Cartier divisors
	are flat local}. 
\end{proof}

The preceding lemma justifies the following
definition.

\begin{df} We say that a closed substack $\cZ$ of a locally Noetherian
	formal algebraic stack
	$\cX$
	is an {\em effective Cartier divisor} on $\cX$ if
		 for any morphism $U \to \cX$ whose source
			is a Noetherian affine formal algebraic space,
			and which is representable by algebraic spaces and
			flat,
			the pull-back $\cZ\times_{\cX} U$
			is an effective Cartier divisor on $U$.
\end{df}

We consider the $W(k)$-scheme $\Spec W(k)[X,Y]/(XY - p)$, which we endow with
a $\Gm$-action via $u\cdot (X,Y) := (uX, u^{-1} Y).$
There is an obvious morphism
$$\Spec W(k)[X,Y]/(XY -p) \to \Spec W(k)[X]=\mathbb A^1$$
given by $(X,Y) \to X$, which is $\Gm$-equivariant (for the action
of~$\Gm$ on~$\mathbb A^1$ given by $u\cdot X:=uX$),
and so induces a morphism
\numequation
\label{eqn:Cartier divisor construction}
[\bigl(\Spec W(k)[X,Y]/(XY -p)\bigr) / \Gm]
\to  [\mathbb A^1/\Gm].
\end{equation}

\begin{lemma}\label{lem: f equals 1 effective Cartier} If $\cX$ is a
  locally Noetherian $p$-adic formal algebraic stack
	which is furthermore flat over $\Spf W(k)$,
        then the groupoid of morphisms
	\[\cX \to [\Spec W(k)[X,Y]/(XY - p) / \Gm]\] is in fact 
	a setoid, and is equivalent to the set of effective Cartier divisors 
	on $\cX$ that are contained in the effective Cartier divisor
	$(\Spec k) \times_{\Spf W(k)} \cX$ on~$\cX$. 
\end{lemma}
\begin{proof}
  Essentially by definition (and taking into account \cite[Lem.\
  8.18]{Emertonformalstacks}), it suffices to prove this in the case
  when $\cX = \Spf B$, where $B$ is a flat Noetherian adic
  $W(k)$-algebra admitting $(p)$ as an ideal of definition.  In this
  case, the restriction map
  \[[\Spec W(k)[X,Y]/(XY - p) / \Gm](\Spec B)\to [\Spec W(k)[X,Y]/(XY -
  p) / \Gm](\Spf B)\] is an equivalence of groupoids. 
Indeed, the
  essential surjectivity follows from the (standard and easily
  verified) fact that if $\{M_i\}$ is a compatible family of locally
  free $B/p^iB$-modules of rank one, then $M := \varprojlim M_i$ is a
  locally free $B$-module of rank one, for which each of the natural
  morphisms $M/p^iM \to M_i$ is an isomorphism.  The full
  faithfulness follows from the fact that a locally free $B$-module
  of rank one is $p$-adically complete, and so is recovered as the
  inverse limit of its compatible family of quotients $\{M/p^iM\}.$

We are therefore reduced to the same statement with $\cX = \Spec B$.  The composite morphism $\Spec B \to  [\mathbb A^1/\Gm]$ induced
  by~\eqref{eqn:Cartier divisor construction} corresponds to giving a
  pair~$(\cD,X)$ where~$\cD$ is a line bundle  on~$\Spec B$, and~$X$
  is a global section of~$\cD^{-1}$. Indeed, giving a morphism $\Spec
  B \to  [\mathbb A^1/\Gm]$ is equivalent
	to giving a $\Gm$-torsor $P \to \Spec B$, together with a $\Gm$-equivariant
	morphism $P \to \mathbb A^1$.   Giving a $\Gm$-torsor 
	$P$ over $\Spec B$ is equivalent to giving an invertible sheaf
        $\cD$ on $\Spec B$
	(the associated $\Gm$-torsor is then obtained by deleting the
	zero section from the line bundle $D\to X$ corresponding to $\cD$),
	and giving a $\Gm$-equivariant morphism $P \to \mathbb A^1$
	is equivalent to giving a global section of $\cD^{-1}$.

It follows that giving a morphism
  $\Spec B \to [\Spec W(k)[X,Y]/(XY - p) / \Gm]$ corresponds to giving
  a line bundle $\cD$ and sections $X \in \cD^{-1}$, $Y \in \cD$ satisfying
  $X Y = p$.  To say that $B$ is flat over $W(k)$ is just to say that
  $p$ is a regular element on $B$, and so we see that $X$ 
  (resp.\ $Y$) is a regular section of $\cD^{-1}$ (resp.\ $\cD$).
  Again, since $p$ is a
  regular element on $B$, we see that $Y$ is uniquely determined by
  $X$ and the equation $X Y = p$, and so giving a morphism
  $\Spec B\to [\Spec W(k)[X,Y]/(XY - p) / \Gm]$ is equivalent to giving a line
  bundle $\cD$ and a regular section $X$ of $\cD^{-1}$, such that
  $pB \subset X\otimes_B \cD \subset \cD^{-1}\otimes_B \cD\iso B$;
  this last condition guarantees the existence of the (then uniquely
  determined) $Y$.

Now giving a line bundle $\cD$ on $\Spec B$ and a regular section
$X\in\cD^{-1}$ is the
same as giving the zero locus $D$ of $X$, which is a Cartier divisor
on $\Spec B$.
(There is a canonical isomorphism $(\cD,X) \cong
\bigl(\cI_D,1\bigr)$, where $\cI_D$ denotes the ideal sheaf of $D$.) 
The condition that $pB \subset X\otimes_B \cD$ is equivalent 
to the condition that $p \in \cI_D$,
i.e.\ that $D$ be contained in $\Spec B/pB$, and we are done.
\end{proof}

\begin{lemma}\label{lem: maps to gauge stack as Cartier
    divisors} 
If $\cS$ is a locally Noetherian $p$-adic formal algebraic stack which
is flat over $W(k)$, 
then giving a morphism $\cS \to \cG_{\eta}$
over $W(k)$ is equivalent to giving a collection
of effective Cartier divisors $\cD_j$ on $\cS$ {\em (}$j = 0,
\ldots,f-1${\em )}, with each $\cD_j$ contained in the Cartier divisor
$\overline{\cS}$ cut out by the equation $p = 0$ on $\cS$ {\em (}i.e.\
the {\em special fibre} of $\cS${\em )}.
\end{lemma}
\begin{proof}
	This follows immediately from Lemma~\ref{lem: f equals 1
          effective Cartier}, by the definition of~$\cG_\eta$.
\end{proof}

\section{Extensions of rank one Breuil--Kisin modules with descent
  data}\label{sec: extensions of rank one Kisin modules}
The goal of this section is to construct certain universal families of
extensions of rank one Breuil--Kisin modules over $\F$ with descent data, and
to use these to describe the generic behaviour of the various irreducible 
components of the special fibres of~$\cC^{\tau,\BT}$ and~$\cZ^{\tau}$.

In Subsection~\ref{subsec:ext generalities} we present some generalities
on extensions of Breuil--Kisin modules.
In Subsection~\ref{subsec:universal families}
we explain how to construct our desired families of extensions.
In Subsection~\ref{subsec: Diamond--Savitt} we recall the fundamental
computations related to extensions of rank one Breuil--Kisin modules
from \cite{DiamondSavitt}, to which the results of
Subsection~\ref{subsec:universal families} will be applied.

We assume throughout this section that $[K':K]$ is not divisible
by~$p$; since we are assuming throughout the paper that $K'/K$ is
tamely ramified, this is equivalent to assuming that $K'$ does not
contain an unramified extension of $K$ of degree~$p$. In our final
applications $K'/K$ will contain unramified extensions of degree at
most~$2$, and $p$ will be odd, so this assumption will be satisfied.
(In fact, we specialize to such a context begining in
Subsection~\ref{subsec:irreducible}.)

\subsection{Extensions of Breuil--Kisin modules with descent data}
\label{subsec:ext generalities}
When discussing the general theory of extensions of Breuil--Kisin
modules, it is convenient to embed the category of Breuil--Kisin modules
in a larger category which is abelian,
contains enough injectives and projectives,
and is closed under passing to arbitrary limits and colimits.
The simplest way to obtain such a category is as the category of modules 
over some ring, and so we briefly recall how a Breuil--Kisin module with
$A$-coefficients and  descent 
data can be interpreted as a module over a certain $A$-algebra.

Let $\gS_A[F]$ denote the twisted polynomial ring over $\gS_A$,
in which the variable $F$ obeys the following commutation relation
with respect to elements $s \in \gS_A$:
$$ F \cdot s = \varphi(s) \cdot F.$$
Let $\gS_A[F, \Gal(K'/K)]$ denote the twisted group ring over $\gS_A[F]$,
in which the elements $g \in \Gal(K'/K)$ commute with $F$,
and obey the following commutation
relations with elements $s  \in \gS_A$:
$$ g \cdot s = g(s) \cdot g.$$
One immediately confirms that giving a left $\gS_A[F, \Gal(K'/K)]$-module $\gM$
is equivalent to equipping the underlying $\gS_A$-module
$\gM$ with a $\varphi$-linear morphism
$\varphi:\gM \to \gM$ and a semi-linear action of $\Gal(K'/K)$
which commutes with $\varphi$.

In particular, if we let $\K{A}$ denote the category of left $\gS_A[F,
\Gal(K'/K)]$-modules, then
a Breuil--Kisin module with descent data from $K'$ to $K$ 
 may naturally be regarded as an object of $\K{A}$.
In the following lemma, we record the fact that extensions of Breuil--Kisin modules
with descent data may be computed as extensions in the category~$\K{A}$.  

\begin{lemma}
\label{lem:ext of a Kisin module is a Kisin module}
If $0 \to \gM' \to \gM \to \gM'' \to 0$ is a short exact sequence
in $\K{A}$, such that $\gM'$ {\em (}resp.\ $\gM''${\em )}
is a Breuil--Kisin module with descent data
of rank $d'$ and height at most $h'$ 
{\em (}resp.\ of rank $d''$ and height at most $h''$\emph{)}, 
then $\gM$ is a Breuil--Kisin module with descent data
of rank $d'+d''$ and height at most $h'+h''$.  

More generally, if
$E(u)^h\in\Ann_{\gS_A}(\coker\Phi_{\gM'})\Ann_{\gS_A}(\coker\Phi_{\gM''})$,
then $\gM$ is a Breuil--Kisin module with descent data of height at most~$h$.
\end{lemma}
\begin{proof}
Note that since $\Phi_{\gM'}[1/E(u)]$ and $\Phi_{\gM''}[1/E(u)]$ are both isomorphisms by
assumption, it follows from the snake lemma that~$\Phi_{\gM}[1/E(u)]$ is
isomorphism.  Similarly we have a short exact sequence of
$\gS_A$-modules \[0\to\coker\Phi_{\gM'}\to\coker\Phi_{\gM}\to\coker\Phi_{\gM''}\to
0.\]
The claims about the height and rank of~$\gM$ follow immediately.
\end{proof}

We now turn to giving an explicit description of the functors
$\Ext^i(\gM, \text{--} \, )$ for a Breuil--Kisin module with descent data
$\gM$. 

\begin{df}
\label{def:explicit Ext complex}
Let $\gM$ be a 
Breuil--Kisin module with $A$-coefficients and descent data (of some 
height).
If $\gN$ is any object of $\K{A}$, then 
we let $C^{\bullet}_{\gM}(\gN)$ denote the complex
$$ 
\Hom_{\gS_A[\Gal(K'/K)]}(\gM,\gN) \to
\Hom_{\gS_A[\Gal(K'/K)]}(\varphi^*\gM,\gN), 
$$
with differential being given by 
$$\alpha \mapsto \Phi_{\gN} \circ \varphi^* \alpha - \alpha \circ \Phi_{\gM}.$$
Also let $\Phi_{\gM}^*$ denote the map $C^0_{\gM}(\gN) 
\to C^1_{\gM}(\gN)$ given by $\alpha \mapsto \alpha \circ
\Phi_{\gM}$. When $\gM$ is clear from the context we will usually suppress it from the
notation and write simply $C^{\bullet}(\gN)$. 
\end{df}

Each $C^i(\gN)$ is naturally an $\gS_A^0$-module. 
The formation of $C^{\bullet}(\gN)$ is evidently functorial in $\gN$,
and is also exact in $\gN$, since $\gM$, and hence also $\varphi^*\gM$,
is projective over $\gS_A$, and since $\Gal(K'/K)$ has prime-to-$p$ order.
Thus the cohomology functors
$H^0\bigl(C^{\bullet}(\text{--})\bigr)$
and
$H^1\bigl(C^{\bullet}(\text{--})\bigr)$
form a $\delta$-functor on $\K{A}$.

\begin{lemma}\label{lem: C computes Hom}
There is a natural isomorphism
$$\Hom_{\K{A}}(\gM,\text{--}\, ) \cong
H^0\bigl(C^{\bullet}(\text{--}\,)\bigr).$$
\end{lemma}
\begin{proof}
This is immediate.
\end{proof}

It follows from this lemma and 
a standard dimension shifting argument (or, equivalently, the theory
of $\delta$-functors) that there is an embedding of functors
\numequation
\label{eqn:ext embedding}
\Ext^1_{\K{A}}(\gM, \text{--} \, ) \hookrightarrow
H^1\bigl(C^{\bullet}(\text{--})\bigr).
\end{equation}

\begin{lemma}\label{lem: C computes Ext^1}
The embedding of functors~\eqref{eqn:ext embedding}
is an isomorphism.
\end{lemma}
\begin{proof}
We first describe the embedding~\eqref{eqn:ext embedding} explicitly.
Suppose that 
$$0 \to \gN \to \gE \to \gM \to 0$$
is an extension in $\K{A}$.  Since $\gM$ is projective over $\gS_A$,
and since $\Gal(K'/K)$ is of prime-to-$p$ order, 
we split this
short exact sequence
over the twisted group ring $\gS_A[\Gal(K'/K)],$ 
say via some element $\sigma \in \Hom_{\gS_A[\Gal(K'/K)]}(\gM,\gE).$
This splitting is well-defined up to the addition of an
element $\alpha \in \Hom_{\gS_A[\Gal(K'/K)]}(\gM,\gN).$

This splitting is a homomorphism in $\K{A}$ if and only
if the element
$$\Phi_{\gE}\circ \varphi^*\sigma - \sigma \circ \Phi_{\gM}
\in \Hom_{\gS_A[\Gal(K'/K)]}(\varphi^*\gM,\gN)$$ vanishes.
If we replace $\sigma$ by $\sigma +\alpha,$ then this element
is replaced by
$$(\Phi_{\gE}\circ \varphi^*\sigma - \sigma \circ \Phi_{\gM})
+  (\Phi_{\gN} \circ \varphi^* \alpha - \alpha \circ \Phi_{\gM}).$$
Thus the coset of 
$\Phi_{\gE}\circ \varphi^*\sigma - \sigma \circ \Phi_{\gM}$
in $H^1\bigl(C^{\bullet}(\gN)\bigr)$ is well-defined, independent
of the choice of $\sigma,$
and this coset is the image of the class of the extension
$\gE$ under the embedding
\numequation
\label{eqn:explicit embedding}
\Ext^1_{\K{A}}(\gM,\gN) \hookrightarrow 
H^1\bigl(C^{\bullet}(\gN)\bigr)
\end{equation}
(up to a possible overall sign, which we ignore, since it doesn't
affect the claim of the lemma).

Now, given any element  
$\nu \in \Hom_{\gS_A[\Gal(K'/K)]}(\varphi^*\gM,\gN)$,
we may give the $\gS_A[\Gal(K'/K)]$-module $\gE := \gN \oplus \gM$ the
structure of a $\gS_A[F,\Gal(K'/K)]$-module as follows: we need to
define a $\varphi$-linear morphism $\gE\to\gE$, or equivalently a linear
morphism $\Phi_{\gE}:\varphi^*\gE\to\gE$. We do this
by setting $$\Phi_{\gE} := \begin{pmatrix} \Phi_{\gN} & \nu \\ 0 & \Phi_{\gM} 
\end{pmatrix}.$$
Then $\gE$ is an extension of $\gM$ by $\gN$,
and if we let $\sigma$ denote the obvious embedding of $\gM$ into $\gE$,
then one computes that 
$$\nu = \Phi_{\gE}\circ \varphi^*\sigma - \sigma \circ \Phi_{\gM} .$$
This shows that~\eqref{eqn:explicit embedding} is an isomorphism, as claimed.
\end{proof}

Another dimension shifting argument, taking into account the preceding
lemma, shows that $\Ext^2_{\K{A}}(\gM,\text{--} \,)$ embeds into
$H^2\bigl( C^{\bullet}(\text{--}) \bigr).$ Since the target of this
embedding vanishes,
we find that the same is true of the source.  This yields the
following corollary.

\begin{cor}
\label{cor:ext2 vanishes}
If $\gM$ is a Breuil--Kisin module with $A$-coefficients and descent data,
then $\Ext^2_{\K{A}}(\gM, \text{--} \, ) = 0.$
\end{cor}
We summarise the above discussion in the following corollary.
\begin{cor}
  \label{cor:complex computes Hom and Ext}If $\gM$ is a 
  Breuil--Kisin module with $A$-coefficients and descent data, and $\gN$ is an
  object of~$\K{A}$, then we have a natural short exact
  sequence \[0\to\Hom_{\K{A}}(\gM,\gN)\to C^0(\gN)\to
    C^1(\gN)\to\Ext^1_{\K{A}}(\gM,\gN)\to 0.\]
\end{cor}

The following lemma records the behaviour of these complexes with
respect to base change.

\begin{lemma}\label{lem:base-change-complexes}
  Suppose that $\gM$, $\gN$ are Breuil--Kisin modules with descent data and
  $A$-coefficients, that $B$ is an $A$-algebra, and that $Q$ is a
  $B$-module. Then the
  complexes $C^{\bullet}_{\gM}(\gN \cotimes_A Q)$ and
  $C^{\bullet}_{\gM \cotimes_A B}(\gN \cotimes_A Q)$ coincide, the
  former complex
  formed with respect to $\K{A}$ and the latter with respect to $\K{B}$.
\end{lemma}

\begin{proof}
Indeed, there is a natural isomorphism  $$\Hom_{\gS_A[\Gal(K'/K)]}(\gM, \gN\cotimes_A Q) \cong
\Hom_{\gS_B[\Gal(K'/K)]}(\gM\cotimes_A B, \gN\cotimes_A Q),$$
and similarly with $\varphi^*\gM$ in place of $\gM$.
\end{proof}

The following slightly technical lemma is
crucial for establishing 
finiteness properties, and also base-change properties,
of Exts of Breuil--Kisin modules.

\begin{lem}
  \label{lem:truncation argument used to prove f.g. of Ext Q version}Let $A$ be
  a $\cO/\varpi^a$-algebra for some $a\ge 1$, suppose that
  $\gM$ 
is a Breuil--Kisin module with descent data and $A$-coefficients,
of height at most~$h$,
and suppose that $\gN$ is a $u$-adically complete, $u$-torsion
free object of $\K{A}$. 

Let $C^\bullet$ be the complex defined
  in \emph{Definition~\ref{def:explicit Ext complex}}, and write~$\delta$ for
  its differential. Suppose that $Q$ is an $A$-module with
the property that $C^i\otimes_A Q$ is $v$-torsion free for $i=0,1$
and $v$-adically  separated for $i=0$.

 Then:
  \begin{enumerate}
  \item For any integer $M\ge (eah+1)/(p-1)$, $\ker (\delta\otimes \id_Q)\cap
    v^MC^0\otimes_AQ=0$.
  \item For any integer $N \ge (peah+1)/(p-1)$, $\delta\otimes\id_Q$ induces an isomorphism
\[(\Phi_{\gM}^*)^{-1}(v^N C^1\otimes_AQ) \isoto v^N (C^1\otimes_AQ).\]
  \end{enumerate}
Consequently, for $N$ as in \emph{(2)} the natural morphism of complexes of $A$-modules
$$[ C^0\otimes_AQ \buildrel \delta\otimes\id_Q \over \longrightarrow C^1\otimes_AQ] \to
[C^0\otimes_AQ/\bigl((\Phi_{\gM}^*)^{-1}(v^N C^1\otimes_AQ) \bigr) \buildrel \delta\otimes\id_Q
\over \longrightarrow C^1\otimes_AQ/v^N C^1\otimes_AQ ]$$
is a quasi-isomorphism.
\end{lem}

Since we are assuming that the $C^i\otimes_AQ$ are $v$-torsion free, 
the expression $v^r C^i(\gN) \otimes_A Q$ may be interpreted as
denoting either $v^r \bigl( C^i(\gN)\otimes_A Q\bigr)$ or
$\bigl(v^r C^i(\gN) \bigr)\otimes_A Q$, the two being naturally
isomorphic.

\begin{rem}\label{rem:truncation-remark} Before giving the proof of Lemma~\ref{lem:truncation
    argument used to prove f.g. of Ext Q version}, we observe that the
  hypotheses on the $C^i \otimes_A Q$ are satisfied if either $Q=A$, or
  else $\gN$ is a projective $\gS_A$-module and $Q$ is a finitely
  generated $B$-module for some finitely
generated $A$-algebra $B$.  (Indeed $C^1 \otimes_A Q$
  is $v$-adically separated as well in these cases.)

(1)  Since $\gM$ is projective of finite rank over $A[[u]]$,
and since $\gN$ is $u$-adically complete and $u$-torsion free,
each  $C^i$ is $v$-adically 
separated
and $v$-torsion free. In particular the hypothesis on~$Q$ is always satisfied
by~$Q=A$. (In fact since $\gN$ is $u$-adically complete it also follows that 
the $C^i$ are $v$-adically complete. Here we use that $\Gal(K'/K)$ has
order prime to $p$ to see that $C^0$ is an $\gS_A^0$-module direct
summand of $\Hom_{\gS_A}(\gM,\gN)$, and similarly for $C^1$.)

(2) Suppose $\gN$ is a projective $\gS_A$-module. Then the $C^i$ are
projective $\gS_A^0$-modules, again using that $\Gal(K'/K)$ has
order prime to $p$. Since each $C^i(\gN)/v C^i(\gN)$ is $A$-flat, it
follows that $C^i(\gN) \otimes_A Q$ is $v$-torsion free. If furthermore $B$ is a finitely
generated $A$-algebra, and $Q$ is a finitely generated $B$-module,
then the $C^i(\gN)\otimes_A Q$ are
$v$-adically separated (being finitely generated modules over
the ring $A[[v]]\otimes_A B$, which is a finitely generated 
algebra over the Noetherian ring $A[[v]]$, and hence is itself
Noetherian).
\end{rem} 
\begin{proof}[Proof of Lemma~{\ref{lem:truncation argument used to
    prove f.g. of Ext Q version}}] Since $p^a=0$ in $A$, there exists $H(u)\in\gS_A$ with
  $u^{e'ah}=E(u)^hH(u)$ in $\gS_A$. Thus the image of $\Phi_{\gM}$
contains $u^{e'ah} \gM=v^{eah}\gM$, and there exists a map $\Upsilon : \gM \to
\varphi^* \gM$ such that $\Phi_{\gM} \circ \Upsilon$ is multiplication by $v^{eah}$.

We begin with~(1). Suppose that 
  $f\in\ker (\delta\otimes\id_Q)\cap v^MC^0\otimes_AQ$. 
Since $C^0\otimes_AQ$ is $v$-adically separated, 
it is enough, applying induction on $M$, 
to show that $f\in v^{M+1}C^0\otimes_AQ$. Since
  $f\in\ker(\delta\otimes\id_Q)$, we have
  $f\circ\Phi_{\gM}=\Phi_{\gN}\circ\varphi^*f$. Since $f\in v^MC^0\otimes_AQ$,
  we have $f \circ \Phi_{\gM} = \Phi_{\gN} \circ \varphi^*f \in
  v^{pM} C^1 \otimes_A Q$. Precomposing with $\Upsilon$ gives 
 $v^{eah} f \in v^{pM} C^0 \otimes_A Q$. 
 Since~$C^0 \otimes_A Q$ is $v$-torsion free, it follows that $f\in
  v^{pM-eah}C^0\otimes_AQ\subseteq u^{M+1}C^0\otimes_AQ$, as required.
  
We now move on to~(2).  Set $M = N - eah$. 
By precomposing with $\Upsilon$ we see that $\alpha \circ \Phi_{\gM} \in v^N C^1\otimes_A Q$ implies $\alpha \in v^M C^0\otimes_A Q$; from this, together with the inequality
$pM \ge N$, it is
straightforward to check that
\[(\Phi_{\gM}^*)^{-1}(v^N C^1\otimes_AQ) = (\delta \otimes \id_Q)^{-1}(v^N
C^1\otimes_AQ)\cap  v^M C^0\otimes_AQ.\]
Note that $M$ satisfies the condition in~(1). To complete the proof we
will show that for any $M$ as in (1) and any $N \ge M + eah$ the map
$\delta$ induces an isomorphism 
\[ (\delta \otimes \id_Q)^{-1}(v^N C^1\otimes_AQ)\cap  v^M C^0\otimes_AQ \isoto v^N C^1\otimes_AQ.\]
 By~(1), $\delta\otimes\id_Q$ induces an injection $(\delta\otimes\id_Q)^{-1}(v^N C^1\otimes_AQ)\cap
    v^M C^0\otimes_AQ\hookrightarrow v^N C^1\otimes_AQ$, so it is enough to show that
    $(\delta\otimes\id_Q)(v^MC^0\otimes_AQ)\supseteq v^N
    C^1\otimes_AQ$. 
Equivalently, we need to show that 
$$
v^N C^1 \otimes_A Q
\to
(C^1\otimes_A Q) / 
(\delta\otimes\id_Q)\bigl(v^M C^0\otimes_A Q) 
$$
is identically zero. Since the formation of cokernels is compatible with tensor products,
we see that this morphism is obtained by tensoring the corresponding morphism
$$
v^N C^1 
\to
C^1/
\delta\bigl(v^M C^0\bigr)
$$
with $Q$ over $A$, so we are reduced to the case $Q=A$. (Recall from
Remark~\ref{rem:truncation-remark}(1) that the
hypotheses of the Lemma are satisfied in this case, and that $C^1$ is 
$v$-adically separated.)

We claim that for
    any $g\in v^NC^1$, we can find an $f\in v^{N-eah}C^0$ such that
    $\delta(f)-g\in v^{p(N-eah)}C^1$. Admitting the claim, given any
    $g\in v^N C^1$, we may find $h\in v^MC^0$ with $\delta(h)=g$ by
    successive approximation in the following way: 
Set $h_0=f$ for~$f$
    as in the claim; then $h_0\in v^{N-eah}C^0\subseteq v^MC^0$, and
    $\delta(h_0)-g\in v^{p(N-eah)}C^1\subseteq v^{N+1}C^1$. Applying
    the claim again with $N$ replaced by $N+1$, and $g$ replaced by
    $g-\delta(h_0)$, we find $f\in v^{N+1-eah}C^0\subseteq v^{M+1}C^0$
    with $\delta(f)-g+\delta(h_0)\in v^{p(N+1-eah)}C^1\subseteq
    v^{N+1}C^1$. Setting $h_1=h_0+f$, and proceeding inductively, we
    obtain a Cauchy sequence converging (in the $v$-adically
complete $A[[v]]$-module $C^0$) to the required element~$h$.

It remains to prove the claim. Since
$\delta(f)=\Phi_{\gN}\circ\varphi^*f-f\circ\Phi_{\gM} $, and since if
$f\in v^{N-eah}C^0$ then $\Phi_{\gN}\circ\varphi^*f\in v^{p(N-eah)}C^1$,
it is enough to show that we can find an $f\in v^{N-eah}C^0$ with
$f\circ\Phi_{\gM}=-g$. Since $\Phi_{\gM}$ is injective, the map
$\Upsilon \circ \Phi_\gM$ is also multiplication by $v^{eah}$, and so
it suffices to take $f$ with $v^{eah} f = -g \circ \Upsilon \in v^N C^0$. 
\end{proof}

\begin{cor}\label{cor: base change completion for complex in free case}
  Let $A$ be a Noetherian 
  $\cO/\varpi^a$-algebra,
  and let $\gM$, $\gN$ be Breuil--Kisin modules with descent data
  and $A$-coefficients. If $B$ is a finitely generated 
$A$-algebra, and $Q$ is a finitely generated $B$-module,
then the natural morphism of complexes of $B$-modules
$$[ C^0(\gN)\otimes_A Q \buildrel {\delta\otimes \id_Q} \over \longrightarrow
C^1(\gN)\otimes_A Q] \to
[C^0(\gN\cotimes_A Q) \buildrel \delta \over \longrightarrow
C^1(\gN\cotimes_A Q)]$$
is a quasi-isomorphism.
\end{cor}
\begin{proof} 
By Remarks~\ref{rem:truncation-remark} and~\ref{rem:completed tensor}(2)  we can apply Lemma~\ref{lem:truncation argument used to
  prove f.g. of Ext Q version} to both
$C^i(\gN\cotimes_AQ)$ and  $C^i(\gN)\otimes_AQ$, and we see that it is enough to show that the
natural morphism of complexes 
\[\begin{adjustbox}{max width=\textwidth}
\begin{tikzcd}
{[\bigl(C^0(\gN)\otimes_A Q \bigr)/
(\Phi_{\gM}^*\otimes \id_Q)^{-1}\bigl(v^N C^1(\gN)\otimes_A Q \bigr) 
\buildrel \delta
\over \longrightarrow \bigl(C^1(\gN)\otimes_A Q\bigr)/\bigl(v^N C^1(\gN)
\otimes_A Q\bigr) ]}
\arrow{d}{} \\
{[C^0(\gN\cotimes_A Q)
/\bigl(\Phi_{\gM}^*)^{-1}(v^N C^1(\gN\cotimes_A Q)\bigr) 
\buildrel \delta
\over \rightarrow C^1(\gN\cotimes_A Q)/v^N C^1(\gN\cotimes_A Q) ]}
\end{tikzcd}
\end{adjustbox}\]
is a quasi-isomorphism. In fact, it is even an isomorphism. 
\end{proof}

\begin{prop}
\label{prop:exts are f.g. over A} Let $A$ be a 
$\cO/\varpi^a$-algebra for some $a\ge 1$, and let $\gM$, $\gN$ be
Breuil--Kisin modules with descent data and $A$-coefficients. 
Then  $\Ext^1_{\K{A}}(\gM,\gN)$ and $\Ext^1_{\K{A}}(\gM,\gN/u^i\gN)$
for $i \ge 1$ are finitely presented $A$-modules.

If
furthermore $A$ is Noetherian, then $\Hom_{\K{A}}(\gM,\gN)$ and
$\Hom_{\K{A}}(\gM,\gN/u^i\gN)$ for $i\ge 1$ are also
finitely presented {\em (}equivalently, finitely generated{\em )} $A$-modules.
\end{prop}
\begin{proof} 
The statements for $\gN/u^i\gN$
  follow easily from those for $\gN$, by considering the short exact sequence $0 \to u^i\gN \to \gN \to
\gN/u^i\gN \to 0$ in $\K{A}$ and applying Corollary~\ref{cor:ext2
  vanishes}.
	By Corollary~\ref{cor:complex computes Hom and Ext}, it is enough to consider the cohomology of the
  complex~$C^\bullet$. By Lemma~\ref{lem:truncation argument used to
    prove f.g. of Ext Q version} with $Q=A$, 
the cohomology of~$C^\bullet$ agrees with the
  cohomology of the induced complex \[C^0/\bigl((\Phi_{\gM}^*)^{-1}(v^N C^1) )\to C^1/ v^N C^1,\] 
for an appropriately chosen value of $N$. It follows that for an
appropriately chosen value of~$N$, $\Ext^1_{\K{A}}(\gM,\gN)$ can be
computed as the cokernel of the induced morphism $C^0/v^N C^0 \to C^1/ v^N C^1$.

Under our hypothesis on~$\gN$, $C^0/v^N C^0$ and $C^1/v^NC^1$ are finitely
generated projective $A$-modules, and thus finitely presented. It follows  that
$\Ext^1_{\K{A}}(\gM,\gN)$ is finitely presented. 

In the case that $A$ is furthermore assumed to be Noetherian, it is
enough to note that since $v^NC^0\subseteq (\Phi_{\gM}^*)^{-1}(v^N C^1)$,
the quotient $C^0/\bigl((\Phi_{\gM}^*)^{-1}(v^N C^1) \bigr)$ is a finitely generated $A$-module.
\end{proof}

\begin{prop}
  \label{prop:descent for Homs of free Kisin modules}Let $A$ be a 
$\cO/\varpi^a$-algebra for some $a\ge 1$,
and let $\gM$ and~$\gN$ be Breuil--Kisin modules with descent
data and $A$-coefficients. Let $B$ be an $A$-algebra, and let
      $f_B:\gM\cotimes_AB\to\gN\cotimes_AB$ be  a morphism of Breuil--Kisin
      modules with $B$-coefficients.

      Then there is a finite type $A$-subalgebra $B'$ of~$B$ and a morphism
      of Breuil--Kisin modules $f_{B'}:\gM\cotimes_A B'\to\gN\cotimes_A B'$
      such that $f_B$ is the base change of~$f_{B'}$.
\end{prop}
\begin{proof}

By Lemmas~\ref{lem: C computes Hom}
and~\ref{lem:base-change-complexes} (the latter applied with $Q=B$) we can and do think
  of $f_B$ as being an element of the kernel of
  $\delta:C^0(\gN\cotimes_A B)\to C^1(\gN\cotimes_A
  B)$, the complex $C^\bullet$ here and throughout this proof denoting
  $C^{\bullet}_{\gM}$ as usual.

Fix $N$ as in
  Lemma~\ref{lem:truncation argument used to prove f.g. of Ext Q version}, and
  write~$\fbar_B$ for the corresponding element of
  $C^0(\gN\cotimes_A B)/v^N=(C^0(\gN)/v^N)\otimes_A B$ (this equality
  following easily from the assumption that $\gM$ and $\gN$ are
  projective $\gS_A$-modules of finite rank). Since $C^0(\gN)/v^N$
  is a projective $A$-module of finite
  rank, 
  it follows 
  that for some finite type
  $A$-subalgebra $B'$ of~$B$, there is an element $\fbar_{B'}\in
  (C^0(\gN)/v^N)\otimes_A B'=C^0(\gN\cotimes_A B')/v^N$ such that
  $\fbar_{B'}\otimes_{B'}B=\fbar_B$. Denote also by $\fbar_{B'}  $ the induced element of \[C^0(\gN\cotimes
_A B')/\bigl(\Phi_{\gM}^*)^{-1}(v^N C^1(\gN\cotimes
_A B')).\] 

By Lemma~\ref{lem:truncation argument used to prove
f.g. of Ext Q version} (and Lemma~\ref{lem: C computes Hom}) we have a
commutative diagram with exact rows \[\xymatrix{0 \ar[r] & H^0(C^{\bullet}(\gN\cotimes
_A B')) \ar[d] \ar[r] & C^0(\gN\cotimes
_A B')/\bigl((\Phi_{\gM}^*)^{-1}(v^N C^1(\gN\cotimes
_A B')) 
\bigr) \ar[r]^-{\delta}\ar[d] & C^1(\gN\cotimes
_A B')/v^N\ar[d]\\ 0 \ar[r] & H^0(C^{\bullet}(\gN\cotimes
_A B)) \ar[r] & C^0(\gN\cotimes
_A B)/\bigl((\Phi_{\gM}^*)^{-1}(v^N C^1 (\gN\cotimes
_A B)) 
\bigr) \ar[r]^-{\delta} & C^1(\gN\cotimes
_A B)/v^N }  \]
in which the vertical arrows 
are induced by $\cotimes_{B'}B$.  By a diagram chase we only need to
show that $\delta(\fbar_{B'})=0$. Since $\delta(f_B)=0$, it is
enough to show that the right hand vertical arrow is an
injection. This arrow can be rewritten as the tensor product of the
injection of $A$-algebras $B'\into B$ with the flat (even projective
of finite rank) $A$-module $C^1(\gN)/v^N$, so the result follows.
\end{proof}
We have the following key base-change result for $\Ext^1$'s of 
Breuil--Kisin modules with descent data.

\begin{prop}
\label{prop:base-change for exts}
Suppose that $\gM$ and $\gN$ are Breuil--Kisin modules with
descent data and coefficients in a $\cO/\varpi^a$-algebra $A$.
Then for any $A$-algebra $B$, and for any $B$-module $Q$,
there are natural isomorphisms
$\Ext^1_{\K{A}}(\gM,\gN)\otimes_A Q \iso
\Ext^1_{\K{B}}(\gM\cotimes_A B, \gN\cotimes_A B) \otimes_B Q
\iso 
\Ext^1_{\K{B}}(\gM\cotimes_A B,\gN\cotimes_A Q).$
\end{prop}
\begin{proof}
We first prove the lemma in the case of an $A$-module $Q$.
It follows from Lemmas \ref{lem: C computes Ext^1}
and~\ref{lem:truncation argument used to prove f.g. of Ext Q version}
that we may compute
$\Ext^1_{\K{A}}(\gM,\gN)$
as the cokernel of the morphism
$$C^0(\gN)/v^N C^0(\gN)
\buildrel \delta \over \longrightarrow C^1(\gN)/v^N C^1(\gN),$$for
some sufficiently large value of $N$ (not depending on $\gN$),
and hence that we may compute
$\Ext^1_{\K{A}}(\gM,\gN)\otimes_A Q$
as the cokernel of the morphism
$$\bigl(C^0(\gN)/v^N C^0(\gN)\bigr) \otimes_A Q
\buildrel \delta \over \longrightarrow 
\bigl(C^1(\gN)/v^N C^1(\gN)\bigr) \otimes_A Q.$$
We may similarly compute
$\Ext^1_{\K{A}}(\gM,\gN\cotimes_A Q)$
as the cokernel of the morphism
$$C^0(\gN\cotimes_A Q)/v^N C^0(\gN\cotimes_A Q)
\buildrel \delta \over \longrightarrow
C^1(\gN\cotimes_A Q)/v^N C^1(\gN\cotimes_A Q).$$
(Remark~\ref{rem:completed tensor}~(2) shows
that $\gN\cotimes_A Q$ satisfies the necessary hypotheses
for Lemma~\ref{lem:truncation argument used to prove f.g. of Ext Q version}
to apply.)
Once we note that the natural morphism
$$\bigl( C^i(\gN)/v^N C^i(\gN)\bigr)\otimes_A Q \to C^i(\gN\cotimes_A Q)/v^N
C^i(\gN\cotimes_A Q)$$
is an isomorphism for $i = 0$ and $1$ (because $\gM$ is a finitely
generated projective $\gS_A$-module), 
we obtain the desired isomorphism
$$\Ext^1_{\K{A}}(\gM,\gN)\otimes_A Q \iso \Ext^1_{\K{A}}(\gM,\gN\cotimes_A Q).$$

If $B$ is an $A$-algebra, and $Q$ is a $B$-module,
then by Lemma~\ref{lem:base-change-complexes} 
there is a natural isomorphism
$$\Ext^1_{\K{A}}(\gM, \gN\cotimes_A Q) \iso
\Ext^1_{\K{B}}(\gM\cotimes_A B, \gN\cotimes_A Q);$$
combined with the preceding base-change result, this yields
one of our claimed isomorphisms, namely
\[\Ext^1_{\K{A}}(\gM,\gN)\otimes_A Q \iso 
\Ext^1_{\K{B}}(\gM\cotimes_A B, \gN\cotimes_A Q).\]
Taking $Q$ to be $B$ itself, we then obtain
an isomorphism
\[\Ext^1_{\K{A}}(\gM,\gN)\otimes_A B \iso 
\Ext^1_{\K{B}}(\gM\cotimes_A B, \gN\cotimes_A B).\]
This allows us to identify
$\Ext^1_{\K{A}}(\gM,\gN)\otimes_A Q$,
which is naturally isomorphic to 
$\bigr(\Ext^1_{\K{A}}(\gM,\gN)\otimes_A B\bigr) \otimes_B Q$,
with 
$\Ext^1_{\K{B}}(\gM\cotimes_A B, \gN\cotimes_A B)\otimes_B Q$,
yielding the second claimed isomorphism.
\end{proof}

In contrast to the situation for extensions
(\emph{cf}.\ Proposition~\ref{prop:base-change for exts}), the formation of
homomorphisms between Breuil--Kisin modules is in general
not compatible with arbitrary base-change, as the following example shows.

\begin{example}\label{example:rank one unramified}
Take $A = (\Z/p\Z)[x^{\pm 1}, y^{\pm 1}]$, and let $\gM_x$ be the
free Breuil--Kisin module of rank one and $A$-coefficients with $\varphi(e)
= xe$ for some generator $e$ of $\gM_x$. Similarly define $\gM_y$ with
$\varphi(e') = ye'$ for some generator $e'$ of $\gM_y$. Then
$\Hom_{\K{A}}(\gM_x,\gM_y)=0$.  On the other hand, if $B=A/(x-y)$ then
$\gM_x \cotimes_A B$ and $\gM_y \cotimes_A B$ are isomorphic, so that 
$\Hom_{\K{B}}(\gM_x \cotimes B, \gM_y \cotimes B) \not\cong
\Hom_{\K{A}}(\gM_x,\gM_y) \otimes_A B$. 
\end{example}

However, it is possible
to establish such a compatibility in some settings.
Corollary~\ref{cor:vanishing of homs non Noetherian}, which
gives a criterion for the vanishing of $\Hom_{\K{B}} (\gM\cotimes_A
B, \gN\cotimes_A B)$ for any $A$-algebra $B$,  is a first example of
a result in this direction. Lemma~\ref{lem: flat base change for
  Homs} deals with flat base change, and Lemma~\ref{lem: vanishing of Kisin module homs implies vanishing on dense open}, which will
be important in Section~\ref{subsec:universal
families}, proves that formation of
homomorphisms is  compatible with base-change over a dense open
subscheme of $\Spec A$. 

\begin{prop}
\label{prop:vanishing of homs}
Suppose that $A$ is a Noetherian $\cO/\varpi^a$-algebra,
and that $\gM$ and $\gN$ are objects of $\K{A}$ that are
finitely generated over $\gS_A$ {\em (}or, equivalently,
over $A[[u]]${\em )}. 
Consider the following conditions:
\begin{enumerate}
\item
$\Hom_{\K{B}} (\gM\cotimes_A B, \gN\cotimes_A B) = 0$
for any finite type $A$-algebra $B$.
\item
$\Hom_{\K{\kappa(\m)}}\bigl(\gM\otimes_A \kappa(\mathfrak m),
\gN\otimes_A \kappa(\mathfrak m) \bigr) = 0$ 
for each maximal ideal $\mathfrak m$ of $A$.
\item
$\Hom_{\K{A}}(\gM, \gN\otimes_A Q) = 0$ 
for any 
finitely generated $A$-module $Q$.
\end{enumerate}
Then we have (1)$\implies$(2)$\iff$(3).  If $A$ is furthermore
Jacobson, then all three conditions are equivalent.
\end{prop}
\begin{proof}
If $\mathfrak m$ is a maximal ideal of $A$, then $\kappa(\mathfrak m)$
is certainly a finite type $A$-algebra, and so evidently~(1) implies~(2). 
It is even a finitely generated $A$-module, and so also~(2) follows
from~(3). 

We next
prove that~(2) implies~(3).
To this end, recall that if $A$ is any ring, and $M$ is any $A$-module,
then $M$ injects into the product of its localizations at all maximal ideals. 
If $A$ is Noetherian, and $M$ is finitely generated, then, by combining
this fact with the Artin--Rees
Lemma, we see that $M$ embeds into the product of its completions at all 
maximal ideals.   Another way to express this is that, if $I$ runs
over all cofinite length ideals in $A$ (i.e.\ all ideals for which $A/I$
is finite length), then $M$ embeds into the projective limit 
of the quotients $M/IM$  (the point being that
this projective limit is the same as the product
over all $\mathfrak m$-adic completions).
We are going to apply this observation with $A$ replaced by $\gS_A$,
and with $M$ taken to be $\gN\otimes_A Q$ for some finitely generated
$A$-module $Q$.

In $A[[u]]$, one sees that $u$ lies in the Jacobson radical (because 
$1  + fu$ is invertible in $A[[u]]$ for every $f \in A[[u]]$), and thus
in every maximal ideal, and so the maximal ideals of $A[[u]]$ are of
the form $(\mathfrak m, u)$, where $\mathfrak m$ runs over the maximal 
ideals of~$A$.
Thus the ideals of the form $(I,u^n)$, where $I$ is a cofinite length
ideal in $A$, are
cofinal in all cofinite length ideals in $A[[u]]$.
Since $\gS_A$ is finite over $A[[u]]$, we see that the ideals
$(I,u^n)$ in $\gS_A$ are also
cofinal in all cofinite length ideals in $A[[u]]$.
Since $A[[u]]$, and hence $\gS_A$, is furthermore Noetherian when $A$ is,
we see that if $Q$ is a
finitely generated $A$-module, and $\gN$ is a finitely generated
$\gS_A$-module,
then $\gN\otimes_A (Q/IQ)$ is $u$-adically complete,
for any cofinite length ideal $I$ in $A$, and
hence equal to the limit over $n$ of $\gN \otimes_A Q/(I,u^n)$.  
Putting this together with the observation of the preceding paragraph,
we see that the natural morphism
$$\gN\otimes_A Q \to \varprojlim_I \gN\otimes_A (Q/IQ)$$ 
(where $I$ runs over all cofinite length ideals of $A$)
is an embedding.
The induced morphism
$$ \Hom_{\K{A}}(\gM,\gN\otimes_A Q)
\to
\varprojlim_I \Hom_{\K{A}}(\gM,\gN\otimes_A (Q/IQ))$$
is then evidently also an embedding.

Thus, to conclude that 
$ \Hom_{\K{A}}(\gM,\gN\otimes_A Q)$
vanishes,
it suffices to show that
$\Hom_{\K{A}}(\gM,\gN\otimes_A (Q/IQ))$  vanishes for each
cofinite length ideal $I$ in $A$.  An easy induction on the
length of $A/I$ reduces this to showing that
$\Hom_{\K{A}}\bigl(\gM,\gN\otimes_A \kappa(\mathfrak m)\bigr),$
or, equivalently, $\Hom_{\K{\kappa(\mathfrak{m})}}\bigl(\gM\otimes_A \kappa(\mathfrak m),
\gN\otimes_A \kappa(\mathfrak m)\bigr),$
vanishes for each maximal ideal~$\mathfrak m$.
Since this is 
the hypothesis of~(2), we see that indeed~(2) implies~(3).

It remains to show that~(3) implies~(1) when $A$ is Jacobson. 
Applying the result
``(2) implies~(3)'' (with $A$ replaced by~$B$, and taking $Q$ in~(3) to be $B$ itself as a $B$-module) to $\gM\cotimes_A B$ and $\gN\cotimes_A B$,
we see that it suffices to prove the vanishing of
$$\Hom_{\K{B}}\bigl( (\gM\cotimes_A B)\otimes_B \kappa(\mathfrak n),
(\gN\cotimes_A B)\otimes_B \kappa(\mathfrak n) \bigr)
= \Hom_{\K{A}}\bigl( \gM, \gN\cotimes_A \kappa(\mathfrak n) \bigr)
$$
for each maximal ideal $\mathfrak n$ of $B$.
Since $A$ is Jacobson, the field $\kappa(\mathfrak n)$ is in fact a
finitely generated
$A$-module, hence $\gN\cotimes\kappa(\mathfrak n) = \gN\otimes_A
\kappa(\mathfrak n)$, and so the desired vanishing is a special case of~(3).
\end{proof}

\begin{cor}
\label{cor:vanishing of homs non Noetherian}
If $A$ is a Noetherian and Jacobson $\cO/\varpi^a$-algebra,
and if $\gM$ and $\gN$ are Breuil--Kisin modules with descent
data and $A$-coefficients, 
then the
following three conditions are equivalent:
\begin{enumerate}
\item
$\Hom_{\K{B}} (\gM\cotimes_A B, \gN\cotimes_A B) = 0$
for any $A$-algebra $B$.
\item
$\Hom_{\K{\kappa(\mathfrak{m})}}\bigl(\gM\otimes_A \kappa(\mathfrak m),
\gN\otimes_A \kappa(\mathfrak m) \bigr) = 0$ 
for each maximal ideal $\mathfrak m$ of $A$.
\item
$\Hom_{\K{A}}(\gM, \gN\otimes_A Q) = 0$ 
for any 
finitely generated $A$-module $Q$.
\end{enumerate}
\end{cor}
\begin{proof}By Proposition~\ref{prop:vanishing of homs}, we need only
  prove that if $\Hom_{\K{B}} (\gM\cotimes_A B, \gN\cotimes_A B)$
  vanishes 
for all finitely generated $A$-algebras~$B$, then it vanishes for all
$A$-algebras~$B$. This is immediate from Proposition~\ref{prop:descent for Homs of free Kisin modules}.
\end{proof}

\begin{cor}
\label{cor:freeness for exts}
Suppose that $\gM$ and $\gN$ are Breuil--Kisin modules with
descent data and coefficients in a Noetherian $\cO/\varpi^a$-algebra $A$,
and that furthermore
$\Hom_{\K{A}}\bigl(\gM\otimes_A \kappa(\mathfrak m),
\gN\otimes_A \kappa(\mathfrak m) \bigr)$ vanishes
for each maximal ideal $\mathfrak m$ of $A$.
Then the $A$-module
$\Ext^1_{\K{A}}(\gM,\gN)$
is projective
of finite rank. 
\end{cor}
\begin{proof}
By Proposition~\ref{prop:exts are f.g. over A},
in order to prove that $\Ext^1_{\K{A}}(\gM,\gN)$
is projective of finite rank over $A$,
it suffices to prove that it is flat over $A$.
For this, it suffices to show that
$Q \mapsto \Ext^1_{\K{A}}(\gM,\gN)\otimes_A Q$
is exact when applied to finitely generated $A$-modules $Q$.
Proposition~\ref{prop:base-change for exts} (together with 
Remark~\ref{rem:completed tensor}~(1)) allows us to identify
this functor with the functor
$Q \mapsto \Ext^1_{\K{A}}(\gM,\gN\otimes_A Q).$
Note that the functor $Q\mapsto\gN\otimes_A Q$ is an exact functor of $Q$,
since $\gS_A$ is a flat $A$-module (as $A$ is Noetherian; see Remark~\ref{rem:projectivity for Kisin modules}(3)).
Thus, taking into account
Corollary~\ref{cor:ext2 vanishes},
we see that it suffices to show that
$\Hom_{\K{A}}(\gM,\gN\otimes_A Q) = 0$
for each finitely generated $A$-module~$Q$,
under the hypothesis that
$\Hom_{\K{A}}\bigl(\gM\otimes_A \kappa(\mathfrak m),
\gN\otimes_A \kappa(\mathfrak m) \bigr) = 0$ 
for each maximal ideal $\mathfrak m$ of~$A$.
This is the implication (2) $\implies$ (3) of  Proposition~\ref{prop:vanishing of homs}.
\end{proof}

\begin{lemma}\label{lem: flat base change for Homs}
Suppose that $\gM$ is a Breuil--Kisin modules with
descent data and coefficients in a Noetherian $\cO/\varpi^a$-algebra
$A$. Suppose that $\gN$ is either a Breuil--Kisin module with
$A$-coefficients, or that $\gN=\gN'/u^N\gN'$, where $\gN'$  a Breuil--Kisin module with
$A$-coefficients and $N\ge 1$. 
Then,
if $B$ is a finitely generated flat 
$A$-algebra, we have a natural isomorphism
\[\Hom_{\K{B}}(\gM\cotimes_{A} B, 
\gN\cotimes_{A} B) \iso \Hom_{\K{A}}(\gM,\gN)\otimes_{A}B.  \] 
\end{lemma}
\begin{proof}
   By Corollary~\ref{cor:complex computes Hom and Ext} and the
   flatness of~$B$,
  we have a left exact sequence
  \[0\to \Hom_{\K{A}}(\gM,\gN)\otimes_AB\to C^0(\gN)\otimes_AB\to
    C^1(\gN)\otimes_AB\]  and therefore (applying
Corollary~\ref{cor: base change completion for complex in free case}
to treat the case that $\gN$ is projective)  a left exact sequence
\[0\to \Hom_{\K{A}}(\gM,\gN)\otimes_AB\to C^0(\gN\cotimes_AB)\to
  C^1(\gN\cotimes_AB).\]
The result follows from Corollary~\ref{cor:complex computes Hom and
  Ext} and Lemma~\ref{lem:base-change-complexes}.
\end{proof}

\begin{lemma}\label{lem: vanishing of Kisin module homs implies vanishing on dense open}
Suppose that $\gM$ is a Breuil--Kisin module with
descent data and coefficients in a Noetherian $\cO/\varpi^a$-algebra
$A$ which is furthermore a domain.
Suppose also that $\gN$ is either a Breuil--Kisin module with
$A$-coefficients, or that $\gN=\gN'/u^N\gN'$, where $\gN'$  is a Breuil--Kisin module with
$A$-coefficients and $N\ge 1$. 
Then there is some nonzero $f\in
A$ with the following property: 
writing 
$\gM_{A_f}=\gM\cotimes_A A_f$ and $\gN_{A_f}=\gN\cotimes_A A_f$, then for any
finitely generated $A_f$-algebra $B$, and any finitely 
generated $B$-module $Q$, there are natural isomorphisms
\begin{multline*}
\Hom_{\K{A_f}}(\gM_{A_f},\gN_{A_f})\otimes_{A_f}Q \iso
\Hom_{\K{B}}(\gM_{A_f}\cotimes_{A_f} B, \gN_{A_f}\cotimes_{A_f} B)\otimes_B Q
\\
\iso
\Hom_{\K{B}}(\gM_{A_f}\cotimes_{A_f} B, \gN_{A_f}\cotimes_{A_f} Q).
\end{multline*}
\end{lemma}
\begin{proof}[Proof of Lemma~{\ref{lem: vanishing of Kisin module homs implies vanishing on dense open}}.]
Note that since $A$ is Noetherian, by Remark~\ref{rem:projectivity for
  Kisin modules}(3) we see that~$\gN$ is $A$-flat. 
 By Corollary~\ref{cor:complex computes Hom and Ext}
  we have an exact sequence
  \[0\to \Hom_{\K{A}}(\gM,\gN)\to C^0(\gN)\to C^1(\gN) \to
    \Ext^1_{\K{A}}(\gM,\gN)\to 0.\]
  Since by assumption $\gM$ is a projective $\gS_A$-module, and 
 $\gN$ is a flat
$A$-module, the $C^i(\gN)$ are also flat $A$-modules.

By Proposition~\ref{prop:exts are f.g. over A},
$\Ext^1_{\K{A}}(\gM,\gN)$ is a finitely generated $A$-module, so
by the generic freeness
theorem~\cite[\href{http://stacks.math.columbia.edu/tag/051R}{Tag
    051R}]{stacks-project} there is some nonzero $f\in A$ such that $\Ext^1_{\K{A}}(\gM,\gN)_f$ is
free over~$A_f$.

Since localisation is exact, we obtain an exact
sequence
\[0\to \Hom_{\K{A_f}}(\gM,\gN)_f \to C^0(\gN)_f \to C^1(\gN)_f \to
  \Ext^1_{\K{A}}(\gM,\gN)_f\to 0\]and therefore (applying
Corollary~\ref{cor: base change completion for complex in free case}
to treat the case that $\gN$ is a Breuil--Kisin module) an exact sequence
\[0\to \Hom_{\K{A_f}}(\gM_{A_f},\gN_{A_f})\to C^0(\gN_{A_f})\to C^1(\gN_{A_f}) \to
  \Ext^1_{\K{A}}(\gM,\gN)_f\to 0.\] 

Since the last three terms are flat over~$A_f$, this sequence remains
exact upon tensoring over $A_f$
with $Q$.
Applying  Corollary~\ref{cor: base change completion for complex in free case}
again to treat the case that $\gN$ is a Breuil--Kisin module, we see that in particular we
have a left exact sequence
\[0\to
\Hom_{\K{A_f}}(\gM_{A_f},\gN_{A_f})\otimes_{A_f}Q
\to
  C^0(\gN_{A_f}\cotimes_{A_f}Q)\to C^1(\gN_{A_f}\cotimes_{A_f}Q),\]
and Corollary~\ref{cor:complex computes Hom and Ext} together with Lemma~\ref{lem:base-change-complexes}
yield one of the desired isomorphisms, namely
$$\Hom_{\K{A_f}}(\gM_{A_f},\gN_{A_f})\otimes_{A_f}Q \iso 
\Hom_{\K{B}}(\gM_{A_f}\cotimes_{A_f}B ,\gN_{A_f}\cotimes_{A_f} Q).$$
If we consider the case when $Q = B$, we  obtain an isomorphism
$$\Hom_{\K{A_f}}(\gM_{A_f},\gN_{A_f})\otimes_{A_f}B \iso 
\Hom_{\K{B}}(\gM_{A_f}\cotimes_{A_f}B ,\gN_{A_f}\cotimes_{A_f} B).$$
Rewriting the tensor product $\text{--}\otimes_{A_f} Q $ as
$\text{--}\otimes_{A_f} B \otimes_B Q,$
we then find that 
$$
\Hom_{\K{B}}(\gM_{A_f}\cotimes_{A_f}B ,\gN_{A_f}\cotimes_{A_f} B)\otimes_B Q
\iso
\Hom_{\K{B}}(\gM_{A_f}\cotimes_{A_f}B ,\gN_{A_f}\cotimes_{A_f} Q),$$
which gives the second desired isomorphism.
\end{proof}

Variants on the preceding result may be proved using other
versions of the generic freeness theorem.

\begin{example}\label{example:rank one unramified redux} Returning to
  the setting of
  Example~\ref{example:rank one unramified},
 one can check using Corollary~\ref{cor:vanishing of homs non
  Noetherian} that the conclusion of Lemma~\ref{lem: vanishing of
  Kisin module homs implies vanishing on dense open} (for $\gM =
\gM_x$ and $\gN = \gM_y)$  holds with $f =
x-y$. In this case all of the resulting $\Hom$ groups
vanish (\emph{cf}.\ also the proof of Lemma~\ref{lem: generically no Homs}).
It then follows from
Corollary~\ref{cor:freeness for exts} that
$\Ext^1_{\K{A}}(\gM,\gN)_{f}$ is projective over $A_f$, so that the proof
of Lemma~\ref{lem: vanishing of Kisin module homs implies vanishing on
  dense open} even goes through with this choice of $f$.
\end{example}

As well as considering homomorphisms and extensions of Breuil--Kisin modules, we need to
consider the homomorphisms and extensions of their associated \'etale $\varphi$-modules;
recall that the passage to associated \'etale $\varphi$-modules amounts
to inverting $u$, and so we briefly discuss this process in the general
context of the category $\K{A}$.

We let $\K{A}[1/u]$ denote the full subcategory of $\K{A}$
consisting of objects on which multiplication by $u$ is invertible.
We may equally well regard it as the category of left
$\gS_A[1/u][F,\Gal(K'/K)]$-modules (this notation being interpreted in
the evident manner).   
There are natural isomorphisms (of bi-modules)
\numequation
\label{eqn:left tensor iso}
\gS_A[1/u]\otimes_{\gS_A} \gS_A[F,\Gal(K'/K)] 
\iso \gS_A[1/u][F,\Gal(K'/K)]
\end{equation}
and
\numequation
\label{eqn:right tensor iso}
\gS_A[F,\Gal(K'/K)] \otimes_{\gS_A} \gS_A[1/u]
\iso \gS_A[1/u][F,\Gal(K'/K)].
\end{equation}
Thus (since $\gS_A \to \gS_A[1/u]$ is a flat morphism of commutative rings)
the morphism of rings $\gS_A[F,\Gal(K'/K)] \to \gS_A[1/u][F,\Gal(K'/K)]$
is both left and right flat.

If $\gM$ is an object of $\K{A}$, then we see from~(\ref{eqn:left tensor
iso}) that 
$\gM[1/u] := \gS_A[1/u]\otimes_{\gS_A} \gM \iso \gS_A[1/u][F,\Gal(K'/K)]
\otimes_{\gS_A[F,\Gal(K'/K)]} \gM$ is naturally an object
of $\K{A}[1/u]$.   Our preceding remarks about flatness show
that $\gM \mapsto \gM[1/u]$ is an exact functor $\K{A}\to \K{A}[1/u]$.

\begin{lemma}\label{lem:ext-i-invert-u}
\begin{enumerate}
\item If $M$ and $N$ are objects
of $\K{A}[1/u]$, then 
there is a natural isomorphism
$$\Ext^i_{\K{A}[1/u]}(M,N) \iso \Ext^i_{\K{A}}(M,N).$$
\item
If $\gM$ is an object of $\K{A}$ and $N$ is an object of $\K{A}[1/u]$,
then there is a natural isomorphism
$$\Ext^i_{\K{A}}(\gM,N) \iso \Ext^i_{\K{A}}(\gM[1/u],N),$$
for all $i\geq 0$.
\end{enumerate}
\end{lemma}
\begin{proof}
The morphism of~(1) can be understood in various ways; for example,
by thinking in terms of Yoneda Exts, and recalling that $\K{A}[1/u]$
is a full subcategory of $\K{A}.$   If instead we think in terms
of projective resolutions, we can begin with a projective resolution
$\gP^{\bullet} \to M$ in $\K{A}$, and then consider the induced
projective resolution $\gP^{\bullet}[1/u]$ of $M[1/u]$.  Noting 
that $M[1/u] \iso M$ for any object $M$ of $\K{A}[1/u]$,
we then find (via tensor adjunction) that $\Hom_{\K{A}}(\gP^{\bullet},
N) \iso \Hom_{\K{A}[1/u]}(\gP^{\bullet}[1/u], N)$,
which induces the desired isomorphism of $\Ext$'s by passing to 
cohomology.

Taking into account the isomorphism of~(1), the claim of~(2) is a general
fact about tensoring over a flat ring map (as can again be seen by
considering projective resolutions). 
\end{proof}

\begin{remark}
The preceding lemma is fact an automatic consequence of the abstract 
categorical properties of our situation:\ the functor $\gM \mapsto \gM[1/u]$
is left adjoint to the inclusion $\K{A}[1/u] \subset\K{A},$
and restricts to (a functor naturally equivalent to) the identity functor
on $\K{A}[1/u]$.
\end{remark}
The following lemma expresses the Hom between \'etale $\varphi$-modules
arising from Breuil--Kisin modules in terms
of a certain direct limit.

\begin{lem} 
  \label{lem:computing Hom as direct limit}Suppose that  $\gM$ is a
   Breuil--Kisin module with descent data in a Noetherian $\cO/\varpi^a$-algebra~$A$, and that~$\gN$ is an object of $\K{A}$ which is
   finitely generated and $u$-torsion free as an
   $\gS_A$-module. 
   Then there is a natural isomorphism
\[ \varinjlim_i\Hom_{\K{A}}(u^i\gM,\gN) \iso
	\Hom_{\K{A}[1/u]}(\gM[1/u],\gN[1/u]),\]
where the transition maps are induced by the inclusions $u^{i+1} \gM 
\subset u^i \gM$.
\end{lem}
\begin{rem}
  \label{rem: maps in direct limit are injections}Note that since
  $\gN$ is $u$-torsion free, the transition maps in the colimit are
  injections, so the colimit is just an increasing union.
\end{rem}
\begin{proof}There are compatible injections $\Hom_{\K{A}}(u^i\gM,\gN) \to
	\Hom_{\K{A}[1/u]}(\gM[1/u],\gN[1/u])$, taking $f'\in
        \Hom_{\K{A}}(u^i\gM,\gN)$ to $f\in\Hom_{\K{A}}(\gM,\gN[1/u])$
        where $f(m)=u^{-i}f'(u^im)$. Conversely, given
        $f\in\Hom_{\K{A}}(\gM,\gN[1/u])$, there is some~$i$ such that
        $f(\gM)\subset u^{-i}\gN$, as required.
\end{proof}
We have the following analogue of Proposition~\ref{prop:vanishing of homs}.
\begin{cor}
  \label{cor:vanishing homs with u inverted}Suppose that  $\gM$ and
  $\gN$ are Breuil--Kisin modules with descent data in a Noetherian $\cO/\varpi^a$-algebra~$A$. 
Consider the following conditions:
\begin{enumerate}
\item
$\Hom_{\K{B}[1/u]} \bigl((\gM\cotimes_A B)[1/u], (\gN\cotimes_A B)[1/u]\bigr) = 0$
for any finite type $A$-algebra $B$.
\item
$\Hom_{\K{\kappa(\m)}[1/u]}\bigl((\gM\otimes_A \kappa(\mathfrak m))[1/u],
(\gN\otimes_A \kappa(\mathfrak m))[1/u] \bigr) = 0$ 
for each maximal ideal $\mathfrak m$ of $A$.
\item
$\Hom_{\K{A}[1/u]}\bigl(\gM[1/u], (\gN\otimes_A Q)[1/u]\bigr) = 0$ 
for any 
finitely generated $A$-module $Q$.
\end{enumerate}
Then we have (1)$\implies$(2)$\iff$(3).  If $A$ is furthermore
Jacobson, then all three conditions are equivalent.
\end{cor}
\begin{proof}
By Lemma~\ref{lem:computing Hom as direct limit}, the three conditions
are respectively equivalent to the following conditions. 
\begin{enumerate}[label=(\arabic*$'$)]
\item
$\Hom_{\K{B}} \bigl(u^i(\gM\cotimes_A B), \gN\cotimes_A B\bigr) = 0$
for any finite type $A$-algebra $B$ and all $i\ge 0$.
\item
$\Hom_{\K{\kappa(\m)}}\bigl(u^i(\gM\otimes_A \kappa(\mathfrak m)),
\gN\otimes_A \kappa(\mathfrak m) \bigr) = 0$ 
for each maximal ideal $\mathfrak m$ of $A$ and all $i\ge 0$.
\item
$\Hom_{\K{A}}\bigl(u^i\gM, \gN\otimes_A Q\bigr) = 0$ 
for any 
finitely generated $A$-module $Q$ and all $i\ge 0$.
\end{enumerate}
Since $\gM$ is projective, the first two conditions are in turn
equivalent to
\begin{enumerate}[label=(\arabic*$''$)]
\item
$\Hom_{\K{B}} \bigl((u^i\gM)\cotimes_A B, \gN\cotimes_A B\bigr) = 0$
for any finite type $A$-algebra $B$ and all $i\ge 0$.
\item
$\Hom_{\K{\kappa(\m)}}\bigl((u^i\gM)\otimes_A \kappa(\mathfrak m),
\gN\otimes_A \kappa(\mathfrak m) \bigr) = 0$ 
for each maximal ideal $\mathfrak m$ of $A$ and all $i\ge 0$.
\end{enumerate}
The result then follows from Proposition~\ref{prop:vanishing of homs}.
\end{proof}

\begin{df}\label{def:kext}
If $\gM$ and $\gN$ are objects of $\K{A}$, then we define
$$\kExt^1_{\K{A}}(\gM,\gN)
:=
\ker\bigl(\Ext^1_{\K{A}}(\gM,\gN)\to\Ext^1_{\K{A}}(\gM[1/u],\gN[1/u])\bigr).$$  
The point of this definition is to capture, in the setting of
Lemma~\ref{lem: Galois rep is a functor if A is actually finite local}, the non-split extensions
of Breuil--Kisin modules whose underlying extension of Galois
representations is split. 
\end{df}

Suppose now that $\gM$ is a Breuil--Kisin module. 
The exact sequence in~$\K{A}$ \[0\to\gN\to \gN[1/u]\to\gN[1/u]/\gN\to 0\]
gives an exact sequence of complexes \[\xymatrix{0\ar[r]&
  C^0(\gN)\ar[d]\ar[r]&
  C^0(\gN[1/u])\ar[d]\ar[r]&C^0(\gN[1/u]/\gN)\ar[d]\ar[r]&0\\ 0\ar[r]&
  C^1(\gN)\ar[r]&
  C^1(\gN[1/u])\ar[r]&C^1(\gN[1/u]/\gN)\ar[r]&0. } \] It follows from
Corollary~\ref{cor:complex computes Hom and Ext},
Lemma~\ref{lem:ext-i-invert-u}(2), and 
the snake lemma that we have an exact
sequence  \numequation\label{eqn: computing kernel of Ext groups}\begin{split}0\to\Hom_{\K{A}}(\gM,\gN)\to\Hom_{\K{A}}(\gM,\gN[1/u])
\qquad \qquad \\
\to\Hom_{\K{A}}(\gM,\gN[1/u]/\gN) \to \kExt^1_{\K{A}}(\gM,\gN)\to
0.\end{split}\end{equation}

\begin{lem}\label{lem: bound on torsion in kernel of Exts}If $\gM$, $\gN$ are Breuil--Kisin modules with descent data and
  coefficients in a Noetherian $\cO/\varpi^a$-algebra~$A$, and $\gN$ has
  height at most~$h$, then $f(\gM)$ 
is killed by $u^i$ for any $f \in \Hom_{\K{A}}(\gM,\gN[1/u]/\gN)$ and
any $i \ge \lfloor e'ah/(p-1) \rfloor$. 
  \end{lem}
  \begin{proof}
    Suppose that $f$ is an  element of
    $\Hom_{\K{A}}(\gM,\gN[1/u]/\gN)$. Then $f(\gM)$ is a finitely
    generated submodule of $\gN[1/u]/\gN$, and it therefore killed
    by~$u^i$ for some $i\ge 0$. Choosing~$i$ to be the exponent of
    $f(\gM)$ (that is, choosing $i$ to be minimal), it follows
    that
    $(\varphi^*f)(\varphi^*\gM)$ has exponent 
    precisely~$ip$. (From the choice of $i$, we see that $u^{i-1}
    f(\gM)$ is nonzero but killed by $u$, i.e., it is just a $W(k')
    \otimes A$-module, and so its pullback by~$\varphi:\gS_A\to\gS_A$
    has exponent precisely $p$.  Then by the flatness  
   of~$\varphi:\gS_A\to\gS_A$  we have
    $u^{ip-1}(\varphi^*f)(\varphi^*\gM)=u^{p-1}\varphi^*(u^{i-1}
    f(\gM)) \neq 0$.)

We claim that $u^{i+e'ah}(\varphi^*f)(\varphi^*\gM)=0$; admitting this, we
deduce that $i+e'ah\ge ip$, as required. To see the claim, take
$x\in\varphi^*\gM$, so that $\Phi_{\gN}((u^i\varphi^*f)(x))=
u^if(\Phi_\gM(x))=0$. It is therefore enough to show that the kernel
of \[\Phi_{\gN}:\varphi^*\gN[1/u]/\varphi^*\gN\to \gN[1/u]/\gN\] is killed
by $u^{e'ah}$; but this follows immediately from an application of the
snake lemma to the commutative diagram \[\xymatrix{0\ar[r]&
  \varphi^*\gN\ar[r]\ar[d]_{\Phi_\gN}&\varphi^*\gN[1/u]\ar[r]\ar[d]_{\Phi_\gN}
&\varphi^*\gN[1/u]/\varphi^*\gN\ar[r]\ar[d]_{\Phi_\gN}&0
\\ 0\ar[r]&
  \gN\ar[r]&\gN[1/u]\ar[r]
&\gN[1/u]/\gN\ar[r]&0}
\]together with the assumption that $\gN$ has height at most~$h$ and
an argument as in the first line of the proof of
Lemma~\ref{lem:truncation argument used to prove f.g. of Ext Q version}.
  \end{proof}

\begin{lem}\label{lem:computing kernel of Ext groups finite level}If $\gM$, $\gN$ are Breuil--Kisin modules with descent data and
  coefficients in a Noetherian $\cO/\varpi^a$-algebra~$A$, and $\gN$ has
  height at most~$h$, then for any $i \ge \lfloor e'ah/(p-1) \rfloor$  we have an exact
  sequence 
 \[\begin{split}0\to\Hom_{\K{A}}(u^i\gM,u^i\gN)\to\Hom_{\K{A}}(u^i\gM,\gN) \qquad \qquad \\
 \to\Hom_{\K{A}}(u^i\gM,\gN/u^i\gN) \to \kExt^1_{\K{A}}(\gM,\gN)\to
0.\end{split}\]
  \end{lem}
\begin{proof} Comparing Lemma~\ref{lem: bound on torsion in kernel of
    Exts} with the proof of Lemma~\ref{lem:computing Hom as direct
    limit}, we see that the direct limit in that proof has stabilised
  at $i$, and we obtain an isomorphism $\Hom_{\K{A}}(\gM,\gN[1/u])
  \toisom \Hom_{\K{A}}(u^i \gM,\gN)$ sending a map $f$ to $f' : u^i m
  \mapsto u^i f(m)$.  The same formula evidently identifies
  $\Hom_{\K{A}}( \gM,\gN)$ with $\Hom_{\K{A}}(u^i\gM,u^i\gN)$ and
  $\Hom_{\K{A}}(\gM,\gN[1/u]/\gN)$ with $\Hom_{\K{A}}(u^i
  \gM,\gN[1/u]/u^i \gN)$. But any map in the latter group has image
 contained  in $\gN/u^i \gN$ (by Lemma~\ref{lem: bound on torsion in kernel of
    Exts} applied to  $\Hom_{\K{A}}(\gM,\gN[1/u]/\gN)$, together with
  the identification in the previous sentence), so that $\Hom_{\K{A}}(u^i
  \gM,\gN[1/u]/u^i \gN) = \Hom_{\K{A}}(u^i
  \gM,\gN/u^i \gN)$.
  \end{proof}

\begin{prop}
  \label{prop: base change for kernel of map to etale Ext}
Let $\gM$ and $\gN$ be Breuil--Kisin modules with descent data and
coefficients in a Noetherian $\cO/\varpi^a$-domain~$A$. 
Then there is some nonzero $f\in
A$ with the following property: if we write 
$\gM_{A_f}=\gM\cotimes_A A_f$ and $\gN_{A_f}=\gN\cotimes_A A_f$, then if~$B$
is any finitely generated $A_f$-algebra, and if $Q$ is any finitely
generated $B$-module, we have natural isomorphisms
\begin{multline*}
\kExt^1_{\K{A_f}}(\gM,\gN)\otimes_{A_f}Q \iso
\kExt^1_{\K{A_f}}(\gM_{A_f}\cotimes_{A_f} B, 
\gN\cotimes_{A_f} B)\otimes_B Q
\\
\iso
\kExt^1_{\K{A_f}}(\gM_{A_f}\cotimes_{A_f} B, 
\gN\cotimes_{A_f} Q).
\end{multline*}
\end{prop}
\begin{proof}In view of Lemma~\ref{lem:computing kernel of Ext groups
    finite level}, this follows from  Lemma~\ref{lem: vanishing of
    Kisin module homs implies vanishing on dense open}, 
  with $\gM$ there being our $u^i\gM$, and $\gN$ being
  each of $\gN$,  $\gN/u^i\gN$ in turn.
\end{proof}

The following result will be crucial in our investigation of the
decomposition of $\cC^{\dd,1}$ and $\cR^{\dd,1}$ into
irreducible components.

\begin{prop}
  \label{prop: we have vector bundles}
Suppose that $\gM$ and $\gN$ are Breuil--Kisin modules with
descent data and coefficients in a Noetherian $\cO/\varpi^a$-algebra $A$
which is furthermore a domain,
and suppose that
$\Hom_{\K{A}}\bigl(\gM\otimes_A \kappa(\mathfrak m),
\gN\otimes_A \kappa(\mathfrak m) \bigr)$ vanishes
for each maximal ideal $\mathfrak m$ of $A$.
Then there is some nonzero $f\in
A$ with the following property: if we write 
$\gM_{A_f}=\gM\cotimes_A A_f$ and $\gN_{A_f}=\gN\cotimes_A A_f$, then 
for any finitely generated $A_f$-algebra $B$,
each of $\kExt^1_{\K{B}}(\gM_{A_f}\cotimes_{A_f} B,\gN_{A_f}\cotimes_{A_f}B)$,
$\Ext^1_{\K{B}}(\gM_{A_f}\cotimes_{A_f} B,\gN_{A_f}\cotimes_{A_f}B)$,
and
$$\Ext^1_{\K{B}}(\gM_{A_f}\cotimes_{A_f}B ,\gN_{A_f}\cotimes_{A_f}B)/
\kExt^1_{\K{A_f}}(\gM_{A_f}\cotimes_{A_f} B,\gN_{A_f}\cotimes_{A_f} B)$$
is a finitely generated projective $B$-module.
\end{prop}
\begin{proof}
Choose $f$ as in 
  Proposition~\ref{prop: base change for kernel of map to etale Ext},
let $B$ be a finitely generated $A_f$-algebra,
and let $Q$ be a finitely generated $B$-module.
By Propositions~\ref{prop:base-change for exts} and~\ref{prop: base
    change for kernel of map to etale Ext}, the morphism
  \[\kExt^1_{\K{B}}(\gM_{A_f}\cotimes_{A_f} B,\gN_{A_f}\cotimes_{A_f} B)\otimes_B Q\to
\Ext^1_{\K{B}}(\gM_{A_f}\cotimes_{A_f} B,\gN_{A_f}\cotimes_{A_f}B)\otimes_B Q\]
  is naturally identified with the morphism
  \[\kExt^1_{\K{B}}(\gM_{A_f}\cotimes_{A_f} B,\gN_{A_f}\cotimes_{A_f} Q)\to
\Ext^1_{\K{B}}(\gM_{A_f}\cotimes_{A_f} B,\gN_{A_f}\cotimes_{A_f} Q);\]
  in particular, it is injective. 
By Proposition~\ref{prop:base-change for exts} and
Corollary~\ref{cor:freeness for exts} we see that
  $\Ext^1_{\K{B}}(\gM_{A_f}\cotimes_{A_f} B,\gN_{A_f}\cotimes_{A_f} B)$ is a
finitely generated projective $B$-module; hence it is also flat.
Combining this with the injectivity just proved, we find that
  \[\Tor^1_B\bigl(Q, \Ext^1_{\K{B}}(\gM\cotimes_{A_f}B,\gN_{A_f}\cotimes_{A_f} B)/
\kExt^1_{\K{B}}(\gM_{A_f}\cotimes_{A_f} B,\gN_{A_f}\cotimes_{A_f} B)\bigr)=0\]
  for every finitely generated $B$-module $Q$, and thus that
  $$\Ext^1_{\K{B}}(\gM_{A_f}\cotimes_{A_f}B,\gN_{A_f}\cotimes_{A_f}B)/
\kExt^1_{\K{B}}(\gM_{A_f}\cotimes_{A_f}B,\gN_{A_f}\cotimes_{A_f}B)$$
 is a finitely generated flat,
and therefore finitely generated projective, $B$-module.
Thus $\kExt^1_{\K{B}}(\gM_{A_f}\cotimes_{A_f}B,\gN_{A_f}\cotimes_{A_f}B)$
is a direct summand of
  the finitely generated projective $B$-module
$\Ext^1_{\K{B}}(\gM_{A_f}\cotimes_{A_f} B,\gN_{A_f}\cotimes_{A_f}B)$, and so is
  itself a finitely generated projective $B$-module.
\end{proof}

\subsection{Families of extensions}
\label{subsec:families of extensions}
Let $\gM$ and $\gN$ be Breuil--Kisin modules with descent data
and $A$-coefficients, so that
$\Ext^1_{\K{A}}(\gM,\gN)$ is an $A$-module.
Suppose that $\psi: V \to \Ext^1_{\K{A}}(\gM,\gN)$ is
a homomorphism of $A$-modules whose source is a projective $A$-module of 
finite rank.
Then we may regard $\psi$ as an element of
$$\Ext^1_{\K{A}}(\gM,\gN)\otimes_A V^{\vee} =
\Ext^1_{\K{A}}(\gM, \gN\otimes_A V^{\vee} ),$$ 
and in this way $\psi$ 
corresponds to an extension
\numequation
\label{eqn:universal extension}
0 \to \gN\otimes_A V^{\vee} \to \gE \to \gM \to 0,
\end{equation}
which we refer to as the {\em family of extensions} of $\gM$ by $\gN$
parametrised by $V$ (or by $\psi$, if we want to emphasise our
choice of homomorphism).  We let $\gE_v$ denote the pushforward of $\gE$ under
the morphism $\gN\otimes_A V^{\vee} \to \gN$
given by evaluation on $v\in V$.
In the special case that 
$\Ext^1_{\K{A}}(\gM,\gN)$ itself is a projective $A$-module of finite rank,
we can let $V$ be $\Ext^1_{\K{A}}(\gM,\gN)$ and take $\psi$ be the
identity map;
in this case we refer to~(\ref{eqn:universal extension}) as
the {\em universal extension} of $\gM$ by $\gN$.
The reason for this terminology is as follows:
if $v \in \Ext^1_{\K{A}}(\gM,\gN)$, 
then $\gE_v$ is the extension of $\gM$ by $\gN$ corresponding to the element
$v$. 

Let $B := A[V^{\vee}]$ denote the symmetric algebra over $A$ generated by
$V^{\vee}$.  
The short exact sequence~(\ref{eqn:universal extension}) is a
short exact sequence of Breuil--Kisin modules with descent data,
and so forming its $u$-adically completed tensor product with $B$ over $A$,
we obtain a short exact sequence 
$$
0 \to \gN\otimes_A V^{\vee} \cotimes_A B \to \gE\cotimes_A B \to
\gM\cotimes_A B
\to 0$$
of Breuil--Kisin modules with descent data over $B$ (see Lemma~\ref{rem: base change of locally free Kisin module is a
    locally free Kisin module}).
Pushing this short exact sequence forward under the natural map
$$V^{\vee} \cotimes_A B = V^{\vee} \otimes_A B \to B$$
induced by the inclusion of $V^{\vee}$ in $B$
and the multiplication map $B\otimes_A B \to B$,
we obtain a short exact sequence
\numequation
\label{eqn:geometric universal extension}
0 \to \gN\cotimes_A B \to \widetilde{\gE} \to \gM\cotimes_A B \to 0
\end{equation}
of Breuil--Kisin modules with descent data over $B$,
which we call the {\em family of extensions} of $\gM$  by $\gN$
parametrised by $\Spec B$ (which we note is (the total space of) the
vector bundle over $\Spec A$ corresponding to the projective
$A$-module~$V$).
 
If $\alpha_v: B \to A$ is the morphism induced
by the evaluation map
$V^{\vee} \to A$ given by some element $v \in V$,
then base-changing~(\ref{eqn:geometric universal extension}) by $\alpha_v$,
we recover the short exact sequence
$$0 \to \gN \to \gE_v \to \gM \to 0.$$
More generally, suppose that $A$ is a $\cO/\varpi^a$-algebra for some
$a\ge 1$, and let  $C$ be any $A$-algebra. Suppose that $\alpha_{\tilde{v}}:
B \to C$ is the morphism induced 
by the evaluation map
$V^{\vee} \to C$ corresponding to some element $\tilde{v} \in C\otimes_A V$.
Then base-changing~(\ref{eqn:geometric universal extension}) by
$\alpha_{\tilde{v}}$
yields a short exact sequence
$$0 \to \gN\cotimes_A C \to \widetilde{\gE}\cotimes_B C
\to \gM\cotimes_A C \to 0,$$
whose associated extension class corresponds
to the image of $\tilde{v}$ 
under the natural morphism
$C\otimes_A V \to C\otimes_A \Ext^1_{\K{A}}(\gM,
\gN) \cong \Ext^1_{\K{C}}(\gM\cotimes_A C, \gN\otimes_A C),$
the first arrow being induced by $\psi$
and the second arrow being the isomorphism of Proposition~\ref{prop:base-change
for exts}. 

\subsubsection{The functor represented by a universal family}
We now suppose that the ring~$A$ and the Breuil--Kisin modules $\gM$ and $\gN$ have the following 
properties:

\begin{assumption}
\label{assumption:vanishing}Let $A$ be a  Noetherian and Jacobson
$\cO/\varpi^a$-algebra for some $a\ge 1$, and assume that for each maximal ideal $\mathfrak m$ of $A$, we have that
$$\Hom_{\K{\kappa(\mathfrak{\m})}}\bigl(\gM\otimes_A \kappa(\mathfrak m) , \gN\otimes_A
\kappa(\mathfrak m)\bigr) = \Hom_{\K{\kappa(\mathfrak{\m})}}\bigl(\gN\otimes_A \kappa(\mathfrak m) , \gM\otimes_A
\kappa(\mathfrak m)\bigr) = 0.$$
\end{assumption}

By Corollary~\ref{cor:freeness for exts}, this assumption implies in particular
that $V:= \Ext^1_{\K{A}}(\gM,\gN)$ is projective of finite rank,
and so we may form $\Spec B := \Spec A[V^{\vee}]$,
which parametrised the universal family of 
extensions. 
We are then able to give the following precise description
of the functor represented by $\Spec B$.

\begin{prop}\label{prop: the functor that Spec B represents}
The scheme $\Spec B$ represents the functor which,
to any $\cO/\varpi^a$-algebra $C$, associates the set of isomorphism
classes of tuples $(\alpha, \gE, \iota, \pi)$, where $\alpha$ is a morphism
$\alpha: \Spec C \to \Spec A$, $\gE$ is a Breuil--Kisin module
with descent data and coefficients in $C$, and $\iota$ and $\pi$ are morphisms
$\alpha^* \gN \to \gE$ and $\gE \to \alpha^* \gM$ respectively,
with the property that $0 \to \alpha^*\gN \buildrel \iota
\over \to \gE \buildrel \pi  \over \to \alpha^* \gM \to 0$
is short exact.
\end{prop}
\begin{proof}
We have already seen that giving a morphism $\Spec C \to \Spec B$
is equivalent to giving the composite morphism $\alpha:\Spec C \to \Spec B 
\to \Spec A$, together with an extension class
$[\gE] \in \Ext^1_{\K{C}}(\alpha^*\gM,\alpha^*\gN).$
Thus to prove the proposition, we just have to show that
any automorphism of $\gE$ which restricts to the identity on $\alpha^*\gN$
and induces the identity on $\alpha^*\gM$ is itself the identity on
$\gE$.   This follows from Corollary~\ref{cor:vanishing of homs non
  Noetherian}, 
together with Assumption~\ref{assumption:vanishing}. 
\end{proof}
Fix an integer $h\ge 0$ so that $E(u)^h\in
\Ann_{\gS_A}(\coker\Phi_{\gM})\Ann_{\gS_A}(\coker\Phi_{\gN})$, so
that by Lemma~\ref{lem:ext of a Kisin module is a Kisin module}, every
Breuil--Kisin module parametrised by $\Spec B$ has height at most~$h$.
There is a natural action of $\Gm\times_{\cO} \Gm$ on $\Spec B$,
given by rescaling each of $\iota$ and $\pi$.
There is also an evident forgetful morphism
$\Spec B \to \Spec A \times_{\cO} \cC^{\dd,a}$, 
given by forgetting $\iota$ and $\pi$, which is evidently invariant
under the $\Gm\times_{\cO} \Gm$-action. (Here and below,
$\cC^{\dd,a}$ denotes the moduli stack defined in Section~\ref{subsec:results from EG
  and PR on finite flat moduli} for our fixed choice of~$h$ and for
$d$ equal to the sum of the ranks of $\gM$ and $\gN$.)  We thus obtain a morphism
\numequation
\label{eqn:ext relation}
\Spec B \times_{\cO} \Gm\times_{\cO} \Gm \to
\Spec B \times_{\Spec A \times_{\cO} \cC^{\dd,a}} \Spec B.
\end{equation}

\begin{cor}\label{cor: monomorphism to Spec A times C}
Suppose that
$\Aut_{\K{C}}(\alpha^*\gM) = \Aut_{\K{C}}(\alpha^*\gN) = C^{\times}$
for any morphism $\alpha: \Spec C \to \Spec A$.
Then the morphism~{\em (\ref{eqn:ext relation})} is an isomorphism,
and consequently the induced morphism
$$[\Spec B/ \Gm\times_{\cO} \Gm] \to \Spec A \times_{\cO} \cC^{\dd,a}$$
is a finite type monomorphism.
\end{cor}
\begin{proof}
By Proposition~\ref{prop: the functor that Spec B represents}, a morphism  \[\Spec C\to \Spec B \times_{\Spec A
  \times_{\cO} \cC^{\dd,a}} \Spec B\]  corresponds to an isomorphism class
of tuples $(\alpha,\beta:\gE\to\gE',\iota,\iota',\pi,\pi')$, where
\begin{itemize}
\item $\alpha$ is a morphism
$\alpha:\Spec C\to\Spec A$,
\item  $\beta:\gE\to\gE'$ is an isomorphism of Breuil--Kisin modules
with descent data and coefficients in $C$, 
\item $\iota:\alpha^* \gN \to \gE$ and $\pi :\gE \to
  \alpha^* \gM$ are morphisms
with the property that $$0 \to \alpha^*\gN \buildrel \iota
\over \to \gE \buildrel \pi  \over \to \alpha^* \gM \to 0$$ is short
exact, 
\item $\iota':\alpha^* \gN \to \gE'$ and $\pi' :\gE' \to
  \alpha^* \gM$ are morphisms
with the property that 
  $$0 \to \alpha^*\gN \buildrel \iota'
\over \to \gE' \buildrel \pi'  \over \to \alpha^* \gM \to 0$$ is
short exact.
\end{itemize}
Assumption~\ref{assumption:vanishing} and Corollary~\ref{cor:vanishing of homs non Noetherian}  together show that 
$\Hom_{\K{C}}(\alpha^*\gN, \alpha^*\gM)~=~0$. It follows that the
composite
$\alpha^*\gN\stackrel{\iota}{\to}\gE\stackrel{\beta}{\to}\gE' $
factors through $\iota'$, and the induced endomorphism of
$\alpha^*\gN$ is injective. Reversing the roles of $\gE$ and $\gE'$,
we see that it is in fact an automorphism of $\alpha^*\gN$, and it
follows easily that $\beta$ also induces an automorphism of
$\alpha^*\gM$. Again, Assumption~\ref{assumption:vanishing} and Proposition~\ref{cor:vanishing of homs non Noetherian} together show that 
$\Hom_{\K{C}}(\alpha^*\gM, \alpha^*\gN) = 0$, from which it follows
easily that $\beta$ is determined by the automorphisms of
$\alpha^*\gM$ and $\alpha^*\gN$ that it induces.

 Since $\Aut_{\K{C}}(\alpha^*\gM) =
\Aut_{\K{C}}(\alpha^*\gN) = C^{\times}$ by assumption, we see that
$\beta \circ \iota, \iota'$ and $\pi,\pi'\circ\beta$ differ only by the action of
$\Gm\times_{\cO}\Gm$, so  the first claim of the corollary
follows. 
The claim regarding the monomorphism is immediate from
Lemma~\ref{lem: morphism from quotient stack is a monomorphism}
below. Finally, note that $[\Spec B/\Gm\times_{\cO} \Gm]$ is 
of finite type over $\Spec A$, while $\cC^{\dd,a}$ has finite type diagonal.
It follows that the morphism 
$[\Spec B / \Gm \times_{\cO} \Gm ] \rightarrow \Spec A\times_{\cO}\cC^{\dd,a}$
is of finite type, as required.
\end{proof} 

\begin{lem}
  \label{lem: morphism from quotient stack is a monomorphism}Let $X$
  be a scheme over a base scheme~$S$, let $G$ be a smooth affine group
  scheme over~$S$, and let $\rho:X\times_S G\to X$ be a {\em
    (}right{\em )} action of $G$
  on~$X$. Let $X\to\cY$ be a $G$-equivariant morphism, whose target is
  an algebraic stack over~$S$ on which $G$ acts trivially. 
Then the induced
  morphism \[[X/G]\to\cY\] is a monomorphism if and only if the
  natural morphism \[X\times_S G\to X\times_{\cY} X\] {\em (}induced by the
  morphisms $\pr_1,\rho:X\times_S G\to X${\em )} is an isomorphism.
\end{lem} 
\begin{proof}
  We have a Cartesian diagram as follows. \[\xymatrix{X\times_S G\ar[r]\ar[d]&
    X\times_\cY X\ar[d]\\ [X/G]\ar[r]& [X/G]\times_\cY[X/G]}\]The
  morphism $[X/G]\to\cY$ is a monomorphism if and only if the bottom horizontal
  morphism of this square is an isomorphism; since the right hand
  vertical arrow is a smooth surjection, this is the case if and only
  if the top horizontal morphism is an isomorphism, as required.
\end{proof}
\subsection{Families of extensions of rank one Breuil--Kisin modules}
\label{subsec:universal families}
In this section we construct universal families of extensions of rank
one Breuil--Kisin modules. We will use these rank two families to study our
moduli spaces of Breuil--Kisin modules, and the corresponding spaces of
\'etale $\varphi$-modules. We show how to compute the dimensions of
these universal families; in the subsequent sections, we will combine
these results with explicit calculations to determine the irreducible
components of our moduli spaces. In particular, we will show that each
irreducible component has a dense open substack given by a family of
extensions.

\subsubsection{Universal unramified twists}
Fix a free Breuil--Kisin module with descent data 
$\gM$ over $\F$, and write $\Phi_i$   for
$\Phi_{\gM,i}:\varphi^*(\gM_{i-1}) \to\gM_{i}$. (Here we are using the
notation of Section~\ref{subsec: kisin modules with dd}, so
that~$\gM_i=e_i\gM$ is cut out by the idempotent~$e_i$ of Section~\ref{subsec: notation}.) 
We will construct the ``universal unramified twist'' of $\gM$.

\begin{df}
\label{def:unramified twist}
If $\Lambda$ is an $\F$-algebra, and if $\lambda \in \Lambda^\times$,
then we define $\gM_{\Lambda,\lambda}$ to be the free Breuil--Kisin
module with descent data and $\Lambda$-coefficients whose underlying
$\gS_\Lambda[\Gal(K'/K)]$-module 
is equal to $\gM\cotimes_\F\Lambda$
(so the usual base change of $\gM$ to $\Lambda$),
and for which $\Phi_{\gM_{\Lambda,\lambda}}: \varphi^*\gM_{\Lambda,\lambda} \to \gM_{\Lambda,\lambda}$ is defined via the $f'$-tuple
$(\lambda \Phi_0,\Phi_1,\ldots,\Phi_{f'-1}).$
We refer to $\gM_{\Lambda,\lambda}$ as the \emph{unramified twist} of $\gM$ by $\lambda$ over $\Lambda$.

If $M$ is a free \'etale $\varphi$-module with descent data, then we
define $M_{\Lambda,\lambda}$ in the analogous fashion. If we write
$X=\Spec \Lambda$, then we will sometimes write $\gM_{X,\lambda}$
(resp.\ $M_{X,\lambda}$) for $\gM_{\Lambda,\lambda}$ (resp.\
$M_{\Lambda,\lambda}$).
\end{df}

As usual, we write $\Gm := \Spec \F[x,x^{-1}].$ 
We may then 
form the rank one Breuil--Kisin module with descent data
$\gM_{\Gm,x},$ which is the universal instance of an unramified twist:
given $\lambda\in\Lambda^\times$, 
there is a corresponding morphism $\Spec\Lambda
\to \Gm$ determined by the requirement that
$x \in \Gamma(\Gm, \cO_{\Gm}^{\times})$ pulls-back to $\lambda$,
and $\gM_{X,\lambda}$ is obtained by pulling back
$\gM_{\Gm,x}$ under this morphism (that is, by base changing under the
corresponding ring homomorphism $\F[x,x^{-1}]\to\Lambda$).

\begin{lemma}
\label{lem:rank one endos}
If $\gM_\Lambda$ is a Breuil--Kisin module of rank one with $\Lambda$-coefficients, then $\End_{\K{\Lambda}}(\gM) = \Lambda.$
Similarly, if $M_\Lambda$ is a \'etale $\varphi$-module of rank
one with $\Lambda$-coefficients, then
$\End_{\K{\Lambda}}(M_\Lambda) = \Lambda.$ 
\end{lemma}
\begin{proof}
We give the proof for~$M_\Lambda$, the argument for~$\gM_\Lambda$ being essentially
identical. One reduces easily to the case where $M_\Lambda$ is free. Since an endomorphism~$\psi$ of~$M_\Lambda$ is in particular an
endomorphism of the underlying $\gS_\Lambda[1/u]$-module, we see that there
is some $\lambda\in \gS_\Lambda[1/u]$ such that $\psi$ is given by
multiplication by~$\lambda$. The commutation relation with $\Phi_{M_\Lambda}$
means that we must have $\varphi(\lambda)=\lambda$, so that
certainly 
(considering the powers of~$u$ in~$\lambda$ of lowest negative and positive degrees)
$\lambda\in W(k')\otimes_{\Zp}\Lambda$, and in fact $\lambda\in
\Lambda$. Conversely, multiplication by any element of~$\Lambda$ is evidently an
endomorphism of~$M_\Lambda$, as required.
\end{proof}

\begin{lem}
  \label{lem:rank one isomorphism over a field}Let $\kappa$ be a field
  of characteristic~$p$, and let $M_\kappa, N_\kappa$ be
  \'etale $\varphi$-modules of rank one with $\kappa$-coefficients and
  descent data. Then any nonzero element of $\Hom_{\K{\kappa}}(M_\kappa, N_\kappa)$
  is an isomorphism. 
\end{lem}
\begin{proof}
  Since $\kappa((u))$ is a field, it is enough to show that if one of
  the induced maps $M_{\kappa,i}\to N_{\kappa,i}$ is nonzero, then
  they all are; but this follows from the commutation relation with~$\varphi$.
\end{proof}

\begin{lemma}\label{lem: isomorphic twists are the same twist} If $\lambda,\lambda' \in \Lambda^{\times}$ and $\gM_{\Lambda,\lambda} \cong \gM_{\Lambda,\lambda'}$
{\em (as Breuil--Kisin modules with descent data over $\Lambda$)}, then $\lambda=\lambda'$. Similarly, if $M_{\Lambda,\lambda}\cong M_{\Lambda,\lambda'}$, then $\lambda=\lambda'$.
\end{lemma}
\begin{proof}
Again, we give the proof for~$M$, the argument for~$\gM$ being
essentially identical. Write $M_i=\F((u))m_i$, and write
$\Phi_i(1\otimes m_{i-1})=\theta_{i}m_{i}$, where $\theta_{i}\ne 0$. There
are $\mu_i\in \Lambda[[u]][1/u]$ such that the given isomorphism
$M_{\Lambda,\lambda}\cong M_{\Lambda,\lambda'}$  takes~$m_i$ to $\mu_im_i$. The
commutation relation between the given isomorphism and~$\Phi_{M}$ imposes the
condition \[\lambda_{i}\mu_{i}\theta_{i}m_{i}=\lambda'_{i}\varphi(\mu_{i-1})\theta_{i}m_{i}\]where
$\lambda_{i}$ (resp.\ $\lambda'_i$) equals~$1$ unless $i=0$, when it equals~$\lambda$
(resp.\ $\lambda'$). 

Thus we have $\mu_{i}=(\lambda'_i/\lambda_i)\varphi(\mu_{i-1})$, so that in
particular $\mu_0=(\lambda'/\lambda)\varphi^{f'}(\mu_0)$. Considering the powers
of~$u$ in~$\mu_0$ of lowest negative and positive degrees we
conclude that $\mu_0\in W(k') \otimes \Lambda$; but then
$\mu_0=\varphi^{f'}(\mu_0)$, so that~$\lambda'=\lambda$, as required.
\end{proof}

\begin{remark} 
If $\gM$ has height at most~$h$, and we let $\cC$ (temporarily) denote the
moduli stack of rank one Breuil--Kisin modules of height at most~$h$ with $\F$-coefficients and
descent data then Lemma~\ref{lem: isomorphic twists are the same twist} can be interpreted as saying that 
the morphism $\Gm \to \cC$ 
that classifies $\gM_{\Gm,x}$ 
is a monomorphism, i.e.\ the diagonal morphism $\Gm \to \Gm
\times_{\cC}
\Gm$ is an isomorphism. Similarly, the morphism $\Gm\to\cR$
(where we temporarily let $\cR$ denote the moduli stack
of rank one \'etale $\varphi$-modules with $\F$-coefficients and descent data)
that classifies $M_{\Gm,x}$ is a monomorphism.
\end{remark}

Now choose another rank one Breuil--Kisin module with descent data $\gN$ over $\F$.
Let $(x,y)$ denote the standard coordinates on $\Gm\times_{\F} \Gm$,
and consider the rank one Breuil--Kisin modules with descent data $\gM_{\Gm\times_{\F} \Gm,x}$ 
and $\gN_{\Gm\times_{\F} \Gm,y}$ 
over $\Gm\times_{\F} \Gm$.

\begin{lemma}\label{lem: generically no Homs}
There is a non-empty irreducible affine open subset $\Spec \Adist$ of $\Gm\times_{\F} \Gm$
whose finite type points 
are exactly the maximal ideals $\mathfrak m$ of $\Gm\times_{\F} \Gm$
such that
 \[\Hom_{\K{\kappa(\mathfrak m)}}\bigl( \gM_{\kappa(\mathfrak m),\xbar}[1/u],
\gN_{\kappa(\mathfrak m),\ybar}[1/u]\bigr)=0\]
{\em (}where we have written $\xbar$ and $\ybar$ to
denote the images of $x$ and $y$ in $\kappa(\mathfrak m)^{\times}${\em )}.

Furthermore, if $R$ is any finite-type $\Adist$-algebra,
and if $\mathfrak m$ is any maximal ideal of~$R$,
then 
\[\Hom_{\K{\kappa(\mathfrak m)}}\bigl( \gM_{\kappa(\mathfrak m),\xbar},
\gN_{\kappa(\mathfrak m),\ybar}\bigr)
= \Hom_{\K{\kappa(\mathfrak m)}}\bigl( \gM_{\kappa(\mathfrak m),\xbar}[1/u],
\gN_{\kappa(\mathfrak m),\ybar}[1/u]\bigr)
= 0,\]
and also
\[
\Hom_{\K{\kappa(\mathfrak m)}}\bigl( \gN_{\kappa(\mathfrak m),\ybar},
\gM_{\kappa(\mathfrak m),\xbar}\bigr)
=
 \Hom_{\K{\kappa(\mathfrak m)}}\bigl( \gN_{\kappa(\mathfrak m),\ybar}[1/u],
\gM_{\kappa(\mathfrak m),\xbar}[1/u]\bigr)=0.\]
In particular, Assumption~{\em \ref{assumption:vanishing}}
is satisfied by $\gM_{\Adist,x}$ and $\gN_{\Adist,y}$.
\end{lemma} 
\begin{proof} 
If  $\Hom\bigl(\gM_{\kappa(\mathfrak m),\xbar}[1/u],\gN_{\kappa(\mathfrak m),\ybar}[1/u])=0$ 
for all maximal ideals~$\m$ of $\F[x,y,x^{-1},y^{-1}]$, then we are
done: $\Spec A^{\dist} = \Gm\times\Gm$. Otherwise, we see that
for some finite extension $\F'/\F$ and some $a,a'\in\F'$, we have a
non-zero morphism $\gM_{\F',a}[1/u]\to\gN_{\F',a'}[1/u]$. By
Lemma~\ref{lem:rank one isomorphism over a field}, this morphism must
in fact be an isomorphism. 
Since $\gM$ and $\gN$ are both defined over $\F$, we furthermore see
that the ratio $a'/a$ lies in $\F$.
We then let $\Spec\Adist$ be the affine open subset of $\Gm\times_{\F}
\Gm$ where $a'x\ne ay$; the claimed property of $\Spec A^{\dist}$ 
then follows easily from 
Lemma~\ref{lem: isomorphic twists are the same twist}.

For the remaining statements of the lemma,
note that if $\mathfrak m$ is a maximal
ideal in a finite type $A^{\dist}$-algebra, then its pull-back to $A^{\dist}$
is again a maximal ideal $\mathfrak m'$ 
of $A^{\dist}$ (since $A^{\dist}$ is Jacobson),
and the vanishing of
$$
\Hom_{\K{\kappa(\mathfrak m)}}\bigl( \gM_{\kappa(\mathfrak m),\xbar}[1/u],
\gN_{\kappa(\mathfrak m),\ybar}[1/u]\bigr)
$$
follows from the corresponding statement for $\kappa(\mathfrak m')$,
together with  Lemma~\ref{lem: flat base change for Homs}.

Inverting $u$ induces an embedding
\[\Hom_{\K{\kappa(\mathfrak m)}}\bigl( \gM_{\kappa(\mathfrak m),\xbar},
\gN_{\kappa(\mathfrak m),\ybar}\bigr)
\hookrightarrow
\Hom_{\K{\kappa(\mathfrak m)}}\bigl( \gM_{\kappa(\mathfrak m),\xbar}[1/u],
\gN_{\kappa(\mathfrak m),\ybar}[1/u]\bigr),\]
and so certainly the vanishing of the target implies the vanishing of
the source.

The statements in which the roles of $\gM$ and $\gN$
are reversed follow from 
Lemma~\ref{lem:rank one isomorphism over a field}. 
\end{proof}

Define
$T := \Ext^1_{\K{\Gm\times_{\F}\Gm}}\bigl(\gM_{\Gm\times_{\F} \Gm,x},\gM_{\Gm\times_{\F}\Gm,y})$;
it follows from 
Proposition~\ref{prop:exts are f.g. over A} that
$T$ is finitely generated over $\F[x,x^{-1},y,y^{-1}],$  
while
Proposition~\ref{prop:base-change for exts}
shows that
$T_\Adist := T\otimes_{\F[x^{\pm 1}, y^{\pm 1}]} \Adist$ is naturally isomorphic to
$\Ext^1_{\K{\Adist}}\bigl(\gM_{\Adist,x}, \gN_{\Adist,y}\bigr)$. (Here and elsewhere
we abuse notation by writing $x$, $y$ for $x|_{\Adist}$, $y|_{\Adist}$.)
Corollary~\ref{cor:freeness for exts} and Lemma~\ref{lem: generically no Homs}
show that $T_\Adist$ is in fact a 
finitely generated projective $\Adist$-module.
If, for any $\Adist$-algebra $B$, we write $T_B := T_\Adist \otimes_\Adist B \iso
 T\otimes_{\F[x^{\pm 1}, y^{\pm 1}]} B$,
then 
Proposition~\ref{prop:base-change for exts} again shows that
$T_B \iso \Ext^1_{\K{B}}\bigl(\gM_{B,x}, \gN_{B,y}\bigr)$.

By Propositions~\ref{prop: base change for kernel of map to etale Ext}
and~\ref{prop: we have vector bundles}, together with
Lemma~\ref{lem: generically no Homs},
there is a nonempty (so dense) affine open subset ~$\Spec\Akfree$ of
  $\Spec \Adist$ with the properties that \[U_\Akfree :=
  \kExt^1_{\K{\Akfree}}(\gM_{\Akfree,x},\gN_{\Akfree,y})\] and
\begin{multline*}
	T_\Akfree/U_\Akfree \\
 \iso
  \Ext^1_{\K{\Akfree}}(\gM_{\Akfree,x},\gN_{\Akfree,y})/\kExt^1_{\K{\Akfree}}(\gM_{\Akfree,x},\gN_{\Akfree,y})
  \end{multline*}
  are finitely generated and projective over $\Akfree$, and furthermore so
  that for all finitely generated $\Akfree$-algebras $B$, the formation of
$\kExt^1_{\K{B}}(\gM_{B,x},\gN_{B,y})$ and 
$\Ext^1_{\K{B}}(\gM_{B,x},\gN_{B,y})/\kExt^1_{\K{B}}(\gM_{B,x},\gN_{B,y})$
is compatible with base change from $U_\Akfree$ and $T_\Akfree/U_\Akfree$ respectively.

We choose a finite rank projective module $V$ over $\F[x,x^{-1},y,y^{-1}]$
admitting a surjection $V \to T$.
Thus, if we write $V_\Adist := V\otimes_{\F[x^{\pm 1}, y^{\pm 1}]} \Adist$,
then the induced morphism $V_\Adist \to T_\Adist$ is a (split) surjection of
$\Adist$-modules.

Following the prescription 
of Subsection~\ref{subsec:families of extensions},
we form the symmetric algebra $\Btwist := \F[x^{\pm 1}, y^{\pm
  1}][V^{\vee}],$ 
and construct the  family of extensions $\widetilde{\gE}$
over $\Spec \Btwist$.    We may similarly form the symmetric algebras
$\Bdist := \Adist[T_{\Adist}^{\vee}]$ and $\Bkfree := \Akfree[T_{\Akfree}^{\vee}]$, and construct the families 
of extensions  $\gEdist$ and $\gEkfree$ over $\Spec
\Bdist$ and $\Spec\Bkfree$ respectively. 
Since $T_\Akfree/U_\Akfree$ is projective, the natural morphism
$T_{\Akfree}^{\vee} \to U_{\Akfree}^{\vee}$ is surjective, and hence 
$\Ckfree := A[U_{\Akfree}^{\vee}]$ is a quotient of $\Bkfree$; geometrically,
$\Spec \Ckfree $ is a subbundle of the vector bundle $\Spec \Bkfree$
over $\Spec A$. 

We write $X := \Spec \Bkfree \setminus \Spec \Ckfree$; it is an open subscheme
of the vector bundle $\Spec \Bkfree$.  
The restriction of 
$\widetilde{\gE}'$ to~$X$ is the universal family of extensions over~$A$ which
do not split after inverting~$u$.

\begin{remark}
	\label{rem:Zariski density}
	Since $\Spec A^{\dist}$ and $\Spec A^{\kfree}$ are irreducible,
	each of the vector bundles $\Spec B^{\dist}$ and $\Spec B^{\kfree}$
 	is also irreducible.  In particular, $\Spec B^{\kfree}$ is Zariski dense
	in $\Spec B^{\dist}$, and if $X$ is non-empty, then it is Zariski dense
	in each of $\Spec B^{\kfree}$ and $\Spec B^{\dist}$.   Similarly,
	$\Spec \Btwist \times_{\Gm\times_{\F}\Gm} \Spec A^{\dist}$ is Zariski dense in
	$\Spec \Btwist$.
\end{remark}

The surjection $V_\Adist \to T_\Adist$ induces a surjection of vector bundles
$\pi: \Spec \Btwist\times_{\Gm\times_{\F}\Gm} \Spec \Adist \to \Spec \Bdist$ over $\Spec \Adist$, 
and there is a natural isomorphism
\numequation
\label{eqn:pull-back iso}
 \pi^*\gEdist \iso
\widetilde{\gE}\cotimes_{\F[x^{\pm 1},y^{\pm 1}]} \Adist.
\end{equation}

The rank two Breuil--Kisin module with descent data
$\widetilde{\gE}$ is classified by a morphism
$\xi:\Spec \Btwist \to \cC^{\dd,1}$; 
similarly,
the rank two Breuil--Kisin module with descent data
$\gEdist$ is classified by a morphism
$\xi^\dist: \Spec \Bdist \to \cC^{\dd,1}.$
If we write $\xi_\Adist$ for the restriction of $\xi$ to the open
subset $\Spec \Btwist \times_{\Gm\times_{\F}\Gm} \Spec \Adist$ of $\Spec \Btwist$,
then the isomorphism~(\ref{eqn:pull-back iso}) shows that
$\xi^\dist\circ \pi = \xi_\Adist$.
We also write $\xi^{\kfree}$ for the restriction of $\xi^{\dist}$ to 
$\Spec B^{\kfree}$, and $\xi_X$ for the restriction of $\xi^{\kfree}$
to $X$.

\begin{lemma}\label{lem: scheme theoretic images coincide}
The scheme-theoretic images {\em (}in the sense
  of~{\em \cite[Def.\ 3.1.4]{EGstacktheoreticimages})} of
  $\xi:\Spec \Btwist\to \cC^{\dd,1}$,
$\xi^\dist:\Spec \Bdist\to \cC^{\dd,1}$,
and $\xi^{\kfree}: \Spec B^{\kfree}\to \cC^{\dd,1}$
all coincide; in particular, the
scheme-theoretic image of~$\xi$ is independent of the choice of
surjection $V\to T$, and the scheme-theoretic image of~$\xi^{\kfree}$ is
independent of the choice of~$\Akfree$. 
  If $X$ is non-empty, then the scheme-theoretic image
  of $\xi_X: X \to \cC^{\dd,1}$ also coincides with these other scheme-theoretic images, and is independent of the choice of $\Akfree$.
\end{lemma} 
\begin{proof} 
	This follows from the various observations about Zariski density 
	made in Remark~\ref{rem:Zariski density}.
\end{proof}

\begin{defn}
	\label{def:scheme-theoretic images}
  We let $\overline{\cC}(\gM,\gN)$ denote the scheme-theoretic image  of
  $\xi^\dist:\Spec \Bdist\to \cC^{\dd,1}$, and
  we let $\overline{\cZ}(\gM,\gN)$ denote the
  scheme-theoretic image of the composite
  $\xi^\dist:\Spec \Bdist\to \cC^{\dd,1}\to \cZ^{\dd,1}$.
Equivalently,
$\overline{\cZ}(\gM,\gN)$ is the scheme-theoretic image of the
composite $\Spec \Bdist\to \cC^{\dd,1}\to \cR^{\dd,1}$ (\emph{cf}.\
\cite[Prop.\ 3.2.31]{EGstacktheoreticimages}), and the scheme-theoretic
image of $\overline{\cC}(\gM,\gN)$
under the morphism $\cC^{\dd,1} \to \cZ^{\dd,1}$.
(Note that Lemma~\ref{lem: scheme theoretic images coincide}
provides various other alternative descriptions of $\overline{\cC}(\gM,\gN)$
(and therefore also~$\overline{\cZ}(\gM,\gN$)) as 
a scheme-theoretic image.)
\end{defn}
\begin{rem}
  \label{rem: C(M,N) is reduced}Note that $\overline{\cC}
		(\gM,\gN)$
and $\overline{\cZ}(\gM,\gN)$
are both reduced (because they are each defined as a scheme-theoretic
		image of~$\Spec\Bdist$, which is reduced by definition).
\end{rem}

As well as scheme-theoretic images, as in the preceding Lemma and Definition,
we will need to consider images of underlying topological spaces.

\begin{lem}
	\label{lem:ext images}
	The image of the morphism on underlying topological spaces
	$| \Spec \Btwist | \to | \cC^{\dd,1}|$ induced by $\xi$ is
	a constructible subset of $| \cC^{\dd,1}|$, and is
	independent of the choice of $V$.
\end{lem}
\begin{proof}
	The fact that the image of $|\Spec \Btwist|$ is a constructible 
	subset of $|\cC^{\dd,1}|$ follows from the fact that
	$\xi$ is a morphism of finite presentation between Noetherian
        stacks; see~\cite[App.\ D]{MR2818725}. 
	Suppose now that $V'$ is another choice of finite rank projective
	$\F[x,x^{-1},y,y^{-1}]$-module surjecting onto~$T$.  Choose
	a finite rank projective module surjecting onto each of 
	$V$ and $V'$, compatible with the given surjections of each these
	sheaves onto $T$. (E.g.\ one could take $W = V \oplus V'$.)  
	Thus it suffices to prove the independence claim of the lemma
	in the case when $V'$ admits a surjection onto $V$ (compatible
	with the maps of each of $V$ and $V'$ onto $T$).
	If we write $B' := \F[x^{\pm 1}, y^{\pm 1}][(V')^{\vee}],$
	then the natural morphism $\Spec B' \to \Spec \Btwist$ is a surjection,
	and the morphism $\xi': \Spec B' \to \cC^{\dd,1}$ is the composite
	of this surjection with the morphism $\xi$.  Thus indeed
	the images of $|\Spec B'|$ and of $|\Spec \Btwist|$ coincide as
	subsets of $|\cC^{\dd,1}|$.
\end{proof}

\begin{df}
	\label{df:constructible images}
	We write $|\cC(\gM,\gN)|$ to denote the constructible subset
	of $|\cC^{\dd,1}|$ described in Lemma~\ref{lem:ext images}.
\end{df}

\begin{remark}
	We caution the reader that we don't define a substack $\cC(\gM,\gN)$
	of $\cC^{\dd,1}$. Rather, we have defined a closed substack
	$\overline{\cC}(\gM,\gN)$ of $\cC^{\dd,1}$, and a constructible subset
	$|\cC(\gM,\gN)|$ of $|\cC^{\dd,1}|$.  It follows from
	the constructions that $|\overline{\cC}(\gM,\gN)|$ is the
	closure in $|\cC^{\dd,1}|$ of $|\cC(\gM,\gN)|$.
\end{remark}

As in Subsection~\ref{subsec:families of extensions},
there is a natural action of $\Gm\times_{\F}\Gm$ on $T$,
and hence on each of $\Spec \Bdist$, $\Spec \Bkfree$ and~$X$,
given by the action of $\Gm$ as automorphisms on each of $\gM_{\Gm\times_{\F} \Gm,x}$ 
and $\gN_{\Gm\times_{\F} \Gm,y}$ (which induces a corresponding action on $T$,
hence on $T_\Adist$ and $T_\Akfree$, and hence on $\Spec \Bdist$ and
$\Spec \Bkfree$).   
Thus we may form the corresponding quotient stacks $[\Spec \Bdist / \Gm\times_{\F} \Gm]$ and  
$[X / \Gm\times_{\F} \Gm],$ each of which admits a natural morphism to~$\cC^{\dd,1}$. 

\begin{rem}
  \label{rem: potential confusion of two lots of Gm times Gm}Note that
  we are making use of two independent copies of $\Gm\times_{\F}\Gm$; one
  parameterises the different unramified twists of $\gM$ and $\gN$, and the other the
  automorphisms of (the pullbacks of) $\gM$ and $\gN$. 
\end{rem}
\begin{defn}
  \label{defn: strict situations}We say that the pair $(\gM,\gN)$ is
  \emph{strict} if $\Spec\Adist=\Gm\times_\F\Gm$.
\end{defn}

Before stating and proving the main result of this subsection,
we prove some lemmas (the first two of which amount to recollections
of standard --- and simple --- facts).

\begin{lem}
  \label{lem: fibre products of finite type}If $\cX\to\cY$ is a
  morphism of stacks over~$S$, with $\cX$ algebraic and of
  finite type over~$S$, and $\cY$ having diagonal which is
  representable by algebraic spaces and of finite type, then
  $\cX\times_{\cY}\cX$ is an algebraic stack of finite type
  over~$S$. 
\end{lem}
\begin{proof}
	The fact that $\cX\times_{\cY} \cX$ is an algebraic 
	stack follows
	from~\cite[\href{http://stacks.math.columbia.edu/tag/04TF}{Tag 04TF}]{stacks-project}.
  Since composites of morphisms of finite type are of finite type, in
  order to show that $\cX\times_{\cY} \cX$ is of finite type over $S$,
  it suffices to show that the natural morphism $\cX\times_{\cY} \cX
  \to \cX\times_S \cX$ is of finite type.  Since this
  morphism is the base-change of the diagonal morphism $\cY \to \cY\times_S \cY,$
  this follows by assumption.
\end{proof}

\begin{lem}\label{lem:fibres of kExt}
The following conditions are equivalent:\\
(1) $\kExt^1_{\K{\kappa(\mathfrak m)}}\bigl( \gM_{\kappa(\mathfrak m),\xbar},
\gN_{\kappa(\mathfrak m),\ybar}\bigr)= 0$
for all maximal ideals $\mathfrak m$ of $\Akfree$.\\
(2) $U_\Akfree = 0$.\\
(3) $\Spec C^{\kfree}$ is the trivial
vector bundle over $\Spec A^{\kfree}$. 
\end{lem}
\begin{proof}
  Conditions~(2) and~(3) are equivalent by definition. Since the formation of
$\kExt^1_{\K{\Akfree}}(\gM_{\Akfree,x},\gN_{\Akfree,y})$ is compatible with base change,
and since $\Akfree$ is Jacobson, (1) is equivalent to
the assumption that
$$\kExt^1_{\K{\Akfree}}(\gM_{\Akfree,x},\gN_{\Akfree,y})=0,$$
i.e.\ that~$U_\Akfree=0$, as required.
\end{proof}

\begin{lemma}
	\label{lem:C to R, with maps to Akfree}
If the equivalent conditions of Lemma~{\em \ref{lem:fibres of kExt}} hold,
	then the natural morphism
	\begin{multline*}
	\Spec B^{\kfree} \times_{\Spec A^{\kfree} \times_\F \cC^{\dd,1}}
		\Spec B^{\kfree}
		\\
		\to
	\Spec B^{\kfree} \times_{\Spec A^{\kfree} \times_\F \cR^{\dd,1}}
		\Spec B^{\kfree}
		\end{multline*}
		is an isomorphism.
\end{lemma}
\begin{proof}
  Since $\cC^{\dd,1} \to \cR^{\dd,1}$ is separated (being
  proper) and representable, the diagonal morphism
  $\cC^{\dd,1} \to \cC^{\dd,1}\times_{\cR^{\dd,1}} \cC^{\dd,1}$
  is a closed immersion, and hence the morphism in the statement
  of the lemma is
  a closed immersion.   Thus, in order to show that it is an
  isomorphism, it suffices to show that it induces a surjection
  on $R$-valued points, for any $\F$-algebra $R$.  Since 
the source and target are of finite type
over $\F$, by Lemma~\ref{lem: fibre products of finite type}, 
we may in fact restrict attention to finite type $R$-algebras.

A morphism $\Spec R\to \Spec\Bkfree
  \times_{\Spec A^{\kfree} \times_{\F} \cC^{\dd,1}}\Spec\Bkfree$
  corresponds to an isomorphism class
of tuples $(\alpha,\beta:\gE\to\gE',\iota,\iota',\pi,\pi')$, where
\begin{itemize}
\item $\alpha$ is a  morphism
$\alpha:\Spec R\to\Spec \Akfree$,
\item  $\beta:\gE\to\gE'$ is an isomorphism of Breuil--Kisin modules
with descent data and coefficients in $R$, 
\item $\iota:\alpha^* \gN \to \gE$, $\iota':\alpha^* \gN \to \gE'$, $\pi:\gE \to
  \alpha^* \gM$ and $\pi':\gE' \to
  \alpha^* \gM$ are morphisms
with the properties that $0 \to \alpha^*\gN \buildrel \iota
\over \to \gE \buildrel \pi  \over \to \alpha^* \gM \to 0$ and  $0 \to \alpha^*\gN \buildrel \iota'
\over \to \gE' \buildrel \pi'  \over \to \alpha^* \gM \to 0$ are both
short exact. 
\end{itemize}

Similarly,
a morphism $\Spec R\to \Spec\Bkfree
  \times_{\Spec A^{\kfree} \times_{\F} \cR^{\dd,1}}\Spec\Bkfree$
  corresponds to an isomorphism class
of tuples $(\alpha,\gE,\gE',\beta,\iota,\iota',\pi,\pi')$, where
\begin{itemize}
\item $\alpha$ is a  morphism
$\alpha:\Spec R\to\Spec \Akfree$,
\item  $\gE$ and $\gE'$ are Breuil--Kisin modules
with descent data and coefficients in $R$, 
and $\beta$ is an isomorphism
 $\beta:\gE[1/u]\to\gE'[1/u]$ of etale $\varphi$-modules with
 descent data and coefficients in $R$,
\item $\iota:\alpha^* \gN \to \gE$, $\iota':\alpha^* \gN \to \gE'$, $\pi:\gE \to
  \alpha^* \gM$ and $\pi':\gE' \to
  \alpha^* \gM$ are morphisms
with the properties that $0 \to \alpha^*\gN \buildrel \iota
\over \to \gE \buildrel \pi  \over \to \alpha^* \gM \to 0$ and  $0 \to \alpha^*\gN \buildrel \iota'
\over \to \gE' \buildrel \pi'  \over \to \alpha^* \gM \to 0$ are both
short exact. 
\end{itemize}

Thus to prove the claimed surjectivity, we have to show that, given
a tuple $(\alpha,\gE,\gE',\beta,\iota,\iota',\pi,\pi')$ associated
to a morphism $\Spec R\to \Spec\Bkfree
  \times_{\Spec A^{\kfree} \times_{\F} \cR^{\dd,1}}\Spec\Bkfree$,
  the isomorphism $\beta$ restricts to an isomorphism
  $\gE \to \gE'.$ 

By 
Lemma~\ref{lem:fibres of kExt}, the
  natural map
  $\Ext^1(\alpha^*\gM,\alpha^*\gN)\to\Ext^1_{\K{R}}(\alpha^*\gM[1/u],\alpha^*\gN[1/u])$
  is injective; so the Breuil--Kisin modules $\gE$ and $\gE'$ are isomorphic. Arguing as in the proof of Corollary~\ref{cor:
    monomorphism to Spec A times C}, we see that~$\beta$ is equivalent
  to the data of an $R$-point of~$\Gm\times_{\cO}\Gm$,
  corresponding to the automorphisms of $\alpha^*\gM[1/u]$ and
  $\alpha^*\gN[1/u]$ that it induces. These restrict to
  automorphisms of $\alpha^*\gM$ and $\alpha^*\gN$, so that
   (again by the proof of Corollary~\ref{cor:
    monomorphism to Spec A times C}) 
  $\beta$ indeed restricts to an
  isomorphism $\gE\to\gE'$, as required.
\end{proof}
	
We now present the main result of this subsection.

\begin{prop}\label{prop: construction of family monomorphing to C and R}
{\em (1)} The morphism
			$\xi^{\dist}$ induces a morphism 
\numequation
\label{eqn:unramified morphism}
[\Spec \Bdist / \Gm \times_{\F} \Gm ] \to
\cC^{\dd,1},
\end{equation}which is representable by algebraic spaces, of finite type,
			and unramified,
whose fibres over finite type points  are of degree $\leq 2$. 
In the strict case, this induced morphism is in fact a monomorphism,
while in general, the restriction $\xi_X$ of $\xi^{\dist}$
induces a finite type monomorphism 
  \numequation
  \label{eqn:monomorphism}
[X / \Gm \times_{\F} \Gm ] \hookrightarrow
\cC^{\dd,1}.
\end{equation}

{\em (2)}
If
$\kExt^1_{\K{\kappa(\mathfrak m)}}\bigl( \gM_{\kappa(\mathfrak m),\xbar},
\gN_{\kappa(\mathfrak m),\ybar}\bigr)=0$
for all maximal ideals $\mathfrak m$ of $\Akfree$, 
then the composite morphism
\numequation
\label{eqn:second unramified morphism}
[\Spec B^{\kfree}/\Gm\times_\F\Gm]\to \cC^{\dd,1}\to\cR^{\dd,1}
\end{equation}
is a representable by algebraic spaces, of finite type,
and unramified,
with fibres of degree $\leq 2.$
In the strict case, this induced morphism is in fact a monomorphism,
while in general,
the composite morphism
\numequation
\label{eqn:second monomorphism}
[X/\Gm\times_\F\Gm]\hookrightarrow \cC^{\dd,1}\to\cR^{\dd,1}
\end{equation}
is a finite type monomorphism. 
\end{prop} 

\begin{remark}\label{rem:explain-hypotheses}
The failure of~\eqref{eqn:unramified morphism} to be a monomorphism in
general is due, effectively,  to the possibility that an extension $\gE$ of some 
$\gM_{R,x}$ by $\gN_{R,y}$ and an extension $\gE'$ of some
$\gM_{R,x'}$ by $\gN_{R,y'}$ might be isomorphic as
Breuil--Kisin modules while nevertheless $(x,y)\neq (x',y')$. As we
will see in the proof,
whenever this happens the map $\gN_{\Lambda,y} \to \gE\to \gE' 
\to \gM_{\Lambda,x'}$ is nonzero, and then
$\gE' \otimes_R \kappa(\frakm)[1/u]$ is split for some maximal ideal $\frakm$
of $R$. This explains why, to obtain a monomorphism,
we can restrict either to the strict case or to the substack of extensions
that are non-split after inverting $u$.
\end{remark}

\begin{remark}
	We have stated this proposition in the strongest form that we
	are able to prove, but in fact its full strength is not required
	in the subsequent applications. 
              In particular, we don't need the precise bounds on the
	degrees of the fibres. 
\end{remark}
\begin{proof}[Proof of Proposition~{\ref{prop:
		construction of family monomorphing to C and R}}]
	By Corollary~\ref{cor: monomorphism to Spec A times C}
  (which we can apply because Assumption~\ref{assumption:vanishing} is
  satisfied, by Lemma~\ref{lem: generically no Homs})
  the natural morphism $[\Spec \Bdist/ \Gm \times_{\F} \Gm ] \to
  \Spec \Adist\times_{\F} \cC^{\dd,1}$ is a finite type monomorphism, 
  and hence so is its restriction to the open substack
  $[X/\Gm\times_{\F} \Gm]$ of its source.

Let us momentarily write $\cX$ to denote either $[\Spec B^{\dist}/
\Gm\times_{\F} \Gm]$  or $[X/\Gm\times_{\F} \Gm]$.  To show that
the finite type morphism
$\cX \to \cC^{\dd,1}$ is representable by algebraic spaces,
resp.\ unramified, resp.\ a
monomorphism,
it suffices to show that the corresponding diagonal morphism
$\cX \to \cX \times_{\cC^{\dd,1}} \cX$ is a monomorphism, resp.\ \'etale, resp.\ 
an isomorphism.

Now since $\cX \to \Spec A^{\dist} \times_{\F} \cC^{\dd,1}$ is a monomorphism,
the diagonal morphism $\cX \to \cX \times_{\Spec A^{\dist}\times_{\F} \cC^{\dd,1}}
\cX$ {\em is} an isomorphism, 
and so it is equivalent to show that the morphism of products
$$\cX \times_{\Spec A^{\dist}\times_{\F} \cC^{\dd,1}} \cX \to
\cX\times_{\cC^{\dd,1}} \cX$$
is a monomorphism, resp.\ \'etale, resp.\ an isomorphism.
This is in turn equivalent to showing the corresponding properties
for the morphisms
\numequation
\label{eqn:first closed immersion}
  \Spec\Bdist\times_{\Spec \Adist\times \cC^{\dd,1}}\Spec\Bdist \to
 \Spec\Bdist \times_{\cC^{\dd,1}}\Spec\Bdist 
 \end{equation}
 or
\numequation
\label{eqn:second closed immersion}
  X \times_{\Spec \Adist\times \cC^{\dd,1}}X \to
 X \times_{\cC^{\dd,1}} X.
 \end{equation}
 Now each of these morphisms is a base-change of the diagonal
 $\Spec A^{\dist}\to \Spec A^{\dist} \times_{\F} \Spec A^{\dist},$
 which is a closed immersion (affine schemes being separated),
 and so is itself a closed immersion.   In particular,
 it is a monomorphism, and so we have proved the representability
 by algebraic spaces
 of each of~(\ref{eqn:unramified morphism}) and~(\ref{eqn:monomorphism}).
Since the source and target of each of these monomorphisms
is of finite type
over~$\F$, by Lemma~\ref{lem: fibre products of finite type}, 
in order to show that either of these monomorphisms is
an isomorphism,
 it suffices to show that it induces a surjection on 
$R$-valued points, for arbitrary finite type $\F$-algebras $R$.
Similarly, to check that the
closed immersion~(\ref{eqn:first closed immersion}) is \'etale,
it suffices to verify that it is formally smooth,
and for this it suffices to verify that it satisfies the
infinitesimal lifting property 
with respect to square zero thickenings of finite type
$\F$-algebras.

A morphism $\Spec R\to \Spec\Bdist
  \times_{\cC^{\dd,1}}\Spec\Bdist$ corresponds to an isomorphism class
of tuples $(\alpha,\alpha',\beta:\gE\to\gE',\iota,\iota',\pi,\pi')$, where
\begin{itemize}
\item $\alpha,\alpha'$ are morphisms
$\alpha,\alpha':\Spec R\to\Spec \Adist$,
\item  $\beta:\gE\to\gE'$ is an isomorphism of Breuil--Kisin modules
with descent data and coefficients in $R$, 
\item $\iota:\alpha^* \gN \to \gE$, $\iota':(\alpha')^* \gN \to \gE'$, $\pi:\gE \to
  \alpha^* \gM$ and $\pi':\gE' \to
  (\alpha')^* \gM$ are morphisms
with the properties that $0 \to \alpha^*\gN \buildrel \iota
\over \to \gE \buildrel \pi  \over \to \alpha^* \gM \to 0$ and  $0 \to (\alpha')^*\gN \buildrel \iota'
\over \to \gE' \buildrel \pi'  \over \to (\alpha')^* \gM \to 0$ are both
short exact. 
\end{itemize}

We begin by proving that~(\ref{eqn:first closed immersion}) satisfies
the infinitesimal lifting criterion (when $R$ is a finite type $\F$-algebra).
Thus we assume given a square-zero ideal $I \subset R$, 
such that the induced morphism
$$\Spec R/I \to \Spec B^{\dist} \times_{\cC^{\dd,1}} \Spec B^{\dist}$$
factors through $\Spec B^{\dist}\times_{\Spec A^{\dist} \times_{\F} \cC^{\dd,1}} 
\Spec B^{\dist}$. In terms of the data
$(\alpha,\alpha',\beta:\gE\to\gE',\iota,\iota',\pi,\pi')$, 
we are assuming that $\alpha$ and $\alpha'$ coincide when restricted
to~$\Spec R/I$, and 
we must show that $\alpha$ and $\alpha'$ themselves coincide.

To this end, we consider the composite
\numequation
\label{eqn:key composite}
\alpha^*\gN\stackrel{\iota}{\to}\gE\stackrel{\beta}{\to}\gE'\stackrel{\pi'}{\to}(\alpha')^*\gM. 
\end{equation}
If we can show the vanishing of this morphism,
then by reversing the roles of $\gE$ and $\gE'$,
we will similarly deduce the vanishing of
$\pi \circ \beta^{-1} \circ \iota'$,  
from which we can conclude that $\beta$ induces an isomorphism between
$\alpha^*\gN$ and $(\alpha')^*\gN$. Consequently, it also induces an
isomorphism between~$\alpha^*\gM$ and~$(\alpha')^*\gM$, so it follows from
Lemma~\ref{lem: isomorphic twists are the same twist} that  $\alpha=\alpha'$,
as required. 

We show the vanishing of~(\ref{eqn:key composite}).
Suppose to the contrary that it doesn't vanish,
so that we have a non-zero morphism
$\alpha^*\gN\to (\alpha')^*\gM.$
It follows from Proposition~\ref{prop:vanishing of homs} that,
for some maximal ideal $\m$ of $R$, there exists a non-zero morphism
\[\alpha^*(\gN)\otimes_R\kappa(\mathfrak m) 
  {\to}(\alpha')^*(\gM)\otimes_R\kappa(\mathfrak
  m).\] 
By assumption $\alpha$ and $\alpha'$ coincide modulo $I$.  Since $I^2 = 0$,
there is an inclusion $I \subset \mathfrak m$,
and so in particular we find that
$$(\alpha')^*(\gM) \otimes_R \kappa(\mathfrak m)
\iso \alpha^*(\gM)\otimes_R \kappa(\mathfrak m).$$
Thus there exists a non-zero morphism
\[\alpha^*(\gN)\otimes_R\kappa(\mathfrak m) 
  {\to}\alpha^*(\gM)\otimes_R\kappa(\mathfrak m).\] 
Then, by Lemma~\ref{lem:rank one isomorphism over a field},
  after inverting~$u$ we obtain an isomorphism
\[\alpha^*(\gN)\otimes_R\kappa(\mathfrak m) [1/u]
  {\iso}\alpha^*(\gM)\otimes_R\kappa(\mathfrak m)[1/u],\] 
contradicting the assumption that $\alpha$ maps $\Spec R$
into $\Spec A^{\dist}$.
This completes the proof that~(\ref{eqn:first closed immersion})
is formally smooth, and hence that~(\ref{eqn:unramified morphism})
is unramified.

We next show that, in the strict case,
the closed immersion~(\ref{eqn:first closed immersion}) 
is an isomorphism, and thus that~(\ref{eqn:unramified morphism})
is actually a monomorphism. As noted above, 
it suffices to show that~(\ref{eqn:first closed immersion})
induces a surjection on $R$-valued points for finite type $\F$-algebras $R$,
which in terms of the data
$(\alpha,\alpha',\beta:\gE\to\gE',\iota,\iota',\pi,\pi')$, 
amounts to showing that necessarily $\alpha = \alpha'$.
Arguing just as we did above,
it suffices show the vanishing of~(\ref{eqn:key composite}).

Again, we suppose for the sake of contradiction that~(\ref{eqn:key composite})
does not vanish. It then follows
from Proposition~\ref{prop:vanishing of homs} that
for some maximal ideal $\m$ of $R$ there exists a non-zero
morphism \[\alpha^*(\gN)\otimes_R\kappa(\mathfrak m) 
  {\to}(\alpha')^*(\gM)\otimes_R\kappa(\mathfrak
  m).\] 
Then, by Lemma~\ref{lem:rank one isomorphism over a field},
  after inverting~$u$ we obtain an isomorphism
  \numequation
  \label{eqn:key isomorphism}
  \alpha^*(\gN)\otimes_R\kappa(\mathfrak
  m)[1/u]
  \iso (\alpha')^*(\gM)\otimes_R\kappa(\mathfrak m)[1/u].
  \end{equation}
In the strict case, such an isomorphism cannot exist by assumption,
and thus~(\ref{eqn:key composite}) must vanish.

We now turn to proving that~(\ref{eqn:second closed immersion}) 
is an isomorphism. Just as in the preceding arguments,
it suffices to show that~(\ref{eqn:key composite}) vanishes, and
if not 
then we obtain an isomorphism~(\ref{eqn:key isomorphism}).
Since 
we are considering points of $X\times X$,
we are given that the
induced extension $\gE'\otimes_R\kappa(\mathfrak m)[1/u]$ is non-split,
so that the base change of the morphism~(\ref{eqn:key composite})
from $R[[u]]$ to $\kappa(\mathfrak m)((u))$ must vanish.  Consequently
the composite $\beta\circ \iota$ induces a non-zero morphism
$\alpha^*(\gN)\otimes_R\kappa(\mathfrak m)[1/u] \to (\alpha')^*(\gN)
\otimes_R\kappa(\mathfrak m)[1/u],$
which, by Lemma~\ref{lem:rank one isomorphism over a field},
must in fact be an isomorphism.  Comparing this isomorphism
with the isomorphism~(\ref{eqn:key isomorphism}),
we find that
$(\alpha')^*(\gN)\otimes_R\kappa(\mathfrak m)[1/u]$
and
$(\alpha')^*(\gM)\otimes_R\kappa(\mathfrak m)[1/u]$
are isomorphic, contradicting the fact that $\alpha'$ maps 
$\Spec R$ to $\Spec A^{\dist}$.
Thus in fact the composite~(\ref{eqn:key composite}) must vanish,
and we have completed the proof that~(\ref{eqn:monomorphism})
is a monomorphism.

To complete the proof of part~(1) of the proposition,
we have to show that the fibres of~(\ref{eqn:unramified morphism})
are of degree at most $2$. We have already observed that $[\Spec \Bdist/ \Gm \times_{\F} \Gm ] \to
  \Spec \Adist\times_{\F} \cC^{\dd,1}$ is a monomorphism, so it is
  enough to check that given a finite extension~$\F'/\F$ and an
  isomorphism class of tuples $(\alpha,\alpha',\beta:\gE\to\gE',\iota,\iota',\pi,\pi')$, where
\begin{itemize}
\item $\alpha,\alpha'$ are distinct morphisms
$\alpha,\alpha':\Spec \F'\to\Spec \Adist$,
\item  $\beta:\gE\to\gE'$ is an isomorphism of Breuil--Kisin modules
with descent data and coefficients in $\F'$, 
\item $\iota:\alpha^* \gN \to \gE$, $\iota':(\alpha')^* \gN \to \gE'$, $\pi:\gE \to
  \alpha^* \gM$ and $\pi':\gE' \to
  (\alpha')^* \gM$ are morphisms
with the properties that $0 \to \alpha^*\gN \buildrel \iota
\over \to \gE \buildrel \pi  \over \to \alpha^* \gM \to 0$ and  $0 \to (\alpha')^*\gN \buildrel \iota'
\over \to \gE' \buildrel \pi'  \over \to (\alpha')^* \gM \to 0$ are both
short exact. 
\end{itemize}
then~$\alpha'$ is determined by the data of~$\alpha$ and~$\gE$. To see
this, note that since we are assuming that~$\alpha'\ne\alpha$, the
arguments above show that~(\ref{eqn:key composite}) does not vanish,
so that (since~$\F'$ is a field), we have an
isomorphism~$\alpha^*\gN[1/u]\isoto(\alpha')^*\gM[1/u]$. Since we are
over~$\Adist$, it follows that~$\gE[1/u]\cong\gE'[1/u]$ is split, and
that we also have an
isomorphism~$\alpha^*\gM[1/u]\isoto(\alpha')^*\gN[1/u]$. Thus
if~$\alpha''$ is another possible choice for~$\alpha'$, we have
$(\alpha'')^*\gM[1/u]\isoto(\alpha')^*\gM[1/u]$ and
$(\alpha'')^*\gN[1/u]\isoto(\alpha')^*\gN[1/u]$,
whence~$\alpha''=\alpha'$ by Lemma~\ref{lem: isomorphic twists are the
  same twist}, as required.

We turn to proving~(2), and thus
assume that
$$\kExt^1_{\K{\kappa(\mathfrak m)}}\bigl( \gM_{\kappa(\mathfrak m),\xbar},
\gN_{\kappa(\mathfrak m),\ybar}\bigr)=0$$
for all maximal ideals $\mathfrak m$ of $\Akfree$.

Lemma~\ref{lem:C to R, with maps to Akfree} shows that
$$\Spec B^{\kfree} \times_{\Spec A^{\kfree}\times_{\F}
	\cC^{\dd,1}}\Spec B^{\kfree} \to
\Spec B^{\kfree} \times_{\Spec A^{\kfree}\times_{\F}
	\cR^{\dd,1}}\Spec B^{\kfree}$$
is an isomorphism, from which we deduce
that
$$[\Spec B^{\kfree}/\Gm\times_{\F} \Gm] \to \Spec A^{\kfree}\times_{\F}
\cR^{\dd,1}$$
is a monomorphism.
Using this as input, the claims of~(2) may be proved in an essentially identical
fashion to those of~(1). 
\end{proof}

 \begin{cor} 
  \label{cor: dimension of families of extensions}
The dimension of $\overline{\cC}(\gM,\gN)$ 
is equal to 
the rank of $T_\Adist$ as a projective $\Adist$-module.  
If 
$$\kExt^1_{\K{\kappa(\mathfrak m)}}\bigl( \gM_{\kappa(\mathfrak m),\xbar},
\gN_{\kappa(\mathfrak m),\ybar}\bigr)=0$$
for all maximal ideals $\mathfrak m$ of $A^{\kfree}$, 
then the dimension of $\overline{\cZ}(\gM,\gN)$ is also equal to this rank,
while 
if
$$\kExt^1_{\K{\kappa(\mathfrak m)}}\bigl( \gM_{\kappa(\mathfrak m),\xbar},
\gN_{\kappa(\mathfrak m),\ybar}\bigr) \neq 0$$
for all maximal ideals $\mathfrak m$ of $A^{\kfree}$, 
then the dimension of $\overline{\cZ}(\gM,\gN)$ is strictly less than this rank.
\end{cor}
\begin{proof} The dimension of $[\Spec\Bdist/\Gm\times_{\F} \Gm]$ is equal to
	the rank of $T_{A^{\dist}}$ (it is the quotient
	by a two-dimensional group of a vector bundle over a two-dimensional base of rank
	equal to the rank of $T_{A^{\dist}}$). By Lemma~\ref{lem: scheme theoretic images coincide},
	$\overline{\cC}(\gM,\gN)$ is the scheme-theoretic image
	of the morphism
	$[\Spec\Bdist/\Gm\times_{\F}\Gm] \to \cC^{\dd,1}$
	provided by
	Proposition~\ref{prop: construction of family monomorphing to
          C and R}(1), which (by that proposition) is representable 
  by algebraic spaces and unramified.
  Since such a morphism is locally quasi-finite
  (in fact, in this particular case, 
  we have shown that the fibres of this morphism have degree at
        most~$2$), \cite[\href{https://stacks.math.columbia.edu/tag/0DS6}{Tag 0DS6}]{stacks-project}
 ensures
	that $\overline{\cC}(\gM,\gN)$ has the claimed dimension.

	If 
$\kExt^1_{\K{\kappa(\mathfrak m)}}\bigl( \gM_{\kappa(\mathfrak m),\xbar},
\gN_{\kappa(\mathfrak m),\ybar}\bigr) = 0$
for all maximal ideals $\mathfrak m$ of $A^{\kfree}$, 
then an identical argument using Proposition~\ref{prop: construction of family monomorphing to
          C and R}(2) implies the claim regarding the dimension
of $\overline{\cZ}(\gM,\gN)$.

Finally, suppose that 
$$\kExt^1_{\K{\kappa(\mathfrak m)}}\bigl( \gM_{\kappa(\mathfrak m),\xbar},
\gN_{\kappa(\mathfrak m),\ybar}\bigr) \neq 0$$
for all maximal ideals $\mathfrak m$ of $A^{\kfree}$. 
Then the composite $[\Spec\Bkfree/\Gm\times_{\F} \Gm] \to \cC^{\dd,1} \to \cR^{\dd,1}$
has the property that for every point $t$ in the source, the fibre over the
image of $t$ has a positive dimensional fibre.  \cite[\href{https://stacks.math.columbia.edu/tag/0DS6}{Tag 0DS6}]{stacks-project} then implies the remaining
	claim of the present lemma. 
\end{proof}

\subsection{Rank one modules over finite fields, and their extensions}
\label{subsec: Diamond--Savitt}

We now wish to apply the results of the previous subsections to study
the geometry of our various moduli stacks. In order to do this, it
will be convenient for us to have an explicit description of the
rank one Breuil--Kisin modules of height at most one with descent data over a
finite field of characteristic $p$, and of their possible extensions. 
Many of the results in this section are proved (for $p>2$)  in~\cite[\S 1]{DiamondSavitt} in the context of
Breuil modules, and in those cases it
is possible simply to translate the relevant statements to the Breuil--Kisin module context.

Assume from now on that $e(K'/K)$ is divisible by $p^{f}-1$, so
  that we are in the setting of~\cite[Remark 1.7]{DiamondSavitt}.
  (Note that the parallel in \cite{DiamondSavitt} of our field
  extension $K'/K$, with ramification and inertial indices $e',f'$ and
  $e,f$ respectively, is the extension $K/L$ with indices $e,f$ and
  $e',f'$ respectively.) 

Let $\F$ be a finite subfield of $\Fpbar$ containing the image of some (so all)
embedding(s) $k'\into\Fpbar$. 
Recall that for each
$g\in\Gal(K'/K)$ we write $g(\pi')/\pi'=h(g)$ with $h(g)\in
\mu_{e(K'/K)}(K') \subset W(k')$. We abuse notation and
denote the image of $h(g)$ in $k'$ again by $h(g)$, so that we obtain
a map 
$\hchar \colon \Gal(K'/K) \to (k')^{\times}$. 
Note that~$\hchar$ restricts to a character on the inertia
subgroup $I(K'/K)$, and is itself a character when $e(K'/K) = p^f-1$.

\begin{lem}
  \label{lem:rank one Kisin modules with descent data}Every rank one
  Breuil--Kisin module of height at most one with descent data and $\F$-coefficients is isomorphic
  to one of the modules $\gM(r,a,c)$ defined by: 
  \begin{itemize}
  \item $\gM(r,a,c)_i=\F[[u]]\cdot m_i$,
  \item $\Phi_{\gM(r,a,c),i}(1\otimes m_{i-1})=a_{i} u^{r_{i}} m_{i}$,
  \item $\ghat(\sum_i m_i)=\sum_i h(g)^{c_i} m_i$ for all $g\in\Gal(K'/K)$, 
  \end{itemize}
where  $a_i\in\F^\times$,  $r_i\in\{0,\dots,e'\}$ and $c_i\in\Z/e(K'/K)$ are sequences
satisfying 
$pc_{i-1}\equiv c_{i}+r_{i}\pmod{e(K'/K)}$, the sums in the third
bullet point run from $0$ to $f'-1$, and the $r_i,c_i,a_i$ are periodic with
period dividing $f$. 

Furthermore, two such modules $\gM(r,a,c)$ and $\gM(s,b,d)$ are
isomorphic if and only if $r_i=s_i$ and $c_i=d_i$ for all $i$, and $\prod_{i=0}^{f-1}a_i=\prod_{i=0}^{f-1}b_i$.
\end{lem}
\begin{proof}
The proof is elementary; see e.g.\ \cite[Thm.~2.1,
Thm.~3.5]{SavittRaynaud} for proofs of analogous results.
\end{proof}

We will sometimes refer to the element $m = \sum_i m_i \in
\gM(r,a,c)$ as the standard generator of $\gM(r,a,c)$.

\begin{rem}
 When $p > 2$
many of the results in this section (such as the above) can
be obtained by translating \cite[Lem.\ 1.3, Cor.\
1.8]{DiamondSavitt} from the Breuil module context to the Breuil--Kisin module context.
We briefly recall the dictionary between these two categories
(\emph{cf.}\ \cite[\S 1.1.10]{kis04}). If $A$ is a finite local
$\Zp$-algebra, write $S_A = S \otimes_{\Zp} A$, where $S$ is Breuil's
ring. We regard $S_A$ as a $\gS_A$-algebra via
$u\mapsto u$, and we let $\varphi:\gS_A\to S_A$ be the composite of this map with
$\varphi$ on $\gS_A$. Then given a Breuil--Kisin module of height at most~$1$ with descent data $\gM$, 
we
set $\cM:=S_A\otimes_{\varphi,\gS_A}\gM$. We have a map $1\otimes\varphi_\gM:S_A\otimes_{\varphi,\gS_A}\gM\to S_A \otimes_{\gS_A}\gM$,
and we set \[\Fil^1\cM:=\{x\in\cM\ :\
(1\otimes \varphi_\gM)(x)\in\Fil^1S_A\otimes_{\gS_A}\gM\subset
S_A\otimes_{\gS_A}\gM\}\]and define $\varphi_1:\Fil^1\cM\to\cM$ as the
composite \[\Fil^1\cM\overset{1\otimes\varphi_\gM}{\longrightarrow}\Fil^1S_A\otimes_{\gS_A}\gM\overset{\varphi_1\otimes
  1}{\longrightarrow}S_A\otimes_{\varphi,\gS_A}\gM=\cM.\]Finally, we define $\hat{g}$ on $\cM$
via $\hat{g}(s\otimes m)=g(s)\otimes \hat{g}(m)$. One checks without difficulty
that this makes $\cM$ a strongly divisible module with descent data 
(\emph{cf.}\ the
proofs of~\cite[Proposition 1.1.11, Lemma 1.2.4]{kis04}).

  In the correspondence described above, the Breuil--Kisin module $\gM((r_i),(a_i),(c_i))$ corresponds to the Breuil module $\cM((e'-r_i),(a_i),(pc_{i-1}))$ 
of~\cite[Lem.\ 1.3]{DiamondSavitt}. 
\end{rem}

\begin{defn}
If $\gM = \gM(r,a,c)$ is a rank one Breuil--Kisin module as described in the 
preceding lemma, we set $\alpha_i(\gM)  :=   (p^{f'-1} r_{i-f'+1} + \cdots +
r_{i})/(p^{f'} - 1)$ (equivalently, $(p^{f-1} r_{i-f+1} +  \cdots 
+ r_{i})/(p^f-1)$). We may abbreviate $\alpha_i(\gM)$ simply as $\alpha_i$
when $\gM$ is clear from the context.

 It follows easily from the congruence $r_i 
\equiv pc_{i-1} - c_i \pmod{e(K'/K)}$ together with the hypothesis 
that $p^f-1 \mid e(K'/K)$  that $\alpha_i \in \Z$ for all $i$. Note
that the $\alpha_i$'s are the unique solution to the system of
equations $p \alpha_{i-1} - \alpha_i = r_i$ for all $i$. Note also
that $(p^f-1)(c_i-\alpha_i) \equiv 0 \pmod{e(K'/K)}$, so that
$\hchar^{c_i-\alpha_i}$ is a character with image in $k^{\times}$.
\end{defn}

\begin{lem}
  \label{lem: generic fibres of rank 1 Kisin modules}We have
  $T(\gM(r,a,c))=\left(\sigma_i\circ\hchar^{c_i-\alpha_{i}}\cdot\ur_{\prod_{i=0}^{f-1}a_i}\right)|_{G_{K_\infty}}$,
  where $\ur_\lambda$ is the unramified character of $G_K$ sending
  geometric Frobenius to $\lambda$. 
\end{lem}

\begin{proof}
 Set $\gN = \gM(0,(a_i),0)$, so that $\gN$ is effectively a Breuil--Kisin module without
 descent data. Then for $\gN$ this result follows from the second paragraph of the
 proof \cite[Lem.~6.3]{MR3164985}. (Note that the functor $T_{\gS}$ of
 \emph{loc.\ cit.} is dual to our functor $T$;\ \emph{cf}.~\cite[A\
 1.2.7]{MR1106901}. Note also that the fact that the base field is
 unramified in \emph{loc.\ cit.} does not change the calculation.) If
 $n = \sum n_i$ is the standard generator of $\gN$ as in Lemma~\ref{lem:rank one Kisin modules with descent data}, let
 $\gamma \in \Zp^{\un} \otimes_{\Zp} (k' \otimes_{\Fp} \F)$ be an element
 so that $\gamma  m \in
 (\cO_{\widehat{\cE^{\text{nr}}}}\otimes_{\gS[1/u]} \gN[1/u])^{\varphi
   = 1}$.

Now for $\gM$ as in the statement of the lemma it is straightforward to verify
that $$\gamma \sum_{i=0}^{f'-1} [\underline{\pi}']^{-\alpha_{i}} \otimes
m_i \in  (\cO_{\widehat{\cE^{\text{nr}}}}\otimes_{\gS[1/u]}
\gM[1/u])^{\varphi=1},$$ and the result follows.
\end{proof}

One immediately deduces the following.

\begin{cor}
  \label{cor: Kisin modules with the same generic fibre} Let $\gM=\gM(r,a,c)$ and
  $\gN=\gM(s,b,d)$ be rank one Breuil--Kisin modules with descent data as
  above.  We have $T(\gM)=T(\gN)$ if and only if $c_i - \alpha_i(\gM)
  \equiv d_i -  \alpha_i(\gN) \pmod{e(K'/K)}$ for some $i$ {\upshape(}hence for all $i${\upshape)} and $\prod_{i=0}^{f-1}a_i=\prod_{i=0}^{f-1}b_i$.
\end{cor}

\begin{lem}
  \label{lem: maps between rank 1 Kisin modules}  In the notation of the
  previous Corollary, there is a nonzero map $\gM\to\gN$
  \emph{(}equivalently, $\dim_{\F} \Hom_{\K{\F}}(\gM,\gN)=1$\emph{)} if
  and only if $T(\gM)=T(\gN)$ and $\alpha_i(\gM) \ge\alpha_i(\gN)$ for each $i$.
\end{lem}

\begin{proof}

  The proof is  essentially the same as that of \cite[Lem.\
  1.6]{DiamondSavitt}. (Indeed, when $p > 2$ this
  lemma can once again be proved by translating directly from
  \cite{DiamondSavitt} to the Breuil--Kisin module context.)
\end{proof}

 Using the material of Section~\ref{subsec:ext
  generalities}, 
one can compute $\Ext^1(\gM,\gN)$ for any pair of rank
one Breuil--Kisin modules $\gM,\gN$ of height at most one. We begin with the
following explicit description of the complex $C^{\bullet}(\gN)$ of  Section~\ref{subsec:ext
  generalities}.

\begin{defn}\label{notn:calh}
 We write  $\Czerofrac = \Czerofrac(\gM,\gN) \subset \F((u))^{\Z/f\Z}$ for 
the space of $f$-tuples $(\mu_i)$ such that each nonzero term of 
$\mu_i$ has degree congruent to $c_i - d_i \pmod{e(K'/K)}$, and set 
$\Czero = \Czerofrac \cap \F[[u]]^{\Z/f\Z}$. 

We further define $\Conefrac  = \Conefrac(\gM,\gN) \subset
\F((u))^{\Z/f\Z}$ to be
 the space of $f$-tuples $(h_i)$ such that each nonzero term of $h_i$
 has degree congruent to $r_i + c_i - d_i \pmod{e(K'/K)}$, and set
 $\Cone = \Conefrac \cap \F[[u]]^{\Z/f\Z}$.   There is a map $\mumap \colon \Czerofrac \to 
\Conefrac$ defined by
\[ \mumap(\mu_i) =  (-a_i u^{r_i} \mu_i + b_i \varphi(\mu_{i-1}) u^{s_i}) \] 
Evidently this restricts to a map
$\mumap \colon \Czero \to \Cone$.
\end{defn}

\begin{lemma}\label{lem:explicit-complex}
There is an isomorphism of complexes 
\[ [ \Czero \xrightarrow{\mumap} \Cone  ] \toisom C^{\bullet}(\gN)\]
in which $(\mu_i) \in \Czero$ is sent to the map $m_i \mapsto \mu_i n_i$
in $C^0(\gN)$, and $(h_i) \in \Cone$ is sent to the map $(1\otimes
m_{i-1}) \mapsto h_i n_i$ in $C^1(\gN)$.
\end{lemma}

\begin{proof}
 Each element of $\Hom_{\gS_{\F}}(\gM,\gN)$ has the form $m_i \mapsto \mu_i
 n_i$ for some $f'$-tuple $(\mu_i)_{i \in \Z/f'\Z}$ of elements of $\F[[u]]$.  
The condition that this map 
is $\Gal(K'/K)$-equivariant 
 is easily seen to be
equivalent to the conditions that $(\mu_i)$ is periodic with period dividing $f$, and
that each nonzero term of $\mu_i$ has degree congruent to 
$c_{i}-d_{i} \pmod{e(K'/K)}$. (For the former consider the action
of a lift to $g \in \Gal(K'/K)$ satisfying $h(g)=1$ of a generator of $\Gal(k'/k)$, and for the
latter consider the action of $I(K'/K)$;\ \emph{cf}.\ the proof of
\cite[Lem.~1.5]{DiamondSavitt}.)   It follows that the map $\Czero
\to C^0(\gN)$ in the statement of the Lemma is an isomorphism. An
essentially identical argument shows that the given map $\Cone \to
C^1(\gN)$ is an isomorphism. 

To conclude, it suffices to observe that if $\alpha \in C^0(\gN)$ is given by $m_i
\mapsto \mu_i n_i$ with $(\mu_i)_i \in \Czero$ then
$\delta(\alpha) \in C^1(\gN)$ is the map
given by $$(1\otimes m_{i-1}) \mapsto (-a_i u^{r_i} \mu_i + b_i
\varphi(\mu_{i-1}) u^{s_i}) n_i,$$ which follows by a direct calculation.
\end{proof}

It follows from Corollary~\ref{cor:complex computes Hom and Ext} 
that $\Ext^1_{\K{\F}}(\gM,\gN) \cong \coker \mumap$.
If $h \in \Cone$, we write $\gP(h)$ for the element of
 $\Ext^1_{\K{\F}}(\gM,\gN)$ represented by $h$ under this isomorphism.

\begin{remark}
  \label{prop: extensions of rank one Kisin modules}Let $\gM=\gM(r,a,c)$ and
  $\gN=\gM(s,b,d)$ be rank one Breuil--Kisin modules with descent data as in
Lemma~{\em \ref{lem:rank one Kisin modules with descent data}}. It
follows from the proof of Lemma~\ref{lem: C computes Ext^1}, and in
particular the description of the map~\eqref{eqn:explicit 
  embedding}  found there, that the extension $\gP(h)$  is given by
the formulas 
  \begin{itemize}
  \item  $\gP_i=\F[[u]]\cdot m_i + \F[[u]]\cdot n_i$,
  \item $\Phi_{\gP,i}(1\otimes n_{i-1})=b_{i} u^{s_{i}}n_{i}$,
    $\Phi_{\gP,i}(1\otimes m_{i-1})=a_{i}u^{r_{i}}m_{i}+h_{i}
    n_{i}$.
  \item $\ghat(\sum_i m_i)=\sum_i h(g)^{c_i}m_i$,
    $\ghat(\sum_i n_i)=h(g)^{d_i} \sum_i n_i$  for all $g\in\Gal(K'/K)$.
  \end{itemize}
From this description it is easy to see that the extension $\gP(h)$
has  height at most $1$ if and only if  each $h_i$ is divisible by $u^{r_i+s_i-e'}$.
\end{remark}

\begin{thm}\label{thm: extensions of rank one Kisin modules}
The dimension of $\Ext^1_{\K{\F}}(\gM,\gN)$ is given by the formula
\[\Delta+\sum_{i=0}^{f-1}\#\biggl\{j\in[0,r_i):j\equiv r_i+c_{i}-d_{i}\pmod{e(K'/K)}\biggr\} \]  
where $\Delta =  \dim_{\F} \Hom_{\K{\F}}(\gM,\gN)$ is $1$ if there is a nonzero map $\gM\to\gN$ and $0$
otherwise, while the subspace consisting of extensions of height at most $1$
has dimension
\[\Delta+\sum_{i=0}^{f-1}\#\biggl\{j\in[\max(0,r_i+s_i-e'),r_i):j\equiv r_i+c_{i}-d_{i}\pmod{e(K'/K)}\biggr\} .\] 
\end{thm}

\begin{proof}
  When $p > 2$, this result (for extensions of height at most $1$) can be obtained by translating
  \cite[Thm.~1.11]{DiamondSavitt} from Breuil modules to Breuil--Kisin
  modules. 
  We argue in the same spirit as \cite{DiamondSavitt} using  the generalities of Section~\ref{subsec:ext generalities}. 

Choose~$N$ as in
Lemma~\ref{lem:truncation argument used to prove f.g. of Ext Q
  version}(2). 
For brevity we
write $C^{\bullet}$ in lieu of $C^{\bullet}(\gN)$. We now use the
description of~$C^{\bullet}$ provided by
Lemma~\ref{lem:explicit-complex}. 
As we have noted, $C^0$ consists of the maps $m_i \mapsto \mu_i
n_i$ with $(\mu_i) \in \Czero$.
Since $(\varphi^*_{\gM})^{-1}(v^N C^1)$ contains precisely the maps $m_i
\mapsto \mu_i n_i$ in $C^0$ such that $v^{N} \mid u^{r_i} \mu_i$, we
compute that $\dim_{\F} C^0/\bigl((\varphi^*_{\gM})^{-1}(v^N C^1)\bigr)$
is the quantity
$$ Nf - \sum_{i=0}^{f-1} \#\biggl\{ j \in [e(K'/K) N-r_i, e(K'/K) N) : j \equiv c_i-d_i
  \pmod{e(K'/K)}\biggr\}.$$ We have
$\dim_{\F} C^1/v^NC^1 = Nf$, so
our formula for the dimension of $\Ext^1_{\K{\F}}(\gM,\gN)$ now follows
from Lemma~\ref{lem:truncation argument used to prove f.g. of Ext Q
  version}.
\end{proof}

\begin{remark}\label{rem:representatives-for-ext}
One can show exactly as in \cite{DiamondSavitt} that each element of $\Ext^1_{\K{\F}}(\gM,\gN)$ can be written uniquely in
the form $\gP(h)$ for $h \in \Cone$
with $\deg(h_i) < r_i$, except that when there exists a nonzero morphism
$\gM\to\gN$, the polynomials $h_i$ for $f \mid i$ may also have a term of degree
$\alpha_0(\gM)-\alpha_0(\gN)+r_0$ in common. Since we will not need
this fact we omit the proof.
\end{remark}

\subsection{Extensions of shape $J$}
\label{sec:extensions-shape-J}

We now begin the work of showing, for each non-scalar tame type $\tau$,
that $\cC^{\tau,\BT,1}$ has $2^f$ irreducible
components, indexed by the subsets $J$ of~$\{0,1,\dots,f-1\}$. We will
also   describe the irreducible
components of~$\cZ^{\tau,1}$. 
The proof of this hinges on examining the extensions considered in
Theorem~\ref{thm: extensions of rank one Kisin modules},
and then applying the results of Subsection~\ref{subsec:universal families}.
We will show that 
most of these families of extensions  have 
positive codimension in $\cC^{\tau,\BT,1}$, and are thus negligible from the
point of view of determining irreducible components.  By a base change
argument, we will also be able to show that we can neglect the irreducible
Breuil--Kisin modules. The rest of Section~\ref{sec: extensions of rank one Kisin modules} is devoted to establishing the
necessary bounds on the dimension of the various families of
extensions, and to studying the map from $\cC^{\tau,\BT,1}$ to
$\cR^{\dd,1}$. 

We now introduce notation that we will use for the remainder of the
paper. 
We fix a
tame inertial type $\tau=\eta\oplus\eta'$ with coefficients in $\Qpbar$.
 We allow the case of scalar
types (that is, the case $\eta=\eta'$). 
Assume that the subfield $\F$ of $\Fpbar$ is large enough so that the reductions modulo $\m_{\Zp}$ of $\eta$ and
$\eta'$ (which by abuse of notation we continue to denote $\eta,\eta'$) have image in $\F$. 
We also fix a uniformiser $\pi$ of~$K$. 

\begin{remark}\label{rk:ordering}
We stress that  when we 
write $\tau=\eta\oplus\eta'$, we are implicitly ordering 
$\eta,\eta'$. Much of the notation in this section depends on 
distinguishing $\eta,\eta'$, as do some of the constructions later in paper (in particular, the 
map to the Dieudonn\'e stack of Section~\ref{subsec: map to 
  Dieudonne stack}). 
\end{remark}

As in Subsection~\ref{subsec: Dieudonne stack}, we make
the following  ``standard choice'' for the extension~$K'/K$: if $\tau$ is a tame
principal series type, we take $K'=K(\pi^{1/(p^f-1)})$, while
if~$\tau$ is a tame cuspidal type, we let $L$ be an unramified
quadratic extension of~$K$, and set $K'=L(\pi^{1/(p^{2f}-1)})$. In
either case $K'/K$ is a Galois extension and $\eta, \eta'$ both factor through
$I(K'/K)$. In the principal series
case, we have $e'=(p^f-1)e$, $f'=f$, and in the cuspidal case we have
$e'=(p^{2f}-1)e$, $f'=2f$.   Either way, we have $e(K'/K) =
p^{f'}-1$.

In either case, it follows from Lemma~\ref{lem:rank one Kisin modules
  with descent data} that a Breuil--Kisin module of rank one with descent data
from $K'$ to $K$ is described by the data of the quantities $r_i,a_i,c_i$ for $0\le i\le
f-1$, and similarly from Lemma~\ref{lem:explicit-complex} that extensions between two such Breuil--Kisin modules are
described by the $h_i$ for $0\le i\le
f-1$. This common description will enable us to treat the principal
series and cuspidal cases largely in parallel.

The character
$\hchar |_{I_K}$ of Section~\ref{subsec: Diamond--Savitt}  is identified via the Artin map $\cO_L^\times \to 
I_L^{\ab} = I_K^{\ab} $ with the reduction map 
$\cO_L^{\times} \to (k')^{\times}$. Thus for each $\sigma \in 
\Hom(k',\Fpbar)$ the map $\sigma \circ 
\hchar|_{I_L}$ is the fundamental character $\omega_{\sigma}$ defined 
in Section~\ref{subsec: notation}. 
Define  $k_i,k'_i\in \Z/(p^{f'}-1)\Z$ for all $i$ by the formulas
$\eta=\sigma_i\circ\hchar^{k_i}|_{I(K'/K)}$ and
$\eta'=\sigma_i\circ\hchar^{k'_i}|_{I(K'/K)}$. In particular we have
$k_i=p^ik_0$, $k'_i=p^ik'_0$ for all $i$.

\begin{defn} Let $\gM = \gM(r,a,c)$ and $\gN=\gM(s,b,d)$ be Breuil--Kisin
  modules of rank one with descent data. We say that the pair
  $(\gM,\gN)$ has \emph{type $\tau$} provided that for all $i$:
  \begin{itemize}
  \item the multisets $\{c_i,d_i\}$ and $\{k_i,k'_i\}$ are equal, and
  \item $r_i + s_i = e'$.
  \end{itemize}
  \end{defn}

  \begin{lemma} The following are equivalent.
    \begin{enumerate}
    \item   The pair $(\gM,\gN)$ has type~$\tau$. 
    \item  Some element of $\Ext^1_{\K{\F}}(\gM,\gN)$ of height at
      most one  satisfies the strong determinant condition and is of type~$\tau$. 
  \item Every element of $\Ext^1_{\K{\F}}(\gM,\gN)$ has height at most
    one, 
satisfies the strong determinant condition, and is of type~$\tau$. 
    \end{enumerate}
(Accordingly, we will sometimes refer to the condition that
$r_i+s_i=e'$ for all~$i$ as the determinant condition.) 
  \end{lemma}

  \begin{proof}
    Suppose first that the pair $(\gM,\gN)$ has type $\tau$. The last
    sentence of Remark~\ref{prop: extensions of rank one Kisin
      modules} shows that every element of $\Ext^1_{\K{\F}}(\gM,\gN)$
    has height at most one.  Let $\gP$ be such an element. The
    condition on the multisets $\{c_i,d_i\}$ guarantees that $\gP$ has
    unmixed type $\tau$. By Proposition~\ref{prop:types are unmixed}
    we see that $\dim_{\F} (\im_{\gP,i}/E(u)\gP_i)_{\teta}$ is
    independent of $\teta$. From the condition that $r_i+s_i=e'$ we
    know that the sum over all $\teta$ of these dimensions is equal to
    $e'$; since they are all equal, each is equal to $e$, and
    Lemma~\ref{lem: determinant condition explicit finite field} tells
    us that $\gP$ satisfies the
    strong determinant condition. This proves that (1) implies (3). 

Certainly (3)
    implies (2), so it remains to check that (2) implies (1). Suppose
    that  
$\gP \in \Ext^1_{\K{\F}}(\gM,\gN)$ has height at most one, 
satisfies the strong determinant condition, and has type $\tau$. The
condition that $\{c_i,d_i\}=\{k_i,k'_i\}$ follows from $\gP$ having
type $\tau$, and the condition that $r_i+s_i=e'$ follows from the last
part of Lemma~\ref{lem: determinant condition explicit finite field}.
  \end{proof}

\begin{df}
\label{df:extensions of shape $J$}
If $(\gM,\gN)$ is a pair of type $\tau$ (resp.\ $\gP$ is an extension 
of type~$\tau$), we define the {\em shape} of $(\gM,\gN)$ (resp.\ of $\gP$) to
be the subset $J := \{ i  \, | \, c_i = k_i\} \subseteq \Z/f'\Z$,
unless $\tau$ is scalar, in which case we define the shape to be the
subset~$\varnothing$. 
(Equivalently, $J$ is in all cases the complement in
$\Z/f'\Z$ of the set  $\{i \, | \, c_i = k'_i\}.$)

Observe that in the cuspidal case the equality $c_i = c_{i+f}$ means
that $i \in J$ if and only if $i+f \not\in J$, so that the set $J$ is
determined by its intersection with any $f$ consecutive integers
modulo $f' = 2f$.  

In the cuspidal case we will say that a subset $J \subseteq \Z/f'\Z$ is a
  shape
if it satisfies $i \in J$ if and only if $i+f\not\in J$; in the
principal series case, we may refer to any subset $J \subseteq \Z/f'\Z$
as a shape.

We define the {\em refined shape} of the pair $(\gM,\gN)$ (resp.\ of $\gP$) to consist of its shape $J$,
together with the $f$-tuple of invariants $r:= (r_i)_{i = 0}^{f-1}$.
If $(J,r)$ is a refined shape that arises from some pair (or extension)
of type $\tau$, then we refer to $(J,r)$ as a refined shape for $\tau$.

We say the pair $(i-1,i)$ is a \emph{transition} for $J$ if $i-1 \in
J$, $i \not\in J$ or vice-versa. (In the first case we sometimes say
that the pair $(i-1,i)$ is a transition out of $J$, and in the latter case a transition
into $J$.) Implicit in many of our arguments below
is the observation that in the cuspidal case $(i-1,i)$ is a transition
if and only if $(i+f-1,i+f)$ is a transition.
\end{df}

\subsubsection{An explicit description of refined shapes}
\label{subsubsec:explicitly refined}
The refined shapes for $\tau$ admit an explicit description.
If $\gP$ is of shape $J$, for some fixed $J \subseteq \Z/f'\Z$
then, since $c_i$, $d_i$ are fixed, we see that the $r_i$ and $s_i$
appearing in $\gP$ are determined
modulo $e(K'/K)=p^{f'}-1$. Furthermore, we see that $r_i+s_i\equiv
0\pmod{p^{f'}-1}$, so that these values are consistent with the 
determinant condition; conversely, if we make any choice of the~$r_i$
in the given residue class modulo $(p^{f'}-1)$, then the $s_i$ are
determined by the determinant condition, and the imposed values
are consistent with the descent data. There are of course only
finitely many choices for the~$r_i$, and so there are only finitely 
many possible refined shapes for $\tau$.

To make this precise, recall that we have the congruence
$$r_i \equiv
pc_{i-1} - c_i \pmod{p^{f'}-1}.$$ We will write $[n]$ for the
least non-negative residue class of $n$ modulo $e(K'/K) =
p^{f'}-1$.

If both $i-1$ and $i$ lie in $J$,
then we have $c_{i-1} = k_{i-1}$ and  $c_i = k_i$. The first of these
implies that $pc_{i-1} = k_i$, and therefore $r_i \equiv 0
\pmod{p^{f'}-1}$. The same conclusion holds if neither $i-1$ and $i$
lie in $J$. Therefore if $(i-1,i)$ is not a transition we may write 
\[ r_i =
  (p^{f'}-1)y_i \quad \text{ and } \quad s_i = (p^{f'}-1)(e-y_i).\]  
with $0 \le y_i \le e$.

Now suppose instead that $(i-1,i)$ is a transition. (In
particular the type $\tau$ is not scalar.)  This
time $pc_{i-1} = d_i$ (instead of $pc_{i-1} = c_i$), so that $r_i
\equiv d_i - c_i \pmod{p^{f'}-1}$. In this case we write
 \[ r_i =
  (p^{f'}-1)y_i - [c_i-d_i] \quad \text{ and } \quad s_i = (p^{f'}-1)(e+1-y_i) -
  [d_i-c_i]\]
with $1 \le y_i \le e$.

Conversely, for fixed shape $J$ one checks that each choice of integers $y_i$ in the ranges
described above gives rise to a refined shape for $\tau$.

If $(i-1,i)$ is not a transition 
and $(h_i) \in \Conefrac(\gM,\gN)$ then non-zero terms of $h_i$ have
degree congruent to $r_i + c_i - d_i \equiv c_i - d_i \pmod{p^{f'}-1}$.
If instead $(i-1,i)$ is a transition 
and $(h_i) \in \Conefrac(\gM,\gN)$ then non-zero terms
of $h_i$  have degree congruent to $r_i + c_i - d_i \equiv 0 \pmod{p^{f'}-1}$.
In either case, comparing with the preceding paragraphs we see that $\#\{ j \in 
  [0,r_i) : j \equiv  r_i + c_i - d_i \pmod{e(K'/K)}\}$ is exactly~$y_i$. 

By Theorem~\ref{thm: extensions of rank one Kisin modules}, we conclude
that for a fixed choice of the~$r_i$ 
the dimension of the
corresponding~$\Ext^1$ is $\Delta + \sum_{i=0}^{f-1} y_i$ (with
$\Delta$ as in the statement of \emph{loc.\ cit.}).
We say that the refined shape $\bigl(J, (r_i)_{i =0}^{f-1}\bigr)$
is \emph{maximal} if the $r_i$ are chosen to be maximal subject to the
above conditions, or equivalently if the $y_i$ are all chosen to be $e$; for each
shape~$J$, there is a unique maximal refined shape~$(J,r)$.

\subsubsection{The sets $\cP_{\tau}$}
\label{sec:sets-cp_tau}

To each tame type $\tau$ we now associate a set $\cP_{\tau}$, which
will be a subset of the set of shapes in $\Z/f'\Z$.  (In Appendix~\ref{sec: appendix on tame
  types} we will recall, following~\cite{MR2392355}, that the 
set $\cP_{\tau}$ parameterises the Jordan--H\"older factors of the
reduction mod~$p$ of 
$\sigma(\tau)$.)

We write $\eta (\eta')^{-1} =
\prod_{j=0}^{f'-1} (\sigma_j \circ \hchar)^{\gamma_j}$ for uniquely defined integers $0
\le \gamma_j \le p-1$ not all equal to $p-1$, so that
\begin{equation}\label{eq:k-gamma}
[k_i - k'_i] = \sum_{j=0}^{f'-1} p^{j} \gamma_{i-j} 
\end{equation}
with subscripts taken modulo $f'$.

If~$\tau$ is scalar then
we set $\cP_\tau=\{\varnothing\}$. Otherwise we let 
$\cP_{\tau}$ be the collection of shapes $J \subseteq \Z/f'\Z$ 
satisfying the conditions:
\begin{itemize}
\item if $i-1\in J$ and $i\notin J$ then $\gamma_{i}\ne p-1$, and
\item if $i-1\notin J$ and $i\in J$ then $\gamma_{i}\ne 0$.
\end{itemize}
When $\tau$ is a cuspidal type, so that $\eta' = \eta^q$, the integers 
$\gamma_j$ satisfy $\gamma_{i+f} = p-1-\gamma_i$ for all $i$; thus the
condition that if $(i-1,i)$ is a transition out of $J$ then $\gamma_i
\neq p-1$ translates exactly into the condition that if $(i+f-1,i+f)$ is a
transition into $J$ then $\gamma_{i+f} \neq 0$.

\subsubsection{Moduli stacks of extensions}\label{subsubsection: stacks of extensions}

We now apply the constructions of stacks and topological spaces of
Definitions~\ref{def:scheme-theoretic images}
and~\ref{df:constructible images} to the families of
extensions considered in Section~\ref{sec:extensions-shape-J}. 

\begin{df}\label{defn: M(j,r)}
If $(J,r)$ is a refined shape for $\tau$, then 
we let $\gM(J,r) := \gM(r,1,c)$ and let
$\gN(J,r) := \gM(s,1,d),$ where $c$, $d$, and $s$ are determined
from $J$, $r$, and $\tau$ according to the 
discussion of~(\ref{subsubsec:explicitly refined}); for instance we
take $c_i
= k_i$ when $i \in J$ and $c_i = k'_i$ when $i \not\in J$.
For the unique
maximal shape~$(J,r)$ refining~$J$, we write simply $\gM(J)$ and $\gN(J)$.
\end{df}

\begin{df}
If $(J,r)$ is a refined shape for $\tau$, then following
Definition~\ref{def:scheme-theoretic images}, we may construct the reduced
closed substack $\overline{\cC}\bigl(\gM(J,r),\gN(J,r)\bigr)$ of $\cC^{\tau,\BT,1}$,
as well as the reduced closed substack $\overline{\cZ}\bigl(\gM(J,r),
\gN(J,r)\bigr)$ of $\cZ^{\tau,1}.$
We introduce the notation $\overline{\cC}(J,r)$ and $\overline{\cZ}(J,r)$
for these two stacks, and note that (by definition) $\overline{\cZ}(J,r)$
is the scheme-theoretic image of $\overline{\cC}(J,r)$ under
the morphism $\cC^{\tau,\BT,1} \to \cZ^{\tau,1}$.
\end{df}

\begin{thm}\label{thm: dimension of refined shapes} 
If $(J,r)$ is any refined shape for $\tau$,
then $\dim \overline{\cC}(J,r) \leq [K:\Qp],$ with equality if and only if $(J,r)$ is
maximal.
\end{thm}
\begin{proof}
  This follows from  Corollary~\ref{cor:
    dimension of families of extensions}, Theorem~\ref{thm:
    extensions of rank one Kisin modules}, and Proposition~\ref{prop:base-change for exts}.  (See also the discussion
  following Definition~\ref{df:extensions of shape $J$}, and note that over
  $\Spec \Adist$, we have $\Delta=0$ by definition.)
\end{proof}

\begin{df} If $J \subseteq \Z/f'\Z$ is a shape, and if $r$ is chosen so that
$(J,r)$ is a maximal refined shape for $\tau$,
then we write $\overline{\cC}(J)$ to denote the closed substack $\overline{\cC}(J,r)$
of $\cC^{\tau,\BT,1}$,
and $\overline{\cZ}(J)$ to denote the closed substack
$\overline{\cZ}(J,r)$ of $\cZ^{\tau,1}$.
Again, we note that by definition
$\overline{\cZ}(J)$ is the scheme-theoretic
image of~$\overline{\cC}(J)$ in~$\cZ^{\tau,1}$. 
\end{df}

We will see later that the~$\overline{\cC}(J)$ are precisely the
irreducible components of~$\cC^{\tau,\BT,1}$; in particular, their
finite type points can correspond to irreducible Galois
representations. While we do not need it in the sequel, we note the
following definition and result, describing the underlying topological
spaces of the loci of reducible Breuil--Kisin modules of fixed refined shape.
\begin{defn}
  For each refined type~$(J,r)$, we write~$|\cC(J,r)^\tau|$ for the
  constructible subset~$|\cC(\gM(J,r),\gN(J,r))|$
  of~$|\cC^{\tau,\BT,1}|$ of Definition~\ref{df:constructible images}
  (where $\gM(J,r)$, $\gN(J,r)$ are the Breuil--Kisin modules
  of Definition~\ref{defn: M(j,r)}). We
  write~$|\cZ(J,r)^\tau|$ for the image of~$|\cC(J,r)^\tau|$
  in~$|\cZ^{\tau,1}|$ (which is again a constructible
  subset). 
\end{defn}

\begin{lem}
  \label{lem: closed points of C(J,r)}The $\Fpbar$-points
  of~$|\cC(J,r)^\tau|$ are precisely the reducible Breuil--Kisin modules
  with $\Fpbar$-coefficients
  of type~$\tau$ and refined shape~$(J,r)$.
\end{lem}
\begin{proof}
  This is immediate from the definition.
\end{proof}

\subsection{\texorpdfstring{$\kExt^1$}{ker-Ext} and vertical
  components}
In this section we will establish some basic facts about $\kExt^1_{\K{\F}}(\gM,\gN)$,
and use these results to study the images of our
irreducible components in $\cZ^{\tau,1}$. 
 Let $\gM = \gM(r,a,c)$ and $\gN = \gM(s,b,c)$ be Breuil--Kisin 
modules as in Section~\ref{subsec:
  Diamond--Savitt}.

Recall from \eqref{eqn: computing kernel of Ext groups} that the
dimension of $\kExt_{\K{\F}}(\gM,\gN)$ is bounded above by the
dimension of $\Hom_{\K{\F}}(\gM,\gN[1/u]/\gN)$;\ more precisely, by
Lemma~\ref{lem: Galois rep is a functor if A is actually finite local}
we find in this setting that 
\numequation\label{eq:ker-ext-formula}\begin{split} \dim_{\F}  \kExt^1_{\K{\F}}(\gM,\gN) = \dim_{\F}
\Hom_{\K{A}}(\gM,\gN[1/u]/\gN)  \\- (\dim_{\F} \Hom_{\F[G_K]}(T(\gM),T(\gN)) -
\dim_{\F} \Hom_{\K{A}}(\gM,\gN)).
\end{split}
\end{equation}

A map $f : \gM \to \gN[1/u]/\gN$ has the form
$f(m_i) = \mu_i n_i$ for some $f'$-tuple of elements $\mu_i \in
\F((u))/\F[[u]]$. By the same argument as in the first paragraph of the proof of
Lemma~\ref{lem:explicit-complex}, such a map belongs to
$C^0(\gN[1/u]/\gN)$ (i.e., it is $\Gal(K'/K)$-equivariant) if and only
if the $\mu_i$ are periodic with
period dividing $f$, and each nonzero term of $\mu_i$ has degree
congruent to $c_i-d_i \pmod{e(K'/K)}$.  One computes that
$\delta(f)(1\otimes m_{i-1}) = (u^{s_i} \varphi(\mu_{i-1}) - u^{r_i}
\mu_i)n_i$ and so $f \in C^0(\gN[1/u]/\gN)$ lies in  $\Hom_{\K{\F}}(\gM,\gN[1/u]/\gN)$
precisely when
\numequation\label{eq:phi-commute-ker-ext} a_i
u^{r_i} \mu_i = b_i \varphi(\mu_{i-1}) u^{s_i}\end{equation} for all
$i$.

\begin{remark}\label{rem:explicit-ker-ext}
Let $f \in  \Hom_{\K{\F}}(\gM,\gN[1/u]/\gN)$ be given as above. Choose
any lifting $\tilde{\mu}_i$ of $\mu_i$ to $\F((u))$. Then (with
notation as in~Definition~\ref{notn:calh}) the tuple
$(\tilde{\mu}_i)$ is an element of $\Czerofrac$,  and we define $h_i =
\mumap(\tilde{\mu}_i)$. Then
$h_i$ lies in $\F[[u]]$ for all $i$, so that $(h_i) \in \Cone$, and a
comparison with Lemma~\ref{lem:explicit-complex} shows that $f$ maps
to the extension class in $\kExt^1_{\K{\F}}(\gM,\gN)$ represented by $\gP(h)$.
\end{remark}

Recall that Lemma~\ref{lem: bound on torsion in kernel of Exts}
 implies that  
nonzero terms appearing in $\mu_i$ have degree at least $-\lfloor
e'/(p-1) \rfloor$. From this we obtain the following trivial bound on $\kExt$.

\begin{lemma}\label{cor:bounds-on-ker-ext}
 We have $\dim_{\F} \kExt^1_{\K{\F}}(\gM,\gN) \le \lceil e/(p-1)
 \rceil f$. 
\end{lemma}

\begin{proof} The degrees of nonzero terms of $\mu_i$ all lie in a single
  congruence class modulo $e(K'/K)$, and are bounded below by
  $-e'/(p-1)$. Therefore
  there are at most $\lceil e/(p-1) \rceil$ nonzero terms, and since
  the $\mu_i$ are periodic with period dividing $f$ the lemma follows.
\end{proof}

\begin{remark}\label{rem:half}
It follows directly from Corollary~\ref{cor:bounds-on-ker-ext} that if
$p > 3$ and $e\neq 1$ then we have  $\dim_{\F}
\kExt^1_{\K{\F}}(\gM,\gN) \le [K:\Qp]/2$, for then $\lceil e/(p-1)
\rceil \le e/2$. Moreover these inequalities are strict if $e >
2$. 
\end{remark}

We will require a more precise computation of
$\kExt^1_{\K{\F}}(\gM,\gN)$ in the setting of
Section~\ref{sec:extensions-shape-J} where the pair $(\gM,\gN)$ has 
maximal refined shape $(J,r)$.  We now return to that  setting and its
notation.

Let $\tau$ be a tame type. We
will find the following notation to be helpful. We let $\gamma_i^* =
\gamma_i$ if $i-1 \not\in J$, and $\gamma_i^* = p-1-\gamma_i$
if $i-1 \in J$. (Here the integers $\gamma_i$ are as in
Section~\ref{sec:sets-cp_tau}. In the case of scalar types this means
that we have $\gamma^*_i = 0$ for all $i$.)
Since $p[k_{i-1}-k'_{i-1}] - [k_i - k'_i] = (p^{f'}-1)\gamma_i$, 
 an elementary but useful calculation
shows that
\numequation\label{eq:gammastar}
p[d_{i-1}-c_{i-1}] - [c_i-d_i] = \gamma_i^* (p^{f'}-1),
\end{equation}
when $(i-1,i)$ is a transition, and that in this case $\gamma_i^*
=0$ if and only if $[d_{i-1}-c_{i-1}] < p^{f'-1}$. Similarly, if
$\tau$ is not a scalar type and $(i-1,i)$ is not a transition then
\numequation\label{eq:gammastar-2}
p[d_{i-1}-c_{i-1}] + [c_i-d_i] = (\gamma_i^*+1) (p^{f'}-1).
\end{equation}

The main computational result of this section is the following.

\begin{prop}\label{prop:ker-ext-maximal}
Let $(J,r)$ be any maximal refined shape for $\tau$, and suppose that
the pair $(\gM,\gN)$ has refined shape $(J,r)$.    Then  
$\dim_{\F} \kExt^1_{\K{\F}} (\gM,\gN)$ is equal to 
\[\# \{ 0 \le i < f \, : \, \text{the pair } (i-1,i) \text{ is a
  transition and } \gamma_i^* = 0 \},\]
except that when $e=1$, $\prod_i a_i = \prod_i b_i$,  and the quantity displayed above is $f$,
 then 
the dimension of $\kExt^1_{\K{\F}} (\gM,\gN)$ is equal to $f-1$.
\end{prop}

\begin{proof}
 The argument has two parts. First we show that $\dim_{\F}
 \Hom_{\K{\F}}(\gM,\gN[1/u]/\gN)$ is precisely the displayed quantity
 in the statement of the Proposition;\ then we check that
 $\dim_{\F} \Hom_{\F[G_K]} (T(\gM),T(\gN)) - \dim_{\F} \Hom_{\K{\F}} (\gM,\gN)$ is
 equal to $1$ in the exceptional case of the statement, and $0$
 otherwise. The result then follows from~\eqref{eq:ker-ext-formula}.

Let  $f : m_i \mapsto \mu_i n_i$ be an element of 
$C^0(\gN[1/u]/\gN)$. 
Since $u^{e'}$ kills $\mu_i$, and all nonzero terms of $\mu_i$ have
degree congruent to
$c_i-d_i \pmod{p^{f'}-1}$, certainly all nonzero terms of $\mu_i$ have
degree at least $-e' + [c_i-d_i]$. On the other hand since the shape
$(J,r)$ is maximal we have $r_i = e' - [c_i-d_i]$ when $(i-1,i)$ is a
transition and $r_i = e'$ otherwise. In either case $u^{r_i}$ kills
$\mu_i$, so that \eqref{eq:phi-commute-ker-ext} becomes simply the
condition that $u^{s_i}$ kills $\varphi(\mu_{i-1})$.

If $(i-1,i)$ is not a transition then $s_i=0$, and we conclude that
$\mu_{i-1}=0$. Suppose instead that $(i-1,i)$ is a transition, so
that $s_i = [c_i-d_i]$. Then all nonzero terms of $\mu_{i-1}$ have
degree at least $-s_i/p > -p^{f'-1} > -e(K'/K)$. Since those terms must have
degree congruent to $c_{i-1}-d_{i-1}  \pmod{p^{f'}-1}$, it follows
  that $\mu_{i-1}$ has at most one nonzero term (of degree
  $-[d_{i-1}-c_{i-1}]$), and this only if $[d_{i-1}-c_{i-1}] <
  p^{f'-1}$, or equivalently $\gamma_i^* = 0$ (as noted
  above). Conversely if $\gamma_i^*=0$ then 
\[ u^{s_i} \varphi(u^{-[d_{i-1}-c_{i-1}]}) = u^{[c_i-d_i] -
    p[d_{i-1}-c_{i-1}]} = u^{-\gamma_i^* (p^{f'}-1)}\] vanishes in
$\F((u))/\F[[u]]$. We conclude that $\mu_{i-1}$ may have a single nonzero term if
and only if $(i-1,i)$ is a transition and $\gamma_i^*=0$, and this
completes the first part of the argument.

Turn now to the second part. 
Looking at Corollary~\ref{cor: Kisin
  modules with the same generic fibre} and Lemma~\ref{lem: maps
  between rank 1 Kisin modules}, to compare
$\Hom_{\F[G_K]}(T(\gM),T(\gN))$ and $\Hom_{\K{\F}}(\gM,\gN)$ we need
to compute the quantities $\alpha_i(\gM)-\alpha_i(\gN)$. By definition
this quantity is equal
to 
\numequation\label{eq:alpha-difference-1}\frac{1}{p^{f'}-1} \sum_{j=1}^{f'}  p^{f'-j} \left( r_{i+j} - s_{i+j}
\right).\end{equation}
Suppose first that $\tau$ is non-scalar. When $(i+j-1,i+j)$ is a transition, we have $r_{i+j}-s_{i+j} = (e-1)(p^{f'}-1) +
[d_{i+j}-c_{i+j}] - [c_{i+j}-d_{i+j}]$, and otherwise we
have $r_{i+j}-s_{i+j} = e( p^{f'}-1)=  (e-1)(p^{f'}-1) + [d_{i+j}-c_{i+j}] +
[c_{i+j}-d_{i+j}]$. Substituting these expressions into
\eqref{eq:alpha-difference-1}, adding and subtracting $\frac{1}{p^{f'}-1}
p^{f'} [d_i-c_i]$, and regrouping gives
\[ -[d_i-c_i] + (e-1) \cdot \frac{p^{f'}-1}{p-1} + \frac{1}{p^{f'}-1}\sum_{j=1}^{f'}
p^{f'-j} \left( p[d_{i+j-1} - c_{i+j-1}] \mp [c_{i+j} - d_{i+j}]
\right),\]
where the sign is $-$ if $(i+j-1,i+j)$ is a transition and $+$ if
not. Applying the formulas~\eqref{eq:gammastar}
and~\eqref{eq:gammastar-2} we conclude that
\numequation\label{eq:alpha-difference}
\alpha_i(\gM)-\alpha_i(\gN) = -[d_i-c_i] +  (e-1) \cdot
\frac{p^{f'}-1}{p-1} + \sum_{j=1}^{f'} p^{f'-j} \gamma^*_{i+j} +
\sum_{j \in S_i} p^{f'-j}
\end{equation}
where the set $S_i$ consists of $1 \le j \le f$ such that $(i+j-1,i+j)$ is not a transition.
Finally, a moment's inspection shows that the same formula still holds
if $\tau$ is scalar (recalling that $J = \varnothing$ in that case).

Suppose that we are in the exceptional case of the proposition, so
that $e=1$, $\gamma_i^*=0$ for all $i$, and every pair $(i-1,i)$ is a
transition. The formula~\eqref{eq:alpha-difference} gives
$\alpha_i(\gM)-\alpha_i(\gN) = -[d_i-c_i]$. Since also $\prod_i a_i =
\prod_i b_i$  the conditions of
Corollary~\ref{cor: Kisin modules with the same generic fibre} are
satisfied, so that $T(\gM)=T(\gN)$;\ but on the other hand
$\alpha_i(\gM) < \alpha_i(\gN)$, so that by Lemma~\ref{lem: maps
  between rank 1 Kisin modules} there are no nonzero maps $\gM\to
\gN$, and $\dim_{\F} \Hom_{\F[G_K]} (T(\gM),T(\gN)) - \dim_{\F}
\Hom_{\K{\F}} (\gM,\gN) = 1.$

If instead we are not in the exceptional case of the
proposition, then either $\prod_i a_i \neq \prod_i b_i$, or else~\eqref{eq:alpha-difference}
gives $\alpha_i(\gM) - \alpha_i(\gN) > -[d_i-c_i]$ for all
$i$. Suppose that $T(\gM) \cong T(\gN)$. It follows from
Corollary~\ref{cor: Kisin modules with the same generic fibre} that
$\alpha_i(\gM) - \alpha_i(\gN) \equiv -[d_i-c_i]
\pmod{e(K'/K)}$. Combined with the previous inequality we deduce that
$\alpha_i(\gM) - \alpha_i(\gN)  > 0$, and Lemma~\ref{lem: maps between
  rank 1 Kisin modules} guarantees the existence of a nonzero map $\gM
\to \gN$. We deduce that in any event $\dim_{\F} \Hom_{\F[G_K]}
(T(\gM),T(\gN)) = \dim_{\F} \Hom_{\K{\F}} (\gM,\gN)$, completing the proof.
\end{proof}

\begin{cor}\label{cor:ker-ext-maximal-nonzero}
Let $(J,r)$ be any maximal refined shape for $\tau$, and suppose that
the pair $(\gM,\gN)$ has refined shape $(J,r)$.    Then  
$\dim_{\F} \kExt^1_{\K{\F}} (\gM,\gN) = 0$ if and only if $J \in \cP_{\tau}$.
\end{cor}

\begin{proof}
  This is immediate from Proposition~\ref{prop:ker-ext-maximal},
  comparing the definition of $\gamma_i^*$ with the defining
 condition on elements of $\cP_{\tau}$, and
  noting that the exceptional case in Proposition~\ref{prop:ker-ext-maximal} can occur only if
  $f$ is even (so in particular $f-1 \neq 0$ in these exceptional cases).
\end{proof}

Recall 
that
$\overline{\cZ}(J)$ is by definition the scheme-theoretic image of~$\overline{\cC}(J)$ in
$\cZ^{\tau,1}$. In the remainder of this section, we show that the $\overline{\cZ}(J)$
with $J\in\cP_{\tau}$ are pairwise distinct irreducible components
of~$\cZ^{\tau,1}$. In Section~\ref{subsec: irred components} below we
will show that they in fact exhaust the irreducible components of~$\cZ^{\tau,1}$.
\begin{thm}\label{thm: identifying the vertical components} 
$\overline{\cZ}(J)$ has dimension at most $[K:\Qp]$, with equality occurring if
and only if~$J\in\cP_{\tau}$. Consequently, the~$\overline{\cZ}(J)$ with
$J\in\cP_{\tau}$ are irreducible components of~$\cZ^{\tau,1}$.
\end{thm}
\begin{proof}The first part is immediate from Corollary~\ref{cor: dimension of
    families of extensions}, Proposition~\ref{prop:base-change for
    exts}, 
  Corollary~\ref{cor:ker-ext-maximal-nonzero} and Theorem~\ref{thm:
    dimension of refined shapes}. Since~$\cZ^{\tau,1}$ is
  equidimensional of dimension~$[K:\Qp]$ by Proposition~\ref{prop:
    dimensions of the Z stacks}, and the~$\overline{\cZ}(J)$ are closed and
  irreducible by construction, the second part follows from the first
  together with~\cite[\href{https://stacks.math.columbia.edu/tag/0DS2}{Tag 0DS2}]{stacks-project}.
\end{proof}

We also note the following result.

\begin{prop}
\label{prop:C to Z mono}
If~$J\in\cP_{\tau}$, then there is a dense open substack $\cU$ of 
$\overline{\cC}(J)$ such that the canonical morphism $\overline{\cC}(J) 
\to \overline{\cZ}(J)$ restricts to an open immersion on $\cU$.
\end{prop}
\begin{proof}
This follows from 
Proposition~\ref{prop: construction of family monomorphing to C and R}
and Corollary~\ref{cor:ker-ext-maximal-nonzero}.
\end{proof}

For later use, we note the following computation. Recall that we write
$\gN(J) = \gN(J,r)$ for the maximal shape $(J,r)$ refining $J$, and that~$\tau=\eta\oplus\eta'$.

\begin{prop}\label{prop:char-calculation}
For each shape $J$ we have 
\[ T(\gN(J)) \cong \eta \cdot \left(\prod_{i=0}^{f'-1} (\sigma_i \circ
  \hchar)^{t_i}\right)^{-1} |_{G_{K_{\infty}}} \] 
where 
\[ t_i =
\begin{cases}
  \gamma_i + \delta_{J^c}(i) & \text{if } i-1 \in J \\
  0 & \text{if } i-1\not\in J.
\end{cases}\]
Here $\delta_{J^c}$ is the characteristic function of the complement
of $J$ in $\Z/f'\Z$, and we are abusing notation by writing $\eta$ for
the function
$\sigma_i \circ \hchar^{k_i}$, which agrees with $\eta$ on $I_K$.

In particular the map $J \mapsto T(\gN(J))$ is injective on $\cP_{\tau}$.
\end{prop}

\begin{remark}\label{rk:cuspidal-char-niveau-1}
In the cuspidal case it is not \emph{a priori} clear that the formula in Proposition~\ref{prop:char-calculation} gives
a character of $G_{K_{\infty}}$ (rather than a character only when
restricted to $G_{L_{\infty}}$), but this is an elementary (if
somewhat painful) calculation using the
definition of the $\gamma_i$'s and the relation $\gamma_i +
\gamma_{i+f} = p-1$.
\end{remark}

\begin{proof}
 We begin by explaining how the final statement follows from the rest
 of the Proposition. First observe that if $J \in \cP_\tau$ then  $0 \le t_i \le p-1$ for all
 $i$. Indeed the only possibility for a contradiction would be if
 $\gamma_i = p-1$ and $i \not\in J$, but then the definition of
 $\cP_\tau$ requires that we cannot have $i-1 \in J$. Next, note that
 we never have $t_i = p-1$ for all $i$. Indeed, this would require $J
 = \Z/f'\Z$ and $\gamma_i=p-1$ for all~$i$, but by definition the
 $\gamma_i$ are not all equal to $p-1$.

The observations in the previous paragraph imply that (for $J \in \cP_\tau$) the character
$T(\gN(J))$ uniquely determines the integers $t_i$, and so it remains
to show that the integers $t_i$ determine the set $J$. If $t_i = 0$
for all $i$, then either $J = \varnothing$ or $J = \Z/f'\Z$ (for
otherwise there is a transition out of $J$, and $\delta_{J^c}(i) \neq
0$ for some $i-1 \in J$). But if $J = \Z/f'\Z$ then $\gamma_i = 0$ for
all $i$ and $\tau$ is scalar;\ but for scalar types we have $\Z/f'\Z
\not\in \cP_\tau$, a contradiction. Thus $t_i =0$ for all $i$ implies $J =
\varnothing$. 

For the rest of this part of the argument, we may therefore suppose $t_i \neq 0$
for some $i$, which forces $i-1 \in J$. The entire set $J$ will then
be determined by recursion if we can show that knowledge of $t_i$ along with
whether or not $i \in J$, determines whether or not $i-1 \in J$. Given
the defining formula for $t_i$, the only possible ambiguity is if $t_i
= 0$ and $\gamma_i +
\delta_{J^c}(i) = 0$, so that $\gamma_i = 0$ and $i \in J$. But the definition of $\cP_{\tau}$ requires $i-1
\in J$ in this case. This completes the proof.

We now turn to proving the formula for $T(\gN(J))$. We will use
Lemma~\ref{lem: generic fibres of rank 1 Kisin modules} applied at
$i=0$, for which we have to compute $\alpha_0 - d_0$ writing $\alpha_0
= \alpha_0(\gN)$. Recall
that we have already computed $\alpha_0(\gM(J)) - \alpha_0(\gN(J))$ in
the proof of Proposition~\ref{prop:ker-ext-maximal}. Since
$\alpha_0(\gM(J)) + \alpha_0(\gN(J)) = e(p^{f'}-1)/(p-1)$, taking the
difference between these formulas gives 
\[ 2 \alpha_0 =  [d_0 - c_0] - \sum_{j=1}^{f'} p^{f'-j} \gamma_j^* +
  \sum_{j \in S_0^c} p^{f'-j} \]
where $S_0^c$ consists of those $1\le j \le f$ such that $(j-1,j)$ is
a transition. Subtract $2[d_0]$ from both sides, and add the expression
$-[k_0-k'_0] + \sum_{j=1}^{f'} p^{f'-j} \gamma_j$ (which vanishes by definition) to the
right-hand side. Note that $[d_0 - c_0] - [k_0-k'_0] - 2[d_0]$ is
equal to $-2[k_0]$ if $0 \not\in J$, and to
$e(K'/K)-2[k_0-k_0']-2[k_0']$ if $0 \in J$.
Since $\gamma_j -
\gamma_j^* = 2\gamma_j - (p-1)$ if $j-1 \in J$ and is $0$ otherwise,
the preceding expression rearranges to give (after dividing by $2$)
\[ \alpha_0 - [d_0]  = -\kappa_0 + \sum_{j-1\in J} p^{f'-j}
  \gamma_j  + \sum_{j-1\in J, j\not\in J} p^{f'-j} = -\kappa_0 +
  \sum_{j=1}^{f'}  p^{f'-j} t_j\]
where $\kappa_0 = [k_0]$ if $0 \not\in J$ and $\kappa_0 = [k_0 - k'_0]
+[k'_0]$ if $0 \in J$. Since in either case $\kappa_0 \equiv k_0
\pmod{e(K'/K)}$ the result now follows from Lemma~\ref{lem: generic fibres of rank 1 Kisin modules}.
\end{proof}

\begin{defn}Let $\rbar:G_K\to\GL_2(\F')$ be representation. Then we
  say that a Breuil--Kisin module~$\gM$ with $\F'$-coefficients is a
  \emph{Breuil--Kisin model of~$\rbar$ of type~$\tau$} if~$\gM$ is an
  $\F'$-point of~$\cC^{\tau,\BT,1}$, and
  $T_{\F'}(\gM)\cong\rbar|_{G_{K_\infty}}$.
\end{defn}
\begin{thm}
  \label{thm: unique serre weight}The~$\overline{\cZ}(J)$, with $J\in\cP_\tau$, are pairwise distinct closed
  substacks of $\cZ^{\tau,1}$. For each $J\in\cP_\tau$, there is a dense set of finite type
  points of $\overline{\cZ}(J)$ with the property that the corresponding Galois
  representations have $\sigmabar_J$ as a Serre weight, and which
  furthermore admit a unique Breuil--Kisin model of type~$\tau$.
\end{thm}
\begin{proof}
Recall from Definition~\ref{def:scheme-theoretic images}
that $\overline{\cZ}(J)$ is defined to be the scheme-theoretic image
of a morphism $\Spec \Bdist \to \cZ^{\dd,1}.$  As in the proof of
Lemma~\ref{lem:ext images}, since the source
and target of this morphism are of finite presentation over $\F$, 
its image is a dense constructible subset of its scheme-theoretic image, 
and so contains a dense open subset, which we may interpret as 
a dense open substack $\cU$ of $\overline{\cZ}(J).$
From the definition of $\Bdist,$
the finite type points of~$\cU$ 
correspond to
reducible Galois representations admitting a model of type~$\tau$
and refined shape~$(J,r)$, for which~$(J,r)$ is maximal. 

That the~$\overline{\cZ}(J)$ are pairwise distinct is immediate  
from the above and~Proposition~\ref{prop:char-calculation}.  
 Combining this observation with Theorem~\ref{thm: dimension of refined shapes}, 
we see that by deleting the intersections of~$\overline{\cZ}(J)$ with
the~$\overline{\cZ}(J',r')$ for all refined shapes~$(J',r')\ne (J,r)$,
we obtain a dense open substack~$\cU'$ whose finite type points have
the property that every Breuil--Kisin model of type~$\tau$ of the
corresponding Galois representation has shape~$(J,r)$. The unicity of such a Breuil--Kisin model then follows 
from Corollary~\ref{cor:ker-ext-maximal-nonzero}.

It remains to show that every such Galois representation $\rbar$ has $\sigmabar_J$ as a
Serre weight. 
Suppose first that $\tau$ is a principal series type. We claim that (writing $\sigmabar_J =
\sigmabar_{\vec{t},\vec{s}} \otimes (\eta' \circ \det)$ as
in Appendix~\ref{sec: appendix on tame
  types}) we
have \[T(\mathfrak{N}(J))|_{I_K}=\eta'|_{I_K} \prod_{i=0}^{f-1}\omega_{\sigma_i}^{s_i+t_i}.\]To see this, note that by
Proposition~\ref{prop:char-calculation} it is enough to show that
$\eta|_{I_K}=\eta'|_{I_K} \prod_{i=0}^{f-1}\omega_{\sigma_i}^{s_i+2t_i}$, which
follows by comparing the central characters of $\sigmabar_J$
and $\sigmabar(\tau)$ (or from a direct computation with the
quantities $s_i,t_i$).

Since~$\det\rbar|_{I_K}=\eta\eta'\varepsilonbar^{-1}$, we
have \[\rbar|_{I_K}\cong \eta'|_{I_K} \otimes \begin{pmatrix}
      \prod_{i=0}^{f-1}\omega_{\sigma_i}^{s_i+t_i} &*\\ 0 & \varepsilonbar^{-1}\prod_{i=0}^{f-1}\omega_{\sigma_i}^{t_i}
    \end{pmatrix}.\] The result then
follows from Lemma~\ref{lem: explicit Serre weights with our
  normalisations}, using Lemma~\ref{lem: list of things we need to
  know about Serre weights}(2) 
and the fact that the fibre of the morphism $\cC^{\tau,\BT,1} \to \cR^{\dd,1}$
above $\rbar$ is nonempty to see that $\rbar$ is not tr\`es ramifi\'ee.

The argument in the cuspidal case proceeds analogously, noting
that if the character $\theta$ (as in Appendix~\ref{sec: appendix on tame
  types}) corresponds to $\widetilde{\theta}$ under local class field theory
then $\widetilde{\theta} |_{I_K} = \eta'
\prod_{i=0}^{f'-1} \omega_{\sigma'_i}^{t_i}$, and that from central
characters we have 
$\eta\eta' = (\widetilde{\theta} |_{I_K})^2 
\prod_{i=0}^{f-1} \omega_{\sigma_i}^{s_i}$.
\end{proof}

\begin{rem}
  \label{rem: unique Serre weight isn't proved here but could be}With
  more work, we could use the results of~\cite{gls13} and our
  results on dimensions of families of extensions to strengthen
  Theorem~\ref{thm: unique serre weight}, showing that there is a
  dense set of finite type points of~$\cZbar(J)$ with the property that the corresponding Galois
  representations have $\sigmabar_J$ as their \emph{unique} Serre
  weight. In fact, we will prove this as part of our work on
  the geometric Breuil--M\'ezard conjecture, and it is an immediate
  consequence of Theorem~\ref{thm:stack version of geometric
    Breuil--Mezard} below (which uses Theorem~\ref{thm: unique serre
    weight} as an input).
\end{rem}
\subsection{Irreducible Galois representations} 
\label{subsec:irreducible}
We now show
that the points of $\cC^{\tau,\BT,1}$ which are irreducible (that
is, cannot be written as an extension of rank one Breuil--Kisin modules) lie
in a closed substack of positive codimension. We begin with the
following useful observation.

\begin{lem}
  \label{lem: closed points of irred}The rank two Breuil--Kisin modules with
  descent data  and $\Fpbar$-coefficients which are irreducible (that
  is, which cannot be written as an extension of rank~$1$ Breuil--Kisin
  modules with descent data)  are
  precisely those whose corresponding \'etale $\varphi$-modules are
  irreducible, or equivalently whose corresponding
  $G_K$-representations are irreducible.
\end{lem}
\begin{proof}Let~$\gM$ be a Breuil--Kisin module with descent data
  corresponding to a finite type point of~$\cC^{\tau,\BT,1}_{\dd}$,
  let~$M=\gM[1/u]$, and let~$\rhobar$ be the $G_K$-representation
  corresponding to~$M$. As noted in the proof of Lemma~\ref{lem: restricting to K_infty doesn't lose information about
    rbar}, $\rhobar$ is reducible if and only
  if~$\rhobar|_{G_{K_\infty}}$ is reducible, and by Lemma~\ref{lem:
    Galois rep is a functor if A is actually finite local}, this is
  equivalent to~$M$ being reducible. That this is in turn
  equivalent to~$\gM$ being reducible may be proved in 
  the same way as~\cite[Lem.\ 5.5]{MR3164985}. 
\end{proof}

Recall that~$L/K$ denotes the unramified quadratic extension; then the
irreducible representations~$\rhobar:G_K\to\GL_2(\Fpbar)$ are all
induced from characters of~$G_L$. Bearing in mind Lemma~\ref{lem:
  closed points of irred}, this means that we can study
irreducible Breuil--Kisin modules via a consideration of base-change of Breuil--Kisin modules
from $K$ to~$L$, and our previous
study of reducible Breuil--Kisin modules.
Since this will require us to consider Breuil--Kisin modules (and moduli
stacks thereof) over both $K$ and $L$, we will have to introduce
additional notation in order to indicate over which of the two fields
we might be working.   We do this simply by adding a subscript `$K$'
or `$L$' to our current notation.  We will also
omit other decorations which are being held fixed
throughout the present discussion. Thus we write $\cC_{K}^{\tau}$
to denote the moduli stack that was previously denoted $\cC^{\tau,\BT,1}$,
and $\cC_{L}^{\tau_{|L}}$ to denote the corresponding
moduli stack for Breuil--Kisin modules over $L$, with the type taken
to be the restriction $\tau_{|L}$ of $\tau$ from $K$ to $L$.
(Note that whether $\tau$ is principal series or cuspidal,
the restriction $\tau_{| L}$ is principal series.)

As usual we fix
a uniformiser~$\pi$ of $K$, which we also take to be our fixed
uniformiser of $L$.
Also, throughout
this section we take $K' =
L(\pi^{1/(p^{2f}-1)})$, so that $K'/L$ is the standard choice of
extension for $\tau$ and $\pi$ regarded as a type and uniformiser for~$L$. 

If~$\gP$ is a Breuil--Kisin module with descent data from~$K'$ to~$L$, then we
let
$\gP^{(f)}$ be the Breuil--Kisin module ~$W(k')\otimes_{\Gal(k'/k),W(k')}\gP$,
where the pullback is given by the non-trivial automorphism
of~$k'/k$. 
In particular, we have $\gM(r,a,c)^{(f)} =
\gM(r',a',c')$ where $r'_i = r_{i+f}$, $a'_i=a_{i+f}$, 
and $c'_i=c_{i+f}$.

We let $\sigma$ denote the non-trivial automorphism of $L$ over $K$,
and write $G :=  \Gal(L/K)= \langle \sigma \rangle$, a cyclic group
of order two.
There is an action $\alpha$ of $G$ on 
$\cC_{L}$
defined via $\alpha_{\sigma}: \gP \mapsto \gP^{(f)}$. 
More precisely, this induces an action
of $G:= \langle \sigma \rangle$ on 
$\cC_{L}^{\tau_{|L}}$
in the strict\footnote{From a $2$-categorical perspective,
	it is natural to relax the notion
of group action on a stack so as to allow natural transformations,
rather than literal equalities, when relating multiplication
in the group to the compositions of the corresponding 
equivalences of categories arising in the definition of an action. 
An action
in which actual equalities hold is then called {\em strict}.  Since
our action is strict, we are spared from having to consider
the various $2$-categorical aspects of the situation that would
otherwise arise.}
sense that
$$\alpha_{\sigma} \circ \alpha_{\sigma} =
\id_{\cC_L^{\tau_{|L}}}.$$

We now define the fixed point stack for this action.

\begin{df}
	\label{def:fixed points}
	We let the fixed point stack
$(\cC_{L}^{\tau_{|L}})^G$ denote the stack
whose $A$-valued points consist of an $A$-valued point $\gM$
of $\cC_L^{\tau_{|L}}$, together with an isomorphism
$\imath: \gM \iso \gM^{(f)}$ which satisfies the cocycle condition
that the composite
$$\gM \buildrel \imath \over \longrightarrow \gM^{(f)} 
\buildrel \imath^{(f)} \over \longrightarrow (\gM^{(f)})^{(f)} = \gM$$
is equal to the identity morphism $\id_{\gM}$.
\end{df}

We now give another description of $(\cC_L^{\tau_{|L}})^G$, in terms 
of various fibre products, which is technically useful. 
This alternate description involves two steps.  In the first step,
we define fixed points of the automorphism $\alpha_{\sigma}$,
without imposing the additional condition that the fixed point data
be compatible with the relation $\sigma^2~=~1$ in~$G$.  Namely, we define
$$
(\cC_{L}^{\tau_{|L}})^{\alpha_{\sigma}}
:= 
\cC_{L}^{\tau_{|L}}
	\underset
{\cC_{L}^{\tau_{|L}}
	\times 
\cC_{L}^{\tau_{|L}}
}
{\times}
\cC_{L}^{\tau_{|L}}
	$$where the first morphism $\cC_{L}^{\tau_{|L}}\to\cC_{L}^{\tau_{|L}}
	\times 
\cC_{L}^{\tau_{|L}}$ is the diagonal, and the second is $\id\times\alpha_\sigma$.
	Working through the definitions,
	one finds
	that an $A$-valued point of $(\cC_L^{\tau_{|L}})^{\alpha_{\sigma}}$
	consists of a pair $(\gM,\gM')$ of objects of $\cC_L^{\tau_{|L}}$
	over $A$, equipped with isomorphisms $\alpha: \gM \iso \gM'$
	and $\beta: \gM \iso (\gM')^{(f)}$.  The morphism 
	$$(\gM,\gM',\alpha,\beta) \mapsto (\gM, \imath),$$
	where $\imath := (\alpha^{-1})^{(f)} \circ \beta: \gM \to \gM^{(f)}$,
	induces an isomorphism between $(\cC_L^{\tau_{|L}})^{\alpha_{\sigma}}$
		and the stack classifying points $\gM$ of $\cC_L^{\tau_{|L}}$
		equipped with an isomorphism 
		$\imath: \gM \to \gM^{(f)}$.
		(However, no cocycle condition has been imposed on $\imath$.)

	Let $I_{\cC_L^{\tau_{|L}}}$ denote the inertia stack
	of $\cC_L^{\tau_{|L}}.$
We define a morphism
$$(\cC_L^{\tau_{|L}})^{\alpha_{\sigma}} \to I_{\cC_L^{\tau_{|L}}}$$
via $$(\gM,\imath) \mapsto (\gM, \imath^{(f)}\circ \imath),$$
where, as in Definition~\ref{def:fixed points},
we regard the composite $\imath^{(f)}\circ \imath$ as an automorphism
of $\gM$ via the identity $(\gM^{(f)})^{(f)} = \gM.$
Of course, we also have the identity
	section $e: \cC_L^{\tau_{|L}} \to I_{\cC_L^{\tau_{|L}}}$.
	We now define
	$$(\cC_L^{\tau_{|L}})^G :=
	(\cC_L^{\tau_{|L}})^{\alpha_{\sigma}}
       	\underset{I_{\cC_L^{\tau_{|L}}}}{\times} 
	\cC_L^{\tau_{|L}}.
	$$
	If we use the description of $(\cC_L^{\tau|_{L}})^{\alpha_{\sigma}}$
	as classifying
	pairs $(\gM,\imath),$ then (just unwinding definitions)
	this fibre product classifies tuples $(\gM,\imath,\gM',\alpha)$,
	where $\alpha$ is an isomorphism $\gM \iso \gM'$ which furthermore
	identifies $\imath^{(f)}\circ \imath$
	with $\id_{\gM'}$.  Forgetting $\gM'$ and $\alpha$ then induces
	an isomorphism between~$(\cC_L^{\tau|_L})^G$, as defined
	via the above fibre product, and the stack defined in 
	Definition~\ref{def:fixed points}.

To compare this fixed point stack to $\cC^\tau_K$, we make the
following observations. Given a Breuil--Kisin module with descent
data from~$K'$ to~$K$, we obtain a Breuil--Kisin module with descent data
from~$K'$ to~$L$ via the obvious forgetful map. Conversely, given a
Breuil--Kisin module~$\gP$ with  descent data
from~$K'$ to~$L$, the additional data required to enrich this to a
Breuil--Kisin module with descent data from~$K'$ to~$K$ can be described as follows as follows:  let $\theta \in \Gal(K'/K)$ denote the unique element which fixes
$\pi^{1/(p^{2f}-1)}$ and acts nontrivially on $L$. Then to enrich
the descent data on $\gP$ to descent data from $K'$ to $K$, it is
necessary and 
sufficient to give an additive map $\hat\theta : \gP \to \gP$ satisfying
$\hat\theta(sm) = \theta(s)\hat\theta(m)$ for all $s \in \gS_{\F}$ and
$m \in \gP$, and such that $\hat\theta
\hat g \hat \theta = \hat g^{p^f}$ for all $g \in \Gal(K'/L)$.

In turn, the data of the additive map $\hat\theta:\gP\to\gP$ is
equivalent to giving the data of the map
$\theta(\hat\theta):\gP\to\gP^{(f)}$ obtained by
composing~$\hat\theta$ with the Frobenius on~$L/K$. The defining
properties of $\hat\theta$ are equivalent to asking that
this map is an isomorphism of Breuil--Kisin modules with descent data
satisfying the cocycle condition of Definition~\ref{def:fixed points};
accordingly, we have a natural morphism  $\cC_K^{\tau} \to
	(\cC_L^{\tau_{|L}})^G$, and a restriction morphism
        \numequation\label{eqn: restriction morphism}\cC_K^{\tau} \to
	\cC_L^{\tau_{|L}}. \end{equation}

The following simple lemma summarises the basic facts about
base-change in the situation we are considering.

\begin{lemma}\label{lem: fixed point stack iso}
	There is an isomorphism $\cC_K^{\tau} \iso
	(\cC_L^{\tau_{|L}})^G$.
\end{lemma}
\begin{proof}
 This follows immediately from the preceding discussion.	
\end{proof}

\begin{rem}
  \label{rem: the R version of the fixed point stack}In the proof of
  Theorem~\ref{thm: irreducible Kisin modules can be ignored} we
  will make use of the following analogue of Lemma~\ref{lem: fixed
    point stack iso} for \'etale $\varphi$-modules. Write~$\cR_K$,
  $\cR_L$ for the moduli stacks of Definition~\ref{defn: R^dd}, i.e.\
  for the moduli stacks of rank~$2$ \'etale $\varphi$-modules with
  descent data respectively to~$K$ or to~$L$. Then we have an action
  of~$G$ on~$\cR_L$ defined
  via~$M\mapsto M^{(f)}:=W(k')\otimes_{\Gal(k'/k),W(k')}M$, and we
  define the fixed point stack~$(\cR_L)^G$ exactly as in
  Definition~\ref{def:fixed points}: namely an $A$-valued point
  of~$(\cR_L)^G$ consists of an $A$-valued point~$M$ of $\cR_L$,
  together with an isomorphism $\iota:M\isoto M^{(f)}$ satisfying the
  cocycle condition. The preceding discussion goes through in this
  setting, and shows that there is an isomorphism
  $\cR_K\isoto (\cR_L)^G$.

We also note that 
the morphisms $\cC_K^\tau \to \cC_L^{\tau_{|L}}$ and
$\cC_K^\tau \to \cR_K$ 
induce a monomorphism
\numequation\label{eqn:C into R mono base change}  \cC_K^{\tau} \hookrightarrow \cC_L^{\tau_{|L}} \times_{\cR_L}
\cR_K\end{equation}
One way to see this is to rewrite this morphism (using the previous discussion) 
as a morphism
$$(\cC_L^{\tau_{|L}})^G \to \cC_L^{\tau_{|L}} \times_{\cR_L} (\cR_L)^G,$$
and note that the descent data via $G$ on an object classified by
the source of this morphism is determined by the induced descent data on its
image in $(\cR_L)^G$.
\end{rem}

We now use the Lemma~\ref{lem: fixed point stack iso} to study the locus of finite type points
of $\cC_K^{\tau}$ which correspond to irreducible Breuil--Kisin modules. 
Any irreducible Breuil--Kisin module over $K$ becomes reducible when restricted to $L$,
and so may be described as an extension
$$0 \to \gN \to \gP \to \gM \to 0,$$
where $\gM$ and $\gN$ are 
Breuil--Kisin modules of rank one with descent data from $K'$ to $L$,
and $\gP$ is 
additionally
equipped with an isomorphism $\gP \cong \gP^{(f)}$,
satisfying the cocycle condition of Definition~\ref{def:fixed
  points}. 

Note that the characters
$T(\gM)$, $T(\gN)$ of $G_{L_\infty}$ are distinct and cannot be
extended to characters of $G_K$. Indeed, this condition is plainly
necessary for an extension~$\gP$ to arise as the base change of an irreducible
Breuil--Kisin module 
(see the
proof of  Lemma~\ref{lem: restricting to K_infty doesn't lose information about
    rbar}). 
Conversely, if $T(\gM)$, $T(\gN)$ of $G_{L_\infty}$ are distinct and cannot be
extended to characters of $G_K$, then  for any $\gP \in \Ext^1_{\K{\F}}(\gM,\gN)$
whose descent data can be enriched to give descent data from $K'$ to $K$, this enrichment is necessarily irreducible. 
In  particular, the existence of such a~$\gP$ implies that 
the descent data
on $\gM$ and $\gN$ cannot be enriched to give descent data from~$K'$ to~$K$. 

We additionally have the following observation. 

\begin{lemma}\label{lem:nonempty-then-map}
If $\gM,\gN$ are such that there is an extension \[0 \to \gN \to \gP
  \to \gM \to 0\] whose descent data can be enriched to give an irreducible
Breuil--Kisin module over~$K$, then there exists a
nonzero map $\gN \to \gM^{(f)}$.
\end{lemma}

\begin{proof}
The composition $\gN
\to \gP \xrightarrow{\hat\theta}  \gP \to \gM$, in which first and
last arrows are the natural inclusions and projections, must be
nonzero (or else $\hat\theta$ would give descent data on $\gN$ from
$K'$ to $K$). It is not itself a map of Breuil--Kisin modules, because $\hat\theta$
is semilinear, but is a map of Breuil--Kisin modules when viewed as a map $\gN \to \gM^{(f)}$.
\end{proof}

We now consider (for our fixed~$\gM$, $\gN$, and working over~$L$
rather than over~$K$) the scheme $\Spec B^{\dist}$ as in
Subsection~\ref{subsec:universal families}. Following
Lemma~\ref{lem:nonempty-then-map}, we assume that there exists a
nonzero map $\gN \to \gM^{(f)}$.  The observations made above show
that we are in the strict case, and thus that
$\Spec A^{\dist} = \Gm\times \Gm$ and that furthermore we may (and do)
set $V = T$.  We consider the fibre product with the restriction morphism~\eqref{eqn: restriction morphism}
$$Y(\gM,\gN):=\Spec B^{\dist} \times_{\cC_L^{\tau_{|L}}}\cC_K^{\tau}.$$

Let $\Gm \hookrightarrow \Gm\times\Gm$ be the diagonal closed immersion,
and let $(\Spec B^{\dist})_{|\Gm}$ denote the pull-back of $\Spec B^{\dist}$
along this closed immersion.
By Lemma~\ref{lem:nonempty-then-map}, the projection $Y(\gM,\gN) \to \Spec B^{\dist}$
factors through
$(\Spec B^{\dist})_{|\Gm},$
and combining this with Lemma~\ref{lem: fixed point stack iso} we see
that  $Y(\gM,\gN)$ may also be described as the fibre product
$$(\Spec B^{\dist})_{|\Gm} \times_{\cC_L^{\tau_{|L}}} (\cC_L^{\tau_{|L}})^G.$$

Recalling the warning of Remark~\ref{rem: potential confusion of two lots of Gm times
  Gm}, Proposition~\ref{prop: construction of family monomorphing to C
  and R} 
now shows that there is a monomorphism
$$[ (\Spec B^{\dist})_{|\Gm} / \Gm\times\Gm] \hookrightarrow \cC_L^{\tau_{|L}},$$
and thus, by
Lemma~\ref{lem: morphism from quotient stack is a monomorphism},
that there is an isomorphism
$$
(\Spec B^{\dist})_{|\Gm}
\times_{\cC_L^{\tau_{|L}}}
(\Spec B^{\dist})_{|\Gm}
\iso
(\Spec B^{\dist})_{|\Gm} \times \Gm\times \Gm.$$
(An inspection of the proof of Proposition~\ref{prop: construction of family monomorphing to C and R} shows that in fact
this result is more-or-less proved directly,
as the key step in proving the proposition.)
An elementary manipulation with fibre products then shows that there
is an isomorphism
$$Y(\gM,\gN) \times_{(\cC_L^{\tau_{|L}})^G} Y(\gM,\gN)
\iso Y(\gM,\gN)\times \Gm\times\Gm,$$
and thus, by another application of 
Lemma~\ref{lem: morphism from quotient stack is a monomorphism},
we find that there is a monomorphism
\numequation\label{eqn: mono from Y}[Y(\gM,\gN)/\Gm\times\Gm] \hookrightarrow (\cC_L^{\tau_{|L}})^G.\end{equation}

We define $\cC_{\irred}$ to be the union over all such pairs $(\gM,\gN)$ of
the scheme-theoretic images of the various projections
$Y(\gM,\gN) \to (\cC_L^{\tau_{|L}})^G$.  Note that this
image 
depends on $(\gM,\gN)$ up to simultaneous 
unramified twists of $\gM$ and $\gN$, and there are only
finitely many such pairs $(\gM,\gN)$ up to such unramified twist. By
definition, $\cC_{\irred}$ is a closed substack of
$\cC^{\tau}_K$ 
which contains every finite
type point of $\cC^{\tau}_K$ corresponding to an irreducible Breuil--Kisin
module.

The following is the main result of this section.

\begin{thm}
  \label{thm: irreducible Kisin modules can be ignored} The
  closed substack  $\cC_{\irred}$ of $\cC_K^{\tau}=\cC^{\tau,\BT,1}$, which contains
 every finite type point of $\cC^{\tau}_K$ corresponding
 to an irreducible Breuil--Kisin module,
  has dimension strictly less than $[K:\Qp]$.
\end{thm}
\begin{proof}
  As noted above, there are only finitely many pairs $(\gM,\gN)$ up to
  unramified twist, so it is enough to show that for each of them, the
  scheme-theoretic image of the monomorphism~\eqref{eqn: mono from Y}
  has dimension less than $[K:\Q_p]$.

By \cite[\href{https://stacks.math.columbia.edu/tag/0DS6}{Tag 0DS6}]{stacks-project},
to prove the present theorem,
it then suffices to show that 
$\dim Y(\gM,\gN) \leq [K:\Q_p] + 1$
(since $\dim \Gm\times\Gm = 2$).
To establish this, it suffices to show, for each point
$x \in \Gm(\F'),$ where $\F'$ is a finite extension of~$\F$, 
that the dimension of the fibre $Y(\gM,\gN)_x$ is bounded by
$[K:\Q_p]$. After relabelling, as we may, the field $\F'$ as $\F$ and
the Breuil--Kisin modules $\gM_x$ and $\gN_x$ as $\gM$ and $\gN$, we may
suppose that in fact $\F'=\F$ and~$x=1$.

Manipulating
the fibre product appearing in the definition of~$Y(\gM,\gN)$, we find
that
\numequation
\label{eqn:Y fibre blahblah}
Y(\gM,\gN)_1 = 
\Ext^1_{\K{\F}}(\gM,\gN) \times_{\cC_L^{\tau_{|_N}}} \cC_K^{\tau}, 
\end{equation}
where the fibre product is taken with respect to the
morphism $\Ext^1_{\K{\F}}(\gM,\gN) \to \cC_L^{\tau}$ that 
associates the corresponding
rank two extension to an extension
of rank one Breuil--Kisin modules,
and the restriction
morphism~\eqref{eqn: restriction morphism}.

In order to bound the dimension of~$Y(\gM,\gN)_1$, it will be easier
to first embed it into another,
larger, fibre product, which we now introduce. Namely, the
monomorphism~\eqref{eqn:C into R mono base change} 
induces a monomorphism
$$Y(\gM,\gN)_1 \hookrightarrow Y'(\gM,\gN)_1 := 
\Ext^1_{\K{\F}}(\gM,\gN) \times_{\cR_L} \cR_K.$$
Any finite type point of this fibre product lies over a fixed isomorphism
class of finite type points
of $\cR_K$ (corresponding to some fixed irreducible Galois
representation); we let $P$ be a choice of such a point.  The
restriction of $P$ then lies in a fixed isomorphism class of finite
type points of $\cR_L$ (namely, the isomorphism
class of the direct sum
$\gM[1/u]\oplus \gN[1/u] \cong \gM[1/u] \oplus \gM^{(f)}[1/u]$).
Thus the projection $Y'(\gM,\gN)_1 \to \cR_K$ factors through
the residual gerbe of $P$, while the morphism $Y'(\gM,\gN)_1
\to \cR_L$ factors through the residual gerbe 
of 
$\gM[1/u]\oplus \gN[1/u] \cong \gM[1/u] \oplus \gM^{(f)}[1/u]$.
Since $P$ corresponds via
Lemma~\ref{lem: Galois rep is a functor if A is actually finite local}
to an irreducible Galois representation,
we find that $\Aut(P) = \Gm$.
Since 
$\gM[1/u]\oplus \gN[1/u] $  corresponds  via
Lemma~\ref{lem: Galois rep is a functor if A is actually finite local}
to the direct sum of two non-isomorphic Galois characters, we find
that $\Aut(\gM[1/u]\oplus \gN[1/u] ) = \Gm \times \Gm$.  

Thus we obtain monomorphisms
\nummultline
\label{eqn:Y fibre}
Y(\gM,\gN)_1 \hookrightarrow 
Y'(\gM,\gN)_1
\\
\hookrightarrow 
\Ext^1_{\K{\F}}(\gM,\gN)
\times_{[\Spec F'//\Gm\times \Gm]} [\Spec F'//\Gm]
\cong
\Ext^1_{\K{\F}}(\gM,\gN)
\times \Gm.
\end{multline}
In Proposition~\ref{prop:irred-bound} we obtain a description of the image of $Y(\gM,\gN)_1$
under this monomorphism which allows us to bound its dimension
by~$[K:\Qp]$, as required.
\end{proof}

We now prove the bound on the dimension of~$Y(\gM,\gN)_1$ 
that we used in the proof of Theorem~\ref{thm: irreducible Kisin
  modules can be ignored}. Before establishing this bound, we make some further remarks.
To begin with, we remind the reader 
that we are working with Breuil--Kisin modules, \'etale $\varphi$-modules, etc.,
over $L$ rather than $K$, so that e.g.\
the structure parameters of $\gM, \gN$ are periodic modulo $f' = 2f$
(not modulo $f$), and the pair $(\gM,\gN)$ has type $\tau|_L$.
We will readily apply various pieces of notation that were
introduced above in the
context of the field $K$, adapted in the obvious manner to the context
of the field $L$.
(This applies in particular to the notation $\Conefrac$, $\Czerofrac$, etc.\
introduced in Definition~\ref{notn:calh}.)

We write
$m, n$ for the standard generators of $\gM$ and
$\gN$. 
The existence of the nonzero map $\gN \to \gM^{(f)}$ implies that
$\alpha_i(\gN) \ge \alpha_{i+f}(\gM)$ for all $i$, and also that
$\prod_i a_i = \prod_i b_i$. Thanks to the latter we will lose no
generality by assuming that $a_i = b_i =1 $ for all $i$. Let
$\tilde m$ be the standard generator for $\gM^{(f)}$.  The map
$\gN \to \gM^{(f)}$ will (up to a scalar) have the form
$n_i \mapsto u^{x_i} \tilde m_{i}$ for integers $x_i$ satisfying
$px_{i-1}-x_i = s_i - r_{i+f}$ for all $i$; thus
$x_i = \alpha_{i}(\gN) - \alpha_{i+f}(\gM)$ for all $i$. Since the
characters $T(\gM)$ and $T(\gN)$ are conjugate we must have
$x_i \equiv d_i - c_{i+f} \pmod{p^{f'}-1}$ for all $i$ (\emph{cf}.\
Lemma~\ref{lem: generic fibres of rank 1 Kisin modules}).  Moreover,
the strong determinant condition $s_i + r_i = e'$ for all $i$ implies
that $x_i = x_{i+f}$.

We stress that we make no claims about the
optimality of the following result; we merely prove ``just what we
need'' for our applications. Indeed the estimates of
\cite{MR2562792,CarusoKisinVar} suggest that improvement should
be possible.

\begin{prop}\label{prop:irred-bound} 
We have  $\dim 
Y(\gM,\gN)_1 \le [K:\Qp]$. 

\end{prop}

\begin{remark}
  Since the image of $Y(\gM,\gN)_1$ in $\Ext^1_{\K{\F}}(\gM,\gN)$ lies
  in $\kExt^1_{\K{\F}}(\gM,\gN)$ with fibres that can be seen to  have dimension at most one,
  many cases of Proposition~\ref{prop:irred-bound} will  already follow from
  Remark~\ref{rem:half} (applied with $L$ in place of~$K$).
\end{remark}

\begin{proof}[Proof of Proposition~\ref{prop:irred-bound}]

Let $\gP = \gP(h)$ be an element of $\Ext^1_{\K{\F}}(\gM,\gN)$ whose
descent data can be enriched to give descent data from $K'$ to $K$, and
let $\tgP$ be such an enrichment.
By Lemma~\ref{lem:nonempty-then-map} (and the discussion preceding
that lemma) 
the \'etale $\varphi$-module $\gP[\frac 1u]$ is
isomorphic to $\gM[\frac 1u] \oplus \gM^{(f)}[\frac 1u]$. All
extensions of the $G_{L_{\infty}}$-representation $T(\gM[\frac 1u] \oplus
\gM^{(f)}[\frac 1u])$ to a representation of $G_{K_{\infty}}$ are
  isomorphic (and given by the induction of $T(\gM[\frac 1u])$ to~$G_{K_\infty}$), 
  so the same is true of the \'etale $\varphi$-modules with
descent data from $K'$ to $K$ that enrich the descent data on
$\gM[\frac 1u] \oplus \gM^{(f)}[\frac 1u]$. One such enrichment, which
we denote $P$, has $\hat\theta$ that interchanges $m$ and
$\tilde m$. Thus $\tgP[\frac 1u]$ is isomorphic to $P$.

As in the proof of Lemma~\ref{lem:nonempty-then-map},  the hypothesis that $T(\gM) \not\cong
T(\gN)$ implies that any non-zero map (equivalently, isomorphism) of
\'etale $\varphi$-modules with descent data $\lambda : \tgP[\frac 1u] \to P$ takes the submodule $\gN[\frac
1u]$ to $\gM^{(f)}[\frac 1u]$. We may scale the map $\lambda$ so that
it  restricts to the map $n_i \to u^{x_i} \tilde m_i$ on $\gN$.  
Then there is an element $\xi \in \F^\times$ so that
$\lambda$ induces multiplication by $\xi$ on the common quotients $\gM[\frac 1u]$.
That is,  the map $\lambda$ may be assumed to have the form
\numequation\label{eq:lambdamap}
\begin{pmatrix}
n_{i} \\ m_{i}
\end{pmatrix} \mapsto
\begin{pmatrix}
u^{x_i} & 0 \\ \nu_i & \xi   
\end{pmatrix}
\begin{pmatrix}
\tilde m_{i} \\ m_{i}
\end{pmatrix}
\end{equation}
for some $(\nu_i) \in \F((u))^{f'}$. The condition that the map
$\lambda$ commutes with the descent data from $K'$ to $L$ is seen to be
equivalent to the condition that nonzero terms in $\nu_i$ have degree
congruent to $c_i -d_i + x_i \pmod{p^{f'}-1}$; or equivalently, if we
define $\mu_i := \nu_i u^{-x_i}$ for all $i$, that the tuple $\mu = (\mu_i)$
is an element of the set $\Czerofrac = \Czerofrac(\gM,\gN)$
of Definition~\ref{notn:calh}.

The condition that $\lambda$ commutes with $\varphi$ can be checked to give
\begin{equation*}
   \varphi \begin{pmatrix}
n_{i-1} \\ m_{i-1}
\end{pmatrix}
=  \begin{pmatrix}
    u^{s_i} & 0 \\ \varphi(\nu_{i-1}) u^{r_{i+f}-x_i} - \nu_i u^{r_i-x_i}
    & u^{r_i} 
  \end{pmatrix}\begin{pmatrix}
n_{i} \\ m_{i}
\end{pmatrix}.
\end{equation*}
The extension $\gP$ is of the form $\gP(h)$, for some $h \in \Cone$
as in Definition~\ref{notn:calh}.
The lower-left entry of the first matrix on the right-hand side of the
above equation must then be $h_i$. Since $r_{i+f}-x_i = s_i - px_{i-1}$,
the resulting condition can be rewritten  as
\[ h_i= \varphi(\mu_{i-1}) u^{s_i} - \mu_i  u^{r_i},\]
or equivalently that $h = \mumap(\mu)$. Comparing with
Remark~\ref{rem:explicit-ker-ext}, we recover  the fact that the
extension class of $\gP$ is an element
of~$\kExt^1_{\K{\F}}(\gM,\gN)$, and the tuple  $\mu$ determines an 
element of the space $\Hzero$ defined as follows.

\begin{defn}\label{defn:calw} The map $\mumap \colon \Czerofrac \to \Conefrac$ induces a map
  $\Czerofrac/\Czero \to \Conefrac/\mumap(\Czero)$, which we also
  denote
  $\mumap$. We let  $\Hzero \subset \Czerofrac/\Czero$
  denote the subspace consisting of elements $\mu$ such that
$\mumap(\mu) \in \Cone/\mumap(\Czero)$. 
\end{defn}

By the discussion following 
Lemma~\ref{lem:explicit-complex}, an element $\mu \in \Hzero$ determines an 
extension $\gP(\mumap(\mu))$. Indeed,
Remark~\ref{rem:explicit-ker-ext} and the proof of \eqref{eqn:
  computing kernel of Ext groups} taken together show that there is a natural isomorphism,
in the style of Lemma~\ref{lem:explicit-complex}, between the morphism
$\mumap : \Hzero \to \Cone/\mumap(\Czero)$ and the connection map
$\Hom_{\K{\F}}(\gM,\gN[1/u]/\gN) \to \Ext^1_{\K{\F}}(\gM,\gN)$, with
$\im\mumap$ corresponding to $\kExt^1_{\K{\F}}(\gM,\gN)$.

Conversely, let $h$ be an element of $\mumap(\Czerofrac) \cap \Cone$,
and set $\nu_i = u^{x_i} \mu_i$. The condition that there is a Breuil--Kisin module
$\tgP$ with descent data from $K'$ to $K$ and $\xi \in \F^{\times}$ such that $\lambda : \tgP[\frac1u]
\to P$ defined as above is an isomorphism is precisely the condition that 
the map $\hat\theta$ on $P$ pulls back via $\lambda$ to a map that
preserves $\gP$. One computes that this pullback is
\begin{equation*}
  \hat\theta \begin{pmatrix}
n_{i} \\ m_{i}
\end{pmatrix}
= \xi^{-1} \begin{pmatrix}
     -\nu_{i+f}  & u^{x_i} \\
 (\xi^2-\nu_i \nu_{i+f}) u^{-x_i} & \nu_i
   \end{pmatrix}
  \begin{pmatrix}
  n_{i+f} \\ m_{i+f}
  \end{pmatrix}
\end{equation*}
recalling that $x_i =x_{i+f}$. 

We deduce that $\hat\theta$ preserves $\gP$
precisely when the $\nu_i$ are
integral and $\nu_i \nu_{i+f} \equiv \xi^2 \pmod{u^{x_i}}$ for
all~$i$.  For~$i$ with $x_i=0$
the latter condition is automatic given the former, which is
equivalent to the condition that $\mu_i$ and $\mu_{i+f}$ are both
integral. If instead $x_i > 0$, then we have the nontrivial
condition $\nu_{i+f} \equiv \xi^2 \nu_{i}^{-1} \pmod{u^{x_i}}$; in other
words that $\mu_i, \mu_{i+f}$ have $u$-adic valuation exactly $-x_i$,
and their principal parts determine one another via the equation
$\mu_{i+f}  \equiv \xi^2 (u^{2x_i} \mu_i )^{-1}
\pmod{1}$. 

 Let
$\mathbf{G}_{m,\xi}$ be the multiplicative group with parameter
$\xi$. We now (using the notation of Definition~\ref{defn:calw}) define  $\Hzero' \subset \Czerofrac/\Czero \times
\mathbf{G}_{m,\xi}$ to be the subvariety 
consisting of the pairs
$(\mu,\xi)$ with exactly the preceding properties; that is, we
regard~$\Czerofrac/\Czero$ as an Ind-affine space in the obvious way, and
define~$\Hzero'$ to be the pairs $(\mu,\xi)$ satisfying 
\begin{itemize}
\item if $x_i=0$ then $\val_i \mu = \val_{i+f} \mu =\infty$,
	and
\item if $x_i >0$ then $\val_i \mu = \val_{i+f} \mu = -x_i$ and $\mu_{i+f}
  \equiv \xi^2  (u^{2x_i} \mu_i)^{-1} \pmod{u^0}$
\end{itemize} 
where we write $\val_i \mu$ for the $u$-adic valuation of $\mu_i$, putting $\val_i \mu = \infty$ when $\mu_i$ is integral.

Putting all this together with~\eqref{eqn:Y fibre blahblah}, we find that the map 
\[ \Hzero' \cap (\Hzero \times \mathbf{G}_{m,\xi}) \to
  Y(\gM,\gN)_1 \] 
sending $(\mu,\xi)$ to the pair $(\gP,\tgP)$ is a well-defined
surjection, 
where $\gP =
\gP(\mumap(\mu))$, $\tgP$ is the enrichment of $\gP$ to a Breuil--Kisin
module with descent data from $K'$ to $K$ in which $\hat\theta$ is
pulled back to  $\gP$ from $P$ via the map $\lambda$ as in
\eqref{eq:lambdamap}. 
(Note that~$Y(\gM,\gN)_1$ is reduced and
of finite type, for
example by~\eqref{eqn:Y fibre}, so the surjectivity can be checked on
$\Fpbar$-points.)
In particular $\dim Y(\gM,\gN)_1 \le \dim \Hzero'.$

Note that $\Hzero'$ will be empty if for some $i$ we have $x_i > 0$ but
$x_i + c_i-d_i \not\equiv 0 \pmod{p^{f'}-1}$ (so that $\nu_i$ cannot
be a $u$-adic unit). 
Otherwise, the dimension of $\Hzero'$ is easily computed to be
$D =  1+\sum_{i=0}^{f-1} \lceil x_i/(p^{f'}-1) \rceil$ (indeed if~$d$ is the number of nonzero
$x_i$'s, then $\Hzero' \cong \Gm^{d+1} \times \Ga^{D-d}$), 
 and since
$x_i \le e'/(p-1)$ we find that $\Hzero'$ has dimension at most  $1 + \lceil e/(p-1)  \rceil f$.
This establishes  the bound  $\dim 
Y(\gM,\gN)_1 \le 1 + \lceil e/(p-1) \rceil f$. 

Since $p > 2$ this bound already establishes the theorem when $e
> 1$. 
 If instead $e=1$ the above bound gives  $\dim Y(\gM,\gN) \le [K:\Qp]
 + 1$. Suppose for the sake of
 contradiction that equality holds. This is only possible if $\Hzero'
 \cong \Gm^{f+1}$, $\Hzero' \subset \Hzero \times \mathbf{G}_{m,\xi}$,
 and $x_i = [d_i - c_i] > 0$ for all
 $i$. 
 Define $\mu^{(i)}
 \in \Czerofrac$ to be the element such that $\mu_{i} =
 u^{-[d_{i}-c_{i}]}$, and $\mu_j = 0$ for $j \neq i$. Let $\F''/\F$ be
 any finite extension such that $\#\F'' > 3$.  For each nonzero $z \in \F''$
 define $\mu_z = \sum_{j \neq i,i+f} \mu^{(i)}  + z \mu^{(i)} + z^{-1}
\mu^{(i+f)}$, so that  $(\mu_z, 1)$ is an element of  $\Hzero'(\F'')$.
Since  $\Hzero' \subset \Hzero \times \mathbb{G}_{m,\xi}$ and $\Hzero$ is
 linear, the differences between the $\mu_z$ for varying $z$ lie in
 $\Hzero(\F'')$, and (e.g.\ by considering $\mu_1 - \mu_{-1}$ and $\mu_1 -
 \mu_{z}$ for any $z \in \F''$ with $z\neq z^{-1}$) we deduce that each
 $\mu^{(i)}$ lies in $\Hzero$. In particular 
each 
$\mumap(\mu^{(i)})$ lies in $\Cone$.

If $(i-1,i)$ were not a transition then (since $e=1$) we would have
either $r_i =0 $ or $s_i = 0$. The former would contradict
$\mumap(\mu^{(i)}) \in \Cone$ (since the $i$th component of
$\mumap(\mu^{(i)})$ would be $u^{-[d_i-c_i]}$, of negative degree),
and similarly the latter would contradict $\mumap(\mu^{(i-1)}) \in
\Cone$. Thus $(i-1,i)$ is a transition for all $i$. In fact the same
observations show more precisely that $r_i \ge  x_i = [d_i-c_i]$ and $s_i
\ge p x_{i-1} = p [d_{i-1}-c_{i-1}]$.  Summing these inequalities and subtracting
$e'$ we obtain $0 \ge p [d_{i-1}-c_{i-1}] - [c_i-d_i]$, and comparing
with  \eqref{eq:gammastar} 
shows that we must also have $\gamma_i^*=0$ for
all $i$. Since $e=1$ and $(i-1,i)$ is a transition for all $i$ the refined shape of the pair $(\gM,\gN)$ is
automatically maximal;\ but then we are in the exceptional case of
Proposition~\ref{prop:ker-ext-maximal}, 
which (recalling the proof of that Proposition) implies that $T(\gM) \cong T(\gN)$. 
This is the desired contradiction.
\end{proof}

\subsection{Irreducible components}\label{subsec: irred components}
We can now use our results on families of extensions of characters to
classify the irreducible components of the stacks~$\cC^{\tau,\BT,1}$
and~$\cZ^{\tau,1}$. In Section~\ref{sec: picture} we will combine
these results with results coming from Taylor--Wiles patching (in
particular the results of~\cite{geekisin,emertongeerefinedBM}, which
we combine in Appendix~\ref{sec:appendix on geom BM}) to describe the
closed points of each irreducible component of~$\cZ^{\tau,1}$ in terms
of the weight part of Serre's conjecture.
\begin{cor}
  \label{cor: the C(J) are the components}Each irreducible component
  of~$\cC^{\tau,\BT,1}$ is of the form~$\overline{\cC}(J)$ for some~$J$;
  conversely, each~$\overline{\cC}(J)$ is an irreducible component of~$\cC^{\tau,\BT,1}$. 
\end{cor}
\begin{rem}
  \label{rem: haven't yet proved the C(J) are distinct}Note that at
  this point we have not established that different sets~$J$ give
  distinct irreducible components~$\overline{\cC}(J)$; we will prove this in Section~\ref{subsec: map to
  Dieudonne stack} below by a consideration of Dieudonn\'e
modules. 
\end{rem}
\begin{proof}[Proof of Corollary~{\ref{cor: the C(J) are the
    components}}]By~Proposition~\ref{prop: C tau is
    equidimensional of the expected dimension}, $\cC^{\tau,\BT,1}$ is
  equidimensional of dimension~$[K:\Qp]$. By construction, the~$\overline{\cC}(J)$ are irreducible
  substacks of~$\cC^{\tau,\BT,1}$, and by Theorem~\ref{thm:
    dimension of refined shapes} they also have dimension~$[K:\Qp]$, so  they are in fact
  irreducible components by~\cite[\href{https://stacks.math.columbia.edu/tag/0DS2}{Tag 0DS2}]{stacks-project}. 
  
  By Theorem~\ref{thm: irreducible Kisin
    modules can be ignored} and Theorem~\ref{thm: dimension of refined
    shapes}, we see that there is a closed substack
  $\cC_{\mathrm{small}}$ of~$\cC^{\tau,\BT,1}$ of dimension strictly
  less than~$[K:\Qp]$, with the property that every finite type point
  of~$\cC^{\tau,\BT,1}$ is a point of at least one of the~$\overline{\cC}(J)$
  or of~$\cC_{\mathrm{small}}$ (or both). (Indeed, we can
  take~$\cC_{\mathrm{small}}$ to be the union of the
  stack~$\cC_{\mathrm{irred}}$ of Theorem~\ref{thm: irreducible Kisin
    modules can be ignored} and the stacks~$\overline{\cC}(J,r)$ for
  non-maximal shapes~$(J,r)$.) 
  Since  $\cC^{\tau,\BT,1}$ is
  equidimensional of dimension~$[K:\Qp]$, it follows 
  that the~$\overline{\cC}(J)$ exhaust
  the irreducible components of~$\cC^{\tau,\BT,1}$, as required. 
\end{proof}

We now deduce a classification of the
irreducible components of $\cZ^{\tau,1}$; Theorem~\ref{thm:stack
  version of geometric Breuil--Mezard} below is a considerable refinement of
this, giving a precise description of the finite type points of the
irreducible components in terms of the weight part of Serre's conjecture.
  \begin{cor}
    \label{cor: components of Z are exactly the Z(J)}The irreducible
    components of~$\cZ^{\tau,1}$ are precisely the~$\overline{\cZ}(J)$
    for~$J\in\cP_\tau$, and if $J\ne J'$ then~$\overline{\cZ}(J)\ne\overline{\cZ}(J')$.
  \end{cor}
  \begin{proof}
    By Theorem~\ref{thm: identifying the vertical components}, if
    $J\in\cP_\tau$ then~$\overline{\cZ}(J)$ is an irreducible component of
    ~$\cZ^{\tau,1}$. 
    Furthermore, these~$\overline{\cZ}(J)$ are pairwise
    distinct by Theorem~\ref{thm: unique serre weight}.

    Since the morphism
    $\cC^{\tau,\BT,1}\to\cZ^{\tau,1}$ is scheme-theoretically
    dominant, it follows from Corollary~\ref{cor: the C(J) are the
      components} that each irreducible component of $\cZ^{\tau,1}$
    is dominated by some~$\overline{\cC}(J)$.  Applying Theorem~\ref{thm:
      identifying the vertical components} again, we see that if
    $J\notin\cP_\tau$ then~$\overline{\cC}(J)$ does not dominate an irreducible
    component, as required.
  \end{proof}

\subsection{Dieudonn\'e modules and the morphism to the gauge stack}\label{subsec: map to
  Dieudonne stack} 
We now study the images of the irreducible components  $\overline{\cC}(J)$
in the gauge stack $\cG_\eta$; 
this amounts to computing
the Dieudonn\'e modules and Galois
representations associated to the extensions of Breuil--Kisin modules that we
considered in Section~\ref{sec: extensions of rank one Kisin
  modules}. 
Suppose throughout this subsection that $\tau$ is a non-scalar type,
and that $(J,r)$ is a maximal refined shape. 
Recall that in the cuspidal case this entails that $i \in J$ if and
only if $i + f \not\in J$.

\begin{lemma}
	\label{lem:Dieudonne modules}
Let $\gP \in \Ext^1_{\K{\F}}(\gM,\gN)$ be an extension of type $\tau$
and refined shape $(J,r)$. Then for $i \in \Z/f'\Z$ we have $F=0$ on $D(\gP)_{\eta,i-1}$ if $i\in
J$, while $V=0$ on $D(\gP)_{\eta,i}$ if
$i\notin J$.
\end{lemma}
\begin{proof}
Recall that $D(\gP) = \gP/u\gP$. Let $w_i$ be the image of $m_i$ in
$D(\gP)$ if $i \in J$, and let $w_i$ be the image of $n_i$ in $D(\gP)$
if $i \not\in J$.  It follows easily from the
definitions that $D(\gP)_{\eta,i}$ is generated over~$\F$ by $w_i$.

Recall that the actions of $F,V$ on $D(\gP)$ are as specified in
Definition~\ref{def: Dieudonne module formulas}. In particular $F$ is
induced by $\varphi$, while $V$ is $\czero^{-1} \mathfrak{V}$ mod $u$ where $\mathfrak{V}$ is the
unique map on $\gP$ satisfying $\mathfrak{V} \circ \varphi =
E(u)$, and $\czero = E(0)$.  For the Breuil--Kisin module $\gP$, we have
\[\varphi(n_{i-1}) = b_i u^{s_i} n_i,\qquad \varphi(m_{i-1}) = a_i u^{r_i}
m_i + h_i n_i,\] and so
 one checks (using that $E(u) = u^{e'}$ in $\F$)
that 
$$\mathfrak{V}(m_i) = a_i^{-1} u^{s_i} m_{i-1} - a_{i}^{-1} b_i^{-1}
h_i n_{i-1} , \qquad \mathfrak{V}(n_i) = b_i^{-1} u^{r_i} n_{i-1}.$$

From Definition~\ref{df:extensions of shape $J$} and the
discussion immediately following it, we recall that if $(i-1,i)$ is not a transition
then $r_i = e'$,
$s_i=0$, and $h_i$ is divisible by $u$ (the latter because nonzero
terms of $h_i$ have degrees congruent to $r_i+c_i-d_i
\pmod{p^{f'}-1}$, and $c_i \not\equiv d_i$ since $\tau$ is non-scalar).
On the other hand if $(i-1,i)$ is a transition, then $r_i , s_i >0$,
and nonzero terms of $h_i$ have degrees divisible by
$p^{f'}-1$; in that case we write $h_i^0$ for the constant coefficient
of $h_i$, and we remark that $h_i^0$ does not vanish identically on $\Ext^1_{\K{\F}}(\gM,\gN)$.

Suppose, for instance, that $i-1 \in J$ and $i \in J$. Then
$w_{i-1}$ and $w_i$ are the images in $D(\gP)$ of $m_{i-1}$ and
$m_{i}$.  From the above formulas we see that $u^{r_i} = u^{e'}$ and
$h_i$ are both divisible by $u$, while on the other hand $u^{s_i} = 1$. We
deduce that $F(w_{i-1}) = 0$ and $V(w_i) = \czero^{-1} a_i^{-1}
w_{i-1}$.  Computing along similar lines,  it is easy to check the following four
cases.

\begin{enumerate}
\item $i-1\in J,i\in J$. Then  $F(w_{i-1}) = 0$ and $V(w_i) = \czero^{-1} a_i^{-1}
w_{i-1}$.

\item $i-1\notin J,i\notin J$. Then $F(w_{i-1})=b_{i}w_{i}$, $V(w_{i})=0$.
\item\label{item: interesting case} $i-1\in J$, $i\notin J$. Then $F(w_{i-1})=h_{i}^0w_{i}$, $V(w_{i})=0$.
\item $i-1\notin J$, $i\in J$. Then $F(w_{i-1})=0$,
  $V(w_{i})=-\czero^{-1} a_{i}^{-1}b_{i}^{-1}h_{i}^0w_{i-1}$.
\end{enumerate}
In particular, if $i\in J$ then $F(w_i)=0$, while if
$i\notin J$ then $V(w_{i+1})=0$. 
\end{proof}

Since $\cC^{\tau,\BT}$ is flat over $\cO$ by Corollary~\ref{cor: Kisin
  moduli
  consequences of local models}, 
it follows from Lemma~\ref{lem: maps to gauge stack as Cartier
    divisors} that the natural morphism $\cC^{\tau,\BT} \to \cG_{\eta}$ 
is determined by an $f$-tuple of effective Cartier divisors $\{\cD_j\}_{0 \le j < f}$
lying in the special fibre $\cC^{\tau,\BT,1}$. 
Concretely, 
$\cD_j$ is the zero locus of~ $X_j$, which is the zero locus
of~$F:D_{\eta,j}\to D_{\eta,j+1}$. 
The zero locus of $Y_j$ (which is the zero locus of~$V:D_{\eta,j+1}\to
D_{\eta,j}$) is another
Cartier divisor $\cD_j'$. 
Since $\cC^{\tau,\BT,1}$ is reduced,
we conclude that each of $\cD_j$ and $\cD_j'$ is simply a union of irreducible components
of  $\cC^{\tau,\BT,1}$, each component appearing precisely once in
precisely one of either $\cD_j$ or $\cD_j'$.

\begin{prop}
\label{prop:Dieudonne divisors}
$\cD_j$ is equal to the union of the irreducible components~$\overline{\cC}(J)$ of
$\cC^{\tau,\BT,1}$ for those $J$ that contain
$j+1$. 
\end{prop}
\begin{proof}
Lemma~\ref{lem:Dieudonne modules} shows
that if $j+1\in J$, then $X_j=0$, while
if $j+1\notin J$, then $Y_j=0$. In the latter case, by an inspection
of case~\eqref{item: interesting case} of the proof of Lemma~\ref{lem:Dieudonne modules}, we have
$X_j=0$ if and only if  $j\in J$ 
and
$h_{j+1}^0=0$. Since~$h_{j+1}^0$ does not vanish identically on an
irreducible component, we see that the irreducible components on which $X_j$
vanishes identically are precisely those for which $j+1\in J$, as
claimed. 
\end{proof}

\begin{thm}
  \label{thm: components of C}The algebraic stack~$\cC^{\tau,\BT,1}$
  has precisely $2^f$ irreducible components, namely the irreducible substacks~$\overline{\cC}(J)$. 
\end{thm}
\begin{proof}
By Corollary~\ref{cor: the C(J) are the components}, we need only show
that if~$J\ne J'$   then $\overline{\cC}(J)\ne\cC(J')$; but this is immediate
from Proposition~\ref{prop:Dieudonne divisors}.
\end{proof}

\section{Moduli stacks of Galois representations and the geometric
Breuil--M\'ezard conjecture}\label{sec: picture} 
We now make a more detailed study of the stacks~$\cZ^{\dd}$
and~$\cZ^{\tau}$. In particular, we 
prove a ``geometric Breuil--M\'ezard'' result, showing in particular that
the finite type points of each irreducible
components 
are precisely described by the weight part of Serre's conjecture. We
also prove a new result on the structure of potentially Barsotti--Tate
deformation rings, Proposition~\ref{prop: generically reduced special fibre deformation
    ring}, showing that their special fibres are generically reduced.

\subsection{Generic reducedness of \texorpdfstring{$\Spec R^{\tau,\BT}_{\rbar}/\varpi$}{a
    deformation ring}}
\label{subsec: generically
  reduced}
We return to the setting of Subsection~\ref{subsec:Galois
  deformation rings}:\ that is, we fix a finite type point $\Spec\F'\to\cZ^{\tau,a}$, where $\F'/\F$
is a finite extension, and let $\rbar:G_K\to\GL_2(\F')$ be the
 corresponding Galois representation.
It follows from Corollary~\ref{cor: R tau BT is a versal
  ring to Z-hat} that $\Spec R^{\tau,a}$ is a closed subscheme of
$\Spec R^{\tau,\BT}_{\rbar}/\varpi^a$, but we have no reason to
believe that equality holds.
It follows from Lemma~\ref{lem: C 1 and Z 1 are the underlying reduced substacks},
together with Lemma~\ref{lem: generic reducedness passes to completions} below, that
$\Spec R^{\tau,1}$ is the underlying reduced subscheme
of~$\Spec R^{\tau,\BT}_{\rbar}/\varpi$, so that equality holds in the case $a = 1$
if and only
if~$\Spec R^{\tau,\BT}_{\rbar}/\varpi$ is reduced. Again, we have no
reason to believe that this holds in general, but the main result of
this section is
  Proposition~\ref{prop: generically reduced special fibre deformation ring} below,
showing that $\Spec R^{\tau,\BT}_{\rbar}/\varpi$  is generically
reduced. We will use this in the proof of our geometric
Breuil--M\'ezard result below. (Recall that a scheme is
generically reduced if it contains an open reduced subscheme whose
underlying topological space is dense.  In the case of a Noetherian
affine scheme $\Spec A$, this is equivalent to requiring
that the localisation of $A$ at each of its minimal
primes is reduced.)

\begin{prop}
  \label{prop: generically reduced special fibre deformation ring}
For any tame type~$\tau$, the scheme $\Spec R^{\tau,\BT}_{\rbar}/\varpi$ is generically reduced, with
  underlying reduced subscheme $\Spec R^{\tau,1}$. 
\end{prop}
 We will deduce
 Proposition~\ref{prop: generically reduced special fibre deformation ring}
 from the following global statement.

  \begin{prop}
    \label{prop: existence of dense open substack of R such that C is a mono}
Let~$\tau$ be a tame type. There is a dense open substack $\cU$ of $\cZ^{\tau}$
      such that $\cU_{/\F}$ is reduced.
  \end{prop}
  \begin{proof}
	  The proposition will follow from an application of
	  Proposition~\ref{prop:opens},
	  and the key to this application will be to find a
	  candidate open substack $\cU^1$ of $\cZ^{\tau,1}$,
which we will do using our study of the irreducible components
of $\cC^{\tau,\BT,1}$ and $\cZ^{\tau,1}$.  

  Recall that, for each $J \in \cP_{\tau}$,
 we let $\overline{\cZ}(J)$ denote the scheme-theoretic image of
 $\overline{\cC}(J)$ under the proper morphism
 $\cC^{\tau,\BT,1} \to \cZ^{\tau,1}$. 
 Each $\overline{\cZ}(J)$ 
 is a closed substack of $\cZ^{\tau,1}$, and so,
  if we let $\cV(J)$ be the complement in $\cZ^{\tau,1}$ of the
  union of the $\overline{\cZ}(J')$ for all $J'\ne J$, $J'\in\cP_\tau$, then $\cV(J)$ is
  a dense open open substack of~$\cZ^{\tau,1}$, by Corollary~\ref{cor:
  components of Z are exactly the Z(J)}.
The preimage $\cW(J)$ of $\cV(J)$ in $\cC^{\tau,\BT,1}$ is therefore a
dense open
substack of $\overline{\cC}(J)$. 
Possibly shrinking $\cW(J)$ further, we may
  suppose by Proposition~\ref{prop:C to Z mono}
that the morphism $\cW(J)\to
\cZ^{\tau,1}$ is a monomorphism. 

The complement $|\overline{\cC}(J)|\setminus |\cW(J)|$ is a closed subset
of $|\overline{\cC}(J)|$, and thus of
$|\cC^{\tau,\BT,1}|$, and its image under the proper morphism
$\cC^{\tau,\BT,1}\to \cZ^{\tau,1}$ is a closed
subset of $|\cZ^{\tau,\BT,1}|$, which is (e.g.\ for dimension reasons) a proper
closed subset of~$|\overline{\cZ}(J)|$; so if we let
$\cU(J)$ be the complement in $\cV(J)$ of this image, then $\cU(J)$ is
open and dense in~$\overline{\cZ}(J)$, 
and the morphism $\cC^{\tau,\BT,1}\times_{\cZ^{\tau,1}}
\cU(J)\to \cU(J)$ is a monomorphism. Set $\cU^1=\cup_J \cU(J)$.
Since the $\cU(J)$ are pairwise disjoint
by construction, $\cC^{\tau,\BT,1}\times_{\cZ^{\tau,1}}
\cU^1\to \cU^1$ is again a monomorphism. By construction (taking into account
Corollary~\ref{cor: components of Z are exactly the Z(J)}), $\cU^1$ is dense in $\cZ^{\tau,1}$.

Now let $\cU$ denote the open substack of $\cZ^{\tau}$ corresponding
to $\cU^1$.  Since $|\cZ^{\tau}| = |\cZ^{\tau,1}|$,
we see that $\cU$ is dense in $\cZ^\tau$. We have seen in the previous
paragraph that the statement of Proposition~\ref{prop:opens}~(5) holds
(taking~$a=1$, $\cX=\cC^{\tau,\BT}$, and~$\cY=\cZ$); so
Proposition~\ref{prop:opens} implies that, for each $a \geq 1$,
the closed immersion
      $\cU\times_{\cZ^{\tau}}\cZ^{\tau,a} \hookrightarrow \cU\times_{\cO} \cO/\varpi^a$
      is an isomorphism.

      In particular,
      since the closed immersion $\cU^1 = \cU\times_{\cZ^{\tau}} \cZ^{\tau,1}
      \to \cU_{/\F}$ is an isomorphism, 
      we may regard $\cU_{/\F}$ as an open substack
      of $\cZ^{\tau,1}$.  Since $\cZ^{\tau,1}$ is
      reduced,
      by Lemma~\ref{lem: C 1 and Z 1 are the underlying reduced substacks},
    so is its open substack $\cU_{/\F}$.
    This completes the proof of the proposition.
  \end{proof}

\begin{cor}
    \label{cor: existence of dense open substack of Z with unique
      Kisin module}
Let~$\tau$ be a tame type. There is a dense open substack $\cU$ of
$\cZ^{\tau}$ such that we have an isomorphism
	$\cC^{\tau,\BT} \times_{\cZ^{\tau}} \cU \iso \cU,$
	as well as isomorphisms
	$$\cU\times_{\cZ^{\tau}} \cC^{\tau,\BT, a} \iso 
	\cU\times_{\cZ^{\tau}}\cZ^{\tau,a}
	\iso \cU\times_{\cO} \cO/\varpi^a,$$
	for each $a \geq 1$.
  \end{cor}
  \begin{proof}This follows from
    Proposition~\ref{prop: existence of dense open substack of R such
      that C is a mono} and Proposition~\ref{prop:opens}.
  \end{proof}

\begin{remark}
	\label{rem:U^a = U mod pi^a} 
More colloquially, Corollary~\ref{cor: existence of dense open substack of Z with unique
      Kisin module} shows that for each tame type~$\tau$, there is an
    open dense substack~$\cU$ of~$\cZ^\tau$ consisting of Galois
    representations which have a unique Breuil--Kisin
    model of type~$\tau$. 
\end{remark}

  \begin{lemma}
    \label{lem: flat pullbacks are reduced scheme theoretically
      dense} 
      If $\cU$ is an open substack of $\cZ^{\tau}$ satisfying the conditions
      of Proposition~{\em \ref{prop: existence of dense open substack of R such that C is a mono}},
      and if $T \to \cZ^{\tau}_{/\F}$ is a smooth morphism
      whose source is a scheme, 
      then $T\times_{\cZ^{\tau}_{/\F}} \cU_{/\F}$ is reduced,
      and is a dense open subscheme of $T$.
  \end{lemma}
  \begin{proof}
Since $\cZ^{\tau,1}_{/\F}$ is a 
    Noetherian algebraic stack (being of finite presentation over $\Spec \F$),
    the open
    immersion \[\cU_{/\F}\to \cZ^{\tau}_{/\F}\]
    is quasi-compact (\cite[\href{https://stacks.math.columbia.edu/tag/0CPM}{Tag 0CPM}]{stacks-project}). 
    Since $T\to \cZ^{\tau}_{/\F}$ is flat (being smooth, by assumption), the
    pullback $T\times_{\cZ^{\tau}_{/\F}} \cU_{/F} \to T$ is an open immersion
    with dense image;\ here we use the fact that for a quasi-compact
	    morphism, the property of being scheme-theoretically dominant
	    is preserved by flat base-change, together with the fact that
	    an open immersion with dense image induces a scheme-theoretically
	    dominant morphism after passing to underlying reduced substacks.
    Since the source of this morphism is smooth over the reduced algebraic
    stack $\cU_{/\F}$, it is itself reduced.
  \end{proof}

The following result is standard, but we recall the proof for the sake of completeness.

  \begin{lem}
    \label{lem: generic reducedness passes to completions}
Let $T$ be a Noetherian scheme, all of whose local rings at
    finite type points are $G$-rings. If $T$ is reduced {\em (}resp.\ generically reduced{\em)},
    then so are all of its complete local rings at finite type points.
  \end{lem}
  \begin{proof}
    Let $t$ be a finite type point of~$T$, and write
    $A:=\cO_{T,t}$. Then $A$ is a (generically) reduced local $G$-ring,
    and we need to show that its completion $\widehat{A}$ is also
    (generically) reduced. Let $\widehat{\p}$ be a (minimal) prime of
    $\widehat{A}$; since $A\to\widehat{A}$ is (faithfully) flat,
    $\widehat{\p}$ lies over a (minimal) prime $\p$ of $A$ by the
    going-down theorem.

Then $A_\p$ is reduced by assumption, and we need to show that
$\widehat{A}_{\widehat{\p}}$ is
reduced. By~\cite[\href{http://stacks.math.columbia.edu/tag/07QK}{Tag
  07QK}]{stacks-project}, it is enough to show that the morphism $A\to
\widehat{A}_{\widehat{\p}}$ is regular. Both $A$ and $\widehat{A}$ are
$G$-rings (the latter
by~\cite[\href{http://stacks.math.columbia.edu/tag/07PS}{Tag
  07PS}]{stacks-project}), so the composite \[A\to \widehat{A}\to
(\widehat{A}_{\widehat{\p}})^{\widehat{}}\] is a composite of regular
morphisms, and is thus a regular morphism
by~\cite[\href{http://stacks.math.columbia.edu/tag/07QI}{Tag
  07QI}]{stacks-project}. 

This composite factors through the natural morphism $A_{\p}\to
(\widehat{A}_{\widehat{\p}})^{\widehat{}}$, so this morphism is also
regular. Factoring it as the composite \[A_{\p}\to \widehat{A}_{\widehat{\p}}\to
(\widehat{A}_{\widehat{\p}})^{\widehat{}},\]it follows
from~\cite[\href{http://stacks.math.columbia.edu/tag/07NT}{Tag
  07NT}]{stacks-project} that $A_{\p}\to \widehat{A}_{\widehat{\p}}$ is
regular, as required.
  \end{proof}
  \begin{proof}[Proof of Proposition~{\ref{prop: generically reduced special fibre deformation
    ring}}]By Corollary~\ref{cor: R tau BT is a versal ring to Z-hat},
  we have a versal morphism \[\Spf
  R_{\rbar}^{\tau,\BT}/\varpi\to \cZ^\tau_{/\F}.\] 
Since $\cZ^{\tau}_{/\F}$ is an algebraic stack of finite presentation
over $\F$ (as $\cZ^{\tau}$ is a
$\varpi$-adic formal algebraic stack of finite presentation over $\Spf \cO$), 
we may apply~\cite[\href{https://stacks.math.columbia.edu/tag/0DR0}{Tag 0DR0}]{stacks-project} to
this morphism so as to find a
smooth morphism $V\to \cZ^{\tau}_{/\F}$ with source a finite
type $\cO/\varpi$-scheme, and a point $v\in V$ with residue
field~$\F'$, such that there is an isomorphism
$\widehat{\cO}_{V,v}\cong R^{\tau,\BT}_{\rbar}/\varpi$,
compatible with the given morphism to~$\cZ^{\tau}_{/\F}$.
Proposition~\ref{prop: existence of dense open substack of R such that C is a mono}
and Lemma~\ref{lem: flat pullbacks are reduced scheme theoretically
      dense} taken together show that $V$ is generically reduced, and so the result follows from Lemma~\ref{lem: generic reducedness passes to completions}.
\end{proof}

\subsection{The geometric Breuil--M\'ezard conjecture}\label{subsec:
  geometric BM} We now study the irreducible components of
$\cZ^{\dd,1}$. We do this by a slightly indirect method, defining
certain formal sums of these irreducible components which we then
compute via the geometric Breuil--M\'ezard conjecture, and in
particular the results of Appendix~\ref{sec:appendix on geom
  BM}. 

By Lemma~\ref{lem: C 1 and Z 1 are the underlying reduced substacks} and Proposition~\ref{prop: dimensions of the Z stacks}, $\cZ^{\dd,1}$
is reduced and equidimensional, and each $\cZ^{\tau,1}$ is a union of some of
its irreducible components. Let $K(\cZ^{\dd,1})$ be the free abelian group
generated by the irreducible components of $\cZ^{\dd,1}$.   We say that an element of $K(\cZ^{\dd,1})$ is
\emph{effective} if the multiplicity of each irreducible
component is nonnegative.  We say that an element of $K(\cZ^{\dd,1})$ is
\emph{reduced and effective} if the multiplicity of each irreducible
component is $0$ or~$1$. 

Let $x$ be a finite type point of $\cZ^{\dd,1}$, corresponding to a
representation $\rbar:G_K\to\GL_2(\F')$. As in Section~\ref{subsec:Galois
  deformation rings}, there is a quotient~$R_x$ of the framed deformation
ring $R_{\rbar}^{[0,1]}/\varpi$ which is a versal ring to
$\cZ^{\dd,1}$ at~$x$; each  $R^{\tau,1}$ is a quotient
of~$R_x$. Indeed, since $\cZ^{\tau,1}$ is a union of irreducible
components of  $\cZ^{\dd,1}$, $\Spec R^{\tau,1}$ is a union of
irreducible components of~$\Spec R_x$. 

Let $K(R_x)$ be the free abelian group generated by the irreducible
components of $\Spec R_x$. By~\cite[\href{https://stacks.math.columbia.edu/tag/0DRB}{Tag 0DRB},\href{https://stacks.math.columbia.edu/tag/0DRD}{Tag 0DRD}]{stacks-project}, there is a
natural multiplicity-preserving surjection from the set of irreducible
components of $\Spec R_x$ to the set of irreducible components
of~$\cZ^{\dd,1}$ which contain~$x$. Using this surjection, we can
define a group homomorphism \[K(\cZ^{\dd,1})\to K(R_x)\]
in the following way: we send any irreducible component $\cZ$ of
$\cZ^{\dd,1}$ which contains~$x$ to the formal sum of the
irreducible components of~$\Spec R_x$ in the preimage of~$\cZ$ under
this surjection, and we send every other irreducible component to~$0$.

\begin{lem}
  \label{lem: can read off reduced and effective from local rings}An
  element~$\cTbar$ of~$K(\cZ^{\dd,1})$ is effective if
  and only if for every finite type point $x$ of~$\cZ^{\dd,1}$, the
  image of $\cTbar$ in $K(R_x)$ is effective. We have $\cTbar=0$ if
  and only if its image is $0$ in every $K(R_x)$.
\end{lem}
\begin{proof}
  The ``only if'' direction is trivial, so we need only consider the
  ``if'' implication. Write $\cTbar=\sum_{\cZbar}a_{\cZbar}\cZbar$,
  where the sum runs over the irreducible components $\cZbar$ of
  $\cZ^{\dd,1}$, and the $a_{\cZbar}$ are integers.

  Suppose first
  that the image of $\cTbar$ in $K(R_x)$ is effective; we then have to
  show that each $a_{\cZbar}$ is nonnegative.
  To see this, fix an irreducible component~$\cZbar$, and choose~$x$
  to be a finite type point of $\cZ^{\dd,1}$ which is contained
  in~$\cZbar$ and in no other irreducible component
  of~$\cZ^{\dd,1}$. Then the image of~$\cTbar$ in $K(R_x)$ is equal to
  $a_{\cZbar}$ times the sum of the irreducible components
  of~$\Spec R_x$. By hypothesis, this must be effective,
  which implies that $a_{\cZbar}$ is nonnegative, as required.

Finally, if the image of $\cTbar$ in $K(R_x)$ is $0$, then
$a_{\cZbar}=0$; so if this holds for all~$x$, then~$\cTbar=0$.
\end{proof}

For each tame type $\tau$, we let $\cZ(\tau)$ denote the formal sum of
the irreducible components of $\cZ^{\tau,1}$, considered as an
element of $K(\cZ^{\dd,1})$. By Lemma~\ref{lem:write Serre weight as
  linear combination of types}, for each non-Steinberg Serre weight
$\sigmabar$ of $\GL_2(k)$, there are integers $n_\tau(\sigmabar)$ such
that $\sigmabar=\sum_\tau n_\tau(\sigmabar)\sigmabar(\tau)$ in the
Grothendieck group of mod $p$ representations of $\GL_2(k)$, where the
$\tau$ run over the tame types. We set
\[\cZ(\sigmabar):=\sum_\tau n_\tau(\sigmabar)\cZ(\tau)\in K(\cZ^{\dd,1}).\]
The integers $n_\tau(\sigmabar)$ are not necessarily unique, but it
follows from the following result that $\cZ(\sigmabar)$  is independent
of the choice of $n_\tau(\sigmabar)$, and is reduced and effective.

\begin{thm}
  \label{thm:stack version of geometric Breuil--Mezard}
  \begin{enumerate}
  \item   Each $\cZ(\sigmabar)$  is an
    irreducible component of $\cZ^{\dd,1}$.
  \item The finite type points of $\cZ(\sigmabar)$ are precisely the
    representations $\rbar:G_K\to\GL_2(\F')$ having $\sigmabar$ as a
    Serre weight.
 \item For each tame type $\tau$, we have
   $\cZ(\tau)=\sum_{\sigmabar\in\JH(\sigmabar(\tau))}\cZ(\sigmabar)$. 
  \item Every irreducible component of $\cZ^{\dd,1}$ is of the form
    $\cZ(\sigmabar)$ for some unique Serre weight~$\sigmabar$. 
    \item For each tame type $\tau$, and each $J\in \cP_\tau$,
	    we have $\cZ(\sigmabar_J)=\overline{\cZ}(J)$.
  \end{enumerate}

\end{thm}
\begin{proof} Let~$x$ be a finite type point of $\cZ^{\dd,1}$
  corresponding to $\rbar:G_K\to\GL_2(\F')$, and
  write $\cZ(\sigmabar)_x$, $\cZ(\tau)_x$ for the images in $K(R_x)$ of
  $\cZ(\sigmabar)$ and $\cZ(\tau)$ respectively. Each $\Spec R^{\tau,1}$
  is a closed subscheme of $\Spec R^{\square}$, the universal framed deformation
  $\cO_{E'}$-algebra for~$\rbar$, so we may regard the $\cZ(\tau)_x$
  as formal sums (with multiplicities) of irreducible subschemes
  of~$\Spec R^{\square}/\pi$. 

  By definition, $\cZ(\tau)_x$ is just the underlying cycle of
  $\Spec R^{\tau,1}$. By Proposition~\ref{prop: generically reduced
    special fibre deformation ring}, this is equal to the underlying
  cycle of $\Spec R_{\rbar}^{\tau,\BT}/\varpi$. Consequently,
  $\cZ(\sigmabar)_x$ is the cycle denoted by~$\cC_{\sigmabar}$ in
  Appendix~\ref{sec:appendix on geom BM}. It follows from
  Theorem~\ref{thm: geometric BM} that:
  \begin{itemize}
  \item $\cZ(\sigmabar)_x$ is effective, and is nonzero precisely when
    $\sigmabar$ is a Serre weight for~$\rbar$.
  \item For each tame type $\tau$, we have
   $\cZ(\tau)_x=\sum_{\sigmabar\in\JH(\sigmabar(\tau))}\cZ(\sigmabar)_x$.
  \end{itemize}
Applying Lemma~\ref{lem: can read off reduced and effective from local
  rings}, we see that each $\cZ(\sigmabar)$ is effective, and that~(3)
holds. Since $\cZ^{\tau,1}$ is reduced, $\cZ(\tau)$ is reduced and
effective, so it follows from~(3) that each~$\cZ(\sigmabar)$ is
reduced and effective. Since $x$ is a finite type point of
$\cZ(\sigmabar)$ if and only if $\cZ(\sigmabar)_x\ne 0$, we have also proved~(2).

Since every irreducible component of~$\cZ^{\dd,1}$ is an irreducible
component of some~$\cZ^{\tau,1}$, in order to prove~(1) and~(4) it
suffices to show that for each~$\tau$, every irreducible component
of~$\cZ^{\tau,1}$ is of the form~$\cZ(\sigmabar_J)$ for some~$J$, and
that each~$\cZ(\sigmabar_J)$ is irreducible. Now, by
Corollary~\ref{cor: components of Z are exactly the Z(J)}, we know
that~$\cZ^{\tau,1}$ has exactly~$\#\cP_\tau$ irreducible components,
namely the~$\cZbar(J')$ for~$J'\in\cP_\tau$. On the other hand,
the~$\cZ(\sigmabar_J)$ are reduced and effective, and since there
certainly exist representations admitting~$\sigmabar_J$ as their
unique Serre weight, it follows from~(2) that for each~$J$, there must
be a~$J'\in\cP_\tau$ such that~$\cZbar(J')$ contributes
to~$\cZ(\sigmabar_J)$, but not to any~$\cZ(\sigmabar_{J''})$
for~$J''\ne J$. 

Since~$\cZ(\tau)$ is reduced and effective, and the sum in~(3) is
over~$\#\cP_\tau$ weights~$\sigmabar$, it follows that we in fact have
$\cZ(\sigmabar_J)=\cZbar(J')$. This proves~(1) and~(4), and to
prove~(5), it only remains to show that~$J'=J$. To see this, note that
by~(2), $\cZ(\sigmabar_J)=\cZbar(J')$ has a dense open substack whose
finite type points have~$\sigmabar_J$ as their unique non-Steinberg Serre weight
(namely the complement of the union of the~$\cZ(\sigmabar')$ for
all~$\sigmabar'\ne\sigmabar_J$). By Theorem~\ref{thm: unique serre
  weight}, it also has a dense open substack whose finite type points
have~$\sigmabar_{J'}$ as a Serre weight. Considering any finite type
point in the intersection of these dense open substacks, we see
that~$\sigmabar_J=\sigmabar_{J'}$, so that~$J=J'$, as required.
\end{proof}

\renewcommand{\theequation}{\Alph{section}.\arabic{subsection}} 

\appendix

\section{Formal algebraic stacks}
\label{sec:formal stacks}
In this appendix we briefly recall some basic definitions and 
facts concerning formal algebraic 
stacks; our primary reference is \cite{Emertonformalstacks}.  We then
develop some simple geometric lemmas which will be applied in the
main body of the paper.

We first recall the definition of a formal algebraic stack
\cite[Def.\ 5.3]{Emertonformalstacks}.

\begin{adf}
	\label{df:formal stack}
An {\em fppf} stack in groupoids $\cX$ over a scheme $S$ is called
a {\em formal algebraic stack} if
there is a morphism $U \to \cX$,
	whose domain $U$ is a formal algebraic space over $S$
	(in the sense 
of~\cite[\href{http://stacks.math.columbia.edu/tag/0AIL}{Tag 0AIL}]{stacks-project}),
and which is representable by algebraic spaces,
smooth, and surjective.
\end{adf}

We will be primarily interested in the case when $S = \Spec \cO$, where,
as in the main body of the paper, $\cO$ is the ring of integers
in a finite extension $E$ of~$\Q_p$.   We let $\varpi$ denote
a uniformiser of $\cO$, and let $\Spf \cO$ denote the affine formal
scheme (or affine formal algebraic space, in the terminology
of~\cite{stacks-project}) obtained by $\varpi$-adically completing
$\Spec \cO$. 

Among all the formal algebraic stacks over $\Spec \cO$,
we single out the $\varpi$-adic formal algebraic stacks as being
of particular interest. The following definition is a particular case of~\cite[Def.~7.6]{Emertonformalstacks}.

\begin{adf}\label{defn: pi adic formal alg stack}
	A formal algebraic stack $\cX$ over $\Spec \cO$
	is called $\varpi$-adic if the canonical map
	$\cX \to \Spec \cO$ factors 
	through $\Spf \cO$,
	and if the induced map $\cX \to \Spf \cO$
	is algebraic, i.e.\ representable by algebraic
	stacks (in the sense 
of~\cite[\href{http://stacks.math.columbia.edu/tag/06CF}{Tag 06CF}]{stacks-project} and \cite[Def.~3.1]{Emertonformalstacks}).
\end{adf}

We refer to \cite{Emertonformalstacks} for the various other notions
related to formal algebraic stacks that we employ.
We also recall the following key lemma, which allows us
to recognise certain Ind-algebraic stacks as being formal algebraic 
stacks \cite[Lem.~6.3]{Emertonformalstacks}.

\begin{alemma}
	\label{lem:ind to formal}
	If $\cX_1 \hookrightarrow \cX_2 \hookrightarrow \cdots \hookrightarrow
	\cX_n \hookrightarrow \cdots$ is a sequence of finite order
	thickenings of
	algebraic stacks,
then $\varinjlim_n \cX_n$ is a formal algebraic stack.
\end{alemma}

    \subsection{Open and closed substacks}
    Suppose first that~$\cX$ is an algebraic stack over
    some scheme $S$.   An open substack of
    $\cX$ is then, by definition
    (\cite[\href{http://stacks.math.columbia.edu/tag/04YM}{Tag
      04YM}]{stacks-project}), a strictly full substack~$\cX'$ such
    that the natural morphism $\cX'\to\cX$ is an open immersion; that
    is, it is representable by algebraic spaces, and an open immersion
    on all pullbacks to algebraic spaces. A closed substack is defined
    in the analogous way. The stack $\cX$ has an underlying
    topological space $|\cX|$, and the open subsets of $|\cX|$ are in
    natural bijection with the open substacks of~$\cX$
    by~\cite[\href{http://stacks.math.columbia.edu/tag/06FJ}{Tag
      06FJ}]{stacks-project}. 

We now note that
the preceding definitions of open and closed substacks in fact apply perfectly
well to a stack over $S$ which is not assumed to be
algebraic (see e.g.\ ~\cite[Def.~3.26]{Emertonformalstacks}). 
In particular, we can apply it in the case when $\cX$ is a formal algebraic
stack.

We make
an important
observation: if $\cX \hookrightarrow \cX'$ is a
morphism of stacks over $S$ which is representable by algebraic spaces
and is a thickening (in the sense of~\cite[\href{http://stacks.math.columbia.edu/tag/0BPN}{Tag 0BPN}]{stacks-project}), 
then pull-back under this morphism induces a bijection between open
substacks of $\cX'$ and open substacks of $\cX$.  (If $\cU \to \cX$
is an open immersion, and $T$ is an $S$-scheme,
we define \[\cU'(T):= \{ T \to \cX' \, | \, \text{ the base-changed morphism }
\cX \times_{\cX'} T \to \cX \text{ factors through } \cU\}.\] We leave
it to the reader to check that $\cU'$ is an open substack of $\cX'$,
and that $\cU \mapsto \cU'$ and $\cU' \mapsto \cX\times_{\cX'} \cU'$ are
mutually inverse; see also \cite[Lem.~3.41]{Emertonformalstacks},
where this result is established in the context of a more general
statement about the topological invariance of the \'etale site.)

Finally,
we say that an open substack $\cU$ of a formal algebraic stack $\cX$ is
\emph{dense} if its underlying topological space $|\cU|$ is dense in
$|\cX|$. Note that this need not imply that it is scheme-theoretically
dense (as is already the case if $\cX$ is a non-reduced scheme).

\subsection{A geometric situation}
\label{subsec:situation}
	 We suppose given a commutative diagram
	 of morphisms of formal algebraic stacks
	 $$\xymatrix{\cX \ar[r] \ar[dr] & \cY \ar[d] \\ & \Spf \cO}$$
         We suppose that each of $\cX$ and $\cY$ is quasi-compact
	 and quasi-separated,
	 and that the horizontal arrow is scheme-theoretically
	 dominant, in the sense of \cite[Def.~6.13]{Emertonformalstacks}.
	 We furthermore suppose that the morphism $\cX\to\Spf\cO$ realises
	 $\cX$ as a finite type $\varpi$-adic formal algebraic stack.

	 Concretely, 
	 if we write $\cX^a := \cX\times_{\cO} \cO/\varpi^a,$
	 then each $\cX^a$ is an algebraic stack,
	 locally of finite type over $\Spec \cO/\varpi^a$, 
	 and there is an isomorphism $\varinjlim_a \cX^a \iso \cX.$
	 Furthermore, the assumption that the horizontal arrow
	 is scheme-theoretically dominant means that we may
	 find an isomorphism $\cY \cong \varinjlim_a \cY^a$,
	 with each $\cY^a$ being a quasi-compact and quasi-separated
	 algebraic stack, and with the transition morphisms being
	 thickenings, such that the morphism $\cX \to \cY$
	 is induced by a compatible family of morphisms
	 $\cX^a \to \cY^a$, each of which is scheme-theoretically
	 dominant. (The~$\cY^a$ are uniquely determined by the
         requirement that for all~$b\ge a$ large enough so that the
         morphism $\cX^a\to\cY$ factors
         through~$\cY\otimes_\cO\cO/\varpi^b$, $\cY^a$ is the
         scheme-theoretic image of the morphism
         $\cX^a\to\cY\otimes_\cO\cO/\varpi^b$. In particular, $\cY^a$
         is a closed substack of $\cY\times_\cO \cO/\varpi^a$.) 

	 It is often the case, in the preceding situation,
	 that $\cY$ is also a $\varpi$-adic formal algebraic stack.
	 For example, we have the following result.
	 (Note that the usual graph argument
	 shows that the morphism $\cX\to \cY$ is necessarily
		 algebraic, i.e.\ representable by algebraic 
	stacks, in the sense 
of~\cite[\href{http://stacks.math.columbia.edu/tag/06CF}{Tag 06CF}]{stacks-project} and \cite[Def.~3.1]{Emertonformalstacks}. Thus it makes sense to 
speak of it being proper, following \cite[Def.~3.11]{Emertonformalstacks}.)

	 \begin{aprop}
		 \label{prop:Y is p-adic formal}
		 Suppose that the morphism $\cX \to \cY$ is proper,
and that $\cY$ is locally Ind-finite type over $\Spec \cO$
{\em (}in the sense of {\em \cite[Rem.~8.30]{Emertonformalstacks})}.
Then $\cY$ is a $\varpi$-adic formal algebraic stack.
	 \end{aprop}
	 \begin{proof}
		 This is an application of \cite[Prop.~10.5]{Emertonformalstacks}.
	 \end{proof}
	
A key point is that, because the formation of scheme-theoretic images
is not generally compatible with non-flat base-change,
the closed immersion
\anumequation
\label{eqn:closed immersion}
\cY^a \hookrightarrow \cY \times_{\cO} \cO/\varpi^a
\end{equation}
is typically {\em not} an isomorphism, even if $\cY$ is a $\varpi$-adic
formal algebraic stack.
Our goal in the remainder of this discussion is to give a criterion
(involving the morphism $\cX\to\cY$)
on an open substack $\cU \hookrightarrow \cY$ which guarantees
that the closed immersion
$\cU\times_{\cY}\cY^a \hookrightarrow \cU\times_{\cO}\cO/\varpi^a$
induced by~(\ref{eqn:closed immersion})
{\em is} an isomorphism.

We begin by establishing a simple lemma.
         For any $a \geq 1$,
	 we have the $2$-commutative diagram
	 \anumequation
	 \label{eqn:base-change}
	 \xymatrix{
	 \cX^a \ar[r] \ar[d] & \cY^a \ar[d] \\
	 \cX \ar[r] & \cY 
 }
         \end{equation}
	 Similarly, if $b \geq a \geq 1$,
	 then we have the $2$-commutative diagram
	 \anumequation
	 \label{eqn:base-change bis}
	 \xymatrix{
	 \cX^a \ar[r] \ar[d] & \cY^a \ar[d] \\
	 \cX^b \ar[r] & \cY^b 
 }
         \end{equation}

	 \begin{alemma}
		 \label{lem:cartesian diagrams}
		 Each of the diagrams~{\em (\ref{eqn:base-change})}
		 and~{\em (\ref{eqn:base-change bis})} 
		 is $2$-Cartesian.
	 \end{alemma}
	 \begin{proof}
		We may embed the diagram~(\ref{eqn:base-change})
	       in the larger $2$-commutative diagram
       $$\xymatrix{\cX^a \ar[r] \ar[d] & \cY^a \ar^-{\text{(\ref{eqn:closed
				       immersion})}}[r] \ar[d] &\cY\otimes_{\cO}
	       \cO/\varpi^a \ar[d] \\ \cX \ar[r] &\cY \ar@{=}[r]& \cY}$$
       Since the outer rectangle is manifestly $2$-Cartesian,
       and since~(\ref{eqn:closed immersion}) is a closed immersion (and thus
       a monomorphism), we conclude that~(\ref{eqn:base-change}) is indeed
       $2$-Cartesian.  

       A similar argument shows that~(\ref{eqn:base-change bis}) 
       is $2$-Cartesian.
	 \end{proof}

         We next note that,
	 since each of the closed immersions $\cY^a \hookrightarrow \cY$
	 is a thickening, giving an open substack $\cU \hookrightarrow \cY$
	 is equivalent to giving an open substack $\cU^a \hookrightarrow \cY^a$
	 for some, or equivalently, every, choice of $ a\geq 1$;
	 the two pieces of data are related by the formulas
	 $\cU^a := \cU\times_{\cY} \cY^a$
	 and $\varinjlim_a \cU^a \iso \cU.$

	 \begin{aprop}
		 \label{prop:opens}
		 Suppose that $\cX \to \cY$ is proper.
		 If $\cU$ is an open substack of~$\cY$,
		 then the following conditions are equivalent:
		 \begin{enumerate}
			 \item
				 The morphism $\cX\times_{\cY} \cU \to \cU$
				 is a monomorphism.
			 \item
				 The morphism $\cX\times_{\cY} \cU \to \cU$
				 is an isomorphism.
			 \item   For every $a \geq 1$,
				 the morphism $\cX^a\times_{\cY^a} \cU^a \to
				 \cU^a$
				 is a monomorphism.
			 \item   For every $a \geq 1$,
				 the morphism $\cX^a\times_{\cY^a} \cU^a \to
				 \cU^a$
				 is an isomorphism.
			 \item   For some $a \geq 1$,
				 the morphism $\cX^a\times_{\cY^a} \cU^a \to
				 \cU^a$
				 is a monomorphism.
			 \item   For some $a \geq 1$,
				 the morphism $\cX^a\times_{\cY^a} \cU^a \to
				 \cU^a$
				 is an isomorphism.
			 \end{enumerate}
			 Furthermore, if these equivalent conditions hold,
			 then the closed immersion  $\cU^a \hookrightarrow
			 \cU\times_{\cO} \cO/\varpi^a$ is an isomorphism,
			 for each $a \geq 1$.
	 \end{aprop}
	 \begin{proof}
		 The key point is that 
		 Lemma~\ref{lem:cartesian diagrams} implies that the diagram
		 $$\xymatrix{\cX^a\times_{\cY^a}\cU^a \ar[r] \ar[d] & \cU^a \ar[d] \\
			 \cX\times_{\cY} \cU \ar[r] & \cU }
		 $$
		 is $2$-Cartesian, for any $a \geq 1$,
		 and similarly, that if $b \geq a \geq 1,$ then the diagram
		 $$\xymatrix{\cX^a \times_{\cY^a}\cU^a\ar[r] \ar[d] & \cU^a \ar[d] \\
			 \cX^b\times_{\cY^b} \cU^b \ar[r] & \cU^b }
		 $$
		 is $2$-Cartesian.
		 Since the vertical arrows of this latter diagram
		 are finite order thickenings,
		 we find (by applying the analogue
		 of~\cite[\href{http://stacks.math.columbia.edu/tag/09ZZ}{Tag 09ZZ}]{stacks-project}
		 for algebraic stacks, whose straightforward deduction
		 from that result we leave to the reader)
		 that the top horizontal arrow
		 is a monomorphism
		 if and only if the bottom horizontal arrow is.
		 This shows the equivalence of~(3) and~(5).
		 Since the morphism $\cX\times_{\cY} \cU \to \cU$
		 is obtained as the inductive limit
		 of the various morphisms $\cX^a\times_{\cY^a} \cU^a
		 \to \cU^a,$
		 we find that~(3) implies~(1) (by applying e.g.\
		 \cite[Lem.~4.11~(1)]{Emertonformalstacks}, which
		 shows that the inductive limit of monomorphisms
		 is a monomorphism),
		 and also that~(4) implies~(2) (the inductive
		 limit of isomorphisms being again an isomorphism).

		 Conversely, if~(1) holds, then the base-changed
                 morphism 
		 $$\cX\times_{\cY} (\cU\times_{\cO} \cO/\varpi^a) \to \cU
		 \times_{\cO} \cO/\varpi^a$$
		 is a monomorphism. 
		 The source of this morphism admits an alternative description
		 as $\cX^a \times_{\cY} \cU,$ which the
		 $2$-Cartesian diagram
		 at the beginning of the proof allows us to identify
		 with 
		 $\cX^a\times_{\cY^a} \cU^a$. 
		 Thus we obtain a monomorphism
		 $$\cX^a\times_{\cY^a} \cU^a \hookrightarrow \cU\times_{\cO}
		 \cO_/\varpi^a.$$

		 Since this monomorphism factors through
		 the closed immersion $\cU^a \hookrightarrow \cU
		 \times_{\cO} \cO/\varpi^a,$
		 we find that each of the morphisms of~(3) is a monomorphism;
		 thus~(1) implies~(3).  Similarly, (2)~implies~(4),
		 and also implies 
		 that the closed immersion
		 $\cU^a \hookrightarrow \cU\times_{\cO} \cO/\varpi^a$
		 is an isomorphism, for each $a \geq 1$.

		 Since clearly (4) implies~(6),
		 while (6) implies~(5), to complete the proof
		 of the proposition, it suffices to show that
		 (5) implies~(6).
		 Suppose then that
		 $\cX^a\times_{\cY^a} \cU^a \to \cU^a$
		 is a monomorphism.  Since $\cU^a \hookrightarrow \cY^a$
		 is an open immersion, it is in particular flat.  Since
		 $\cX^a \to \cY^a$ is scheme-theoretically dominant and
		 quasi-compact (being proper), any flat base-change of this
		 morphism is again scheme-theoretically dominant,
		 as well as being proper.
		 Thus we see that $\cX^a\times_{\cY^a} \cU^a \to \cU^a$ is a
		 scheme-theoretically dominant proper monomorphism,
		 i.e.\ a scheme-theoretically dominant closed immersion,
		 i.e.\ an isomorphism, as required.
	 \end{proof}

\section{Serre weights and tame types}\label{sec: appendix on tame
  types} 
We begin by recalling some results from~\cite{MR2392355} on the Jordan--H\"older factors
of the reductions modulo $p$ of lattices in principal series
and cuspidal representations of $\GL_2(k)$,
following~\cite[\S3]{emertongeesavitt} (but with slightly different
normalisations than those of \emph{loc.\ cit.}).

Let $\tau$ be a tame inertial type. Recall from Section~\ref{subsec:
  notation} that we associate a representation~$\sigma(\tau)$ of
$\GL_2(\cO_K)$ to~$\tau$ as follows: if $\tau
\simeq \eta \oplus \eta'$ is a tame principal series type, then we set
$\sigma(\tau) := \Ind_I^{\GL_2(\cO_K)} \eta'\otimes \eta$, while 
if $\tau=\eta\oplus\eta^q$ is a tame cuspidal type, then $\sigma(\tau)$ is the
inflation to $\GL_2(\cO_K)$ of the cuspidal representation of $\GL_2(k)$
denoted by~$\Theta(\eta)$ in~\cite{MR2392355}. (Here we have
identified~$\eta,\eta'$ with their composites with~$\Art_K$.)

 Write $\sigmabar(\tau)$ for the
semisimplification of the reduction modulo~$p$ of (a
$\GL_2(\cO_K)$-stable $\cO$-lattice in) $\sigma(\tau)$. The action
of~$\GL_2(\cO_K)$ on~$\sigmabar(\tau)$ factors through~$\GL_2(k)$, so
the Jordan--H\"older factors~$\JH(\sigmabar(\tau))$ of~$\sigmabar(\tau)$ are Serre weights.
By the results of~\cite{MR2392355}, these Jordan--H\"older factors of
$\sigmabar(\tau)$ are pairwise non-isomorphic, and are parametrised
by the set $\cP_\tau$ (defined in Section~\ref{sec:sets-cp_tau})
in a fashion that we now recall.

Suppose first that $\tau=\eta\oplus\eta'$ is a tame principal series 
type. Set $f'=f$ in this case.
We define $0 \le \gamma_i \le p-1$ (for $i \in \Z/f\Z$) 
to be the unique integers not all equal to $p-1$ such that 
$\eta (\eta')^{-1} = \prod_{i=0}^{f-1}
\omega_{\sigma_i}^{\gamma_i}$. If instead $\tau = \eta \oplus \eta'$
is a cuspidal type, set $f'=2f$. We define $0 \le \gamma_i \le p-1$ (for $i \in \Z/f'\Z$) 
to be the unique integers such that 
$\eta (\eta')^{-1} = \prod_{i=0}^{f'-1}
\omega_{\sigma'_i}^{\gamma_i}$. Here the $\sigma'_i$ are the embeddings
$l \to \F$, where $l$ is the quadratic extension of $k$, 
$\sigma'_0$ is a fixed choice of embedding extending $\sigma_0$, and
$(\sigma'_{i+1})^p = \sigma'_i$ for all $i$.

If~$\tau$ is scalar then
we set $\cP_\tau=\{\varnothing\}$. 
Otherwise we have $\eta\ne\eta'$, and as in
Section~\ref{sec:sets-cp_tau} we let
$\cP_{\tau}$ be the collection of subsets $J \subset \Z/f'\Z$ 
satisfying the conditions:
\begin{itemize}
\item if $i-1\in J$ and $i\notin J$ then $\gamma_{i}\ne p-1$, and
\item if $i-1\notin J$ and $i\in J$ then $\gamma_{i}\ne 0$
\end{itemize}
and, in the cuspidal case, satisfying the further condition that $i
\in J$ if and only if $i+f \not\in J$.

The Jordan--H\"older factors of $\sigmabar(\tau)$ are by definition
Serre weights, and are
parametrised by $\cP_{\tau}$ as follows (see~\cite[\S3.2, 3.3]{emertongeesavitt}). For any $J\subseteq \Z/f'\Z$, we let $\delta_J$ denote
the characteristic function of $J$, and if $J \in \cP_{\tau}$
we define $s_{J,i}$ by
\[s_{J,i}=\begin{cases} p-1-\gamma_{i}-\delta_{J^c}(i)&\text{if }i-1 \in J \\
  \gamma_{i}-\delta_J(i)&\text{if }i-1\notin J, \end{cases}\]
and we set $t_{J,i}=\gamma_{i}+\delta_{J^c}(i)$ if $i-1\in J$ and $0$
otherwise. 

In the principal series case we let
$\sigmabar(\tau)_J
:=\sigmabar_{\vec{t},\vec{s}}\otimes\eta'\circ\det$;
the $\sigmabar(\tau)_J$ are precisely the Jordan--H\"older factors of
$\sigmabar(\tau)$. 

In the cuspidal case, one checks that $s_{J,i} = s_{J,i+f}$ for all
$i$, and also that  the character $\eta' \cdot
\prod_{i=0}^{f'-1} (\sigma'_i)^{t_{J,i}} : l^{\times} \to
\F^{\times}$ factors as $\theta \circ N_{l/k}$ where $N_{l/k}$
is the norm map. We let $\sigmabar(\tau)_J
:=\sigmabar_{0,\vec{s}}\otimes\theta \circ\det$;
the $\sigmabar(\tau)_J$ are precisely the Jordan--H\"older factors of
$\sigmabar(\tau)$. 

\begin{aremark}\label{arem: wtf were we thinking in EGS}
  The parameterisations above are easily deduced from those given in
  \cite[\S3.2, 3.3]{emertongeesavitt} for the Jordan--H\"older factors
  of the representations $\Ind_I^{\GL_2(\cO_K)} \eta'\otimes \eta$
  and~$\Theta(\eta)$. (Note that there is a minor mistake
  in~\cite[\S3.1]{emertongeesavitt}: since the conventions
  of~\cite{emertongeesavitt} regarding the inertial Langlands
  correspondence agree with those of~\cite{geekisin}, the explicit
  identification of~$\sigma(\tau)$ with a principal series or cuspidal
  type in ~\cite[\S3.1]{emertongeesavitt} is missing a dual. The
  explicit parameterisation we are using here is of course independent
  of this issue.

  This mistake has the unfortunate effect that various
  explicit formulae in~\cite[\S7]{emertongeesavitt} need to be
  modified in a more or less obvious fashion; note that since
  ~$\sigma(\tau)$ is self dual up to twist, all formulae can be fixed
  by making twists and/or exchanging~$\eta$ and~$\eta'$. In
  particular, the definition of the strongly divisible module
  before~\cite[Rem.\ 7.3.2]{emertongeesavitt} is incorrect as written,
  and can be fixed by either reversing the roles of~$\eta,\eta'$ or
  changing the definition of the quantity~$c^{(j)}$ defined there.)
\end{aremark}

\begin{aremark}
In the cuspidal case, write $\eta$ in the form 
$(\sigma'_0)^{(q+1)b+1+c}$ where $0\le b\le q-2$, $0\le c\le 
q-1$. 
Set $t'_{J,i} = t_{J,i+f}$ for integers $1\le i \le f$. Then one can check
that $\sigmabar(\tau)_J = \sigmabar_{\vec{t}',\vec{s}} \otimes
(\sigma_0^{(q+1)b + \delta_J(0)} \circ \det).$
\end{aremark}

We now recall some facts about the set of Serre weights~$W(\rbar)$
associated to a representation $\rbar:G_K\to\GL_2(\Fpbar)$. 

\begin{adefn}\label{def:de rham of type}
 We say that a crystalline representation $r : G_K \to
  \GL_2(\Qpbar)$ has \emph{type $\sigmabar_{\vec{t},\vec{s}}$}
  provided that for each embedding $\sigma_j : k \into \F$ there is an
  embedding $\widetilde{\sigma}_j : K \into \Qpbar$ lifting $\sigma_j$
  such that the $\widetilde{\sigma}_j$-labeled Hodge--Tate weights
  of~$r$ are $\{-s_j-t_j,1-t_j\}$, and the remaining $(e-1)f$ pairs of  Hodge--Tate weights
  of~$r$ are all $\{0,1\}$. \emph{(}In particular the representations of
   type $\sigmabar_{\vec{0},\vec{0}}$ \emph{(}the trivial weight\emph{)} are the same as those of Hodge type $0$.\emph{)}
\end{adefn}

\begin{adefn}\label{def:serre weights}
Given a representation $\rbar:G_K\to\GL_2(\Fpbar)$ we define $W(\rbar)$
to be the set of Serre weights $\sigmabar$ such that $\rbar$ has a
crystalline lift of type $\sigmabar$.
\end{adefn}

It follows easily from the formula
$\sigmabar_{\vec{t},\vec{s}}^\vee=\sigmabar_{-\vec{s}-\vec{t},\vec{s}}$ 
that $\sigmabar \in W(\rbar)$ if and only if $\sigmabar^\vee$ is in
the set of Serre weights associated to $\rbar^\vee$  
in~\cite[Defn.\ 4.1.3]{gls13}.  
  
There are several definitions of the set $W(\rbar)$ in the literature,
which by the papers~\cite{blggu2,geekisin,gls13} are known to be
equivalent (up to normalisation).  While the preceding definition 
is perhaps the most compact, it is the description of $W(\rbar)$
via the Breuil--M\'ezard conjecture that appears to be the most
amenable to generalisation;\ see Theorem~\ref{thm: geometric BM} for
an instance of this description, and \cite{MR3871496} for much more discussion.

 Recall that~$\rbar$ is~\emph{tr\`es ramifi\'ee} if it is a
    twist of an extension of the trivial character by the mod~$p$
    cyclotomic character, and if furthermore the splitting field of its projective
    image is \emph{not} of the form
    $K(\alpha_1^{1/p},\dots,\alpha_s^{1/p})$ for some
    $\alpha_1,\dots,\alpha_s\in\cO_K^\times$. 
\begin{alemma}\label{lem: list of things we need to know about Serre weights}
  \begin{enumerate}
  \item If ~$\tau$ is a tame type, then $\rbar$ has a potentially
    Barsotti--Tate lift of type~$\tau$ if and only
    if $W(\rbar)\cap\JH(\sigmabar(\tau))\ne 0$. 
  \item The following conditions are equivalent:
    \begin{enumerate}
    \item $\rbar$ admits a potentially Barsotti--Tate lift of some tame type.
    \item $W(\rbar)$ contains a non-Steinberg Serre weight.
    \item $\rbar$ is not tr\`es    ramifi\'ee.
    \end{enumerate}
  \end{enumerate}
\end{alemma}
\begin{proof}
(1)   By the main result of~\cite{gls13}, and bearing in mind the
differences between our conventions and those of~\cite{geekisin} as
recalled in Section~\ref{subsec: notation}, we have
$\sigmabar\in W(\rbar)$ if and only if $\sigmabar^\vee\in \WBT(\rbar)$,
where $\WBT(\rbar)$ is the set of weights defined
  in~\cite[\S3]{geekisin}. 
  By~\cite[Cor.\ 3.5.6]{geekisin} (bearing in mind once again the
  differences between our conventions and those of~\cite{geekisin}),
  it follows that we have $W(\rbar)\cap\JH(\sigmabar(\tau))\ne 0$ if
  and only if $e(R^{\square,0,\tau}/\pi)\ne 0$ in the notation of
  \emph{loc.\ cit.},
  and 
 by definition $\rbar$ has a potentially Barsotti--Tate lift of
  type~$\tau$ if and only if $R^{\square,0,\tau}\ne 0$. It follows
  from~\cite[Prop.\ 4.1.2]{emertongeerefinedBM} 
that
  $R^{\square,0,\tau}\ne 0$ if and only if
  $e(R^{\square,0,\tau}/\pi)\ne 0$, as required.

  (2) By part~(1), condition~(a) is equivalent to $W(\rbar)$
  containing a Serre weight occurring as a Jordan--H\"older factor of
  ~$\sigmabar(\tau)$ for some tame type~$\tau$. It is easily seen
  (either by inspection, or by Lemma~\ref{lem:write Serre weight as
    linear combination of types} below) that the Serre weights
  occurring as Jordan--H\"older factors of the
  ~$\sigmabar(\tau)$ are precisely the non-Steinberg Serre weights,
  so~(a) and~(b) are equivalent.

  Suppose that~(a) holds; then~$\rbar$ becomes finite flat over a tame
  extension. However the restriction to a tame extension of a tr\`es
  ramifi\'ee representation is still tr\`es
  ramifi\'ee, and therefore not finite flat, so~(c) also
  holds. Conversely, suppose for the sake of contradiction that~(c)
  holds, but that~(b) does not hold, i.e.\ that $W(\rbar)$ consists of
  a single Steinberg weight.

Twisting, we
can without loss of generality assume that~$W(\rbar)=\{\sigmabar_{\vec{0},\vec{p-1}}\}$. By~\cite[Cor.\ A.5]{geekisin} we
can globalise~$\rbar$, and then the hypothesis that~$W(\rbar)$
contains $\sigmabar_{\vec{0},\vec{p-1}}$ implies that it has a semistable lift of
Hodge type~$0$. If this lift were in fact crystalline, then~$W(\rbar)$
would also contain the weight $\sigmabar_{\vec{0},\vec{0}}$ by~(1). So
this lift is not crystalline, and in particular the monodromy
operator~$N$ on the corresponding weakly admissible module is
nonzero. But then $\ker(N)$ is a free filtered submodule of rank $1$,
and since the lift has Hodge type~$0$, $\ker(N)$ is in fact a weakly
admissible submodule. It follows that the lift is an unramified twist of an
extension of $\varepsilon^{-1}$ by the  trivial
character, so that~$\rbar$ is an unramified twist of an extension of
$\varepsilonbar^{-1}$ by the trivial character. But we are assuming
that~(c) holds, so~$\rbar$ is finite flat, so that by~(1), $W(\rbar)$
contains the weight $\sigmabar_{\vec{0},\vec{0}}$, a contradiction.
\end{proof}

\begin{alem}
  \label{lem: explicit Serre weights with our normalisations} 
  Suppose that~$\sigmabar_{\vec{t},\vec{s}}$ is a non-Steinberg
  Serre weight. Suppose that~$\rbar:G_K\to\GL_2(\Fpbar)$ is a reducible
  representation satisfying \[\rbar|_{I_K}\cong
    \begin{pmatrix}
      \prod_{i=0}^{f-1}\omega_{\sigma_i}^{s_i+t_i} &*\\ 0 & \varepsilonbar^{-1}\prod_{i=0}^{f-1}\omega_{\sigma_i}^{t_i}
    \end{pmatrix}
    ,\] and that $\rbar$ is not tr\`es ramifi\'ee.
Then $\sigmabar_{\vec{t},\vec{s}}\in W(\rbar)$.
  \end{alem}
  \begin{proof} 
 Write $\rbar$ as an extension of characters $\overline{\chi}$ by $\overline{\chi}'$.
 It is straightforward from the classification of crystalline
 characters as in \cite[Lem.\ 5.1.6]{MR3871496} that there exist 
 crystalline lifts $\chi,\chi'$ of
 $\overline{\chi},\overline{\chi}'$ so that $\chi,\chi'$ have Hodge--Tate
 weights $1-t_j$ and $-s_j-t_j$ respectively at one embedding lifting each $\sigma_j$ and
 Hodge--Tate weights $1$ and $0$ respectively at the others. In the
 case that $\rbar$ is not the twist of an extension of
 $\varepsilonbar^{-1}$ by $1$ the result follows because the corresponding
    $H^1_f(G_K,\chi'\otimes\chi^{-1})$ agrees with the
    full~$H^1(G_K,\chi'\otimes\chi^{-1})$ (as a consequence of the
    usual dimension formulas for~$H^1_f$, \cite[Prop.\
    1.24]{nekovar}).  

If $\rbar$ is twist
of an extension of $\varepsilonbar^{-1}$ by $1$, the assumption that 
 $\sigmabar_{\vec{t},\vec{s}}$ is non-Steinberg implies $s_j = 0$
    for all $j$. The hypothesis that
$\rbar$ is not tr\`es ramifi\'ee guarantees that 
    $\rbar\otimes\prod_{i=0}^{f-1}\omega_{\sigma_i}^{-t_i} $ is finite
    flat, so has a Barsotti--Tate lift, and we deduce  that
    $\sigmabar_{\vec{t},\vec{0}} \in W(\rbar)$. 
\end{proof}

\section{The geometric Breuil--M\'ezard Conjecture for potentially
  Barsotti--Tate representations}\label{sec:appendix on geom BM} In this
appendix, by combining the methods of~\cite{emertongeerefinedBM}
and~\cite{geekisin} we prove a special case of the geometric Breuil--M\'ezard
conjecture~\cite[Conj.\ 4.2.1]{emertongeerefinedBM}. This result is
``globalised'' in Section~\ref{sec: picture}.

Let $K/\Qp$ be a finite extension, and let $E/\Qp$ be another finite extension,
with ring of integers~$\cO$, uniformiser~$\varpi$, 
and residue field~$\F$. We assume that $E$ is
sufficiently large, and in particular that~$E$ contains~$K$. Let
$\rbar:G_K\to\GL_2(\F)$ be a continuous representation, and let
$R^\square_{\rbar}$ be the universal framed deformation $\cO$-algebra for~$\rbar$. For each tame
type~$\tau$, let $R^\square_{\rbar,0,\tau}$ be the reduced and $p$-torsion
free 
quotient of $R^\square_{\rbar}$ whose $\Qpbar$-points correspond
to the potentially Barsotti--Tate lifts of~$\rbar$ of
type~$\tau$. In Section~\ref{sec: picture} we denote this ring by
$R^{\tau,\BT}_{\rbar}$, but we use the more cumbersome
notation~$R^\square_{\rbar,0,\tau}$ here to make it easier for the
reader to refer to~\cite{emertongeerefinedBM} and~\cite{geekisin}. 

By \cite[Prop.\ 4.1.2]{emertongeerefinedBM}, $R^\square_{\rbar,0,\tau}/\varpi$ is zero
if $\rbar$ has no potentially Barsotti--Tate lifts of type~$\tau$, and otherwise
it is equidimensional of dimension $4+[K:\Qp]$. Each $\Spec
R^\square_{\rbar,0,\tau}/\varpi$ is a closed subscheme of $\Spec
R^\square_{\rbar}/\varpi$, and we write $Z(R^\square_{\rbar,0,\tau}/\varpi)$ for the
corresponding cycle, as in~\cite[Defn.\ 2.2.5]{emertongeerefinedBM}. (This
is a formal sum of the irreducible components of $\Spec
R^\square_{\rbar,0,\tau}/\varpi$, weighted by the multiplicities with which they occur.)

\begin{alem}
  \label{lem:write Serre weight as linear combination of types}If $\sigmabar$ is a
non-Steinberg Serre weight of $\GL_2(k)$, then there are
integers $n_\tau(\sigmabar)$ such that
$\sigmabar=\sum_\tau n_\tau(\sigmabar)\sigmabar(\tau)$ in the Grothendieck group
of mod $p$ representations of $\GL_2(k)$, where the $\tau$ run over the tame types.
\end{alem}
\begin{proof}This is an immediate consequence of the surjectivity of the natural
  map from the Grothendieck group of $\Qpbar$-representations of $\GL_2(k)$ to
  the Grothendieck group of $\Fpbar$-representations of $\GL_2(k)$ \cite[\S III,
  Thm.\ 33]{MR0450380}, together with the observation that the reduction of the
  Steinberg representation of $\GL_2(k)$ is precisely $\sigmabar_{\vec{0},\vec{p-1}}$.
  \end{proof}
Let $\sigmabar$ be a non-Steinberg Serre weight of $\GL_2(k)$, so that by
Lemma~\ref{lem:write Serre weight as linear combination of types} we can
write \anumequation\label{eqn: defn of n_tau(sigmabar)}\sigmabar=\sum_\tau n_\tau(\sigmabar)\sigmabar(\tau)\end{equation}  in the Grothendieck group
of mod $p$ representations of $\GL_2(k)$. Note that the integers $n_\tau(\sigmabar)$ are not uniquely determined; however,
all our constructions elsewhere in this paper will be (non-obviously!) independent of the choice of the $n_\tau(\sigmabar)$. We also
write \[\sigmabar(\tau)=\sum_{\sigmabar}m_{\sigmabar}(\tau)\sigmabar;\] since
$\sigmabar(\tau)$ is multiplicity-free, each $m_{\sigmabar}(\tau)$ is equal to
$0$ or $1$. Then\[\sigmabar=\sum_{\sigmabar'}\left(\sum_\tau
  n_\tau(\sigmabar)m_{\sigmabar'}(\tau)\right)\sigmabar',\]and
therefore\anumequation\label{eqn: orthogonality of sigmabar tau}\sum_\tau
  n_\tau(\sigmabar)m_{\sigmabar'}(\tau)=\delta_{\sigmabar,\sigmabar'}.\end{equation}

For each non-Steinberg Serre weight $\sigmabar$, we
set \[\cC_{\sigmabar}:=\sum_\tau
n_\tau(\sigmabar)Z(R^\square_{\rbar,0,\tau}/\varpi),\]where the sum ranges over the
tame types $\tau$, and the integers $n_\tau(\sigmabar)$ are as in~(\ref{eqn:
  defn of n_tau(sigmabar)}). By definition this is a formal sum with (possibly
negative) multiplicities of irreducible
subschemes of~$\Spec
R^\square_{\rbar}/\varpi$; recall that we say that it is \emph{effective} if all of
the multiplicities are non-negative.
\begin{athm}
  \label{thm: geometric BM}Let $\sigmabar$ be a non-Steinberg Serre weight. Then
  the cycle $\cC_{\sigmabar}$ is effective, 
and is nonzero
  precisely when $\sigmabar\in W(\rbar)$. 
  It is
  independent of the choice of integers $n_\tau(\sigmabar)$ satisfying~{\em
    (\ref{eqn: defn of n_tau(sigmabar)})}. For
  each tame type $\tau$, we
  have \[Z(R^\square_{\rbar,0,\tau}/\varpi)=\sum_{\sigmabar\in\JH(\sigmabar(\tau))}\cC_{\sigmabar}.\] 
\end{athm}
\begin{proof}
  We will argue exactly as in the proof of~\cite[Thm.\
  5.5.2]{emertongeerefinedBM} (taking $n=2$), and we freely use the
  notation and definitions of~\cite{emertongeerefinedBM}. Since
  $p>2$, 
  we have $p\nmid n$ and thus a suitable globalisation $\rhobar$ exists
  provided that~\cite[Conj.\ A.3]{emertongeerefinedBM} holds for
  $\rbar$. Exactly as in the proof of~\cite[Thm.\ 5.5.4]{emertongeerefinedBM}, this follows from the proof of
  Theorem A.1.2 of~\cite{geekisin} (which shows that $\rbar$ has a
  potentially Barsotti--Tate lift) and Lemma 4.4.1 of \emph{op.cit}.\ 
  (which shows that any potentially Barsotti--Tate representation is
  potentially diagonalizable). These same results also show that the equivalent
  conditions of~\cite[Lem.\ 5.5.1]{emertongeerefinedBM} hold in the case that
  $\lambda_v=0$ for all $v$, and in particular in the case that $\lambda_v=0$
  and $\tau_v$ is tame for all $v$, which is all that we will require.

By~\cite[Lem.\ 5.5.1(5)]{emertongeerefinedBM}, we see that for each choice of
tame types $\tau_v$, we have 
\anumequation\label{eqn: first BM eqn}
Z(\overline{R}_\infty/\varpi)=\sum_{\otimes_{v|p}\sigmabar_v}\prod_{v|p}m_{\sigmabar_v}(\tau_v)Z'_{\otimes_{v|p}\sigmabar_v}(\rhobar).\end{equation}
Now, by definition we
have  \anumequation\label{eqn: second BM eqn}Z(\Rbarinfty/\varpi)=\prod_{v|p} Z(R^{\square}_{\rbar,0,\tau_v}/\varpi)\times Z
    (\F[[[x_1,\dots,x_{q-[F^+:\Q]n(n-1)/2},t_1,\dots,t_{n^2} ]]).\end{equation} Fix a
    non-Steinberg Serre weight $\sigmabar=\otimes_v\sigmabar_v$, and sum over
    all choices of types $\tau_v$, weighted by
    $\prod_{v|p}n_{\tau_v}(\sigmabar_v)$. We obtain \[\sum_\tau\prod_{v|p}n_{\tau_v}(\sigmabar_v)\prod_{v|p} Z(R^{\square}_{\rbar,0,\tau_v}/\varpi)\times Z
    (\F[[[x_1,\dots,x_{q-[F^+:\Q]n(n-1)/2},t_1,\dots,t_{n^2}
    ]])\]\[=\sum_\tau\prod_{v|p}n_{\tau_v}(\sigmabar_v)\sum_{\otimes_{v|p}\sigmabar'_v}\prod_{v|p}m_{\sigmabar'_v}(\tau_v)Z'_{\otimes_{v|p}\sigmabar'_v}(\rhobar) \]which
    by~(\ref{eqn: orthogonality of sigmabar tau}) simplifies to \anumequation\label{eqn: third BM eqn}\prod_{v|p}\cC_{\sigmabar_v}\times Z
    (\F[[[x_1,\dots,x_{q-[F^+:\Q]n(n-1)/2},t_1,\dots,t_{n^2}
    ]])=Z'_{\otimes_{v|p}\sigmabar_v}(\rhobar).\end{equation} Since
    $Z'_{\otimes_{v|p}\sigmabar_v}(\rhobar)$ is effective by definition (as it
    is defined as a positive multiple of the support cycle of a patched module),
    this shows that every $\prod_{v|p}\cC_{\sigmabar_v}$ is effective. We conclude
    that either every $\cC_{\sigmabar}$ is effective, or that every
    $-\cC_{\sigmabar}$ is effective.

Substituting~(\ref{eqn: third BM eqn}) and~(\ref{eqn: second BM eqn}) into~(\ref{eqn:
  first BM eqn}), we see that \anumequation\label{eqn: 4th BM eqn}\prod_{v|p} Z(R^{\square}_{\rbar,0,\tau_v}/\varpi)\times Z
    (\F[[[x_1,\dots,x_{q-[F^+:\Q]n(n-1)/2},t_1,\dots,t_{n^2} ]])\end{equation} \[=\prod_{v|p}\left(\sum_{\sigmabar\in\JH(\sigmabar(\tau))}\cC_{\sigmabar_v}\right)\times Z
    (\F[[[x_1,\dots,x_{q-[F^+:\Q]n(n-1)/2},t_1,\dots,t_{n^2}
    ]],\]and we deduce that either
    $Z(R^\square_{\rbar,0,\tau}/\varpi)=\sum_{\sigmabar}m_{\sigmabar}(\tau)\cC_{\sigmabar}$
    for all $\tau$, or
    $Z(R^\square_{\rbar,0,\tau}/\varpi)=-\sum_{\sigmabar}m_{\sigmabar}(\tau)\cC_{\sigmabar}$
    for all $\tau$.

Since each $Z(R^\square_{\rbar,0,\tau}/\varpi)$ is effective, the second
possibility holds if and only if every $-\cC_{\sigmabar}$ is effective (since
either all the $-\cC_{\sigmabar}$ are effective, or all the $\cC_{\sigmabar}$
are effective). It remains to show that this possibility leads to a
contradiction. Now, if
$Z(R^\square_{\rbar,0,\tau}/\varpi)=-\sum_{\sigmabar}m_{\sigmabar}(\tau)\cC_{\sigmabar}$
for all $\tau$, then substituting into the definition $\cC_{\sigmabar}=\sum_\tau
n_\tau(\sigmabar)Z(R^\square_{\rbar,0,\tau}/\varpi)$, we
obtain \[\cC_{\sigmabar}=\sum_{\sigmabar'}\left(\sum_{\tau}n_\tau(\sigmabar)m_{\sigmabar'}(\tau)\right)\left(-\cC_{\sigmabar'}\right),\]and
applying~(\ref{eqn: orthogonality of sigmabar tau}), we obtain
$\cC_{\sigmabar}=-\cC_{\sigmabar}$, so that $\cC_{\sigmabar}=0$ for all
$\sigmabar$.  Thus all the $\cC_{\sigmabar}$ are effective, as claimed.

Since $Z'_{\otimes_{v|p}\sigmabar_v}(\rhobar)$ by definition
depends only on (the global choices in the Taylor--Wiles method, and)
$\otimes_{v|p}\sigmabar_v$, and \emph{not} on the particular choice of the
$n_\tau(\sigmabar)$, it follows from~(\ref{eqn: third BM eqn})
that~$\cC_{\sigmabar}$ is also independent of this choice.

Finally, note that by definition
  $Z'_{\otimes_{v|p}\sigmabar_v}(\rhobar)$  is nonzero precisely when
  $\sigmabar_v$ is in the set $\WBT(\rbar)$ defined in~\cite[\S3]{geekisin}; but
  by the main result of~\cite{gls13}, this is precisely the set~$W(\rbar)$.
\end{proof}

\begin{rem}
  \label{rem: could have done BM for wildly ramified
    types}As we do not use wildly ramified types elsewhere in the
  paper, we have restricted the statement of Theorem~\ref{thm:
    geometric BM} to the case of tame types; but the statement
admits a natural extension
to the case
  of wildly ramified inertial types (with some components now occurring with multiplicity greater than
  one), and the proof goes through unchanged in this more general setting.
\end{rem}

\bibliographystyle{amsalpha}
\bibliography{dieudonnelattices}
\end{document}